\documentclass[onefignum,onetabnum]{siamonline190516} 

\usepackage{lipsum}
\usepackage{amsfonts}
\usepackage{graphicx}
\usepackage{epstopdf}
\usepackage{algorithmic}
\usepackage{dsfont}
\usepackage{enumitem}
\setlist[enumerate]{leftmargin=.5in}
\setlist[itemize]{leftmargin=.5in}

\usepackage{array}
\usepackage{arydshln}
\usepackage{multirow}

\usepackage{cleveref}
\usepackage{autonum}

\usepackage{pdflscape}
\usepackage{dsfont}
\usepackage[caption=false]{subfig} 
\ifpdf
  \DeclareGraphicsExtensions{.eps,.pdf,.png,.jpg}
\else
  \DeclareGraphicsExtensions{.eps}
\fi

\newsiamremark{remark}{Remark}
\newsiamremark{hypothesis}{Hypothesis}
\crefname{hypothesis}{Hypothesis}{Hypotheses}
\newsiamthm{claim}{Claim}

\theoremstyle{plain}

\numberwithin{assumption}{section}

\newtheorem{manualassumptioninner}{Assumption}
\newenvironment{manualassumption}[1]{%
  \renewcommand\themanualassumptioninner{#1}%
  \manualassumptioninner
}{\endmanualassumptioninner}

\newtheorem{manualtheoreminner}{Theorem}
\newenvironment{manualtheorem}[1]{%
  \renewcommand\themanualtheoreminner{#1}%
  \manualtheoreminner
}{\endmanualtheoreminner}

\makeatletter
\newcommand{\customlabel}[2]{%
   \protected@write \@auxout {}{\string \newlabel {#1}{{#2}{\thepage}{#2}{#1}{}} }%
   \hypertarget{#1}{}
}
\makeatother

\DeclareMathOperator*{\argsup}{arg\,sup}

\numberwithin{equation}{section}
\numberwithin{theorem}{section}

\headers{Parameter Estimation for the McKean-Vlasov SDE}{L. Sharrock, N. Kantas, P. Parpas, G. A. Pavliotis}

\title{Parameter Estimation for the McKean-Vlasov Stochastic Differential Equation \thanks{\color{header1} \textbf{Funding:} \color{black} The first author was funded by the EPSRC CDT in the Mathematics of Planet Earth (grant number EP/L016613/1) and the National Physical Laboratory. The second, third, and fourth authors were partially funded under a J.P. Morgan A.I. Research Award (2022). The fourth author was partially supported by the EPSRC (grant number EP/P031587/1).}
}

\author{L. Sharrock\thanks{Department of Mathematics, Imperial College London, South Kensington, London, SW7 2AZ, UK
  (\email{louis.sharrock16@imperial.ac.uk}, \email{n.kantas@imperial.ac.uk}, \email{g.pavliotis@imperial.ac.uk})}
\and N. Kantas\footnotemark[2]
\and P. Parpas\footnotemark[3]\thanks{Department of Computing, Imperial College London, South Kensington, London, SW7 2AZ, UK (\email{panos.parpas@imperial.ac.uk})}
\and G.A. Pavliotis\footnotemark[2]
 }

\usepackage{amsopn}

\ifpdf
\hypersetup{
  pdftitle={Online Parameter Estimation for the McKean-Vlasov Stochastic Differential Equation},
  pdfauthor={L. Sharrock, N. Kantas}
}
\fi

\begin{document}

\maketitle

\begin{abstract}
We consider the problem of parameter estimation for a stochastic McKean-Vlasov equation, and the associated system of weakly interacting particles. We study two cases: one in which we observe multiple independent trajectories of the McKean-Vlasov SDE, and another in which we observe multiple particles from the interacting particle system. In each case, we begin by establishing consistency and asymptotic normality of the (approximate) offline maximum likelihood estimator, in the limit as the number of observations $N\rightarrow\infty$. We then propose an online maximum likelihood estimator, which is based on a continuous-time stochastic gradient ascent scheme with respect to the asymptotic log-likelihood of the interacting particle system. We characterise the asymptotic behaviour of this estimator in the limit as $t\rightarrow\infty$, and also in the joint limit as $t\rightarrow\infty$ and $N\rightarrow\infty$. In these two cases, we obtain a.s. or $\mathbb{L}_1$ convergence to the stationary points of a limiting contrast function, under suitable conditions which guarantee ergodicity and uniform-in-time propagation of chaos. We also establish, under the additional condition of global strong concavity, $\mathbb{L}_2$ convergence to the unique maximiser of the asymptotic log-likelihood of the McKean-Vlasov SDE, with an asymptotic convergence rate which depends on the learning rate, the number of observations, and the dimension of the non-linear process. Our theoretical results are supported by two numerical examples, a linear mean field model and a stochastic opinion dynamics model.
\end{abstract}



\section{Introduction} \label{sec:introduction}
In this paper, we consider a family of McKean-Vlasov stochastic differential equations (SDEs) on $\mathbb{R}^d$, parametrised by $\theta\in\mathbb{R}^p$, of the form
\begin{align}
\mathrm{d}x_t^{\theta}&= B(\theta,x_t^{\theta},\mu_t^{\theta})\mathrm{d}t + \sigma(x_t^{\theta})\mathrm{d}w_t,~~~t\geq 0 \label{MVSDE} \\
\mu_t^{\theta} &= \mathcal{L}(x_t^{\theta}), \label{MVSDE2}
\end{align}
where $B:\mathbb{R}^p\times\mathbb{R}^d\times\mathcal{P}(\mathbb{R}^d)\rightarrow\mathbb{R}^d$, $\sigma:\mathbb{R}^d\rightarrow\mathbb{R}^{d\times d}$ are Borel measurable functions, $(w_t)_{t\geq0}$ is a $\mathbb{R}^d$-valued standard Brownian motion, and $\mathcal{L}(x_t^{\theta})$ denotes the law of of $x_t^{\theta}$. We assume that $x_0\in\mathbb{R}^d$, or that $x_0$ is a $\mathbb{R}^d$-valued random variable with law $\mu_0$, independent of $(w_t)_{t\geq 0}$. This equation is non-linear in the sense of McKean \cite{McKean1966,McKean1967,Sznitman1991}; in particular, the drift coefficient depends on the law of the solution, in addition to the solution itself. We will restrict our attention to the case in which the dependence on the law only enters linearly in the drift, namely, that
\begin{equation}
B(\theta,x,\mu) = b(\theta,x) + \int_{\mathbb{R}^d}\phi(\theta,x,y)\mu(\mathrm{d}y), \label{bigB}
\end{equation}
for some Borel measurable functions $b:\mathbb{R}^p\times\mathbb{R}^d\rightarrow\mathbb{R}^d$ and $\phi:\mathbb{R}^p\times\mathbb{R}^d\times\mathbb{R}^d\rightarrow\mathbb{R}^d$. This choice of dynamics, while not the most general possible, is sufficiently broad for many applications of interest. Moreover, it includes the popular case in which $b$ and $\phi$ both have gradient forms, that is, $b(\theta,x) = -\nabla V_{\theta}(x)$ and $\phi(\theta,x,y) = -\nabla W_{\theta}(x-y)$, in which case $V_{\theta}$ and $W_{\theta}$ are referred to as the confinement potential and the interaction potential, respectively (e.g., \cite{Durmus2020,Malrieu2001}).
 
The McKean-Vlasov SDE arises naturally as the mean field limit ($N\rightarrow\infty$) of the weakly interacting particle system (IPS)
\begin{equation}
\mathrm{d}x_t^{\theta,i,N} = B(\theta,x_t^{\theta,i,N},\mu_t^{\theta,N})\mathrm{d}t+ \sigma(x_t^{\theta,i,N})\mathrm{d}w_t^{i}~,~~~i=1,\dots,N \label{IPS}
\end{equation}
where $(w_t^{i})_{t\geq 0}$ are $N$ independent $\mathbb{R}^d$-valued independent standard Brownian motions, $x_0^{i}$ are a family of i.i.d. $\mathbb{R}^d$-valued random variables with common law $\mu_0$, independent of $\smash{(w_t^{i})_{t\geq 0}}$, and $\smash{\mu_t^{\theta,N} = \frac{1}{N}\sum_{i=1}^N \delta_{x_t^{\theta,i,N}}}$
is the empirical law of the interacting particles. In particular, under relatively weak assumptions, it is well known that the empirical law $\mu_{t}^{\theta,N}\rightarrow\mu_t^{\theta}$ weakly as $N\rightarrow\infty$ (e.g., \cite{Oelschlager1984}). This phenomenon is commonly known as the {propagation of chaos} \cite{Sznitman1991}. 

The McKean-Vlasov SDE also has a natural connection to a non-linear, non-local partial differential equation on the space of probability measures (e.g., \cite{Cattiaux2008}). In particular, under some regularity conditions on $b$ and $\phi$, one can show that $\mathcal{L}(x_t^{\theta})$ is absolutely continuous with respect to the Lebesgue measure for all $t\geq 0$ \cite{McKean1967,Tamura1984} and its density, which we will denote by $u_t^{\theta}$, satisfies a non-linear partial differential equation of the form
\begin{equation}
\frac{\partial u_t^{\theta}(x)}{\partial t} = \nabla \left[\frac{1}{2}\sigma(x)\sigma^T(x)\nabla u_t^{\theta}(x) + u_t^{\theta}(x)\left[b(\theta,x)+\int_{\mathbb{R}^d}\phi(\theta,x,y)u_t^{\theta}(y)\mathrm{d}y\right]\right]. \label{PDE}
\end{equation}
In the particular case that $b(x) = \nabla V(x)$ and $\phi(x,y) = \nabla W(x-y)$, this is commonly referred to as the {granular media equation} or the {McKean-Vlasov-Fokker-Planck equation} (e.g., \cite{Benachour1998,Cattiaux2008}). 

\subsection{Literature Review}
 The systematic study of McKean-Vlasov SDEs was first initiated by McKean \cite{McKean1966} in the 1960s, inspired by Kac's programme in Kinetic Theory \cite{Kac1956}. We refer to \cite{Funaki1984,Meleard1996,Sznitman1991} 
 for some other classical references. In the last two decades, the study of non-linear diffusions has continued to receive considerable attention, with extensive results on well-posedness (e.g., \cite{ChaudrudeRaynal2020,Huang2019}), existence and uniqueness (e.g., \cite{Bauer2018,Funaki1984,Hammersley2018,Jourdain2008,Li2016,Mishura2020,Wang2018}), 
ergodicity (e.g., \cite{Bashiri2020,Bolley2013,Butkovsky2014,Carrillo2006,Cattiaux2008,DelMoral2019,Eberle2019,Herrmann2010,Malrieu2001,Tugaut2013}),
  and propagation of chaos (e.g., \cite{Benachour1998,Bustos2008,Durmus2020,Malrieu2001,Malrieu2003}). 
  This has no doubt been motivated, at least in part, by the increasing number of applications for McKean-Vlasov SDEs, including in statistical physics \cite{Benedetto1997}, 
 multi-agent systems \cite{Benachour1998}, 
  mean-field games \cite{Carmona2018,Cardaliaguet2019,Cardaliaguet2018},
   stochastic control \cite{Buckdahn2017},
 filtering \cite{Crisan2010}, mathematical biology (including neuroscience \cite{Baladron2012} and structured models of population dynamics \cite{Burger2007,Mogilner1999}),
  epidemic dynamics \cite{Ball2020}, social sciences (including opinion dynamics \cite{Chazelle2017,Goddard2022} and cooperative behaviours \cite{Canuto2012}), financial mathematics \cite{Giesecke2020}, high dimensional sampling \cite{Liu2017,Korba2020}, and the analysis of 
 neural networks \cite{Hu2020,Mei2018,Rotskoff2019,Sirignano2020}. 

Despite the recent renewed interest in the study of McKean-Vlasov SDEs, however, the problem of parameter estimation for this class of equations has received relatively little attention. This is contrast to the wealth of literature on parameter inference in linear (i.e., not measure-dependent) diffusion processes (e.g., \cite{Bishwal2008,Borkar1982,Kutoyants1984,Liptser2001}).
 Recently, Wen et al. \cite{Wen2016} established the asymptotic consistency and asymptotic normality of the (offline) maximum likelihood estimator (MLE) for a broad class of McKean-Vlasov SDEs, based on continuous observation of $(x_t)_{t\in[0,T]}$. These results have since been extended by Liu et al. to the path-dependent case \cite{Liu2020}. We also mention the work of Catalot and Laredo \cite{Catalot2021a,Catalot2021,Catalot2022}, who have studied parametric inference for a particular class of nonlinear self-stabilising SDEs using an approximate log-likelihood function, again based on continuous observation of the non-linear diffusion process, and established asymptotic properties (consistency, normality, convergence rates) of the resulting estimators in several asymptotic regimes (e.g., small noise and long time limit). On a related topic, Gomes et al. \cite{Gomes2019} have considered parameter estimation for a McKean-Vlasov PDE, based on independent realisations of the associated non-linear SDE, in the context of models for pedestrian dynamics. 
 
Despite the recent renewed interest in the study of McKean-Vlasov SDEs, however, the problem of parameter estimation for this class of equations has received relatively little attention. This is contrast to the wealth of literature on parameter inference in linear (i.e., not measure-dependent) diffusion processes (e.g., \cite{Borkar1982,Kutoyants1984,Liptser2001}).
 Recently, Wen et al. \cite{Wen2016} established the asymptotic consistency and asymptotic normality of the (offline) maximum likelihood estimator (MLE) for a broad class of McKean-Vlasov SDEs, based on continuous observation of $(x_t)_{t\in[0,T]}$. These results have since been extended by Liu et al. to the path-dependent case \cite{Liu2020}. We also mention the work of Catalot and Laredo \cite{Catalot2021a,Catalot2021,Catalot2022}, who have studied parametric inference for a particular class of nonlinear self-stabilising SDEs using an approximate log-likelihood function, again based on continuous observation of the non-linear diffusion process, and established asymptotic properties (consistency, normality, convergence rates) of the resulting estimators in several asymptotic regimes (e.g., small noise and long time limit). On a related topic, Gomes et al. \cite{Gomes2019} have considered parameter estimation for a McKean-Vlasov PDE, based on independent realisations of the associated non-linear SDE, in the context of models for pedestrian dynamics. 

In a slightly different framework, Maestra and Hoffmann \cite{DellaMaestra2022} consider non-parametric estimation of the drift-term in a McKean-Vlasov SDE, and the solution of the corresponding non-linear Fokker-Planck equation, based on continuous observation of the associated IPS over a fixed time horizon, namely $(x_t^{i,N})_{t\in[0,T]}^{i=1,\dots,N}$, in the limit as $N\rightarrow\infty$. The authors obtain adaptive estimators based on the solution map of the Fokker-Planck equation, and prove their optimality in a minimax sense. Moreover, in the case that $b(x) = -\nabla V(x)$ and $\phi(x,y) = -\nabla W(x-y)$, the authors derive an estimator of the interaction potential, and establish its consistency. We also refer to \cite{Lu2019,Lang2021,Lu2021,Comte2022,Yao2022} for some other recent contributions on non-parametric inference for IPSs. 
 
Despite these recent contributions, however, there are no existing works which tackle the problem of {online} parameter estimation for McKean-Vlasov SDEs. The main purpose of this paper is to address this gap. There is significant motivation for this approach. In comparison to classical (offline) methods, which process observations in a batch fashion, online methods perform inference in real time, can track changes in parameters over time, are more computationally efficient, and have significantly smaller storage requirements. Even for standard diffusion processes, literature on online parameter estimation is somewhat sparse, with some notable recent exceptions \cite{Bhudisaksang2021,Sirignano2017a,Sirignano2020a,Surace2019}. The problem of recursive estimation in continuous-time stochastic processes was first rigorously analysed by Levanony et al. \cite{Levanony1994}. These authors propose an online MLE which, irrespective of initial conditions, is consistent and asymptotically efficient. This estimator, however, involves computing gradients of a Girsanov log-likelihood, $\mathcal{L}_t(\theta)$, every time a new observation arrives; as a result, it is computationally expensive, and cannot be implemented in a truly online fashion, since $\nabla_{\theta}\mathcal{L}_t(\theta)$ depends on the entire trajectory of the process $x_t$. This problem has more recently been revisited by Sirignano and Spiliopoulos \cite{Sirignano2017a,Sirignano2020a}, who propose an online statistical learning algorithm - `stochastic gradient descent in continuous time' - for the estimation of the parameters in a fully observed ergodic diffusion process. These authors establish the a.s. convergence of this estimator to the stationary points of a suitably defined objective function \cite{Sirignano2017a}, and, under additional assumptions, also obtain an $\mathbb{L}^{p}$ convergence rate and a central limit theorem \cite{Sirignano2020a}. These results have since also been extended to partially observed diffusion processes \cite{Surace2019}, and jump-diffusion processes \cite{Bhudisaksang2021}.
 
There also exists relatively little previous literature on statistical inference for IPSs, in the limit as the number of particles $N\rightarrow\infty$. In the context of parameter estimation, the mean field regime was first analysed by Kasonga \cite{Kasonga1990}, who considered a system of interacting diffusion processes, depending linearly on some unknown parameter, and established that the MLE based on continuous observations over a fixed time interval $[0,T]$ is consistent and asymptotically normal in the limit as $N\rightarrow\infty$. Bishwal \cite{Bishwal2011} later extended these results to the case in which the parameter to be estimated is a function of time, proving consistency and asymptotic normality of the sieve estimator (in the case of continuous observations) and an approximate MLE (in the case of discrete observations). In this paper, we extend the results in \cite{Kasonga1990} in another direction, establishing consistency and asymptotic normality of the offline MLE when the parametrisation is not linear.

More recently, Giesecke et al. \cite{Giesecke2020} have established the asymptotic properties (consistency, asymptotic normality, and asymptotic efficiency) of an approximate MLE for a much broader class of dynamic interacting stochastic systems, widely applicable in financial mathematics, which additionally allow for discontinuous (i.e., jump) dynamics. In addition, Chen \cite{Chen2021} has established the optimal convergence rate for the MLE in an interacting parameter system with linear interaction for $\phi$, simultaneously in the large $N$ (mean-field limit) and large $T$ (long-time dynamics) regimes. For other more recent contributions, we refer also to \cite{Amorino2022,Maestra2022,Messenger2022,Pavliotis2022}. None of the these works, however, considers parameter estimation for the IPS in the online setting.

\subsection{Contributions}
The main contributions of this paper relate to both the methodology and the theory of parameter estimation for the McKean-Vlasov SDE \eqref{MVSDE} - \eqref{MVSDE2}. Regarding methodology:
\begin{itemize}
\item We discuss how one can formulate an appropriate approximation to the true likelihood function in this problem, under various modelling assumptions.
\item We distinguish between cases in which the data consists of multiple independent samples from the McKean-Vlasov SDE (Case I), or multiple interacting trajectories from the IPS (Case II). 
\end{itemize}
In each of these cases, we perform a rigorous asymptotic analysis of the MLE, with a focus on online parameter estimation. Our main theoretical contributions can be summarised as follows:
\begin{itemize}
\item In Case II, we establish asymptotic consistency and asymptotic normality of the offline MLE, in the limit as the number of particles $N\rightarrow\infty$. Our results generalise those in \cite{Kasonga1990} to the case in which $b$ and $\phi$ depend non-linearly on the parameter.
\item In Case I and Case II, we propose implementable online estimators for the parameters of the McKean-Vlasov SDE, which evolve according to continuous-time stochastic gradient ascent algorithms with respect to the asymptotic log-likelihood function of the corresponding IPS. 
\item We establish that these estimators converge, either almost surely (in the limit as $t\rightarrow\infty$) or in $\mathbb{L}^1$ (in the limit as  $t\rightarrow\infty$ and $N\rightarrow\infty$) to the stationary points of certain contrast functions. These results hold under assumptions which guarantee ergodicity and, in the double asymptotic framework, uniform-in-time propagation of chaos.
\item We establish, under the additional condition of global strong concavity, that these estimators converge in $\mathbb{L}^2$ to the true parameter $\theta_0$, i.e., the unique maximiser of the asymptotic log-likelihood of the McKean-Vlasov SDE. We also provide explicit convergence rates, which depend on the dimension of the non-linear process, the number of observations, and the learning rate.
\end{itemize}
Finally, we provide numerical examples to illustrate the application of these results to two cases of interest, namely, a linear mean-field model, and a stochastic opinion dynamics model. It is worth emphasising that, given the connection between the McKean-Vlasov SDE \eqref{MVSDE} - \eqref{MVSDE2} and the non-linear, non-local PDE \eqref{PDE}, the results of this paper are also applicable when one is primarily interested in parameter estimation for the non-linear PDE \eqref{PDE}. 

\subsection{Paper Organisation}
The remainder of this paper is organised as follows. In Section \ref{sec:problem}, we formulate the estimation problem, and propose a recursive estimator for the McKean-Vlasov SDE. In Section \ref{sec:results}, we state our assumptions and our main results regarding the asymptotic properties of the offline and online MLEs. In Section \ref{sec:proofs}, we provide the proofs of these results. In Section \ref{sec:numerics}, we provide several numerical examples illustrating the performance of the proposed algorithm. Finally, in Section \ref{sec:conclusions}, we provide some concluding remarks.

\subsection{Additional Notation}
We will assume throughout this paper that $(x_t^{\theta})_{t\geq 0}$ is defined on a complete probability space $(\Omega,\mathcal{F},\mathbb{P}_{\theta})$, equipped with filtration $(\mathcal{F}_t)_{t\geq 0}$, and denote the corresponding expectation by $\mathbb{E}_{\theta}$. In addition, if $(x_t^{\theta})_{t\geq 0}$ is a solution with $x_0 = x\in\mathbb{R}^d$, we will occasionally make explicit the dependence on the initial condition by writing $\smash{x_t^{\theta}:= x_{t}^{\theta,x}}$, $\mu_t^{\theta}:= \mu_t^{\theta,x}$, and $\mathbb{E}_{\theta}:=\mathbb{E}_{\theta,x}$. We will use $\langle\cdot,\cdot\rangle$ and $||\cdot||$ to denote, respectively, the Euclidean inner product and the corresponding norm on $\mathbb{R}^d$. We write $\mathcal{P}(\mathbb{R}^d)$ and $\mathcal{P}_p(\mathbb{R}^d)$, $p>0$, for the collection of all probability measures on $\mathbb{R}^d$, and the collection of all probability measures on $\mathbb{R}^d$ with finite $p^{\text{th}}$ moment. In a slight abuse of notation, we will frequently write $\mu(||\cdot||^p)$ for the $p^{\text{th}}$ moment of $\mu$; that is, $\mu(||\cdot||^p)= \int_{\mathbb{R}^d} ||x||^p\mu(\mathrm{d}x)$. For $\mu,\nu\in\mathcal{P}_p(\mathbb{R}^d)$, we write $\mathbb{W}_{p}(\mu,\nu)$ to denote the Wasserstein distance between $\mu$ and $\nu$, viz
\begin{equation}
\mathbb{W}_{p}(\mu,\nu) = \inf_{\pi\in\Pi(\mu,\nu)}\left[\int_{\mathbb{R}^d\times\mathbb{R}^d}||x-y||^p\pi(\mathrm{d}x,\mathrm{d}y)\right]^{\frac{1}{\max\{1,p\}}}.
\end{equation}
where $\Pi(\mu,\nu)$ is the set of all couplings of $\mu,\nu$. That is, if $\pi\in\Pi(\mu,\nu)$, then $\pi(A\times \mathbb{R}^d) = \mu(A)$ and $\pi(\mathbb{R}^d\times A) = \nu(A)$ for all $A\in\mathcal{B}(\mathbb{R}^d)$.

\section{Parameter Estimation for the McKean-Vlasov SDE}
\label{sec:problem}
We will assume, throughout this paper, that there exists a true, static parameter $\theta_0\in\mathbb{R}^{p}$ which generates observations $(x_t)_{t\geq 0}:=(x_t^{\theta_0})_{t\geq 0}$ of the McKean-Vlasov SDE \eqref{MVSDE}, \color{black} with corresponding law $(\mu_t)_{t\geq 0}:= (\mu_t^{\theta_0})_{t\geq 0}$. Thus, in our notation, we will suppress the dependence of the observed path, and its law, on the true parameter $\theta_0$. We will make the same assumption when instead we observe trajectories of the IPS \eqref{IPS}, in which case the observations are $\smash{(x_t^{i,N})_{t\geq 0}^{i=1,\dots,N} = (x_t^{\theta_0,i,N})_{t\geq 0}^{i=1,\dots,N}}$, with corresponding empirical law $\smash{(\mu_t^{N})_{t\geq 0} := (\mu_t^{\theta_0,N})_{t\geq 0}}$.

\subsection{The Likelihood Function} \label{sec:likelihood}
 Let $\mathbb{P}_{t}^{\theta}$ denote the probability measure induced by a path $(x_s^{\theta})_{s\in[0,t]}$ of the McKean-Vlasov SDE \eqref{MVSDE}. Then, under certain regularity conditions, to be specified below, one can use the Girsanov formula to obtain a likelihood function as (e.g., \cite{Wen2016})
\begin{align}
\mathcal{L}_t(\theta)
=\log \frac{\mathrm{d}\mathbb{P}_t^{\theta}}{\mathrm{d}\mathbb{P}_t^{\theta_0}}&= \int_0^t  \left\langle\big[B(\theta,x_s,\mu_s^{\theta}) - B(\theta_0,x_s,\mu_s)\big],(\sigma(x_s) \sigma^T(x_s))^{-1} 
\mathrm{d}x_s \right\rangle \label{log_lik_1} \\
&\hspace{5mm}-\frac{1}{2}\int_0^t\left[ \left|\left|\sigma^{-1}(x_s) B(\theta,x_s,\mu_s^{\theta})\right|\right|^2-\left|\left|\sigma^{-1}(x_s)B(\theta_0,x_s,\mu_s)\right|\right|^2\right]\mathrm{d}s. \nonumber 
\end{align}

In the case that the diffusion coefficient $\sigma$ depends on the unknown parameter $\theta$, the measures $\{\mathbb{P}_{\theta}^t\}$ are, in general, mutually singular (e.g., \cite[Section 1.3]{Kutoyants2004}). We thus adopt the standard condition of parameter independence for the diffusion coefficient, and for convenience set $\sigma = 1$ (e.g., \cite{Borkar1982,Levanony1994,Wen2016}). For linear diffusion processes (in the sense of McKean), there are several approaches which can be applied when $\sigma$ does depend on an unknown parameter, including those based  on a quasi log-likelihood function \cite{Heyde1994,Hutton1986}, or on minimising a least squares type function for the diffusion coefficient \cite{Sirignano2017a}. In principle, the methods introduced in this paper can be extended to either of these cases.

In order to proceed, it will be convenient to define the functions $G:\mathbb{R}^p\times \mathbb{R}^d\times \mathcal{P}(\mathbb{R}^d) \color{black}\times \mathcal{P}(\mathbb{R}^d) \color{black}\rightarrow\mathbb{R}^d$ and $L:\mathbb{R}^p\times \mathbb{R}^d\times \mathcal{P}(\mathbb{R}^d)\color{black} \times \mathcal{P}(\mathbb{R}^d) \color{black} \rightarrow\mathbb{R}$ according to $G(\theta,x,\mu,\mu'):= B(\theta,x,\mu) - B(\theta_0,x,\mu)$ and $L(\theta,x,\mu,\mu'):= -\frac{1}{2}||G(\theta,x,\mu,\mu')||^2$. We can then write the log-likelihood function as
\begin{align}
\mathcal{L}_t(\theta)&= \int_0^t \underbrace{-\tfrac{1}{2}\left|\left| B(\theta,x_s,\mu_s^{\theta}) - B(\theta_0,x_s,\mu_s)\right|\right|^2}_{:=L(\theta,x_s,\mu_s^{\theta},\mu_s)}\mathrm{d}s + \int_0^t \big \langle \underbrace{B(\theta,x_s,\mu_s^{\theta}) - B(\theta_0,x_s,\mu_s)}_{:=G(\theta,x_s,\mu_s^{\theta},\mu_s)},\mathrm{d}w_s\big\rangle. \label{log_likelihood}
\end{align}

In general, while we may observe a sample path $(x_t)_{t\geq 0}$ of the McKean-Vlasov SDE, typically we will not have direct access to its law $(\mu_t)_{t\geq 0}$. As such, it is generally not possible to compute the likelihood function $\mathcal{L}_t(\theta)$ in \eqref{log_likelihood} directly. On this basis, even if one is interested in fitting data to the McKean-Vlasov SDE, it will typically be necessary to approximate the corresponding likelihood function. In order to make such an approximation, we will assume that we can simultaneously observe multiple continuous sample paths. Indeed, this is much more typical of the data that we observe in practice. Within this framework, there are now two main possibilities, as we outline below.

\subsubsection{Case I}
The first possibility is to assume the observed paths are $N$ independent instances $(x_t^{i})_{t\geq 0}^{i=1,\dots,N}$ of the McKean-Vlasov SDE \eqref{MVSDE}. In this case, we can approximate the true log-likelihood function $\mathcal{L}_t(\theta)$ using 
\begin{align} 
\mathcal{L}^{[N]}_t(\theta):=\frac{1}{N}\sum_{i=1}^N\mathcal{L}_t^{[i,N]}(\theta) 
=\frac{1}{N}\sum_{i=1}^N
\bigg[\int_0^t L(\theta,x_s^{i},\mu_s^{[N]},\mu_s^{[N]})\mathrm{d}s+\int_0^t \big\langle G(\theta,x_s^{i},\mu_s^{[N]},\mu_s^{[N]}),\mathrm{d}w^{i}_s\big\rangle \bigg], \hspace{-5mm}
\label{log_lik_independent}
\end{align}
where $\smash{\mu_t^{[N]} = \frac{1}{N}\sum_{i=1}^N \delta_{x_t^{i}}}$ denotes the empirical measure of the sample paths. In this approximation, the functions $\smash{\mathcal{L}_t^{[i,N]}(\theta)}$, $i=1,\dots,N$, can be viewed as $N$ \color{black} `Monte Carlo esque' approximations of $\smash{\mathcal{L}_t(\theta)}$, obtained by replacing both \color{black} $\mu_t^{\theta}$ and $\mu_t$ in \eqref{log_likelihood} by the empirical law $\smash{\mu_t^{[N]}:=\mu_t^{\theta_0,[N]}}$. The approximation $\smash{\mathcal{L}_t^{[N]}(\theta)}$ then follows by independence. Alternatively, this approximation can also be obtained by substituting the observations of the McKean-Vlasov SDE  $(x_t^{i})_{t\geq 0}^{i=1,\dots,N}$ into the log-likelihood function of the corresponding IPS (see below). 

\subsubsection{Case II} 
The second possibility is to assume that the observed paths correspond to the trajectories of $N$ particles $\smash{(x_t^{i,N})_{t\geq 0}^{i=1,\dots,N}}$ from the IPS \eqref{IPS}. In this case, we can approximate $\mathcal{L}_t(\theta)$ by the Girsanov log-likelihood function for the IPS, which is given by (e.g., \cite{Bishwal2011,Chen2021,Kasonga1990,Maestra2022})
\color{black}
\begin{align} 
\mathcal{L}^{N}_t(\theta):=\frac{1}{N}\sum_{i=1}^N\mathcal{L}_t^{i,N}(\theta)
=\frac{1}{N}\sum_{i=1}^N
\bigg[\int_0^t L(\theta,x_s^{i,N},\mu_s^N,\mu_s^N)\mathrm{d}s+\int_0^t \big\langle G(\theta,x_s^{i,N},\mu_s^N,\mu_s^N),\mathrm{d}w^{i}_s\big\rangle \bigg],
\label{log_lik_particles}  
\end{align}
where $\smash{\mu_t^N = \frac{1}{N}\sum_{j=1}^N \delta_{x_t^{j,N}}}$ 
denotes the empirical measure of the IPS, and we have included $\smash{\frac{1}{N}}$ as a normalisation factor. The connection between the functions $\mathcal{L}_t^{N}(\theta)$ and $\mathcal{L}_t^{[N]}(\theta)$ is now abundantly clear. In particular, they are identical as functions of the data. 

\begin{table}[!t]
\renewcommand{\arraystretch}{2.0}
\setlength\tabcolsep{4.0pt}
\footnotesize{
\begin{center}
\begin{tabular}{|m{1.3cm}|m{4.6cm}m{2.51cm}|m{2.41cm}m{2.41cm}|} 
\hline 

\multirow{1}{1.3cm}{\centering{\bf Case}}  & 
\multirow{1}{4.6cm}{\centering\bf{Data-Generating Model}} &
\multirow{1}{2.51cm}{\bf{Observation(s)}} &
\multicolumn{2}{c|}{\centering{\bf Likelihood Function}}
 \\[-3mm]
 
 \multirow{1}{1.3cm}{\centering{\bf }}  & 
\multirow{1}{4.6cm}{\bf{}} &
\multirow{1}{2.51cm}{\bf{}} &
\multirow{1}{2.41cm}{\centering{\emph{Approximate}}} &
\multirow{1}{2.41cm}{\centering{\emph{True}}} 
 \\
 
\hline

 \multirow{1}{1.3cm}{\centering{Case I}} & 
 \multirow{1}{4.6cm}{\centering{MVSDE} \eqref{MVSDE} } & 
\multirow{1}{2.51cm}{\centering$(x_t^{i})^{i=1,\dots,N}_{t\geq 0}$} & 
\multirow{1}{2.41cm}{\centering{$\mathcal{L}_t^{[N]}(\theta)$ in \eqref{log_lik_independent}}} &
\multirow{1}{2.41cm}{\centering{$\mathcal{L}_t(\theta)$ in \eqref{log_likelihood}}} 
 \\
 
  \hdashline
  
  \multirow{1}{1.3cm}{\centering{Case II}}  &  
\multirow{1}{4.6cm}{\centering{IPS} \eqref{IPS}}  & 
 \multirow{1}{2.51cm}{\centering$(x_t^{i,N})^{i=1,\dots,N}_{t\geq 0}$} &  
\multirow{1}{2.41cm}{\centering{$\mathcal{L}_t^N(\theta)$ in \eqref{log_lik_particles}}} &
\multirow{1}{2.41cm}{\centering{$\mathcal{L}_t^N(\theta)$ in \eqref{log_lik_particles}}} 
\\
 
 \hline

\end{tabular}
\end{center}
}
\vspace{5mm}
\caption{Parameter Estimation: Summary of Different Cases}\label{tab:summary2}
\end{table}

\subsubsection{Discussion}
These two cases are summarised in Table \ref{tab:summary2}. Case I is of particular interest if one believes that the true data-generating model is the McKean-Vlasov SDE. In particular, the function $\smash{\mathcal{L}_t^{[N]}(\theta)}$ in \eqref{log_lik_independent} should be regarded as an approximation of the true log-likelihood $\mathcal{L}_t(\theta)$ in \eqref{log_likelihood}, which can be computed under more realistic modelling assumptions. Meanwhile, Case II is of particular interest if one believes that the true data-generating model is the IPS, for a fixed and finite number of particles. Indeed, in this scenario, the function $\mathcal{L}_t^{i,N}(\theta)$ in \eqref{log_lik_particles} can and should be regarded as the true likelihood function. Nonetheless, it will be interesting to consider the behaviour of this function in the mean-field limit, in order to analyse any differences that arise with Case I.

In the limit as $N\rightarrow\infty$, standard propagation-of-chaos results (e.g., \cite{Malrieu2001}) show that the dynamics of the observations in Cases I and II will coincide. In our results, we will establish rigorously that this also holds for the different implied `likelihood' functions, $\mathcal{L}_t^{N}(\theta)$ and $\smash{\mathcal{L}_t^{[N]}(\theta)}$. This should not be a surprise given the similarities between these two functions.

\subsection{Offline Parameter Estimation} \label{sec:offline}
In the offline setting, the objective is to estimate the true parameter $\theta_{0}$ after receiving a batch of data over a fixed time interval $[0,t]$. In this case, a natural objective is to maximise the log-likelihood of the Mckean-Vlasov SDE in order to obtain the MLE
\begin{equation}
\hat{\theta}_t = \argsup_{\theta\in\mathbb{R}^{p}}\mathcal{L}_t(\theta). \label{MLE}
\end{equation} 
In practice, however, we cannot compute this estimator, since the log-likelihood is a function not only of the sample trajectory $(x_t)_{t\geq 0}$, but also on its law $(\mu_t)_{t\geq 0}$. Instead, we will consider the estimators obtained by maximising the (approximate) log-likelihood of the McKean-Vlasov SDE in \eqref{log_lik_approx}, and the log-likelihood of the IPS in \eqref{log_lik_particles}.

\subsubsection{Case I}
We first consider the case in which we observe $N$ independent paths $(x_t^{i})_{s\in[0,t]}^{i=1,\dots,N}$ from the McKean-Vlasov SDE \eqref{MVSDE}. In particular, we are interested in analysing the asymptotic properties of the approximate MLE
\begin{equation}
\hat{\theta}_t^{[N]} = \argsup_{\theta\in\mathbb{R}^p}\mathcal{L}_t^{[N]}(\theta),
\end{equation}
in the limit as the number of observations $N\rightarrow\infty$. Other than some related results in \cite{Maestra2022} (see, e.g., Proposition 10), we are not aware of any works which study the asymptotic properties of this approximate maximum likelihood estimator. In Theorems \ref{offline_theorem1} - \ref{offline_theorem2}, we establish consistency and asymptotic normality of this estimator as $N\rightarrow\infty$, given a fixed time horizon $t$.

\subsubsection{Case II}
We next consider the case in which we observe $N$ trajectories $\smash{(x_t^{i,N})_{s\in[0,t]}^{i=1,\dots,N}}$ following the dynamics of the IPS \eqref{IPS}. In this case, we aim instead to maximise the value of $\mathcal{L}_t^N(\theta)$, and are thus interested in the asymptotic properties of the following MLE
\begin{equation}
\hat{\theta}_t^N = \argsup_{\theta\in\mathbb{R}^p}\mathcal{L}_t^{N}(\theta).
\end{equation}
The asymptotic properties of this estimator as $t\rightarrow\infty$, for fixed $N$, are covered by well established results for parameter estimation in standard SDEs (e.g., \cite{Bishwal2008,Levanony1994,Liptser2001}). Conversely, there are relatively few results on the properties of this MLE in the limit as $N\rightarrow\infty$, aside from in the case of a linear parametrisation \cite{Bishwal2011,Kasonga1990}.\footnote{Since this preprint first appeared on the arXiv, we note that Maestra and Hoffmann \cite{Maestra2022} have considerably strengthened existing results on this estimator in the limit as $N\rightarrow\infty$, establishing local asymptotic normality and global minimax optimality.} We thus find it instructive to revisit this problem. In Theorems \ref{offline_theorem1} - \ref{offline_theorem2}, we extend previous results to the more general and possible non-linear setting (in the sense of parametrisation), establishing consistency and asymptotic normality of this estimator as $N\rightarrow\infty$, for fixed $t$.

\subsection{Online Parameter Estimation} \label{sec:online}
In the online setting, our objective is to estimate the true parameter $\theta_{0}$ in real time, using the continuous stream of observations. In this case, a natural objective function is the asymptotic (or average) log-likelihood function $\tilde{\mathcal{L}}(\theta)$ of the McKean-Vlasov SDE, viz 
\begin{equation}
\tilde{\mathcal{L}}(\theta) = \lim_{t\rightarrow\infty}\frac{1}{t}\mathcal{L}_t(\theta) \stackrel{a.s.}{=} \int_{\mathbb{R}^d}L(\theta,x,\mu_{\infty}^{\theta},\mu_{\infty})\mu_{\infty}(\mathrm{d}x). \label{asymptotic_MVSDE}
\end{equation}
where $\mu_{\infty}^{\theta}(\mathrm{d}x)$ denotes the unique invariant measure of $(x_t^{\theta})_{t\geq 0}$, which exists under suitable conditions on \eqref{MVSDE}, and similar to before we write $\mu_{\infty}:= \mu_{\infty}^{\theta_0}$. \color{black} In the spirit of \cite{Bhudisaksang2021,Sirignano2017a,Surace2019}, 
one can recursively maximise this objective function by defining an estimator $(\theta_t)_{t\geq 0}$ which follows the gradient of the integrand of the log-likelihood $\mathcal{L}_t(\cdot)$ in \eqref{log_likelihood}, evaluated with the current parameter estimate. \color{black} This represents a stochastic estimate for the direction of steepest ascent of the asymptotic log-likelihood $\tilde{\mathcal{L}}(\cdot)$. \color{black} In our case, initialised at $\theta_{\text{init}}\in\mathbb{R}^p$, this yields
\begin{align}
\mathrm{d}\theta_t&= \gamma_t\bigg(\underbrace{\nabla_{\theta}L(\theta_t,x_t,\mu_t^{\theta_t},\mu_t)\mathrm{d}t}_{\text{(noisy) ascent term}} + \underbrace{\nabla_{\theta}B(\theta_t,x_t,\mu_t^{\theta_t})\mathrm{d}w_t}_{\text{noise term}}\bigg) \label{theta_ideal_2}
\end{align} 
\color{black}
where $\smash{\gamma_t:\mathbb{R}_{+}\rightarrow \mathbb{R}^p_{+}}$ is a positive, non-increasing function known as the \emph{learning rate} and, \color{black} in a slight abuse of notation, $\smash{(\mu_t^{\theta_t})_{t\geq 0}}$ denotes the law of $\smash{(x_t^{\theta_t})_{t\geq 0}}$, the solution of \eqref{MVSDE} integrated with the online parameter estimate. \color{black} One also arrives at this estimator by considering a `least-squares' type objective for the drift function (see \cite{Sirignano2017a}).
This evolution equation does indeed represent a continuous-time stochastic gradient ascent scheme on the asymptotic log-likelihood function.
To see this, we can rewrite the parameter update equation \eqref{theta_ideal_2} in the form
\begin{align}
\mathrm{d}\theta_t
&=\gamma_t\bigg(\underbrace{\nabla_{\theta}\tilde{\mathcal{L}}(\theta_t)\mathrm{d}t}_{\text{(true) ascent term}}+\underbrace{(\nabla_{\theta}L(\theta_t,x_t,\mu_t^{\theta_t},\mu_t)-\nabla_{\theta}\tilde{\mathcal{L}}(\theta_t))\mathrm{d}t}_{\text{fluctuations term}} + \underbrace{\nabla_{\theta}B(\theta_t,x_t,\mu_t^{\theta_t})\mathrm{d}w_t}_{\text{noise term}}\bigg) \label{eq212}
\end{align}
The first term in this decomposition represents the true ascent direction $\nabla_{\theta}\tilde{\mathcal{L}}(\theta_t)$, the second term the deviation between the stochastic gradient ascent direction \color{black} $\nabla_{\theta} L(\theta_t,x_t,\mu_t^{\theta_t},\mu_t)$ \color{black} and the true gradient ascent direction $\nabla_{\theta}\tilde{\mathcal{L}}(\theta_t)$, while the third term is a zero-mean noise term. Heuristically, we might expect that, provided the learning rate decreases (sufficiently quickly), the ascent term will dominate the fluctuations term and the noise term for sufficiently large $t$. If this is the case, we could then reasonably expect $\theta_t$ to converge to a local maximum of $\tilde{\mathcal{L}}(\theta)$. 

\color{black} Unfortunately, there are several challenges associated with implementing the `ideal' online estimator \eqref{theta_ideal_2}. In particular, this estimator depends on $(\mu_t^{\theta_t})_{t\geq 0}$, the law of the solution of the McKean-Vlasov SDE \eqref{MVSDE}, integrated with the online parameter estimate. We will thus seek approximate, easy-to-implement estimators based on the approximations to the log-likelihood given in \eqref{log_lik_independent} and \eqref{log_lik_particles}. Following our earlier discussion, we consider two cases.  \color{black}

\subsubsection{Case I}
\color{black}
We begin with the case in which the $N$ sample paths $\smash{(x_t^{i})_{t\geq 0}^{i=1,\dots,N}}$ correspond to independent replicates of the McKean-Vlasov SDE \eqref{MVSDE}. In this case, it is natural to consider an estimator $\smash{(\theta_t^{[i,N]})_{t\geq 0}}$ which evolves according to the gradient of the integrand of the function $\mathcal{L}_t^{[i,N]}(\theta)$ in \eqref{log_lik_independent}. This yields the update equation
\begin{align}
\mathrm{d}\theta_t^{[i,N]}
&=\gamma_t\left[\nabla_{\theta}L(\theta_t^{[i,N]},x_t^{i},\mu_t^{[N]},\mu_t^{[N]})\mathrm{d}t+\nabla_{\theta} B(\theta_t^{[i,N]},x_t^{i},\mu_t^{[N]})\mathrm{d}w_t^{i}\right], \label{theta_approx_}
\end{align}
Alternatively, one could average over all of the trajectories, which is equivalent to defining an estimator $(\theta_t^{[N]})_{t\geq 0}$ which evolves according to the gradient of the integrand of the function $\mathcal{L}_t^{[N]}(\theta)$ in \eqref{log_lik_independent}, viz 
\begin{align}
\mathrm{d}\theta_t^{[N]}
&=\gamma_t\frac{1}{N}\sum_{i=1}^N\left[\nabla_{\theta}L(\theta_t^{[N]},x_t^{i},\mu_t^{[N]},\mu_t^{[N]})\mathrm{d}t+\nabla_{\theta} B(\theta_t^{[N]},x_t^{i},\mu_t^{[N]})\mathrm{d}w_t^{i}\right]. \label{theta_approx0_}
\end{align} 
The advantage of \eqref{theta_approx_} is that the computation can be performed locally at each particle, following a message passing step for retrieving the empirical law $\smash{\mu_t^N}$. It is thus convenient for a distributed implementation. On the other hand, \eqref{theta_approx0_} will typically be more accurate, as we will later demonstrate (see, in particular, the bounds in Theorem \ref{theorem1_2_star}). These two schemes can be seen as stochastic gradient ascent algorithms for maximising the average of the approximate `likelihood' function of the $i^{\text{th}}$ realisation of the McKean-Vlasov SDE or the approximate `likelihood' function of $N$ realisations of the McKean-Vlasov SDE, as defined by \eqref{log_lik_independent}. That is,  
\begin{equation}
\tilde{\mathcal{L}}^{[i,N]}(\theta) = \lim_{t\rightarrow\infty}\frac{1}{t}\mathcal{L}_t^{[i,N]}(\theta)~~~\text{or}~~~
\tilde{\mathcal{L}}^{[N]}(\theta) = \lim_{t\rightarrow\infty}\frac{1}{t}\mathcal{L}_t^{[N]}(\theta). \label{asymptotic_approx}
\end{equation}
Under suitable conditions on the learning rate, we can expect, based on a decomposition analogous to \eqref{eq212}, that the online estimates defined in \eqref{theta_approx_} and \eqref{theta_approx0_} will converge to local maxima of $\smash{\tilde{\mathcal{L}}^{[i,N]}(\theta)}$ and $\smash{\tilde{\mathcal{L}}^{[N]}(\theta)}$ as $t\rightarrow\infty$. In Theorems \ref{theorem1_1} and \ref{theorem1_1_star}, we show rigorously that this is indeed the case. Meanwhile, by the law of large numbers, we expect that $\smash{\tilde{\mathcal{L}}^{[i,N]}(\theta)}$ and $\smash{\tilde{\mathcal{L}}^{[N]}(\theta)}$ will converge to a limiting contrast function $\smash{\underline{\tilde{\mathcal{L}}}(\theta)}$ 
 as $N\rightarrow\infty$. In the joint limit as $t\rightarrow\infty$ and $N\rightarrow\infty$ it thus seems reasonable to hypothesise that the estimates \eqref{theta_approx_} and \eqref{theta_approx0_} will in fact converge to stationary points of $\smash{\underline{\tilde{\mathcal{L}}}(\theta)}$. In Theorem \ref{theorem1_1_star}, we establish rigorously that this is the case, and provide an explicit characterisation of the limiting contrast function as 
\begin{equation}
\tilde{\underline{\mathcal{L}}}(\theta) := \lim_{t,N\rightarrow\infty} \frac{1}{t} \mathcal{L}_t^{[N]}(\theta) =\int_{\mathbb{R}^d} L(\theta,x,\mu_{\infty},\mu_{\infty}) \mu_{\infty}(\mathrm{d}x). 
\label{log_lik_approx}
\end{equation} 
It is worth noting that this function is distinct from $\tilde{\mathcal{L}}(\theta)$, the asymptotic log-likelihood of the McKean-Vlasov SDE, as defined in \eqref{asymptotic_MVSDE}. Importantly, however, these functions are both maximised at the true parameter $\theta=\theta_0$. 

\subsubsection{Case II}
We now consider the case in which the observations $\smash{(x_t^{i,N})_{t\geq 0}^{i=1,\dots,N}}$ correspond to the trajectories of the IPS \eqref{IPS}. In this case, we can use essentially the same update equations. This should not be surprising, given that the `approximations' of the log-likelihood of the McKean-Vlasov SDE given in \eqref{log_lik_independent} and \eqref{log_lik_particles} were identical up to specification of the data. To be explicit, we now have
\begin{align}
\mathrm{d}\theta_t^{i,N}
&=\gamma_t\left[\nabla_{\theta}L(\theta_t^{i,N},x_t^{i,N},\mu_t^N,\mu_t^N)\mathrm{d}t+\nabla_{\theta} B(\theta_t^{i,N},x_t^{i,N},\mu_t^{N})\mathrm{d}w_t^{i}\right]. \label{theta_approx}
\end{align}
and
\begin{align}
\mathrm{d}\theta_t^{N}
&=\gamma_t\frac{1}{N}\sum_{i=1}^N\left[\nabla_{\theta}L(\theta_t^N,x_t^{i,N},\mu_t^N,\mu_t^N)\mathrm{d}t+\nabla_{\theta} B(\theta_t^N,x_t^{i,N},\mu_t^{N})\mathrm{d}w_t^{i}\right]. \label{theta_approx0}
\end{align} 
These two schemes now represent stochastic gradient ascent algorithms for maximising the average log-likelihood of the $i^{\text{th}}$ particle in the IPS, or the average log-likelihood of all of the particles in the IPS, respectively. That is, the functions
\begin{equation}
\tilde{\mathcal{L}}^{i,N}(\theta) = \lim_{t\rightarrow\infty}\frac{1}{t}\mathcal{L}_t^{i,N}(\theta)~~~\text{or}~~~
\tilde{\mathcal{L}}^N(\theta) = \lim_{t\rightarrow\infty}\frac{1}{t}\mathcal{L}_t^{N}(\theta). \label{asymptotic_IPS}
\end{equation}
Based on similar arguments to before, we can expect that the estimates defined in \eqref{theta_approx} and \eqref{theta_approx0} will converge to local maxima of $\smash{\tilde{\mathcal{L}}^{i,N}(\theta)}$ and $\smash{\tilde{\mathcal{L}}^N(\theta)}$ as $t\rightarrow\infty$. We establish this rigorously in Theorems \ref{theorem2_1} - \ref{theorem2_2}. Meanwhile, due to uniform-in-time propagation of chaos, we might also reasonably expect $\smash{\tilde{\mathcal{L}}^{i,N}(\theta)}$ and $\smash{\tilde{\mathcal{L}}^{N}(\theta)}$ to converge to $\smash{\underline{\tilde{\mathcal{L}}}(\theta)}$ as $N\rightarrow\infty$, the same limiting function as obtained in Case I. Thus, in the long-time and large-particle regime, we should obtain convergence results very similar to those obtained before. This is the subject of Theorems \ref{theorem2_1_star} - \ref{theorem2_2_star}. 

\section{Main Results}
\label{sec:results}
We are now almost ready to outline our main results. Before doing so, we let us elaborate on our standing assumptions. We begin with the following integrability assumption on the initial condition. 

\begin{manualassumption}{A.1} \label{assumption_init}
The initial law satisfies $\mu_0\in\mathcal{P}_k(\mathbb{R}^d)$ for all $k\in\mathbb{N}$.
\end{manualassumption}

This condition, alongside the conditions below, guarantees that the solutions of the McKean Vlasov SDE and the IPS have bounded moments of all orders (see Proposition \ref{prop_moment_bounds}), and so do their invariant measures (see Proposition \ref{lemma_invariant_moment_bounds}). In turn, this ensures that one can control the polynomial growth of the log-likelihood and its derivatives (in the offline case), and the polynomial growth of the solutions of the relevant Poisson equations (in the online case). This condition has also appeared in other studies on parameter estimation for the McKean-Vlasov SDE and its associated IPS (see, e.g., Assumption 1 in \cite{Maestra2022}). 

Regarding the drift function $B(\theta_0,\cdot,\cdot):\mathbb{R}^d\times\mathcal{P}(\mathbb{R}^d)\rightarrow\mathbb{R}^d$ we will work under the following rather classical assumptions. 

\begin{manualassumption}{B.1} \label{assumption1}
The function $b(\theta_0,\cdot):\mathbb{R}^d\rightarrow\mathbb{R}^d$ has the following properties.
\begin{itemize}
\item[(i)] $b(\theta_0,\cdot)$ is locally Lipschitz. That is, for all $x,x'\in\mathbb{R}^d$ such that $||x||, ||x'||\leq R$, there exists $0<L_{\theta_0,1}<\infty$ such that 
\begin{equation}
||b(\theta_0,x)-b(\theta_0,x')||\leq L_{\theta_0,1}||x-x'||.
\end{equation}
\item[(ii)] $b(\theta_0,\cdot)$ is `monotonic'. That is, for all $x,x'\in\mathbb{R}^d$, there exists $\alpha_{\theta_0}>0$ such that 
\begin{equation}
\left \langle x-x', b(\theta_0,x) - b(\theta_0,x')\right\rangle\leq -\alpha_{\theta_0} ||x-x'||^2.
\end{equation}
\end{itemize}
\end{manualassumption}

\begin{manualassumption}{B.2} \label{assumption2}
The function $\phi(\theta_0,\cdot,\cdot):\mathbb{R}^{d}\times\mathbb{R}^d\rightarrow\mathbb{R}^d$ has the following property.
\begin{itemize}
\item[(i)] $\phi(\theta_0,\cdot,\cdot)$ is globally Lipschitz. In particular, there exists $0<2L_{\theta_0,2}<\alpha_{\theta_0}$ such that, for all $x,y,x',y'\in\mathbb{R}^d$, 
\begin{equation}
||\phi(\theta_0,x,y)-\phi(\theta_0,x',y')||\leq L_{\theta_0,2}(||x-x'||+||y-y'||).
\end{equation}
\end{itemize}
\end{manualassumption}

These two conditions are used to establish existence and uniqueness of the strong solution to the McKean-Vlasov SDE (e.g., {\cite[Theorem 2.2]{Meleard1996}}), uniform moment bounds, uniform-in-time propagation of chaos, and the existence of, and exponential convergence to, a unique invariant measure (see Appendix \ref{appendix_existing_results}).  In the literature on non-linear diffusions, it is typical to consider the case in which $b(\theta,x) = -\nabla V(\theta,x)$ for some confinement potential $V$, and $\phi(\theta,x,y) = -\nabla W(\theta,x-y)$ for some interaction potential $W$. 
In this case, classical conditions which can be used in place Conditions \ref{assumption1} -  \ref{assumption2} (e.g., convexity) may be found in \cite{Malrieu2001}. 

Let us remark briefly upon some weaker conditions under which these results still hold, and therefore under which the main results of our paper will also still hold (albeit with some additional technical overhead). In the case that there is no confinement potential (i.e. $V\equiv 0$), and the interaction potential is uniformly convex with gradient that is locally Lipschitz with polynomial growth, Malrieu established uniform-in-time propagation of chaos and exponential convergence to equilibrium \cite{Malrieu2003}. Cattiaux et al. \cite{Cattiaux2008} later established the same results in the case that the interaction potential is degenerately convex. Meanwhile, in \cite{Carrillo2003,Carrillo2006}, the authors establish exponential convergence to equilibrium under the strict convexity condition $\mathrm{Hess}(V+2W)\geq \beta I_{d}$, for some $\beta>0$. 

In the case that $V+2W$ is not convex, far fewer results are available; indeed, without additional conditions on $V$ and $W$, even the existence of a unique stationary distribution is not guaranteed (see, e.g., \cite{Herrmann2010}). This being said, Bolley et al. \cite{Bolley2013} proved uniform exponential convergence to equilibrium in both degenerately convex, and weakly non-convex cases. More recently, \cite{Durmus2020,Eberle2019} have established uniform-in-time propagation of chaos and exponential convergence to equilibrium in the non-convex case, provided the confinement potential $V$ is strictly convex outside a ball, and the interaction potential is globally Lipschitz with sufficiently small Lipschitz constant. For some other relevant recent contributions, see also \cite{Carrillo2020,Delgadino2021,Lacker2018,Liu2021a}. 

In addition to Assumptions \ref{assumption1} - \ref{assumption2}, we will also impose the following regularity conditions.

\begin{manualassumption}{C.1} \label{assumption3}
The functions $b:\mathbb{R}^p\times\mathbb{R}^d\rightarrow\mathbb{R}^d$ and $\phi:\mathbb{R}^p\times\mathbb{R}^d\times\mathbb{R}^d\rightarrow\mathbb{R}^d$ have the following properties.
\begin{itemize}
\item[(i)] $b(\theta,\cdot)\in C^{2+\alpha}(\mathbb{R}^d)$, $\phi(\theta,\cdot,\cdot)\in C^{2+\alpha}(\mathbb{R}^d,\mathbb{R}^d)$  with $\alpha\in(0,1)$, $\nabla_{\theta} b(\cdot,x),\nabla_{\theta} \phi(\cdot,x,y)\in\mathcal{C}^2(\mathbb{R}^p)$ for all $x,y\in\mathbb{R}^d$, $\frac{\partial^2}{\partial x^2} \nabla_{\theta} b \in\mathcal{C}(\mathbb{R}^p,\mathbb{R}^d)$, $\frac{\partial^2}{\partial x^2}\nabla_{\theta} \phi \in\mathcal{C}(\mathbb{R}^p,\mathbb{R}^d,\mathbb{R}^d)$, and $\nabla_{\theta}^{i} b(\theta,\cdot)\in\mathcal{C}^{1+\alpha}(\mathbb{R}^d)$, $\nabla_{\theta}^{i} \phi(\theta,\cdot,\cdot)\in\mathcal{C}^{1+\alpha}(\mathbb{R}^d,\mathbb{R}^d)$, $i=1,2$, uniformly in $\theta\in\mathbb{R}^p$ for some $\alpha\in(0,1)$.
\item[(ii)] For all $\theta\in\mathbb{R}^p$, the functions $\nabla_{\theta}^{k}b(\theta,\cdot,\cdot)$ and $\nabla_{\theta}^{k}\phi(\theta,\cdot,\cdot)$ are locally Lipschitz with polynomial growth for $k=0,1,2,3$. That is, there exist constants $q_{k},K_{k}<\infty$ such that
\begin{align}
\left|\left|\nabla_{\theta}^{k} b(\theta,x) - \nabla_{\theta}^{k}b(\theta,x') \right|\right| &\leq K_{k}\left[||x-x'||\right]\left[1+||x||^{q_{k}} + ||x'||^{q_{k}}\right]\\
\left|\left|\nabla_{\theta}^{k} \phi(\theta,x,y) - \nabla_{\theta}^{k}\phi(\theta,x',y')\right|\right| &\leq K_{k}\left[||x-x'||+||y-y'||\right] \\
&\hspace{10mm}\left[1+||x||^{q_{k}}+||x'||^{q_{k}}+||y||^{q_{k}}+||y'||^{q_{k}}\right].
\end{align}
\end{itemize}
\end{manualassumption}

In the offline setting, these conditions are required in order to control the growth of the log-likelihood function and its derivatives. In the online setting, these conditions are required in order to control the ergodic behaviour of the solution of the McKean-Vlasov SDE, and the associated IPS, which is central to establishing convergence of the online MLE. In particular, they ensure that fluctuation terms of the form $\smash{\int_{0}^t \gamma_s(\nabla_{\theta}L(\theta_s,x_s,\mu_s^{\theta},\mu_s) - \nabla_{\theta}\tilde{\mathcal{L}}(\theta_s))\mathrm{d}s}$, as in \eqref{eq212}, tend to zero sufficiently quickly as $t\rightarrow\infty$. Using an approach which is now well established in the analysis of stochastic approximation algorithms in continuous time \cite{Bhudisaksang2021,Sharrock2022c,Sharrock2020a,Sirignano2017a,Sirignano2020a,Surace2019}, we control such terms by rewriting them in terms of the solutions of some related Poisson equations. Condition \ref{assumption3} ensures that these solutions are unique, and that they grow at most polynomially in a suitable sense (see Lemma \ref{lemma_poisson_3} in  Appendix \ref{sec:theorem1_lemmas}). We note, as in \cite{Sirignano2020a}, that our convergence results also hold under a slightly weaker version of Condition \ref{assumption3}(ii), in which $b(\theta,x)$ and $\phi(\theta,x,y)$ are allowed to grow linearly with respect to the parameter $\theta$.

We should remark that, for the sake of convenience and to remain in line with much of the recent literature, we have restricted our attention to the case in which the measure enters only linearly in the drift coefficient $B(\theta,x,\mu)$. As such, our main conditions are stated in terms of the functions $b:\mathbb{R}^p\times\mathbb{R}^d\times\mathbb{R}^d$ and $\phi:\mathbb{R}^p\times\mathbb{R}^d\times\mathbb{R}^d\rightarrow\mathbb{R}^d$. Our main results, however, can be extended straightforwardly to more general choices of interaction function, under suitable conditions on $B:\mathbb{R}^p\times\mathbb{R}^d\times\mathcal{P}(\mathbb{R}^d)\rightarrow\mathbb{R}^d$. In particular, in the online setting, we simple require conditions which guarantee the existence of a unique invariant measure, and uniform-in-time propagation of chaos. As an example, we can replace Condition \ref{assumption3}(ii) by $||\nabla_{\theta} B(\theta,x,\mu)||\leq K[1+||x||^q + \mu(||\cdot||^q)]$.

\subsection{Offline Parameter Estimation}
For our results on offline parameter estimation, we will require the following additional assumptions. 
\begin{manualassumption}{D.1} \label{offline_assumption_2_1}
For all $t> 0$, and for all $\theta\in\mathbb{R}^p$, the function $m_t:\mathbb{R}^p\rightarrow\mathbb{R}$, defined according to
\begin{equation}
m_t(\theta) = \int_0^t \int_{\mathbb{R}^d} L(\theta,x,\mu_s,\mu_s)\mu_s(\mathrm{d}x)\mathrm{d}s
\end{equation}
satisfies  $\sup_{||\theta-\theta_0||>\delta}m_t(\theta)<0$ a.s. $\forall\delta>0$.
\end{manualassumption}

\begin{manualassumption}{D.2} \label{offline_assumption_2_2}
For all $t> 0$, the matrix $I_t(\theta_0) = [I_t(\theta_0)]_{k,l=1,\dots,p}\in \mathbb{R}^{p\times p}$, defined according to
\begin{equation}
[I_t(\theta_0)]_{kl}= \int_0^t \int_{\mathbb{R}^d}  [\nabla_{\theta}B(\theta_0,x,\mu_s)]_{k}[\nabla_{\theta}B(\theta_0,x,\mu_s)]_{l}\mu_s(\mathrm{d}x)\mathrm{d}s
\end{equation}
is positive-definite, with $\lambda^T I_t(\theta_0)\lambda$ increasing for all $\lambda\in\mathbb{R}^p$, and $I_0(\theta_0)=0$. 
\end{manualassumption}

The first of these two conditions relates to parameter identifiability, guaranteeing the uniqueness of $\theta_0$ as the optimal parameter in the sense of some asymptotic cost, and is necessary in order to establish consistency of the MLE as $N\rightarrow\infty$.  It can be seen, in some sense, as an analogue of the classical condition used to obtain consistency in the long-time regime (e.g., \cite{Borkar1982}, \cite[pp. 137-139]{Ljung1976}, \cite[pp. 252 - 253]{Levanony1994} \cite[Condition {$A_5$}]{Rao1981}). It is also closely related to the so-called `coercivity condition', introduced in \cite{Bongini2017}, which appears in the study of non-parametric inference for IPSs (see also \cite{Li2021,Lu2020,Lu2021,Lu2019}). Meanwhile, the second condition is necessary in order to establish asymptotic normality, and can be seen as a generalisation of a similar condition introduced in \cite{Kasonga1990} (see also \cite{Bishwal2011}).

We are now ready to state our two main results in the offline case. 

\begin{manualtheorem}{1.1} \label{offline_theorem1}
Assume that Conditions \ref{assumption_init}, \ref{assumption1} - \ref{assumption2}, \ref{assumption3}, and \ref{offline_assumption_2_1} hold. Let $\Theta\subseteq\mathbb{R}^p$ be a compact set, and suppose $\theta_0\in\Theta$. Then, for all $t> 0$, $\hat{\theta}_{t}^{[N]}$ and $\hat{\theta}_t^N$ are weakly consistent estimators of $\theta_0$ as $N\rightarrow\infty$. That is, as $N\rightarrow\infty$,
\vspace{-2mm} 
\begin{align}
\hat{\theta}_t^{[N]}  \stackrel{\mathrm{\mathbb{P}}}{\longrightarrow} \theta_0~~~\text{and}~~~\hat{\theta}_t^N  \stackrel{\mathrm{\mathbb{P}}}{\longrightarrow} \theta_0.
\end{align}
\end{manualtheorem}

\begin{proof}
See Section \ref{sec:offline_proof1}.
\end{proof}

\begin{manualtheorem}{1.2} \label{offline_theorem2}
Assume that Conditions  \ref{assumption_init}, \ref{assumption1} - \ref{assumption2}, \ref{assumption3}, and \ref{offline_assumption_2_1} - \ref{offline_assumption_2_2} hold. Let $\Theta\subseteq\mathbb{R}^p$ be a compact set, and suppose $\theta_0\in\Theta$. Then, for all $t> 0$, $N^{\frac{1}{2}}(\hat{\theta}^{[N]}_t - \theta_0)$ and $N^{\frac{1}{2}}(\hat{\theta}^{N}_t - \theta_0)$ are asymptotically normal with mean zero and variance $I^{-1}_t(\theta_0)$. 
That is, as $N\rightarrow\infty$,
\vspace{-2mm}
\begin{equation}
N^{\frac{1}{2}}(\hat{\theta}_t^{[N]}-\theta_0)\stackrel{\mathcal{D}}{\longrightarrow} \mathcal{N}(0,I^{-1}_t(\theta_0))~~~\text{and}~~~N^{\frac{1}{2}}(\hat{\theta}_t^N-\theta_0)\stackrel{\mathcal{D}}{\longrightarrow} \mathcal{N}(0,I^{-1}_t(\theta_0)).
\end{equation}
\end{manualtheorem}

\begin{proof}
See Section \ref{sec:offline_proof2}.
\end{proof}

Interestingly, the properties of $\hat{\theta}_t^{[N]}$ (Case I) and $\hat{\theta}_t^N$ (Case II) are identical as the number of observations $N\rightarrow\infty$. This should not be surprising, given that the dynamics of the McKean-Vlasov SDE and the IPS, and thus the functions $\smash{\mathcal{L}_t^{[N]}(\theta)}$ and $\smash{\mathcal{L}_t^{N}(\theta)}$, coincide as $N\rightarrow\infty$.

\subsection{Online Parameter Estimation}
In the online case, we will also require some additional assumptions. We first proceed with some additional assumptions which will be required in order to establish our $\mathbb{L}^2$ convergence results (Theorems \ref{theorem2_2} and \ref{theorem2_2_star}, and Theorems \ref{theorem1_1_star} and \ref{theorem1_2_star}).  

\begin{manualassumption}{E.1} \label{assumption_bound}
There exists a positive constant $R<\infty$, and an almost everywhere positive function $\kappa:\mathbb{R}^d\times\mathcal{P}(\mathbb{R}^d)\times\mathcal{P}(\mathbb{R}^d)\rightarrow\mathbb{R}$, such that, for all $||\theta||\geq R$, 
\begin{equation}
\langle \nabla_{\theta} L(\theta,x,\mu,\mu),\theta\rangle \leq - \kappa(x,\mu)||\theta||^2.
\end{equation}
\end{manualassumption}

\begin{manualassumption}{E.2} \label{assumption_bound2}
Define the function $\tau:\mathbb{R}^{p}\times\mathbb{R}^{d}\times\mathcal{P}(\mathbb{R}^d)\rightarrow\mathbb{R}$ 
according to 
\begin{align}
\tau(\theta,x,\mu)& = \big\langle \nabla_{\theta} B(\theta,x,\mu)\nabla_{\theta}B^T(\theta,x,\mu)\frac{\theta}{||\theta||},\frac{\theta}{||\theta||}\big\rangle^{\frac{1}{2}} 
\end{align}
Then, there exists $0<q,K<\infty$ such that, for all $\theta,\theta'\in\mathbb{R}^p$, for all $x,y\in\mathbb{R}^d$, 
\begin{align}
|\tau(\theta,x,\mu) - \tau(\theta',x,\mu)||&\leq K||\theta-\theta'||(1+||x||^q+||\mu(||\cdot||^{2})||^{\frac{q}{2}})
\end{align}
\end{manualassumption}

These two conditions ensure, via the comparison theorem (e.g., \cite{Ikeda1977,Yamada1973}), that the online parameter estimates generated by the McKean-Vlasov SDE and the IPS, namely $(\theta_t)_{t\geq 0}$ and $\smash{(\theta_t^{i,N})_{t\geq 0}}$, have uniformly bounded moments (see Lemma \ref{lemma_theta_moments}). We refer to \cite{Khasminskii2012} for some more general conditions under which this result still holds.  The first condition relates to the drift terms in the two parameter update equations, and can be seen as a recurrence condition; the second condition relates to the diffusion terms, and can be seen as an extension of Condition \ref{assumption2}. This condition was introduced in \cite{Sirignano2017a}, and has since also appeared in \cite{Bhudisaksang2021}. Instead of these conditions, one could instead use a projection or truncation device to ensure that the online parameter estimates remain bounded (e.g., \cite{Chen1987}). This is fairly common within the stochastic approximation literature (e.g., \cite{Borkar2008}).

In order to establish consistency, we will also require strong concavity for at least one of the asymptotic objective functions.

\begin{manualassumption}{F.1} \label{assumption4''}
The function $\tilde{\mathcal{L}}^{[i,N]}(\theta)$, or equivalently $\tilde{\mathcal{L}}^{i,N}(\theta)$, is strongly concave with constant $\eta>0$, for all $N\in\mathbb{N}$, $i=1,\dots,N$. 
\end{manualassumption}

\begin{manualassumption}{F.2} \label{assumption4}
The function $\smash{\underline{\tilde{\mathcal{L}}}(\theta)}$ is strongly concave with constant ${\eta}>0$.
\end{manualassumption}

The first condition implies that $\tilde{\mathcal{L}}^{i,N}(\theta)$ and $\tilde{\mathcal{L}}^{[i,N]}(\theta)$ have unique maximisers for each fixed $N\in\mathbb{N}$, and will be used when we consider asymptotics in the long-time regime, for a fixed number of particles. The second condition implies that $\smash{\underline{\tilde{\mathcal{L}}}(\theta)}$ has a unique maximiser, and will be relevant when we consider the double asymptotic framework as $t\rightarrow\infty$ and $N\rightarrow\infty$. It is worth emphasising that neither Assumption \ref{assumption4''}, nor Assumption \ref{assumption4} is redundant. In particular, one of these conditions may hold, while the other does not. For example, in the linear mean field model studied in Section \ref{sec:numerics1}, we will see that Assumption \ref{assumption4''} is satisfied, while Assumption \ref{assumption4} is not (see Appendix \ref{app:verification}).

It remains only to specify our conditions on the learning rate.

\begin{manualassumption}{G.1} \label{assumption0}
The learning rate $\gamma_t:\mathbb{R}_{+}\rightarrow\mathbb{R}_{+}$ is a positive, non-increasing function such that 
$\int_{0}^{\infty}\gamma_t\mathrm{d}t = \infty$, $\int_{0}^{\infty}\gamma^2_t\mathrm{d}t <\infty$, $\int_{0}^{\infty}\gamma'_t\mathrm{d}t< \infty$.
Moreover, there exists $p>0$ such that 
$\lim_{t\rightarrow\infty}\gamma_t^2 t^{2p+\frac{1}{2}}=0$. 
\end{manualassumption}

\begin{manualassumption}{G.2} \label{assumption0_1}
Let $\Phi_{s,t} = \exp(-2\eta\int_{s}^{t} \gamma_u\mathrm{d}u)$, for the constant $\eta$ defined in Condition \ref{assumption4}. The learning rate $\gamma_t:\mathbb{R}_{+}\rightarrow\mathbb{R}_{+}$ satisfies $\int_{0}^{t}\gamma^2_s\Phi_{s,t}\mathrm{d}s = O(\gamma_t)$, $\int_{0}^{t}\gamma'_t\Phi_{s,t}\mathrm{d}s=O(\gamma_t)$, $\int_0^{t}\gamma_s\Phi_{s,t}\mathrm{d}s = O(1)$, $\int_0^t \gamma_s\Phi_{s,t} e^{-\lambda s}\mathrm{d}s = O(\gamma_t)$, and $\Phi_{1,t} = O(\gamma_t)$. 
\end{manualassumption}

The first of these conditions represents the continuous-time analogue of the standard step-size condition used in the convergence analysis of stochastic approximation algorithms in discrete time (e.g., \cite{Robbins1951,Sirignano2017a}). The second condition, first introduced in \cite{Sirignano2020a}, is specific to stochastic gradient descent in continuous time, and is only required in order to establish our $\mathbb{L}^2$ convergence results (Theorems \ref{theorem2_2} and \ref{theorem2_2_star}, and Theorems \ref{theorem1_2} and \ref{theorem1_2_star}). 
A standard choice of learning rate which satisfies both of these conditions is $\gamma_t = C_{\gamma}(C_{\theta,0} + t)^{-1}$, where $C_{\gamma}, C_{\theta,0}>0$ are positive constants such that $C_{\gamma}\eta>1$.

We are now ready to state our main results in the online case. These results are summarised in Table \ref{tab:summary}. 

\subsubsection{Case I} We begin, once more, with the case in which we observe $N$ independent sample paths $\smash{(x_t^{i})^{i=1,\dots,N}_{t\geq 0}}$ from the McKean-Vlasov SDE \eqref{MVSDE}. We thus generate online parameter estimates according to \eqref{theta_approx_} or \eqref{theta_approx0_}. We will consider two asymptotic regimes. 

\begin{itemize}
\item[(i)] $t\rightarrow\infty$ and $N\in\mathbb{N}$. We begin by considering asymptotics only in the long-time regime, with the number of sample paths fixed and finite. In this case, we can establish a.s. convergence of \eqref{theta_approx_} and \eqref{theta_approx0_} to the stationary points of the contrast functions $\tilde{\mathcal{L}}^{[i,N]}(\theta)$ and $\tilde{\mathcal{L}}^{[N]}(\theta)$, respectively. Under additional assumptions on these functions, including strong concavity (Condition \ref{assumption4''}), we also obtain $\mathbb{L}_2$ convergence to their unique maximisers, with a convergence rate which depends on the learning rate $(\gamma_t)_{t\geq 0}$. 

\begin{manualtheorem}{2.1} \label{theorem1_1}
Assume that Conditions \ref{assumption_init}, \ref{assumption1} - \ref{assumption2}, \ref{assumption3}, and \ref{assumption0} hold. Then, for all $N\in\mathbb{N}$, $i=1,\dots,N$, it holds a.s. that
\begin{alignat}{2}
\lim_{t\rightarrow\infty} ||\nabla_{\theta}\tilde{\mathcal{L}}^{[i,N]}({\theta}^{[i,N]}_t)||&=\lim_{t\rightarrow\infty} ||\nabla_{\theta}\tilde{\mathcal{L}}^{[N]}({\theta}^{[N]}_t)||= 0.
\end{alignat}
\end{manualtheorem}

\begin{manualtheorem}{2.2} \label{theorem1_2}
Assume that Conditions \ref{assumption_init}, \ref{assumption1} - \ref{assumption2}, \ref{assumption3}, \ref{assumption_bound} - \ref{assumption_bound2}, \ref{assumption4''}, and \ref{assumption0} - \ref{assumption0_1} hold. Then, for all $N\in\mathbb{N}$, $i=1,\dots,N$, and for sufficiently large $t$, there exist positive constants $K_{\theta_0,1},K_{\theta_0,2}$, 
such that
\begin{align}
\mathbb{E}_{\theta_0}\left[||\theta_t^{[i,N]}-\theta_{0}||^2\right] &\leq \left(K_{\theta_0,1}+K_{\theta_0,2}\right)\gamma_t, \label{convergence_rate-} \\
\mathbb{E}_{\theta_0}\left[||\theta_t^{[N]}-\theta_{0}||^2\right] &\leq (K_{\theta_0,1}+\frac{K_{\theta_0,2}}{N})\gamma_t. \label{convergence_rate_2-}
\end{align}
\end{manualtheorem}

\end{itemize}

\begin{table}[!t]
\color{black}
\renewcommand{\arraystretch}{2.5}
\setlength\tabcolsep{3.4pt}
\footnotesize{
\begin{center}
\begin{tabular}{|m{0.8cm}|m{1.59cm}|m{1.6cm}|m{2.4cm}|m{1.6cm}|m{5.9cm}|}
\hline

\multirow{1}{0.8cm}{\centering{\bf Case}}  & 
\multirow{1}{1.59cm}{\centering{\bf Limit(s)}} & 
\multirow{1}{1.6cm}{\centering{\bf Theorems}}  & 
\multirow{1}{2.4cm}{\centering{\bf Parameter Estimates}} & 
 \multirow{1}{1.6cm}{\centering \bf{Objective Function}} & 
\multirow{1}{5.9cm}{\centering{\bf Convergence Rate}}
 \\ 
 
\hline
 
   \multirow{2}{0.8cm}{\centering{Case I}}  & 
   \multirow{2}{1.59cm}{\centering{$t\rightarrow\infty$, $N\in\mathbb{N}$}} &
 \multirow{2}{1.6cm}{\centering{\ref{theorem1_1} - \ref{theorem1_2}}}  & 
\multirow{1}{2.4cm}{\centering{$\theta_t^{[i,N]}$ from \eqref{theta_approx_}}} &
\multirow{1}{1.6cm}{\centering{$\tilde{\mathcal{L}}^{[i,N]}(\theta)$}}  & 
\multirow{1}{5.9cm}{\centering{$(K_{\theta_0,1}+K_{\theta_0,2})\gamma_t$}} 
\\

& 
&
&
\multirow{1}{2.4cm}{\centering{$\theta_t^{[N]}$ from \eqref{theta_approx0_}}} &
\multirow{1}{1.6cm}{\centering{$\tilde{\mathcal{L}}^{[N]}(\theta)$}}  & 
\multirow{1}{5.9cm}{\centering{$(K_{\theta_0,1} + \dfrac{K_{\theta_0,2}}{N})\gamma_t$}} 
 \\
 
 \hdashline
 
 \multirow{2}{0.8cm}{\centering{Case I}}  & 
 \multirow{2}{1.59cm}{\centering{$t\rightarrow\infty$, $N\rightarrow\infty$}} &
 \multirow{2}{1.6cm}{\centering{\ref{theorem1_1_star} - \ref{theorem1_2_star}}}  & 
\multirow{1}{2.4cm}{\centering{$\theta_t^{[i,N]}$ from \eqref{theta_approx_}}} &
\multirow{2}{1.6cm}{\centering{$\tilde{\underline{\mathcal{L}}}(\theta)$}}  & 
\multirow{1}{5.9cm}{\centering{$\left(K_{\theta_0,1}+K_{\theta_0,2}\right)\gamma_t + {K_{\theta_0,3}}\alpha(N)$}} 
\\

& 
&
&
\multirow{1}{2.4cm}{\centering{$\theta_t^{[N]}$ from \eqref{theta_approx0_}}} &
&
\multirow{1}{5.9cm}{\centering{$(K_{\theta_0,1}+\dfrac{K_{\theta_0,2}}{N})\gamma_t+{K_{\theta_0,3}}\alpha(N)$}} 
 \\
 
 \hdashline
  \hdashline
  
 \multirow{2}{0.8cm}{\centering{Case II}}  & 
 \multirow{2}{1.59cm}{\centering{$t\rightarrow\infty$, $N\in\mathbb{N}$}} &
 \multirow{2}{1.6cm}{\centering{\ref{theorem2_1} - \ref{theorem2_2}}}  & 
\multirow{1}{2.4cm}{\centering{$\theta_t^{i,N}$ from \eqref{theta_approx}}} &
\multirow{1}{1.6cm}{\centering{$\tilde{\mathcal{L}}^{i,N}(\theta)$}}  & 
\multirow{1}{5.9cm}{\centering{$(K_{\theta_0,1}^{\dagger}+K_{\theta_0,2}^{\dagger})\gamma_t$}} 
\\

& 
&
&
\multirow{1}{2.4cm}{\centering{$\theta_t^{N}$ from \eqref{theta_approx0}}} &
\multirow{1}{1.6cm}{\centering{$\tilde{\mathcal{L}}^N(\theta)$}}  & 
\multirow{1}{5.9cm}{\centering{$(K_{\theta_0,1}^{\dagger} + \dfrac{K_{\theta_0,2}^{\dagger}}{N})\gamma_t$}} 
 \\
 
 \hdashline 

\multirow{2}{0.8cm}{\centering{Case II}}  & 
\multirow{2}{1.59cm}{\centering{$t\rightarrow\infty$, $N\rightarrow\infty$}} &
\multirow{2}{1.6cm}{\centering{\ref{theorem2_1_star} - \ref{theorem2_2_star}}}  & 
\multirow{1}{2.4cm}{\centering{$\theta_t^{i,N}$ from \eqref{theta_approx}}} &
\multirow{2}{1.6cm}{\centering{$\tilde{\underline{\mathcal{L}}}(\theta)$}}  & 
\multirow{1}{5.9cm}{\centering{$\smash{(K_{\theta_0,1}^{\dagger}+K_{\theta_0,2}^{\dagger})\gamma_t+ {K_{\theta_0,3}^{\dagger}}\alpha(N) + \dfrac{K_{\theta_0,4}^{\dagger}}{N^{1/2}}}$}} 
\\

& 
&
&
\multirow{1}{2.4cm}{\centering{$\theta_t^{N}$ from \eqref{theta_approx0}}} &
&
\multirow{1}{5.9cm}{\centering{$\smash{(K_{\theta_0,1}^{\dagger}+\dfrac{K_{\theta_0,2}^{\dagger}}{N})\gamma_t+ {K_{\theta_0,3}^{\dagger}}\alpha(N) + \dfrac{K_{\theta_0,4}^{\dagger}}{N^{1/2}}}$}} 
 \\
 
\hline

 \hline

\end{tabular}
\end{center}
\vspace{5mm}
}
\caption{Online Parameter Estimation: Summary of Main Results}\label{tab:summary}
\end{table}

\begin{itemize}
\item[(ii)] $t\rightarrow\infty$ and  $N\rightarrow\infty$. We next consider joint asymptotics in the many observations and the long-time regime. In this case, we establish $\mathbb{L}^1$ convergence of \eqref{theta_approx_} and \eqref{theta_approx0_} to the stationary points of $\smash{\underline{\tilde{\mathcal{L}}}(\theta)}$, the contrast function defined in \eqref{log_lik_approx}. Under the additional assumption of strong log-concavity for this contrast function (Condition \ref{assumption4}), among other technical assumptions, we then obtain convergence to the unique maximiser of $\tilde{\mathcal{L}}(\theta)$, with a convergence rate which depends on the learning rate $(\gamma_t)_{t\geq 0}$, the number of sample paths $N$, and the dimension of the non-linear process $d$. \\

It is worth emphasising that, even though we now take $N\rightarrow\infty$, the first of these two results does not guarantee convergence to the stationary points of $\tilde{\mathcal{L}}(\theta)$, the true asymptotic log-likelihood function of the McKean-Vlasov SDE. This should be seen as the penalty for replacing $\smash{(\mu_t^{\theta_t})_{t\geq 0}}$ in the `theoretically correct' (but intractable) update equation \eqref{theta_ideal_2}, which depends on the path of the online parameter estimates, by $\smash{(\mu_t^{[N]})_{t\geq 0}}$ in the approximate (but easy-to-implement) update equations \eqref{theta_approx_} and \eqref{theta_approx0_}, which depends only on the true data-generating parameter $\theta_0$.

\begin{manualtheorem}{2.1$^{*}$} \label{theorem1_1_star}
Assume that Conditions \ref{assumption_init}, \ref{assumption1} - \ref{assumption2}, \ref{assumption3}, and \ref{assumption0} hold. Then, in $\mathbb{L}^1$, it holds that
\begin{alignat}{2}
\lim_{t,N\rightarrow\infty}||\nabla_{\theta}\underline{\tilde{\mathcal{L}}}({\theta}^{[i,N]}_t)||=\lim_{t,N\rightarrow\infty} ||\nabla_{\theta}\underline{\tilde{\mathcal{L}}}({\theta}^{[N]}_t)||=0.
\end{alignat}
Suppose, in addition, that $\Theta_0=\{\theta\in\mathbb{R}^p: \nabla_{\theta} \underline{\tilde{\mathcal{L}}}(\theta)=0\} = \{\theta_0\}$. Then, in $\mathbb{L}^1$, it also holds that 
\begin{alignat}{2}
\lim_{t,N\rightarrow\infty}||\nabla_{\theta}{\tilde{\mathcal{L}}}({\theta}^{[i,N]}_t)|| =\lim_{t,N\rightarrow\infty} ||\nabla_{\theta}{\tilde{\mathcal{L}}}({\theta}^{[N]}_t)||=0
\end{alignat}
\end{manualtheorem}

\begin{manualtheorem}{2.2$^{*}$} \label{theorem1_2_star}
Assume that Conditions \ref{assumption_init}, \ref{assumption1} - \ref{assumption2}, \ref{assumption3}, \ref{assumption_bound} - \ref{assumption_bound2}, \ref{assumption4}, and \ref{assumption0} - \ref{assumption0_1} hold. Then, for sufficiently large $t$, there exist positive constants $K_{\theta_0,1},K_{\theta_0,2}, K_{\theta_0,3}$, such that 
\begin{align}
\mathbb{E}_{\theta_0}\left[||\theta_t^{[i,N]}-\theta_{0}||^2\right] &\leq \left(K_{\theta_0,1}+K_{\theta_0,2}\right)\gamma_t + {K_{\theta_0,3}}\alpha(N), \label{convergence_rate-} \\
\mathbb{E}_{\theta_0}\left[||\theta_t^{[N]}-\theta_{0}||^2\right] &\leq (K_{\theta_0,1}+\frac{K_{\theta_0,2}}{N})\gamma_t + {K_{\theta_0,3}}\alpha(N),\label{convergence_rate_2-}
\end{align}
where $\alpha:\mathbb{N}\rightarrow\mathbb{R}_{+}$ is a function defined according to
\begin{equation}
\alpha(N)= \left\{
\begin{array}{lll} 
N^{-\frac{1}{4}} & \text{if} & d=1 \\
N^{-\frac{1}{4}}\log(1+N)^{\frac{1}{2}} & \text{if} & d=2 \\
N^{-\frac{1}{2d}} & \text{if} & d\geq 3.
\end{array}
\right. \label{eq:alpha}
\end{equation}
\end{manualtheorem}

\end{itemize}

\subsubsection{Case II} 
We now turn our attention to the case in which we observe $N$ particles $(x_t^{i})_{t\geq 0}^{i=1,\dots,N}$ from the IPS \eqref{IPS}, and generate online parameter estimates according to \eqref{theta_approx} or \eqref{theta_approx0}. Once again, we will consider two asymptotic regimes.

\begin{itemize}
\item[(i)] $N\in\mathbb{N}$ and $t\rightarrow\infty$. We first consider asymptotics in the long-time regime, for a fixed number of particles. This is perhaps the most natural asymptotic regime for Case II, in which implicitly we have assumed that the true data-generating model is a finite system of interacting particles. Thus, arguably, our objective should simply be to maximise the `partial' asymptotic log-likelihood of the IPS, $\tilde{\mathcal{L}}^{i,N}(\theta)$, or the `complete' asymptotic log-likelihood of the IPS, $\tilde{\mathcal{L}}^N(\theta)$, for finite $N$, since this corresponds to the true log-likelihood of the data-generating process. \\

We begin by establishing a.s. convergence of \eqref{theta_approx} or \eqref{theta_approx0} to the stationary points of these two functions. Assuming also strong concavity for these asymptotic log-likelihood functions (Condition \ref{assumption4''}), we then obtain $\mathbb{L}_2$ convergence to their unique maximisers, with an $\mathbb{L}_2$ convergence rate which depends on the learning rate $(\gamma_t)_{t\geq 0}$. 

\begin{manualtheorem}{3.1} \label{theorem2_1}
Assume that Conditions \ref{assumption_init}, \ref{assumption1} - \ref{assumption2}, \ref{assumption3}, and \ref{assumption0} hold. Then, for all $N\in\mathbb{N}$, $i=1,\dots,N$, it holds a.s. that 
\begin{align}
\lim_{t\rightarrow\infty} ||\nabla_{\theta}\tilde{\mathcal{L}}^{i,N}(\theta_t^{i,N})|| = \lim_{t\rightarrow\infty} ||\nabla_{\theta}\tilde{\mathcal{L}}^{N}(\theta_t^N)|| =  0.
\end{align}
\end{manualtheorem}

\begin{manualtheorem}{3.2} \label{theorem2_2}
Assume that Conditions \ref{assumption_init}, \ref{assumption1} - \ref{assumption2}, \ref{assumption3}, \ref{assumption_bound} - \ref{assumption_bound2}, \ref{assumption4''}, and \ref{assumption0} - \ref{assumption0_1} hold. Then, for all $N\in\mathbb{N}$, $i=1,\dots,N$, and for sufficiently large $t$, there exist positive constants $K_{\theta_0,1}^{\dagger},K_{\theta_0,2}^{\dagger}$, 
such that
\begin{align}
\mathbb{E}_{\theta_0}\left[||\theta_t^{i,N}-\theta_{0}||^2\right] &\leq \left(K_{\theta_0,1}^{\dagger}+K_{\theta_0,2}^{\dagger}\right)\gamma_t. \label{convergence_rate2''} \\
\mathbb{E}_{\theta_0}\left[||\theta_t^{N}-\theta_{0}||^2\right] &\leq (K_{\theta_0,1}^{\dagger}+\frac{K_{\theta_0,2}^{\dagger}}{N})\gamma_t. \label{convergence_rate2''_2}
\end{align}
\end{manualtheorem}

\item[(ii)] $N\rightarrow\infty$ and $t\rightarrow\infty$. 
Finally, for completeness, we consider joint asymptotics in the large particle and the long-time regime. Given that the assumed form of the data-generating mechanism in this case (the IPS) and in the previous case (the McKean-Vlasov SDE) coincide as $N\rightarrow\infty$ and $t\rightarrow\infty$, it is natural to analyse whether the same is true for the resulting parameter estimates. \\

Similar to before, we first establish $\mathbb{L}^1$ convergence of \eqref{theta_approx} and \eqref{theta_approx0} to the stationary points of the contrast function $\smash{\underline{\tilde{\mathcal{L}}}(\theta)}$. Under the additional assumption of strong concavity for this contrast function (Condition \ref{assumption4}), we then obtain $\mathbb{L}^2$ convergence to the unique maximiser of $\tilde{\mathcal{L}}(\theta)$, with a convergence rate which once again depends on the learning rate $(\gamma_t)_{t\geq 0}$, the number of particles $N$, and the dimension of each of the particles $d$.   

\begin{manualtheorem}{3.1$^{*}$} \label{theorem2_1_star}
Assume that Conditions \ref{assumption_init}, \ref{assumption1} - \ref{assumption2}, \ref{assumption3}, and \ref{assumption0} hold. Then, in $\mathbb{L}^1$, it holds that
\begin{alignat}{2}
\lim_{t,N\rightarrow\infty} ||\nabla_{\theta}\underline{\tilde{\mathcal{L}}}({\theta}^{i,N}_t)||=\lim_{t,N\rightarrow\infty} ||\nabla_{\theta}\underline{\tilde{\mathcal{L}}}({\theta}^{N}_t)||=0.
\end{alignat}
Suppose, in addition, that $\Theta_0=\{\theta\in\mathbb{R}^p: \nabla_{\theta} \underline{\tilde{\mathcal{L}}}(\theta)=0\} = \{\theta_0\}$. Then, in $\mathbb{L}^1$, it also holds that 
\begin{alignat}{2}
\lim_{t,N\rightarrow\infty}||\nabla_{\theta}{\tilde{\mathcal{L}}}({\theta}^{i,N}_t)|| =\lim_{t,N\rightarrow\infty} ||\nabla_{\theta}{\tilde{\mathcal{L}}}({\theta}^{N}_t)||=0
\end{alignat}
\end{manualtheorem}

\begin{manualtheorem}{3.2$^{*}$} \label{theorem2_2_star}
Assume that Conditions \ref{assumption_init}, \ref{assumption1} - \ref{assumption2}, \ref{assumption3}, \ref{assumption_bound} - \ref{assumption_bound2}, \ref{assumption4}, and \ref{assumption0} - \ref{assumption0_1} hold. Then, for sufficiently large $t$, and for $N\geq 1$, $1\leq i \leq N$, there exist positive constants $\smash{K_{\theta_0,1}^{\dagger},K_{\theta_0,2}^{\dagger},K_{\theta_0,3}^{\dagger},K_{\theta_0,4}^{\dagger}}$, such that
\begin{align}
\mathbb{E}_{\theta_0}\left[||\theta_t^{i,N}-\theta_{0}||^2\right] &\leq (K_{\theta_0,1}^{\dagger}+K_{\theta_0,2}^{\dagger})\gamma_t + {K_{\theta_0,3}^{\dagger}}\alpha(N) + \frac{K_{\theta_0,4}^{\dagger}}{N^{\frac{1}{2}}}, \label{convergence_rate} \\[2mm]
\mathbb{E}_{\theta_0}\left[||\theta_t^{N}-\theta_{0}||^2\right] &\leq (K_{\theta_0,1}^{\dagger}+\frac{K_{\theta_0,2}^{\dagger}}{N})\gamma_t + {K_{\theta_0,3}^{\dagger}}\alpha(N) + \frac{K_{\theta_0,4}^{\dagger}}{N^{\frac{1}{2}}}, \label{convergence_rate_2}
\end{align}
where $\alpha:\mathbb{N}\rightarrow\mathbb{R}_{+}$ is the function defined in \eqref{eq:alpha}.
\end{manualtheorem}
\end{itemize}

\subsubsection{Discussion}

Let us briefly compare the results obtained in Case I (Theorems \ref{theorem1_1} - \ref{theorem1_2} and Theorems \ref{theorem1_1_star} - \ref{theorem1_2_star}) and Case II (Theorems \ref{theorem2_1} - \ref{theorem2_2} and Theorems \ref{theorem2_1_star} - \ref{theorem2_2_star}). As noted previously, the online parameter estimates in both of these cases follow the same update equations; the only difference is in the assumed form of the data-generating model. We thus expect that the results obtained in these two cases will be similar, if not identical. In Theorems \ref{theorem1_1_star} - \ref{theorem1_2_star} and Theorems \ref{theorem2_1_star} - \ref{theorem2_2_star}, this is indeed seen to be the case. 

In particular, Theorem \ref{theorem1_1_star} and Theorem \ref{theorem2_1_star} establish that, regardless of the data-generating mechanism, the online parameter estimates generated via \eqref{theta_approx_} - \eqref{theta_approx0_} and \eqref{theta_approx} - \eqref{theta_approx0} converge to the stationary points of $\smash{\underline{\tilde{\mathcal{L}}}(\theta)}$ as $N\rightarrow\infty$ and $t\rightarrow\infty$. 
Meanwhile, Theorem \ref{theorem1_2_star} and Theorem \ref{theorem2_2_star} establish $\mathbb{L}^2$ convergence to the true parameter $\theta_0$. It is worth noting that there is a difference between the rates established in Theorem \ref{theorem1_2_star} and Theorem \ref{theorem2_2_star}. In particular, in Case II (Theorem \ref{theorem2_2_star}) there is an additional $\smash{{O}({N^{-\frac{1}{2}}})}$ term. We can interpret this term as a penalty for the mismatch between the asymptotic likelihood function implied by the data-generating model in Case II, namely the IPS \eqref{IPS}, and the asymptotic likelihood function of the limiting McKean-Vlasov SDE \eqref{MVSDE}. 

\section{Proof of Main Results} \label{sec:proofs}
In this section, we provide proofs of our main results. Many of these proofs will rely on additional auxiliary lemmas; in the interest of brevity, the statements and proofs of these lemmas have been deferred to the appendices.

\subsection{Offline Parameter Estimation}
We first provide proofs of our results in the offline case, namely Theorem \ref{offline_theorem1} and Theorem \ref{offline_theorem2}. 

\subsubsection{Proof of Theorem \ref{offline_theorem1}} \label{sec:offline_proof1}
We begin by establishing consistency of the offline MLEs $\smash{\hat{\theta}_t^{[N]}}$ (Case I) and $\smash{\hat{\theta}_t^{N}}$ (Case II) as $N\rightarrow\infty$. In the interest of brevity, we will provide full details of the proof for $\smash{\theta_t^{N}}$ (Case II), noting that essentially identical arguments can be used for $\smash{\theta_t^{[N]}}$ (Case I).\footnote{In particular, the arguments in this proof remain valid if we replace, e.g.,  $\smash{(x_s^{i,N})_{s\in[0,t]}}$ by $\smash{(x_s^{i})_{s\in[0,t]}}$ and $\smash{(\mu_s^{N})_{s\in[0,t]}}$ by $\smash{(\mu_s^{[N]})_{s\in[0,t]}}$, since all of the required properties (e.g., uniformly bounded moments) hold both for the McKean-Vlasov SDE and the IPS.}

We should emphasise that, throughout this proof, the value of $t$ will be fixed and finite. This being said, our method of proof will broadly follow the classical approach for establishing strong consistency of the MLE in a different asymptotic regime, namely, in the limit as $t\rightarrow\infty$ (e.g., \cite{Borkar1982}). Since we consider an entirely different asymptotic regime, however, at times we will need to rely on slightly different arguments (e.g., Lemma \ref{theorem1_lemma1}), and, of course, different conditions (e.g., Condition \ref{offline_assumption_2_1}).

\begin{proof}
Let $\mathbb{P}^{\theta}_{t,N}$ denote the probability measure induced by $\smash{(x_s^{\theta,i,N})_{s\in[0,t]}^{i=1,\dots,N}}$. We begin with the observation that, since $\Theta\subseteq\mathbb{R}^p$ is compact, for all $t\geq 0$, and for all $N\in\mathbb{N}$, there exists $\hat{\theta}_t^N\in\Theta$ such that
\begin{equation}
\left.\frac{\mathrm{d}\mathbb{P}^{\theta}_{t,N}}{\mathrm{d}\mathbb{P}_{t,N}^{\theta_0}}\right|_{\theta = \hat{\theta}_t^N} \geq \frac{\mathrm{d}\mathbb{P}^{\tilde{\theta}}_{t,N}}{\mathrm{d}\mathbb{P}^{\theta_0}_{t,N}}~~~\text{a.s.}
\end{equation}
for all $\tilde{\theta}\in\Theta$. We thus have, setting $\tilde{\theta}=\theta_0$, that $\smash{{\mathrm{d}\mathbb{P}^{\theta}_{t,N}}/{\mathrm{d}\mathbb{P}_{t,N}^{\theta_0}}|_{\theta = \hat{\theta}_t^N}  \geq 1}$ a.s., from which it follows, by definition of the log-likelihood, that 
\begin{align} 
\mathcal{L}_t^{N}(\hat{\theta}_t^N) 
&=\frac{1}{N}\sum_{i=1}^N\bigg[\int_0^t \left\langle  G(\theta,x_s^{i,N},\mu_s^N,\mu_s^N),\mathrm{d}w^{i}_s\right\rangle_{\theta=
\hat{\theta}_t^N}-\frac{1}{2}\int_0^t \left|\left|G(\hat{\theta}_t^N,x_s^{i,N},\mu_s^N,\mu_s^N)\right|\right|^2\mathrm{d}s\bigg]\\[2mm]
&\geq 0~~~\text{a.s.}  
\end{align}
It follows straightforwardly that
\begin{align}
&\frac{1}{N}\sum_{i=1}^N\int_0^t \left\langle G(\theta,x_s^{i,N},\mu_s^N,\mu_s^N),\mathrm{d}w^{i}_s\right\rangle_{\theta=\hat{\theta}_t^N}\geq\frac{1}{2N}\sum_{i=1}^N\int_0^t \left|\left| G(\hat{\theta}_t^N,x_s^{i,N},\mu_s^N,\mu_s^N)\right|\right|^2\mathrm{d}s\geq 0. \label{eq4_4}
\end{align}
In addition, by Lemma \ref{theorem1_lemma1}, we have that $\frac{1}{N}\sum_{i=1}^N\int_0^t \langle G(\theta,x_s^{i,N},\mu_s^N,\mu_s^N),\mathrm{d}w^{i}_s\rangle_{\theta=\hat{\theta}_t^N}\stackrel{\mathbb{P}}{\rightarrow}0$ as $N\rightarrow\infty$. Thus, taking the limit as $N\rightarrow\infty$, we have 
\begin{align}
\lim_{N\rightarrow\infty} \frac{1}{N}\sum_{i=1}^N\int_0^t \left|\left| G(\hat{\theta}_t^N,x_s^{i,N},\mu_s^N,\mu_s^N)\right|\right|^2\mathrm{d}s\stackrel{\mathbb{P}}{\rightarrow} 0.
\label{eq_4_7}
\end{align}
We next observe, making use of the Cauchy-Schwarz inequality, that
\begin{align}
&\bigg|\frac{1}{N}\sum_{i=1}^N\int_0^t \big|\big|G(\theta,x_s^{i,N},\mu_s^N,\mu_s^N)\big|\big|^2\mathrm{d}s-\frac{1}{N}\sum_{i=1}^N\int_0^t \big|\big|G(\theta',x_s^{i,N},\mu_s^N,\mu_s^N)\big|\big|^2\mathrm{d}s\bigg| \label{G_change} \\[3mm]
&\leq\bigg[\frac{1}{N}\sum_{i=1}^N\int_0^t \big|\big|G(\theta,x_s^{i,N},\mu_s^N,\mu_s^N)- G(\theta',x_s^{i,N},\mu_s^N,\mu_s^N)\big|\big|^2\mathrm{d}s\bigg]^{\frac{1}{2}} \\
&\hspace{3mm}\cdot\bigg[\frac{1}{N}\sum_{i=1}^N\int_0^t\big|\big|G(\theta,x_s^{i,N},\mu_s^N,\mu_s^N)+G(\theta',x_s^{i,N},\mu_s^N,\mu_s^N)\big|\big|^2\mathrm{d}s\bigg]^{\frac{1}{2}} \nonumber \\[2mm]
&\leq K ||\theta-\theta'||\bigg[\frac{1}{N}\sum_{i=1}^N\int_0^t \big|\big|\frac{1}{N}\sum_{j=1}^N (1+||x_s^{i,N}||^q+||x_s^{j,N}||^q)\big|\big|^2\mathrm{d}s\bigg]^{\frac{1}{2}} \label{eq4_9} \\
&\hspace{20mm}\cdot\bigg[\frac{2}{N}\sum_{i=1}^N\bigg[\int_0^t\big|\big|G(\theta,x_s^{i,N},\mu_s^N,\mu_s^N)\big|\big|^2\mathrm{d}s+\int_0^t\big|\big|G(\theta',x_s^{i,N},\mu_s^N,\mu_s^N,)\big|\big|^2\mathrm{d}s\bigg]\bigg]^{\frac{1}{2}} \hspace{-5mm} \nonumber
\end{align}
where in the final line we have used Conditions \ref{assumption3}(i) - \ref{assumption3}(ii). In addition, the uniform moment bounds on the IPS (Proposition \ref{prop_moment_bounds}), which follow from Condition \ref{assumption_init}, together with Condition \ref{assumption3}(ii), imply that all terms on the RHS of this inequality 
are bounded. 
It follows immediately that the function $\frac{1}{N}\sum_{i=1}^N \int_0^t \big|\big|G(\theta,x_s^{i,N},\mu_s^N,\mu_s^N)\big|\big|^2\mathrm{d}s$
is Lipschitz continuous in $\theta$, uniformly in $N$. Combining this with \eqref{eq_4_7}, we have that, as $N\rightarrow\infty$,
\begin{align}
\hat{\theta}_t^N \stackrel{\mathbb{P}}{\rightarrow} \mathcal{D}_t^N& = \bigg\{\theta\in\Theta:\lim_{N\rightarrow\infty}\frac{1}{N}\sum_{i=1}^N\int_0^t \big|\big|G(\theta,x_s^{i,N},\mu_s^N,\mu_s^N)\big|\big|^2\mathrm{d}s=0\bigg\} \label{eq4_13}
\end{align}
by which we we mean more precisely that $\inf_{\theta\in\mathcal{D}_t}||\hat{\theta}^N_t-\theta||\stackrel{\mathbb{P}}{\rightarrow}0$ as $N\rightarrow\infty$.
 It remains to observe that, by a repeated application of the McKean-Vlasov Law of Large Numbers (Proposition \ref{prop_lln}), as $N\rightarrow\infty$, and for all $t> 0$, we have
 \begin{align}
 \mathcal{D}_t^N \stackrel{\mathbb{P}}{\rightarrow} \mathcal{D}_t = \bigg\{\theta\in\Theta:\int_0^t \left[\int_{\mathbb{R}^d}\left|\left|G(\theta,x,\mu_s,\mu_s)\right|\right|^2\mu_s(\mathrm{d}x)\right]\mathrm{d}s=0\bigg\} = \{\theta_0\}, \label{eq4_15}
 \end{align}
where in the second equality we have also made use of the identifiability condition in Condition \ref{offline_assumption_2_1}. It follows immediately, combining \eqref{eq4_13} and \eqref{eq4_15} that, for all fixed $t>0$, as $N\rightarrow\infty$, $\hat{\theta}_t^N \stackrel{\mathbb{P}}{\longrightarrow} \theta_0$.  
\end{proof}

\subsubsection{Proof of Theorem \ref{offline_theorem2}} \label{sec:offline_proof2}

The proof of this theorem, similarly to the previous proof, combines well known techniques used to establishing strong consistency of the MLE as $t\rightarrow\infty$ (e.g., \cite{Levanony1994}) with ideas relevant to the asymptotic regime as $N\rightarrow\infty$ (e.g., \cite{Kasonga1990}). Once again, we emphasise that throughout this proof the value of $t$ will be fixed and finite, and we will consider the limit only as $N\rightarrow\infty$. 

\begin{proof}
We begin by considering a Taylor expansion of $\nabla_{\theta} \mathcal{L}_t^{N}(\theta)$ around the true value of the parameter $\theta=\theta_0$, viz, 
\begin{equation}
0 = \nabla_{\theta}\mathcal{L}_t^N(\hat{\theta}_t^N)  = \nabla_{\theta}\mathcal{L}_t^N(\theta_0) + (\theta_t^N-\theta_0)\nabla^2_{\theta}\mathcal{L}_t(\bar{\theta}_t^N)
\end{equation}
where $\bar{\theta}_t^N$ is point in the segment connecting $\hat{\theta}_t^N$ and $\theta_{0}$. The validity of this expansion is based on the sample path continuity of the log-likelihood and its derivatives. It follows that 
\begin{equation}
N^{\frac{1}{2}}(\hat{\theta}_t^N-\theta_0)\nabla^2_{\theta}\mathcal{L}_t^N(\bar{\theta}_t^N) = -N^{\frac{1}{2}}\nabla_{\theta}\mathcal{L}_t^{N}(\theta_0) 
\end{equation}
To deal with the terms in this equation, we will rely extensively on a multivariate version of Rebolledo's Central Limit Theorem \cite{Rebolledo1980}, as stated in \cite[Corollary to Theorem 2]{Kasonga1990}. Let us begin by considering the RHS. First observe that 
\begin{align}
N^{\frac{1}{2}}\nabla_{\theta}\mathcal{L}_t^N(\theta_0) &= N^{-\frac{1}{2}}\sum_{i=1}^N\int_0^t \left\langle  \nabla_{\theta} B(\theta_0,x_s^{i,N},\mu_s^N),\mathrm{d}w_s^{i} \right\rangle \\
&+N^{-\frac{1}{2}}\sum_{i=1}^N\int_0^t \nabla_{\theta} B(\theta_0,x_s^{i,N},\mu_s^N)G(\theta_0,x_s^{i,N},\mu_s^N,\mu_s^N)\mathrm{d}s  \nonumber \\
&= N^{-\frac{1}{2}}\sum_{i=1}^N\int_0^t \left\langle \nabla_{\theta} B(\theta_0,x_s^{i,N},\mu_s^N),\mathrm{d}w_s^{i} \right\rangle
\end{align}
where in the second line we have used the fact that, by definition, $G(\theta_0,\cdot,\cdot)=0$ is identically equal to zero. It follows, using also Condition \ref{assumption3}(ii) (the polynomial growth property) and Proposition \ref{prop_moment_bounds} (uniform moment bounds for the solutions of the IPS), that for all $t\geq 0$, $\smash{(N^{\frac{1}{2}}\nabla_{\theta}\mathcal{L}_t^N(\theta_0))_{N\in\mathbb{N}}}$ is a sequence of local square integrable martingales, which implies that the first condition of \cite[Corollary to Theorem 2]{Kasonga1990} is satisfied. 

Next, observe that the process $\smash{(N^{\frac{1}{2}}\nabla_{\theta} \mathcal{L}_t^N(\theta))_{t\geq 0}}$ is continuous (in time), and thus the second condition of \cite[Corollary to Theorem 2]{Kasonga1990} (the Lindenberg condition) is satisfied. Finally, we have that, for all $k,l=1,\dots,p$, as $N\rightarrow\infty$,
\begin{align}
&\big\langle \big[N^{\frac{1}{2}} \nabla_{\theta}{\mathcal{L}}_t^N(\theta_0)\big]_{k},\big[N^{\frac{1}{2}} \nabla_{\theta}{\mathcal{L}}_t^N(\theta_0)\big]_{l}\big\rangle \label{eq_420} \\
&\hspace{10mm} = \frac{1}{N}\sum_{i=1}^N\int_0^t [\nabla_{\theta} B(\theta_0,x_s^{i,N},\mu_s^N)]_{k}[\nabla_{\theta} B(\theta_0,x_s^{i,N},\mu_s^N)]_{l}\mathrm{d}s \label{eq_421}\\
&\hspace{10mm}\stackrel{\mathbb{P}}{\rightarrow} \int_0^t \left[\int_{\mathbb{R}^d} \left[\nabla_{\theta} B(\theta_0,x,\mu_s)\right]_{k}\left[\nabla_{\theta} B(\theta_0,x,\mu_s)\right]_{l}\mu_s(\mathrm{d}x)\right]\mathrm{d}s= \left[I_t(\theta_0)\right]_{kl}, \label{eq_422}
\end{align}
where in the final line, we have used a repeated application of the weak law of large numbers for the empirical distribution of the IPS (Proposition \ref{prop_lln}), and the definition of $I_t(\theta)$ (see Condition \ref{offline_assumption_2_2}). Thus, the final condition in \cite[Corollary to Theorem 2]{Kasonga1990} is satisfied. It follows from this result that
\begin{equation}
-N^{\frac{1}{2}}\nabla_{\theta}\mathcal{L}_t^N(\theta_0)\stackrel{\mathcal{D}}{\longrightarrow}\mathcal{N}_p(0,I_t(\theta_0)).
\end{equation}
It remains to prove that $\smash{\nabla^2_{\theta}\mathcal{L}_t^N(\bar{\theta}_t^N)\stackrel{\mathbb{P}}{\longrightarrow} -I_t(\theta_0)}$.
In fact, since $\smash{\hat{\theta}_t^N\stackrel{\mathbb{P}}{\rightarrow}\theta_0}$ as $\smash{N\rightarrow\infty}$ by Theorem \ref{offline_theorem1}, the continuity of $\smash{\{\nabla_{\theta}^2\mathcal{L}^N_t(\cdot)\}_{N\in\mathbb{N}}}$ in $\theta$ implies that this limit holds provided we can establish that $\smash{\nabla_{\theta}^2\mathcal{L}_t^N(\theta_0)\stackrel{\mathbb{P}}{\longrightarrow} -I_t(\theta_0)}$.  To do so, let us begin with the observation, via a simple calculation, we have that
\begin{align}
\left[\nabla^2_{\theta}\mathcal{L}_t^N(\theta_0)\right]_{kl} 
&= \frac{1}{N}\sum_{i=1}^N \int_0^t  \left[\nabla_{\theta}^2B(\theta_0,x_s^{i,N},\mu_s^N)\right]_{kl}\mathrm{d}w^{i}_s \label{eq_4_21} \\
&-\frac{1}{N}\sum_{i=1}^N\int_0^t [\nabla_{\theta} B(\theta_0,x_s^{i,N},\mu_s^N)]_{k}[\nabla_{\theta} B(\theta_0,x_s^{i,N},\mu_s^N)]_{l}\mathrm{d}s \nonumber
\end{align}
Arguing as in the proof of Lemma \ref{theorem1_lemma1} (see Appendix \ref{appendixB}), we can show that, as $N\rightarrow\infty$, we have
\begin{equation}
 \frac{1}{N}\sum_{i=1}^N\int_0^t \nabla_{\theta}^2B(\theta_0,x_s^{i,N},\mu_s^N)\mathrm{d}w^{i}_s \stackrel{\mathbb{P}}{\longrightarrow}0. \label{eq_4_22}
\end{equation}
Moreover, we have already established, c.f. 
\eqref{eq_422}, that, as $N\rightarrow\infty$, we have
\begin{align}
&\frac{1}{N}\sum_{i=1}^N\int_0^t [\nabla_{\theta} B(\theta_0,x_s^{i,N},\mu_s^N)]_{k}[\nabla_{\theta} B(\theta_0,x_s^{i,N},\mu_s^N)]_{l}\mathrm{d}s\stackrel{\mathbb{P}}{\longrightarrow} [I_t(\theta_0)]_{kl}.  \label{eq_4_23} 
\end{align}
It follows, substituting \eqref{eq_4_22} - \eqref{eq_4_23} into \eqref{eq_4_21}, that $\nabla^2_{\theta}\mathcal{L}_t^N(\theta_0) \stackrel{\mathbb{P}}{\longrightarrow}- I_t(\theta_0)$ as $N\rightarrow\infty$. By our previous remarks, this completes the proof. 
\end{proof}

\subsection{Online Parameter Estimation}
We now provide proofs of our main results in the online case; namely Theorems \ref{theorem1_1} - \ref{theorem1_2}, \ref{theorem1_1_star} - \ref{theorem1_2_star}, \ref{theorem2_1} - \ref{theorem2_2}, and \ref{theorem2_1_star} - \ref{theorem2_2_star}. 

Before we proceed, it will be necessary to introduce some additional notation. As in Section \ref{sec:likelihood}, we will write $\smash{(x_t^{i})_{t\geq 0}}$ for a solution of the McKean-Vlasov SDE \eqref{MVSDE}, where the Brownian motion $(w_t)_{t\geq 0}$ is replaced by $\smash{(w_t^{i})_{t\geq 0}}$. We will still write $\smash{(\mu_t)_{t\geq 0}}$ for the law of this solution. We also now write
\begin{align}
\mathcal{L}_t^{i}(\theta) &=\int_0^t L(\theta,x_s^{i},\mu_s^{\theta},\mu_s)\mathrm{d}s+ \int_0^t \langle G(\theta,x_s^{i},\mu_s^{\theta},\mu_s),\mathrm{d}w_s^{i}\rangle. \label{log_likelihood_i} \\
\underline{\mathcal{L}}_t^{i}(\theta) &=\int_0^t L(\theta,x_s^{i},\mu_s,\mu_s)\mathrm{d}s+ \int_0^t \langle G(\theta,x_s^{i},\mu_s,\mu_s),\mathrm{d}w_s^{i}\rangle. \label{log_likelihood_i_approx}
\end{align}
In addition, we now define $\smash{\hat{x}^{N}_t= (x_t^{1,N},\dots,x_{t}^{N,N})^T\in(\mathbb{R}^d)^N}$, the process consisting of the concatenation of the $N$ solutions of the IPS \eqref{IPS}. This process is the solution of the following SDE on $(\mathbb{R}^d)^N$
\begin{equation}
\mathrm{d}\hat{x}^{N}_t = \hat{B}(\theta,\hat{x}^{N}_t)\mathrm{d}t + \mathrm{d}\hat{w}_t^N, \label{big_SDE}
\end{equation}
where $\hat{w}_t^N$ is a $\smash{(\mathbb{R}^d)^N}$-valued Brownian motion, and the function $\smash{\hat{B}(\theta,\cdot):(\mathbb{R}^d)^N\rightarrow(\mathbb{R}^d)^N}$ is of the form $\hat{B}(\theta,\hat{x}^{N})= (\hat{B}^{1,N}(\theta,\hat{x}^{N}),\dots,\hat{B}^{N,N}(\theta,\hat{x}^{N}))^T$, where, for $i=1,\dots,N$, $\hat{B}^{i,N}(\theta,\cdot):(\mathbb{R}^d)^N\rightarrow\mathbb{R}^{d}$ is defined according to
\begin{align}
\hat{B}^{i,N}(\theta,\hat{x}^{N}) &= b(\theta,x^{i,N}) + \frac{1}{N}\sum_{j=1}^N \phi(\theta,x^{i,N},x^{j,N}). \label{hatBfunc}
\end{align}
It will also be useful to define the functions $\hat{G}^{i,N}(\theta,\cdot):(\mathbb{R}^d)^N\rightarrow\mathbb{R}^d$ and $\hat{L}^{i,N}(\theta,\cdot):(\mathbb{R}^d)^N\rightarrow\mathbb{R}$ according to 
$\hat{G}^{i,N}(\theta,\hat{x}^N)= \hat{B}^{i,N}(\theta,\hat{x}^N) - \hat{B}^{i,N}(\theta_{0},\hat{x}^N)$ and $\hat{L}^{i,N}(\theta,\hat{x}) = -\frac{1}{2}||\hat{G}^{i,N}(\theta,\hat{x}^N)||^2$.
Finally, we will write $\smash{\hat{\mu}_{t}^N= \mathcal{L}(\hat{x}^N_t)}$ to denote the law of $\smash{\hat{x}_t^N = (x_t^{1,N},\dots,x_t^{N,N})}$. 

\subsubsection{Proof of Theorem \ref{theorem1_1}, Theorem \ref{theorem1_1_star}, Theorem \ref{theorem2_1}, and Theorem \ref{theorem2_1_star}} 
\label{sec:proof1}
We first turn our attention to our $\mathbb{L}^1$ convergence results. We will focus, in particular, on the proofs of Theorems \ref{theorem2_1} and \ref{theorem2_1_star}, later demonstrating how the same approach can be used to obtain Theorems \ref{theorem1_1} and \ref{theorem1_1_star}.

\begin{proof}
Using the triangle inequality, we can decompose $\tilde{\underline{\mathcal{L}}}(\theta_t^{i,N})$, the approximate asymptotic log-likelihood of the McKean-Vlasov SDE evaluated at the online parameter estimates generated by the IPS, as 
\begin{alignat}{3}
||\nabla_{\theta}\underline{\tilde{\mathcal{L}}}({\theta}^{i,N}_t)||
&\leq \underbrace{||\nabla_{\theta}\underline{\tilde{\mathcal{L}}}({\theta}^{i,N}_t) - \tfrac{1}{t}\nabla_{\theta}\underline{\mathcal{L}}_t^{i}({\theta}^{i,N}_t)||}_{\rightarrow~0 \text{ as $t\rightarrow\infty~~\forall N\in\mathbb{N}$ by Lemma \ref{lemma0}}}
+\underbrace{||\tfrac{1}{t}\nabla_{\theta}{\underline{\mathcal{L}}}_t^{i}({\theta}^{i,N}_t) - \tfrac{1}{t}\nabla_{\theta}{\mathcal{L}}_t^{[i,N]}({\theta}^{i,N}_t)||}_{\rightarrow~0\text{ as $N\rightarrow\infty~~\forall t\in\mathbb{R}_{+}$   by Lemma \ref{lemmaC_1}}} \\[2mm]
&~~~+ \underbrace{||\tfrac{1}{t}\nabla_{\theta}{\mathcal{L}}_t^{[i,N]}({\theta}^{i,N}_t) - \tfrac{1}{t}\nabla_{\theta}{\mathcal{L}}_t^{i,N}({\theta}^{i,N}_t)||}_{\rightarrow~0\text{ as $N\rightarrow\infty~~\forall t\in\mathbb{R}_{+}$   by Lemma \ref{lemmaC}}} 
+\underbrace{||\tfrac{1}{t}\nabla_{\theta}{\mathcal{L}}_t^{i,N}({\theta}^{i,N}_t) - \nabla_{\theta}\tilde{\mathcal{L}}^{i,N}({\theta}^{i,N}_t)||}_{\rightarrow~0\text{ as $t\rightarrow\infty~~\forall N\in\mathbb{N}$ by Lemma \ref{lemma0_1}}} \\[2mm]
&~~~+\underbrace{||\nabla_{\theta}\tilde{\mathcal{L}}^{i,N}({\theta}^{i,N}_t)||}_{\rightarrow~0\text{ as $t\rightarrow\infty~~\forall N\in\mathbb{N}$ by Lemma \ref{lemmaD}}}.\label{decomposition1}
\end{alignat}
or, almost identically, as
\begin{alignat}{3}
||\nabla_{\theta}\underline{\tilde{\mathcal{L}}}({\theta}^{N}_t)||
&\leq \underbrace{||\nabla_{\theta}\underline{\tilde{\mathcal{L}}}({\theta}^{N}_t) - \tfrac{1}{t}\nabla_{\theta}\underline{{\mathcal{L}}}_t^{i}({\theta}^{N}_t)||}_{\rightarrow~0\text{ as $t\rightarrow\infty~~\forall N\in\mathbb{N}$ by Lemma \ref{lemma0}}}
+ \underbrace{||\tfrac{1}{t}\nabla_{\theta}{\underline{\mathcal{L}}}_t^{i}({\theta}^{N}_t) - \tfrac{1}{t}\nabla_{\theta}{\mathcal{L}}_t^{[N]}({\theta}^{N}_t)||}_{\rightarrow~0\text{ as $N\rightarrow\infty~~\forall t\in\mathbb{R}_{+}$   by Lemma \ref{lemmaC_1}}} \\[2mm]
&~~~+ \underbrace{||\tfrac{1}{t}\nabla_{\theta}{\mathcal{L}}_t^{[N]}({\theta}^{N}_t) - \tfrac{1}{t}\nabla_{\theta}{\mathcal{L}}_t^{N}({\theta}^{N}_t)||}_{\rightarrow~0\text{ as $N\rightarrow\infty~~\forall t\in\mathbb{R}_{+}$   by Lemma \ref{lemmaC}}} 
+\underbrace{||\tfrac{1}{t}\nabla_{\theta}{\mathcal{L}}_t^{N}({\theta}^{N}_t) - \nabla_{\theta}\tilde{\mathcal{L}}^{N}({\theta}^{N}_t)||}_{\rightarrow~0\text{ as $t\rightarrow\infty~~\forall N\in\mathbb{N}$ by Lemma \ref{lemma0_1}}}  \\[2mm]
&~~~\underbrace{||\nabla_{\theta}\tilde{\mathcal{L}}^{N}({\theta}^{N}_t)||}_{\rightarrow~0\text{ as $t\rightarrow\infty~~\forall N\in\mathbb{N}$ by Lemma \ref{lemmaD}}}. \label{decomposition2}
\end{alignat}
where $\underline{\tilde{\mathcal{L}}}(\theta)$ is defined in \eqref{log_lik_approx}, $\tilde{\mathcal{L}}^{[i,N]}(\theta)$ and $\tilde{\mathcal{L}}^{[N]}(\theta)$ are defined in \eqref{asymptotic_approx}, and $\tilde{\mathcal{L}}^{i,N}(\theta)$ and $\tilde{\mathcal{L}}^{N}(\theta)$ are defined in \eqref{asymptotic_IPS}. 

In both \eqref{decomposition1} and \eqref{decomposition2}, the final limit holds a.s.. This proves Theorem \ref{theorem2_1}. Meanwhile, all of the limits hold in $\mathbb{L}^1$. This proves the first statement in Theorem \ref{theorem2_1_star}. Now suppose that $\Theta_0 = \{\theta:\nabla_{\theta}\underline{\tilde{\mathcal{L}}}(\theta)=0\}=\{\theta_0\}$. Then, from the previous result, we have $\theta_t^{i,N}\rightarrow\theta_0$ in $\mathbb{L}^1$. By continuity, it follows immediately that $||\nabla_{\theta}{\tilde{\mathcal{L}}}(\theta_t^{i,N})||\rightarrow ||\nabla_{\theta}{\tilde{\mathcal{L}}}(\theta_0)|| = 0$ in $\mathbb{L}^1$, using the definition of $\tilde{\mathcal{L}}(\theta)$ in \eqref{asymptotic_MVSDE}. This proves the second statement in Theorem \ref{theorem2_1_star}.  

Theorem \ref{theorem1_1} and \ref{theorem1_1_star} are proved in the same fashion. In this case, in \eqref{decomposition1} and \eqref{decomposition2}, we must now replace $\smash{\theta_t^{i,N}}$ by $\smash{\theta_t^{[i,N]}}$, $\smash{\mathcal{L}_t^{i,N}(\cdot)}$ by $\smash{\mathcal{L}_t^{[i,N]}(\cdot)}$, and $\smash{\tilde{\mathcal{L}}^{i,N}(\cdot)}$ by $\smash{\tilde{\mathcal{L}}^{[i,N]}(\cdot)}$ in \eqref{decomposition1}, and $\smash{\theta_t^{N}}$ by $\smash{\theta_t^{[N]}}$, $\smash{\mathcal{L}_t^{N}(\cdot)}$ by $\smash{\mathcal{L}_t^{[N]}(\cdot)}$, and $\smash{\tilde{\mathcal{L}}^{N}(\cdot)}$ by $\smash{\tilde{\mathcal{L}}^{[N]}(\cdot)}$ in \eqref{decomposition2}. The third term in \eqref{decomposition1} and \eqref{decomposition2} now vanishes entirely. The limits for the remaining terms in \eqref{decomposition1} and \eqref{decomposition2}, established in Lemmas \ref{lemma0}, \ref{lemma0_1}, \ref{lemmaC_1}, and \ref{lemmaD} all still hold. Thus, arguing as above, we obtain the claimed results. 
\end{proof}

Before we proceed to the proofs of Lemmas \ref{lemma0} - \ref{lemmaD}, let us provide a brief high level overview of these results. We note that the proofs of these Lemmas will rely on several auxiliary results. The statements and proofs of these results can be found in \ref{appendix_existing_results} and \ref{sec:theorem1_lemmas}.
\begin{itemize}
\item[(i)] In Lemmas \ref{lemma0} and \ref{lemma0_1}, we establish the existence of $\tilde{\underline{\mathcal{L}}}(\theta)$, $\tilde{\mathcal{L}}^{i,N}(\theta)$, and $\tilde{\mathcal{L}}^N(\theta)$, as well as their derivatives. We provide explicit expressions for these functions in terms of the unique invariant measure of the McKean-Vlasov SDE or the IPS, and prove appropriate convergence results as $t\rightarrow\infty$ (both a.s. and in $\mathbb{L}^{1}$). 
\item[(iii)] In Lemma \ref{lemmaC}, we prove that, for all $t\geq 0$, the gradient of the particle based approximation of the log-likelihood of the McKean-Vlasov SDE, as defined in \eqref{log_lik_independent}, converges to the gradient of the log-likelihood of the IPS as $N\rightarrow\infty$ (in $\mathbb{L}^1$). We also provide $\mathbb{L}^1$ convergence rates. The proof of this result relies on classical uniform-in-time propagation of chaos results \cite{Bustos2008,Malrieu2001}. 
\item[(iv)] In Lemma \ref{lemmaC_1}, we prove that, for all $t\geq 0$, the gradient of the particle based approximation of the log-likelihood of the McKean-Vlasov SDE in \eqref{log_lik_independent} converges to the gradient of the `approximate' log-likelihood of the McKean-Vlasov SDE in \eqref{log_likelihood_i_approx} as $N\rightarrow\infty$ (in $\mathbb{L}^1$). We also provide $\mathbb{L}^1$ convergence rates. The proof relies on existing results on the convergence of the empirical law of the McKean-Vlasov SDE \cite{Fournier2015}.
\item[(vi)] In Lemma \ref{lemmaD}, we establish that, for all $N\in\mathbb{N}$, the gradient of the asymptotic log-likelihood of the IPS, evaluated at the relevant online parameter updates, converges to zero as $t\rightarrow\infty$ (both a.s. and in $\mathbb{L}^1$). This result can be seen as a generalisation of \cite[Theorem 2.4]{Sirignano2017a}.
\end{itemize}

\begin{lemma} \label{lemma0}
Assume that Conditions \ref{assumption_init}, \ref{assumption1} - \ref{assumption2}, and \ref{assumption3} hold. Then the processes 
$\frac{1}{t}\nabla_{\theta}^{k}\underline{\mathcal{L}}_t^{i}(\theta)$, $k=0,1,2$, 
converge, both a.s. and in $\mathbb{L}^1$, to the functions
\begin{alignat}{3}
  \nabla_{\theta}^{k}\tilde{\underline{\mathcal{L}}}(\theta) &= \int_{\mathbb{R}^d}\nabla_{\theta}^{k}L(\theta,x,\mu_{\infty},\mu_{\infty})\mu_{\infty}(\mathrm{d}x).
\end{alignat}
\end{lemma}

\begin{proof}
We will prove Lemma \ref{lemma0} for the function $\underline{\mathcal{L}}_t(\cdot)$, in the case $k=0$. The results for $k=1,2$, are proved similarly.
 \footnote{Regarding the results for $k=1,2$, we note that the processes $\frac{1}{t}\nabla_{\theta}^{k}\underline{\mathcal{L}}_t^{i}(\theta)$, and hence also $\nabla_{\theta}^{k}\underline{\tilde{\mathcal{L}}}(\theta)$, exist due to Condition \ref{assumption3}. With this established, the proof when $k=1,2$ is almost identical to the proof when $k=0$.}
 Let us begin by recalling the definition of $\frac{1}{t}\mathcal{L}_t^{i}(\theta)$, viz
\begin{align}
\frac{1}{t}\underline{\mathcal{L}}_t^{i}(\theta) 
&=\underbrace{\frac{1}{t}\int_0^t L(\theta,x_s^{i},\mu_s, \mu_s)\mathrm{d}s}_{I_1^N(\theta,t)} + \underbrace{\frac{1}{t}\int_0^t \langle G(\theta,x_s^{i},\mu_s,\mu_s),\mathrm{d}w_s^{i}}_{I_2^N(\theta,t)}  \rangle \label{log_lik}
\end{align}
We will first consider the first term on the RHS. We will characterise the asymptotic behaviour of this term via the following decomposition
\begin{align}
\underbrace{\frac{1}{t}\int_0^t L(\theta,x_s^{i},\mu_s,\mu_s)\mathrm{d}s}_{I_1^N(\theta,t)} 
&= \underbrace{\frac{1}{t}\int_0^t\left[ L(\theta,x_s^{i},\mu_s,\mu_s)-L(\theta,x_s^{i},\mu_{\infty},\mu_{\infty})\right]\mathrm{d}s}_{I_{1,1}^{N}(\theta,t)}
\\
&\hspace{5mm}+\underbrace{\frac{1}{t}\int_0^t L(\theta,x_s^{i},\mu_{\infty},\mu_{\infty})\mathrm{d}s 
}_{I_{1,2}^{N}(\theta,t)} \label{eq4_3_4} 
\end{align}
where $\mu_{\infty}$ is the unique invariant measure of $(x_t)_{t\geq 0}$, which exists as a consequence of Proposition \ref{prop_invariant_measure} (see Appendix \ref{appendix_existing_results}). We begin with the observation that, as $t\rightarrow\infty$, we have 
\begin{align}
\frac{1}{t}\int_0^t [L(\theta,x_s^{i},\mu_s,\mu_s)-L(\theta,x_s^{i},\mu_{\infty},\mu_{\infty})]\mathrm{d}s&\stackrel{\mathrm{a.s.}}{\longrightarrow} 0 \label{eq4_32}  \\
\frac{1}{t}\int_0^t L(\theta,x_s^{i},\mu_{\infty},\mu_{\infty})\mathrm{d}s&\overset{\mathrm{a.s.}}{\underset{\mathbb{L}^1}\longrightarrow}\tilde{\underline{\mathcal{L}}}(\theta), \label{eq4_33} 
\end{align}
the former by Proposition \ref{prop_invariant_measure} (see Appendix \ref{appendix_existing_results}) and the latter by an appropriate version of the ergodic theorem (e.g., \cite[Chapter X]{Revuz1999}).  
Let us now demonstrate that $\smash{I_{1,1}^N(\theta,t)}$ also converges to zero in $\mathbb{L}^1$. Using Condition \ref{assumption3}, Proposition \ref{prop_moment_bounds} (moment bounds for the McKean-Vlasov SDE), Proposition \ref{prop_invariant_measure} (the exponential contractivity of the McKean-Vlasov SDE), and Proposition \ref{lemma_invariant_moment_bounds} (moment bounds for the invariant measure of the McKean-Vlasov SDE), we have
\begin{align}
&\left|\left|L(\theta,x_s^{i},\mu_s,\mu_s)-L(\theta,x_s^{i},\mu_{\infty},\mu_{\infty})\right|\right| 
\leq K_{\theta_0} \big[1+||x_s^{i}||^{q}\big]e^{-\lambda_{\theta_0} s} \label{eq4.37}
\end{align}
where $\lambda_{\theta_0} = \alpha_{\theta_0} - 2L_{\theta_0,2}>0$, and $\alpha_{\theta_0}, \lambda_{\theta_0}>0$ are the constants defined in Condition \ref{assumption2}. It follows, making use once more of Proposition \ref{prop_moment_bounds}, and allowing the value of the constant to increase between inequalities, that
\begin{align}
\mathbb{E}_{\theta_0}[|I_{1,1}^{N}(t)|] 
&\leq \frac{1}{t}\int_0^t K_{\theta_0}\big(1+\mathbb{E}_{\theta_0}\big[||x_s^{i}||^{q}\big]\big)e^{-\lambda_{\theta_0} s}\mathrm{d}s\leq \frac{K_{\theta_0}}{t}\int_0^t e^{-\lambda_{\theta_0} s}\mathrm{d}s 
\leq \frac{K_{\theta_0}(1-e^{-\lambda_{\theta_0} t})}{\lambda_{\theta_0} t}, \label{eq316}
\end{align}
so that the convergence of $\smash{I_{1,1}^N(\theta,t)}$ to zero does also hold in $\mathbb{L}^1$. We thus have, substituting 
\eqref{eq4_33} into \eqref{eq4_3_4}, that $\smash{I_{1}^N(\theta,t)\rightarrow\tilde{\mathcal{L}}(\theta)}$, both a.s. and in $\mathbb{L}^1$. 

We now turn our attention $I_2^N(\theta,t)$. Using the It\^o's isometry, Condition \ref{assumption3}(ii) (the polynomial growth of $B$ and therefore also $G$), Proposition \ref{prop_moment_bounds} (the bounded moments of the McKean-Vlasov SDE), and Lemma \ref{lemma1} (the asymptotic growth rate of the moments of the McKean-Vlasov SDE), we have 
\begin{align}
\mathbb{E}_{\theta_0}\left[\left|\int_0^t \langle G(\theta,x_s^{i},\mu_s,\mu_s),\mathrm{d}w_s^{i}\rangle\right|^2\right] &= \mathbb{E}_{\theta_0}\left[\int_0^t ||G(\theta,x_s^{i},\mu_s,\mu_s)||^2\mathrm{d}s\right] \\
&\leq \mathbb{E}_{\theta_0}\left[\int_0^t K\left(1+||x_s^{i}||^q+\mu_s(||\cdot||^q) \right)\mathrm{d}s\right]  
\leq K_{\theta_0}t\bigg[1+\sqrt{t}\bigg]
\end{align}
where, in the final inequality, we have once more made explicit the dependence of the constant $K_{\theta_0}$ on the true parameter $\theta_0$. It follows that for each $\theta\in\mathbb{R}^p$, 
\begin{equation}
\mathbb{E}_{\theta_0}\left[||\frac{1}{t}\langle G(\theta,x_s,\mu_s,\mathrm{d}w_s\rangle||^2\right] \leq \frac{K_{\theta_0}(1+\sqrt{t})}{t},
\end{equation}
and so this term converges in $\mathbb{L}^2$ (and hence $\mathbb{L}^1$) to zero.

It remains only to demonstrate a.s. convergence of this term to zero. To do so, consider the local martingale 
\begin{align}
M_t &= \int_0^t \frac{1}{s} \langle G(\theta,x_s^{i},\mu_s^{\theta},\mu_s),\mathrm{d}w_s^{i}\rangle \\
&=\frac{1}{t} \int_0^t \langle G(\theta,x_s^{i},\mu_s^{\theta},\mu_s),\mathrm{d}w_s^{i}\rangle + \int_0^t \frac{1}{s^2} \left[\int_0^s \langle G(\theta,x_u^{i},\mu_u^{\theta},\mu_u),\mathrm{d}w_u^{i}\rangle \right]\mathrm{d}s,
\end{align}
where the second equality follows from It\^o's Lemma. Using the It\^o isometry, Condition \ref{assumption3}(ii) (the polynomial growth of $G$), and Proposition \ref{prop_moment_bounds} (the bounded moments of the McKean-Vlasov SDE), and arguing as above, we have
\begin{align}
\sup_{t>0}\mathbb{E}_{\theta_0}\left[|M_t|^2\right]& = \mathbb{E}_{\theta_0}\left[\int_0^{\infty} \frac{1}{s^2}\mathbb{E}_{\theta_0}\left[||G(\theta,x_s^{i},\mu_s^{\theta},\mu_s)||^2\right]\mathrm{d}s\right] \\
&\leq  K_{\theta_0}\left[\int_0^t \frac{1}{s^2}\left(1+\mathbb{E}_{\theta_0}\left[||x_s^{i}||^q\right]\right)\mathrm{d}s\right]<\infty. 
\end{align}
By Doob's martingale convergence theorem \cite{Doob1953}, there thus exists a finite random variable $M_{\infty}$ such that $M_t\rightarrow M_{\infty}$ a.s. It follows immediately that $\frac{1}{t} \int_0^t \langle G(\theta,x_s^{i},\mu_s^{\theta},\mu_s),\mathrm{d}w_s^{i}\rangle$ also converges to zero a.s., as claimed. Putting everything together, we thus have that $\frac{1}{t}\mathcal{L}_t^{i}(\theta)$ converges to $\tilde{\mathcal{L}}(\theta)$ both a.s. and in $\mathbb{L}^1$
\end{proof}

\begin{lemma} \label{lemma0_1}
Assume that Conditions \ref{assumption_init}, \ref{assumption1} - \ref{assumption2}, and \ref{assumption3}  hold. Then, for all $N\in\mathbb{N}$, the processes $\frac{1}{t}\nabla_{\theta}^{k}\mathcal{L}^{i,N}_t(\theta)$ and $\frac{1}{t}\nabla_{\theta}^{k}\mathcal{L}^{N}_t(\theta)$, $m=0,1,2$, 
converge, both a.s. and in $\mathbb{L}^1$, to the functions
\begin{alignat}{3}
\nabla_{\theta}^{k} \tilde{\mathcal{L}}^{i,N}(\theta)= \int_{(\mathbb{R}^d)^N} \nabla_{\theta}^{k} \hat{L}^{i,N}(\theta,\hat{x}^N)\hat{\mu}^{N}_{\infty}(\mathrm{d}\hat{x}^N)~~~,~~~\nabla_{\theta}^{k} \tilde{\mathcal{L}}^{N}(\theta)&= \frac{1}{N}\sum_{i=1}^N\nabla_{\theta}^{k} \tilde{\mathcal{L}}^{i,N}(\theta).
\end{alignat}
\end{lemma}

\begin{proof} 
We will begin by proving that the statement holds for the function $\mathcal{L}_t^{i,N}(\theta)$. The proof, in this case, is very similar to the proof of Lemma \ref{lemma0}, with some simplifications. We will provide a sketch of the proof, signposting differences with the previous proof where necessary. As previously, we will only consider the case $k=0$, with the results for $k=1,2$ proved analogously. We begin by recalling the definition of the function $\frac{1}{t}\mathcal{L}_t^{i,N}(\theta)$ from \eqref{log_lik_particles}, which we now write in the form
\begin{align}
\frac{1}{t}\mathcal{L}^{i,N}_t(\theta) 
&=\frac{1}{t}\int_0^t \hat{L}^{i,N}(\theta,\hat{x}_s^N)\mathrm{d}s+\frac{1}{t}\int_0^t \langle \hat{G}^{i,N}(\theta,\hat{x}_s^N),\mathrm{d}w_s^{i} \rangle   \label{log_lik_2} 
\end{align}
We begin with the first term on the RHS. By Proposition \ref{prop_invariant_measure}, the IPS admits a unique invariant measure $\smash{\hat{\mu}_{\infty}^N\in\mathcal{P}((\mathbb{R}^d)^N)}$. Thus, for all $N\in\mathbb{N}$,  by the ergodic theorem (e.g., \cite[Chapter X]{Revuz1999}) we have that as $t\rightarrow\infty$,
\begin{align}
\frac{1}{t}\int_0^t \hat{L}^{i,N}\left(\theta,\hat{x}^N_s\right)\mathrm{d}s 
&\overset{\mathrm{a.s.}}{\underset{\mathbb{L}^1}\longrightarrow}  \int_{(\mathbb{R}^d)^N} \hat{L}^{i,N}\left(\theta,\hat{x}^N\right)\hat{\mu}_{\infty}(\mathrm{d}\hat{x}^N) 
= \tilde{\mathcal{L}}^{i,N}(\theta),
\end{align}
It remains to bound the second term on the RHS of \eqref{log_lik_2}. We show that this term converges to zero a.s. and in $\mathbb{L}^1$, and satisfies the required convergence rate, using essentially identical arguments to those used in the proof of Lemma \ref{lemma0}. This concludes the proof. 
We now turn our attention to the function $\mathcal{L}_t^N(\theta)$. The proof of the statements regarding this function now follows easily. In particular, using the definition of $\mathcal{L}_t^N(\theta)$, c.f. \eqref{log_lik_particles}, and the results above, we have
\begin{equation}
\frac{1}{t}\mathcal{L}_t^N(\theta) = \frac{1}{t}\left[\frac{1}{N}\sum_{i=1}^N \mathcal{L}_t^{i,N}(\theta)\right] = \frac{1}{N}\sum_{i=1}^N \left[\frac{1}{t}\mathcal{L}_t^{i,N}(\theta)\right]\overset{\mathrm{a.s.}}{\underset{\mathbb{L}^1}\longrightarrow} \frac{1}{N}\sum_{i=1}^N \tilde{\mathcal{L}}^{i,N}(\theta) = \tilde{\mathcal{L}}^N(\theta).
\end{equation}
\end{proof}

\begin{lemma} \label{lemmaC}
Assume that Conditions \ref{assumption_init}, \ref{assumption1} - \ref{assumption2}, and \ref{assumption3} hold. Then, for all $\theta\in\mathbb{R}^p$, for all $t>0$, for all $i=1,\dots,N$, we have, in $\mathbb{L}^1$, that
\begin{align}
\lim_{N\rightarrow\infty} ||\tfrac{1}{t}\nabla_{\theta}\mathcal{L}_t^{i,N}(\theta)||=||\tfrac{1}{t}\nabla_{\theta}\mathcal{L}_t^{[i,N]}(\theta)||~~~\text{and}~~~ 
\lim_{N\rightarrow\infty} ||\tfrac{1}{t}\nabla_{\theta}\mathcal{L}_t^{N}(\theta)||=||\tfrac{1}{t}\nabla_{\theta}\mathcal{L}_t^{[N]}(\theta)||.
\end{align}
In addition, there exists a positive constant $K_{\theta_0}$ such that, for all $\theta\in\mathbb{R}^p$, for all $t>0$, for all $N\in\mathbb{N}$, and for all $i=1,\dots,N$, 
\begin{align}
\mathbb{E}_{\theta_0}\left[\left|\left|\frac{1}{t}\nabla_{\theta}\mathcal{L}_t^{i,N}(\theta) - \frac{1}{t}\nabla_{\theta}\mathcal{L}_t^{[i,N]}(\theta)\right|\right|\right]&\leq \frac{K_{\theta_0}}{\sqrt{N}}\left(1+\frac{1}{\sqrt{t}}\right),
\end{align}
and this bound also holds if $\mathcal{L}_t^{i,N}(\cdot)$ and $\mathcal{L}_t^{[i,N]}$ are replaced by $\mathcal{L}_t^N(\cdot)$ and $\mathcal{L}_t^{[N]}(\cdot)$.
\end{lemma}

\begin{proof}
We begin by proving that the two statements hold for $\mathcal{L}_t^{i,N}(\theta)$. First recall that
\begin{align}
\frac{1}{t}\nabla_{\theta}{\mathcal{L}}^{i,N}_t(\theta)&= \underbrace{\frac{1}{t}\int_0^t \nabla_{\theta} L(\theta,x_s^{i,N},\mu_s^{N},\mu_s^{N})\mathrm{d}s}_{I^{i,N}_{1}(\theta,t)}+ \underbrace{\frac{1}{t}\int_0^t \left\langle \nabla_{\theta} G(\theta,x_s^{i,N},\mu_s^{N},\mu_s^N),\mathrm{d}w_s^{i}\right\rangle}_{I^{i,N}_{2}(\theta,t)} \\
\frac{1}{t}\nabla_{\theta}{\mathcal{L}}^{[i,N]}_t(\theta)&= \underbrace{\frac{1}{t}\int_0^t \nabla_{\theta} L(\theta,x_s^{i},\mu_s^{[N]},\mu_s^{[N]})\mathrm{d}s}_{I^{[i,N]}_{1}(\theta,t)}+ \underbrace{\frac{1}{t}\int_0^t \left\langle \nabla_{\theta} G(\theta,x_s^{i},\mu_s^{[N]},\mu_s^{[N]}),\mathrm{d}w_s^{i}\right\rangle}_{I^{[i,N]}_{2}(\theta,t)}
\end{align}
We will begin by seeking a bound for $\mathbb{E}_{\theta_0}||I_{2}^{i,N}(\theta,t)-I_{2}^{[i,N]}(\theta,t)||$. Using Assumption \ref{assumption3}, Proposition \ref{prop_moment_bounds} (bounded moments of the McKean-Vlasov SDE and the IPS), and Proposition \ref{prop_chaos} (uniform-in-time propagation-of-chaos), there exists a constant $K_{\theta_0}>0$ (independent of $\theta$), such that
\begin{align}
\sup_{t\geq 0}\mathbb{E}_{\theta_0}\left[||\nabla_{\theta} G(\theta,x_t^{i,N},\mu_t^N,\mu_t^N) - \nabla_{\theta} G(\theta,x_t^{i},\mu_t^{[N]},\mu_t^{[N]})||^2\right] \\
\leq K_{\theta_0}\sup_{t\geq 0}\mathbb{E}_{\theta_0}\left[ ||x_t^{i} - x_t^{i,N}||^2 + \mathbb{W}^2_2(\mu_t^N,\mu_t^{[N]})\right] \leq \frac{K_{\theta_0}}{N}. \label{prop_chaos_G}
\end{align}
It follows, using also the It\^o isometry and Fubini's Theorem, that for all $\theta\in\mathbb{R}^p$, and for all $t>0$, 
\begin{align}
& \mathbb{E}_{\theta_0}\left[||I_2^{i,N}(\theta,t) - I_2^{[i,N]}(\theta,t)||^2\right] \\
 &\leq \frac{K_{\theta_0}}{t^2}\int_0^t\mathbb{E}_{\theta_0}\left[||\nabla_{\theta} G(\theta,x_s^{i,N},\mu_s^N,\mu_s^N) - \nabla_{\theta} G(\theta,x_s^{i},\mu_s^{[N]},\mu_s^{[N]})||^2\right] \mathrm{d}s
\leq \frac{K_{\theta_{0}}}{Nt}, \label{eQ32}
 \end{align}
 and so, by the H\"older inequality, 
 \begin{equation}
\mathbb{E}_{\theta_0}\big[||I_2^{i,N}(\theta,t) - I_2^{[i,N]}(\theta,t)||\big]\leq \frac{K_{\theta_0}}{\sqrt{Nt}}. \label{eQ346}
\end{equation}
In much the same fashion, we can obtain a bound for $\mathbb{E}_{\theta_0}||I_{1}^{i,N}(\theta,t)-I_{1}^{[i,N]}(\theta,t)||$. We begin by noting that
\begin{align}
 &\sup_{t\geq 0}\mathbb{E}_{\theta_0}\left[\left| \left| \nabla_{\theta}L(\theta,x_s^{i,N},\mu_s^N,\mu_s^N) - \nabla_{\theta} L(\theta,x_s^{i},\mu_s^{[N]},\mu_s^{[N]})\right|\right|\right]\\[1mm]
 &=\sup_{t\geq 0} \mathbb{E}_{\theta_0}\big[|| \underbrace{\nabla_{\theta} G(\theta,x_t^{i,N},\mu_t^N,\mu_t^N)}_{Y_N}\underbrace{G(\theta,x_t^{i,N},\mu_t^N,\mu_t^N)}_{Z_N} - \underbrace{\nabla_{\theta} G(\theta,x_t^{i},\mu_t^{[N]},\mu_t^{[N]})}_{Y_{[N]}}\underbrace{G(\theta,x_t^{i},\mu_t^{[N]},\mu_t^{[N]})}_{Z_{[N]}}||\big] \hspace{-15mm} \\[-5mm]
 &\leq \sup_{t\geq 0} \left[\mathbb{E}_{\theta_0}\left[||Z_{N}||^2\right]^{\frac{1}{2}}  \mathbb{E}_{\theta_0}\left[||Y_{N}-Y_{[N]}||^2\right]^{\frac{1}{2}}+ \mathbb{E}_{\theta_0}\left[||Y_{[N]}||^2\right]^{\frac{1}{2}} \mathbb{E}_{\theta_0}\left[||Z_{N}-Z_{[N]}||^2\right]^{\frac{1}{2}}\right]
\end{align}
where the final line follows from the triangle inequality and the Cauchy-Schwarz inequality.
Now, using Condition \ref{assumption3} (the polynomial growth of $G$) and Proposition \ref{prop_moment_bounds} (bounded moments of the McKean-Vlasov SDE and the IPS) for the two left-most inequalities, and arguing as in \eqref{prop_chaos_G} for the two right-most inequalities, 
there exist constants $K_{\theta_0,1},K_{\theta_0,2},K_{\theta_0,3},K_{\theta_0,4}>0$ (all independent of $\theta$), such that 
\begin{alignat}{3}
&\sup_{t\geq 0} \mathbb{E}_{\theta_0}\big[ ||\underbrace{G(\theta,x_t^{i,N},\mu_t^N,\mu_t^N)}_{Z_N}||^2 \big] \leq K_{\theta_0,1}, \label{eQ34a} \\
&\sup_{t\geq 0}\mathbb{E}_{\theta_0}||\underbrace{\nabla_{\theta} G(\theta,x_t^{i,N},\mu_t^N,\mu_t^N)}_{Y_{N}} - \underbrace{\nabla_{\theta} G(\theta,x_t^{i},\mu_t^{[N]},\mu_t^{[N]})}_{Y_{[N]}}||^2\leq \frac{K_{\theta_0,2}}{N} \hspace{10mm}  \label{eQ34b}   \\
&\sup_{t\geq 0} \mathbb{E}_{\theta_0}\big[ ||\underbrace{\nabla_{\theta} G(\theta,x_t^{i},\mu_t^{[N]},\mu_t^{[N]})}_{Y_{[N]}}||^2 \big] \leq K_{\theta_0,3} \label{eQ34c} \\
&\sup_{t\geq 0}\mathbb{E}_{\theta_0}||\underbrace{G(\theta,x_t^{i,N},\mu_t^N,\mu_t^N)}_{Z_N} - \underbrace{G(\theta,x_t^{i},\mu_t^{[N]},\mu_t^{[N]})}_{Z_{[N]}}||^2\leq \frac{K_{\theta_0,4}}{N} . \label{eQ35}
\end{alignat}
Substituting \eqref{eQ34a} - \eqref{eQ35} into the previous bound, we finally arrive at
\begin{align}
\sup_{t\geq 0}\mathbb{E}_{\theta_0}\left[\left| \left| \nabla_{\theta}L(\theta,x_t^{i,N},\mu_t^N,\mu_t^N) - \nabla_{\theta} L(\theta,x_t^{i},\mu_t^{[N]},\mu_t^{[N]})\right|\right|\right]&\leq K_{\theta_0,1}^{\frac{1}{2}}  \left[ \frac{K_{\theta_0,2}}{N}\right]^{\frac{1}{2}} + K_{\theta_0,3}^{\frac{1}{2}} \left[ \frac{K_{\theta_0,4}}{N}\right]^{\frac{1}{2}} \\
&\leq \frac{K_{\theta_0}}{\sqrt{N}}
\end{align}
and so, arguing similarly to \eqref{eQ32}, that for all $t>0$, 
\begin{align}
\mathbb{E}_{\theta_0}\left[||I_{1}^{i,N}(\theta,t)-I_{1}^{[i,N]}(\theta,t)||\right]
\leq  \frac{K_{\theta_0}}{\sqrt{N}}. \label{eQ360}
\end{align}
Combining inequalities \eqref{eQ346} and \eqref{eQ360}, and making use of the triangle inequality one final time, we have that the desired result. This establishes convergence in $\mathbb{L}^1$ as $N\rightarrow\infty$, for all $t>0$. It remains only to establish that the statements of the lemma also hold for $\mathcal{L}_t^N(\theta)$, which now follows almost trivially. We omit the calculations, which are essentially identical to those used at the end of the proof of Lemma \ref{lemma0_1}. 
\end{proof}

\begin{lemma} \label{lemmaC_1}
Assume that Conditions \ref{assumption_init}, \ref{assumption1} - \ref{assumption2}, and \ref{assumption3} hold. Then, for all $\theta\in\mathbb{R}^p$, for all $t\geq0$, for all $i=1,\dots,N$, we have, in $\mathbb{L}^1$, that
\begin{align}
\lim_{N\rightarrow\infty} ||\tfrac{1}{t}\nabla_{\theta}\mathcal{L}_t^{[i,N]}(\theta)||=||\tfrac{1}{t}\nabla_{\theta}\underline{\mathcal{L}}_t^{i}(\theta)||~~~\text{and}~~~ 
\lim_{N\rightarrow\infty} ||\tfrac{1}{t}\nabla_{\theta}\mathcal{L}_t^{[N]}(\theta)||=||\tfrac{1}{t}\nabla_{\theta}\underline{\mathcal{L}}_t^{i}(\theta)||.
\end{align}
In addition, there exists a positive constant $K_{\theta_0}$ such that, for all $\theta\in\mathbb{R}^p$, for all $t>0$, for all $N\in\mathbb{N}$, and for all $i=1,\dots,N$, 
\begin{align}
\mathbb{E}_{\theta_0}\left[\left|\left|\frac{1}{t}\nabla_{\theta}\underline{\mathcal{L}}_t^{i}(\theta) - \frac{1}{t}\nabla_{\theta}\mathcal{L}_t^{[i,N]}(\theta)\right|\right|\right]&\leq K_{\theta_0}{\alpha(N)}\left(1+\frac{1}{\sqrt{t}}\right),
\end{align}
where the function $\alpha:\mathbb{N}\rightarrow\mathbb{R}_{+}$ is defined according to 
\begin{equation}
\alpha(N)= \left\{
\begin{array}{lll} 
N^{-\frac{1}{4}} & \text{if} & d=1 \\
N^{-\frac{1}{4}}\log(1+N)^{\frac{1}{2}} & \text{if} & d=2 \\
N^{-\frac{1}{2d}} & \text{if} & d\geq 3.
\end{array}
\right. \label{eq:alpha2}
\end{equation}
Moreover, this bound also holds if $\mathcal{L}_t^{[i,N]}(\cdot)$ is replaced by $\mathcal{L}_t^{[N]}(\cdot)$.
\end{lemma}

\begin{proof}
The proof of this result is very similar to the proof of the previous result. We will just sketch the main details. Recall that
\begin{align}
\frac{1}{t}\nabla_{\theta}{\mathcal{L}}^{[i,N]}_t(\theta)&= \underbrace{\frac{1}{t}\int_0^t \nabla_{\theta} L(\theta,x_s^{i},\mu_s^{[N]},\mu_s^{[N]})\mathrm{d}s}_{I^{[i,N]}_{1}(\theta,t)}+ \underbrace{\frac{1}{t}\int_0^t \left\langle \nabla_{\theta} G(\theta,x_s^{i},\mu_s^{[N]},\mu_s^{[N]}),\mathrm{d}w_s^{i}\right\rangle}_{I^{[i,N]}_{2}(\theta,t)} \\
\frac{1}{t}\nabla_{\theta}{\underline{\mathcal{L}}}^{i}_t(\theta)&= \underbrace{\frac{1}{t}\int_0^t \nabla_{\theta} L(\theta,x_s^{i},\mu_s,\mu_s)\mathrm{d}s}_{I^{i}_{1}(\theta,t)}+ \underbrace{\frac{1}{t}\int_0^t \left\langle \nabla_{\theta} G(\theta,x_s^{i},\mu_s,\mu_s),\mathrm{d}w_s^{i}\right\rangle}_{I^{i}_{2}(\theta,t)} 
\end{align}
Using Assumption \ref{assumption3} and Proposition \ref{prop_moment_bounds} (bounded moments of the McKean-Vlasov SDE), there exists a constant $K_{\theta_0}>0$ (independent of $\theta$), such that
\begin{align}
\sup_{t\geq 0}\mathbb{E}_{\theta_0}\left[||\nabla_{\theta} G(\theta,x_t^{i},\mu_t^{[N]},\mu_t^{[N]}) - \nabla_{\theta} G(\theta,x_t^{i},\mu_t,\mu_t)||^2\right] &\leq K_{\theta_0}\sup_{t\geq 0}\mathbb{E}_{\theta_0}\left[ \mathbb{W}_2(\mu_t^{[N]},\mu_t)\right] \\
&\leq K_{\theta_0} \alpha(N)^2 .
\end{align}
where the final inequality follows from Theorem 1 in \cite{Fournier2015} (the rate of convergence of the empirical law of the McKean-Vlasov SDE to the true law of the McKean-Vlasov SDE). It follows, arguing as in Lemma \ref{lemmaC}, that
 \begin{align}
 \mathbb{E}_{\theta_0}\big[||I_1^{[i,N]}(\theta,t) - I_1^{i}(\theta,t)||\big] &\leq K_{\theta_0} \sqrt{\alpha(N)^2}~~~,~~~\mathbb{E}_{\theta_0}\big[||I_2^{[i,N]}(\theta,t) - I_2^{i}(\theta,t)||\big] \leq K_{\theta_0} \sqrt{\frac{\alpha(N)^2}{t}}. \label{eq:i2_bound}
\end{align}
Combining these bounds, and using the triangle inequality, we obtain the stated bound for $\mathcal{L}_t^{[i,N]}(\cdot)$. The result for $\mathcal{L}_t^{[N]}(\cdot)$ follows by arguing as at the end of Lemma \ref{lemma0_1}. 
\end{proof}

\begin{lemma} \label{lemmaD}
Assume that Conditions  \ref{assumption_init}, \ref{assumption1} - \ref{assumption2}, \ref{assumption3}, and \ref{assumption0} hold. Then, for all $N\in\mathbb{N}$, we have, both almost surely and in $\mathbb{L}^1$, that
\begin{align}
\lim_{t\rightarrow\infty}||\nabla_{\theta}\tilde{\mathcal{L}}^{i,N}({\theta}^{i,N}_t)||=0~~~\text{and}~~~\lim_{t\rightarrow\infty}||\nabla_{\theta}\tilde{\mathcal{L}}^{N}({\theta}^{N}_t)||=0.
\end{align}
\end{lemma}

\begin{proof}
We will prove the first statement of the lemma, with the second proved identically.
We will use a modified version of the approach in \cite{Sirignano2017a}, which itself is a continuous-time version of the approach first introduced in \cite{Bertsekas2000}.
We will require the following additional notation. Define an arbitrary constant $\kappa>0$, with $\lambda = \lambda(\kappa)>0$ to be determined. Set $\sigma=0$, and define the cycle of random stopping times 
$0 = \sigma_0\leq \tau_1\leq \sigma_1\leq \tau_2\leq \sigma_2\leq \dots$
via
\begin{align}
\tau_k &= \inf\big\{t>\sigma_{k-1}:||\nabla_{\theta}\tilde{\mathcal{L}}^{i,N}(\theta^{i,N}_t)||\geq \kappa\big\} \label{tau} \\[2mm]
\sigma_k&=\sup\big\{ t>\tau_k:\tfrac{1}{2}||\nabla \tilde{\mathcal{L}}^{i,N}({\theta}^{i,N}_{\tau_k})||\leq ||\nabla\tilde{\mathcal{L}}^{i,N}({\theta}^{i,N}_{s})||\leq 2|| \label{sigma} \nabla\tilde{\mathcal{L}}^{i,N}({\theta}^{i,N}_{\tau_k})||~~\forall s\in[\tau_k,t],\\
&\hspace{30mm}~~~\int_{\tau_k}^t \gamma(s)\mathrm{d}s\leq \rho\big\} 
\end{align}
The purpose of these stopping times is to control the periods of time for which $||\nabla\tilde{\mathcal{L}}^{i,N}(\theta_t^{i,N})||$ is close to zero, and those for which it is away from zero.  In addition, let $\eta>0$, and set $\sigma_{k,\eta} = \sigma_k+\eta$. We are now ready to prove this result. First consider the case in which there are a finite number of stopping times $\tau_k$. In this case, there exists finite $t_0$ such that, for all $t\geq t_0$, $\smash{||\nabla_{\theta}\tilde{\mathcal{L}}^{i,N}(\theta^{i,N}_t)||<\kappa}$. Now consider the case in which there are an infinite number of stopping times $\tau_k$. Then, using Lemmas \ref{sub_lemma_3} - \ref{sub_lemma_4} (see Appendix \ref{sec:theorem1_lemmas}), there exist $0<\beta_1<\beta$, and $k_0\in\mathbb{N}$, such that for all $k\geq k_0$, almost surely,
\begin{align}
\tilde{\mathcal{L}}^{i,N}(\theta^{i,N}_{\sigma_k}) - \tilde{\mathcal{L}}^{i,N}(\theta_{\tau_k}^{i,N})\geq \beta 
~~~\text{and}~~~\tilde{\mathcal{L}}^{i,N}(\theta^{i,N}_{\tau_k}) - \tilde{\mathcal{L}}^{i,N}(\theta_{\sigma_{k-1}}^{i,N})\geq -\beta_1. 
\end{align}
It follows straightforwardly that
\begin{align}
\tilde{\mathcal{L}}^{i,N}(\theta^{i,N}_{\tau_{n+1}}) - \tilde{\mathcal{L}}^{i,N}(\theta^{i,N}_{\tau_{k_0}})&= \sum_{k=k_0}^n\left[\tilde{\mathcal{L}}^{i,N}(\theta^{i,N}_{\sigma_k}) - \tilde{\mathcal{L}}^{i,N}(\theta_{\tau_k}^{i,N})+\tilde{\mathcal{L}}^{i,N}(\theta^{i,N}_{\tau_{k+1}}) - \tilde{\mathcal{L}}^{i,N}(\theta_{\sigma_{k}}^{i,N})\right] 
\\
&\geq (n+1-k_0)(\beta-\beta_1) \label{eq4104}
\end{align}
Since $\beta-\beta_1>0$, this implies that $\smash{\tilde{\mathcal{L}}^{i,N}(\theta^{i,N}_{\tau_{n+1}})\rightarrow\infty}$ as $n\rightarrow\infty$. But this contradicts Lemma \ref{lemmab} (see Appendix \ref{sec:theorem1_lemmas}), which states that $\tilde{\mathcal{L}}^{i,N}(\theta)$ is bounded from above. Thus, there must exist a finite time $t_0$ such that, for all $t\geq t_0$, $\smash{||\tilde{\mathcal{L}}^{i,N}(\theta_t^{i,N})||<\kappa}$. Since our original choice of $\kappa$ was arbitrary, this completes the proof that, for all $N\in\mathbb{N}$, $\lim_{t\rightarrow\infty}||\nabla_{\theta}\tilde{\mathcal{L}}^{i,N}(\theta_t^{i,N})||=0$ a.s. Finally, we observe that, by Lemma \ref{lemmab}, $||\nabla_{\theta}\tilde{\mathcal{L}}^{i,N}(\theta)||$ is bounded above for all $\theta\in\mathbb{R}^p$. Thus, we also have convergence in $\mathbb{L}^1$ via Lebesgue's dominated convergence theorem (e.g., \cite[Chapter 5]{Weir1973}). 
\end{proof}

\subsection{Proof of Theorem \ref{theorem1_2}, Theorem \ref{theorem1_2_star}, Theorem \ref{theorem2_2}, and Theorem \ref{theorem2_2_star}} \label{sec:proof2}
\begin{proof}
The proof of this result closely follows the proof of Theorem 2.7 in \cite{Sirignano2020a}, adapted appropriately to our particular case. We will focus on Theorem \ref{theorem2_2_star}, later outlining how our approach can be adapted to obtain the other results. We begin by rewriting the parameter update equation in the following form
\begin{align}
\mathrm{d}\theta_t^{i,N} 
&= \gamma_t\nabla_{\theta}\tilde{\underline{\mathcal{L}}}(\theta_t^{i,N})\mathrm{d}t + \gamma_t\big(\nabla_{\theta}\tilde{\mathcal{L}}^{i,N}(\theta_t^{i,N}) - \nabla_{\theta}\tilde{\underline{\mathcal{L}}}(\theta_t^{i,N})\big)\mathrm{d}t  \label{eq4110}  \\
&+ \gamma_t\big(\nabla_{\theta}L(\theta_t^{i,N},x_t^{i,N},\mu_t^N)-\nabla_{\theta}\tilde{\mathcal{L}}^{i,N}(\theta_t^{i,N})\big)\mathrm{d}t  + \gamma_t\nabla_{\theta}B(\theta_t^{i,N},x_t^{i,N},\mu_t^N)\mathrm{d}w_t^{i}
\end{align}
Using a first order Taylor expansion, we have that 
\begin{align}
\nabla_{\theta}\tilde{\underline{\mathcal{L}}}(\theta_t^{i,N}) &= \nabla_{\theta}\tilde{\underline{\mathcal{L}}}(\theta_0) + \nabla_{\theta}^2\tilde{\underline{\mathcal{L}}}(\underline{\tilde{\theta}}_t^{i,N})(\theta_t^{i,N}- \theta_{0}) =  \nabla^2\tilde{\underline{\mathcal{L}}}(\tilde{\underline{\theta}}_t^{i,N})(\theta_t^{i,N}- \theta_{0}) \label{eq_4_107} 
\end{align}
where $\nabla^2\underline{\mathcal{L}}(\cdot)$ denotes the Hessian, and $\underline{\tilde{\theta}}_t^{i,N}$ is a point in the segment connecting $\theta_t^{i,N}$ and $\theta_0$. Substituting \eqref{eq_4_107} into \eqref{eq4110}, we obtain the following equations for $Z_t^{i,N} = \theta_t^{i,N} - \theta_{0}$
\begin{align}
\mathrm{d}Z_t^{i,N}&= \gamma_t\nabla^2_{\theta}\tilde{\underline{\mathcal{L}}}(\tilde{\underline{\theta}}_t^{i,N})Z_t^{i,N}\mathrm{d}t + \gamma_t\big(\nabla_{\theta}\tilde{\mathcal{L}}^{i,N}(\theta_t^{i,N}) - \nabla_{\theta}\tilde{\underline{\mathcal{L}}}(\theta_t^{i,N})\big)\mathrm{d}t  \\
&+ \gamma_t\big(\nabla_{\theta}L(\theta_t^{i,N},x_t^{i,N},\mu_t^N)-\nabla_{\theta}\tilde{\mathcal{L}}^{i,N}(\theta_t^{i,N})\big)\mathrm{d}t + \gamma_t\nabla_{\theta}B(\theta_t^{i,N},x_t^{i,N},\mu_t^N)\mathrm{d}w_t^{i}.
\end{align}
Applying It\^o's formula to the function $||\cdot||^2$,
and using the strong concavity of $\tilde{\underline{\mathcal{L}}}(\theta)$ (Condition \ref{assumption4}), it follows that
\begin{align}
\mathrm{d}||Z_t^{i,N}||^2 + 2\eta\gamma_t||Z_t^{i,N}||^2\mathrm{d}t  &\leq \gamma_t\big\langle Z_t^{i,N}, \nabla_{\theta}\tilde{\mathcal{L}}^{i,N}(\theta_t^{i,N})-\nabla_{\theta}\tilde{\underline{\mathcal{L}}}(\theta_t^{i,N})\big\rangle\mathrm{d}t \label{eq_4_111} \\
&+\gamma_t\big\langle Z_t^{i,N}, \nabla_{\theta}L(\theta_t^{i,N},x_t^{i,N},\mu_t^N)-\nabla_{\theta}\tilde{\mathcal{L}}^{i,N}(\theta_t^{i,N})\big\rangle\mathrm{d}t \nonumber \\
&+ \gamma_t\big\langle Z_t^{i,N}, \nabla_{\theta}B(\theta_t^{i,N},x_t^{i,N},\mu_t^N)\mathrm{d}w_t^{i}\big\rangle \\
&+ \gamma_t^2 \big|\big|\nabla_{\theta} B(\theta_t^{i,N},x_t^{i,N},\mu_t^N)\big|\big|_{F}^2\mathrm{d}t \nonumber
\end{align}
where $||\cdot||_{F}$ is the Frobenius norm. Now, let us define the function $\Phi_{t,t'}= \exp[-2\eta\int_{t}^{t'}\gamma_u\mathrm{d}u]$, with $\partial_{t}\Phi_{t,t'} = 2\eta\gamma_t \Phi_{t,t'}$. Using the product rule, and \eqref{eq_4_111},  we obtain
\begin{align}
\mathrm{d}\left[\Phi_{t,t'}||Z_t^{i,N}||^2\right] &= \Phi_{t,t'}\left[\mathrm{d}||Z_t^{i,N}||^2 + 2\eta\gamma_t||Z_t^{i,N}||^2\mathrm{d}t\right] \\
&\leq \gamma_t\Phi_{t,t'}\big\langle Z_t^{i,N}, \nabla_{\theta}\tilde{\mathcal{L}}^{i,N}(\theta_t^{i,N})-\nabla_{\theta}\tilde{\underline{\mathcal{L}}}(\theta_t^{i,N})\big\rangle\mathrm{d}t \\
&+\gamma_t\Phi_{t,t'}\big\langle Z_t^{i,N}, \nabla_{\theta}L(\theta_t^{i,N},x_t^{i,N},\mu_t^N)-\nabla_{\theta}\tilde{\mathcal{L}}^{i,N}(\theta_t^{i,N})\big\rangle\mathrm{d}t \nonumber \\
&+ \gamma_t\Phi_{t,t'}\big\langle Z_t^{i,N}, \nabla_{\theta}B(\theta_t^{i,N},x_t^{i,N},\mu_t^N)\mathrm{d}w_t^{i}\big\rangle + \gamma_t^2\Phi_{t,t'}\big|\big|\nabla_{\theta} B(\theta_t^{i,N},x_t^{i,N},\mu_t^N)\big|\big|_{F}^2\mathrm{d}t \nonumber
\end{align} 
Finally, rewriting this in integral form, setting $t'=t$, and taking expectations, we arrive at
\begin{align}
\mathbb{E}_{\theta_0}\left[||Z_t^{i,N}||^2\right] &\leq  \mathbb{E}_{\theta_0}\left[\Phi_{1,t}||Z_1^{i,N}||^2\right] + \mathbb{E}_{\theta_0}\left[\int_1^t \gamma_s\Phi_{s,t}\big\langle Z_s^{i,N}, \nabla_{\theta}\tilde{\mathcal{L}}^{i,N}(\theta_s^{i,N})-\nabla_{\theta}\tilde{\underline{\mathcal{L}}}(\theta_s^{i,N})\big\rangle\mathrm{d}s\right] \label{eq_defs} \hspace{-15mm} \\ 
&\hspace{2mm}+ \mathbb{E}_{\theta_0}\left[\int_1^t\gamma_s\Phi_{s,t}\big\langle Z_s^{i,N}, \nabla_{\theta}L(\theta_s^{i,N},x_s^{i,N},\mu_s^N)-\nabla_{\theta}\tilde{\mathcal{L}}^{i,N}(\theta_s^{i,N})\big\rangle\mathrm{d}s\right] \nonumber \\
&\hspace{2mm}+ \mathbb{E}_{\theta_0}\left[\int_1^t\gamma_s^2\Phi_{s,t} \big|\big|\nabla_{\theta} B(\theta_s^{i,N},x_s^{i,N},\mu_s^N)\big|\big|_{F}^2\mathrm{d}s\right] \nonumber \\[2mm]
&=\mathbb{E}_{\theta_0}\left[\Omega_{t,i,N}^{(1)}\right] + \mathbb{E}_{\theta_0}\left[\Omega_{t,i,N}^{(2)}\right] + \mathbb{E}_{\theta_0}\left[\Omega_{t,i,N}^{(3)}\right] + \mathbb{E}_{\theta_0}\left[\Omega_{t,i,N}^{(4)}\right]. \label{eq_defs_2}
\end{align}
We will deal with each of these terms separately, beginning with $\Omega_{t,i,N}^{(1)}$. For this term, by Lemma \ref{lemma_theta_moments} (the moment bounds for $\theta_s^{i,N}$), and Condition \ref{assumption0} (the conditions on the learning rate), for sufficiently large $t$ we have that
\begin{equation}
\mathbb{E}_{\theta_0}\left[\Omega_{t,i,N}^{(1)}\right] 
= \Phi_{1,t}\mathbb{E}_{\theta_0}\left[||Z_1^{i,N}||^2\right] \leq K_{\theta_0}^{(1)} \gamma_t \label{eq3155}
\end{equation}
We now turn our attention to $\smash{\Omega_{t,i,N}^{(2)}}$. For this term, using Lemmas \ref{lemma0}, \ref{lemma0_1}, \ref{lemmaC}, and \ref{lemmaC_1}, we have that, for all $\theta\in\mathbb{R}^p$, 
\begin{align}
||\nabla_{\theta}\tilde{\mathcal{L}}^{i,N}(\theta) - \nabla_{\theta}\tilde{\underline{\mathcal{L}}}(\theta)|| &= \lim_{t\rightarrow\infty} ||\nabla_{\theta}\tilde{\mathcal{L}}^{i,N}(\theta) - \nabla_{\theta}\tilde{\underline{\mathcal{L}}}(\theta)|| \\
&\leq  \lim_{t\rightarrow\infty} \left[ ||\nabla_{\theta}\tilde{\underline{\mathcal{L}}}(\theta) - \tfrac{1}{t}\nabla_{\theta}\underline{\mathcal{L}}_t^{i}(\theta)| + ||\tfrac{1}{t}\nabla_{\theta}\underline{\mathcal{L}}_t^{i}(\theta) - \tfrac{1}{t}\nabla_{\theta}{\mathcal{L}}_t^{[i,N]}(\theta)||\right. \\ 
&\hspace{10mm}\left.+||\tfrac{1}{t}\nabla_{\theta}{\mathcal{L}}_t^{[i,N]}(\theta) - \tfrac{1}{t}\nabla_{\theta}{\mathcal{L}}_t^{i,N}(\theta)|| + ||\tfrac{1}{t}\nabla_{\theta}{\mathcal{L}}_t^{i,N}(\theta) - \nabla_{\theta}\tilde{\mathcal{L}}^{i,N}(\theta)||\right] \label{l_bound_a} \\
&\leq K_{\theta_0,1}\alpha(N) + \frac{K_{\theta_0,2}}{N^{\frac{1}{2}}}, \hspace{-4mm} \label{l_bound}
\end{align}
where $\alpha:\mathbb{N}\rightarrow\mathbb{R}_{+}$ is the function defined in \eqref{eq:alpha}.
Substituting this bound into \eqref{eq_defs}, and using Condition \ref{assumption0} (the conditions on the learning rate) to bound the final integral, we obtain
\begin{align}
\mathbb{E}_{\theta_0}\left[\Omega_{t,i,N}^{(2)}\right] &\leq \int_{1}^t \gamma_s\Phi_{s,t} \mathbb{E}_{\theta_0}\left[||Z_s^{i,N}|| \hspace{.5mm}\sup_{\theta_{s}^{i,N}} ||\nabla_{\theta}\tilde{\mathcal{L}}^{i,N}(\theta_s^{i,N}) - \nabla_{\theta}\tilde{\underline{\mathcal{L}}}(\theta_s^{i,N})||\right] \mathrm{d}s  \\
&\leq \left[K_{\theta_0,1}\alpha(N) + \frac{K_{\theta_0,2}}{N^{\frac{1}{2}}}\right] \int_{1}^t \gamma_s\Phi_{s,t}\mathrm{d}s \leq K_{\theta_0}^{(2,1)}\alpha(N) + \frac{K_{\theta_0}^{(2,2)}}{N^{\frac{1}{2}}}. \label{eq4168} 
\end{align}
We now turn our attention to $\smash{\Omega_{t,i,N}^{(3)}}$. We will analyse this term by constructing an appropriate Poisson equation. Let us define $R^{i,N}(\theta,\hat{x}^N) = \langle \theta-\theta_{0},\nabla_{\theta}\hat{L}^{i,N}(\theta,\hat{x}^N)-\nabla_{\theta}\tilde{\mathcal{L}}^{i,N}(\theta)\rangle$, 
where, as previously, $\hat{x}^N = (x^{1,N},\dots,x^{N,N})$. This function satisfies all of the conditions of Lemma \ref{lemma_poisson_2}. Thus, by Lemma \ref{lemma_poisson_2}, the Poisson equation
\begin{equation}
\mathcal{A}_{x} v^{i,N}(\theta,\hat{x}^N) = R^{i,N}(\theta,\hat{x}^N)~~~,~~~\int_{\mathbb{R}^{d}} v^{i,N}(\theta,\hat{x}^N)\hat{\mu}^{N}_{\infty}(\mathrm{d}\hat{x}^N)=0
\end{equation}
has a unique twice differentiable solution which satisfies $\sum_{j=0}^2 |\frac{\partial^j v^{i,N}}{\partial \theta^{i}}(\theta,\hat{x}^N)| + |\frac{\partial^2 v^{i,N}}{\partial \theta\partial x}(\theta,\hat{x}^N)|\leq K(1+||x^{i,N}||^q + \frac{1}{N}\sum_{j=1}^N ||x^{j,N}||^q)$. Now, by It\^o's formula, we have that
\begin{align}
v^{i,N}(\theta_t^{i,N},\hat{x}_t^N) - v^{i,N}(\theta_s^{i,N},\hat{x}_s^N)
&= \int_s^t \mathcal{A}_{\theta} v^{i,N}(\theta_u^{i,N},\hat{x}^N_u)\mathrm{d}u + \int_s^{t} \mathcal{A}_{\hat{x}^N} v^{i,N}(\theta_u^{i,N},\hat{x}^N_u)\mathrm{d}u \hspace{-5mm} \label{eq_4147} \\
&+\int_s^t \gamma_u\partial_{\theta}v^{i,N}(\theta_u^{i,N},\hat{x}^N_u)\nabla_{\theta}\hat{B}^{i,N}(\theta_u,\hat{x}^N_u)\mathrm{d}w_u^{i}  \\\
&+\int_s^t \partial_{x}v^{i,N}(\theta_u^{i,N},\hat{x}_u^N)\mathrm{d}\hat{w}_u^N\nonumber \\
&+\int_s^{t}\gamma_u\left[\partial_{\theta}\partial_{\hat{x}}v^{i,N}(\theta_u^{i,N},\hat{x}_u^N)\nabla_{\theta} \hat{B}^{i,N}(\theta_u^{i,N},\hat{x}_u^N)\right]\mathrm{d}u \nonumber
\end{align}
where $\hat{w}_u^N$ was defined in \eqref{big_SDE}. It follows, now writing $v_t^{i,N}:= v^{i,N}(\theta_t^{i,N},\hat{x}_t^N)$, that
\begin{align}
R^{i,N}(\theta_t^{i,N},\hat{x}_t^N)\mathrm{d}t& = \mathcal{A}_{\hat{x}^N} v^{i,N}(\theta_t^{i,N},\hat{x}_t^N)\mathrm{d}t \\
&=\mathrm{d}v_t^{i,N} - \mathcal{A}_{\theta}v^{i,N}(\theta_t^{i,N},\hat{x}_t^N)\mathrm{d}t - \gamma_t\partial_{\theta}v^{i,N}(\theta_t^{i,N},\hat{x}_t^N)\nabla_{\theta}\hat{B}^{i,N}(\theta_t^{i,N},\hat{x}_t^N)\mathrm{d}w_t^{i} \\
&~~~- \partial_{\hat{x}}v^{i,N}(\theta_t^{i,N},\hat{x}_t^N)\mathrm{d}\hat{w}_t^N- \gamma_t\left[\partial_{\theta}\partial_{\hat{x}}v^{i,N}(\theta_t^{i,N},\hat{x}_t^N)\nabla_{\theta} \hat{B}^{i,N}(\theta_t^{i,N},\hat{x}_t^N)\right]\mathrm{d}t \nonumber 
\end{align}
Thus, we can rewrite $\Omega_{t,i,N}^{(3)}$ as
\begin{align}
\Omega_{t,i,N}^{(3)} &= \int_1^t\gamma_s\Phi_{s,t}\underbrace{\left\langle \theta_s^{i,N}-\theta_{0}, \nabla_{\theta}\hat{L}^{i,N}(\theta_s^{i,N},\hat{x}_s^N)-\nabla_{\theta}\tilde{\mathcal{L}}^{i,N}(\theta_s^{i,N})\right\rangle\mathrm{d}s}_{R^{i,N}(\theta_s^{i,N},\hat{x}_s^N)\mathrm{d}s} \label{eq3163} \\[-2.5mm]
&=\int_1^t \gamma_s \Phi_{s,t}\mathrm{d}v_s^{i,N}- \int_{1}^t\gamma_s\Phi_{s,t}\mathcal{A}_{\theta}v^{i,N}(\theta_s^{i,N},\hat{x}_s^N)\mathrm{d}s \\
&-  \int_{1}^t\gamma_s^2\Phi_{s,t}\partial_{\theta}v^{i,N}(\theta_s^{i,N},\hat{x}_s^N)\nabla_{\theta}\hat{B}^{i,N}(\theta_s^{i,N},\hat{x}_s^N)\mathrm{d}w_s^{i} \\
&-  \int_{1}^t\gamma_s\Phi_{s,t}\partial_{\hat{x}}v^{i,N}(\theta_s^{i,N},\hat{x}_s^N)\mathrm{d}\hat{w}^N_s \\
&-  \int_{1}^t\gamma^2_s\Phi_{s,t}\partial_{\theta}\partial_{x}v^{i,N}(\theta_s,\hat{x}_s^N)\nabla_{\theta} \hat{B}^{i,N}(\theta_s^{i,N},\hat{x}_s^N)\mathrm{d}s \nonumber
\end{align}
We can rewrite the first term in this expression by applying It\^o's formula to $f(s,v_s) = \gamma_s\Phi_{s,t}v_s$. This yields 
\begin{equation}
\gamma_t\Phi_{t,t}v_t^{i,N} - \gamma_1\Phi_{1,t}v_1^{i,N} = \int_1^t \gamma_s\Phi_{s,t}\mathrm{d}v_s^{i,N} + \int_{1}^t\dot{\gamma}_s\Phi_{s,t}v_s^{i,N}\mathrm{d}s + \int_{1}^t2\eta\gamma^2_s\Phi_{s,t}v_s^{i,N}\mathrm{d}s. 
\end{equation}
Substituting the resulting expression for $\int_1^t \gamma_s\Phi_{s,t}\mathrm{d}v_s^{i,N}$ into \eqref{eq3163}, and taking expectations, we obtain
\begin{align}
\mathbb{E}_{\theta_0}\left[\Omega_{t,i,N}^{(3)}\right] &= \mathbb{E}_{\theta_0}\left[\gamma_t\Phi_{t,t}v^{i,N}(\theta_t^{i,N},\hat{x}_t^N)\right]- \mathbb{E}_{\theta_0}\left[\gamma_1\Phi_{1,t}v^{i,N}(\theta_1^{i,N},\hat{x}_1^N)\right] \\
&- \mathbb{E}_{\theta_0}\left[\int_{1}^t\dot{\gamma}_s\Phi_{s,t}v^{i,N}(\theta_s^{i,N},\hat{x}_s^N)\mathrm{d}s\right] - \mathbb{E}_{\theta_0}\left[\int_{1}^t2\eta\gamma^2_s\Phi_{s,t}v^{i,N}(\theta_s^{i,N},\hat{x}_s^N)\mathrm{d}s\right] \nonumber \\
&- \mathbb{E}_{\theta_0}\left[\int_{1}^t\gamma_s\Phi_{s,t}\mathcal{A}_{\theta}v^{i,N}(\theta_s^{i,N},\hat{x}_s^N)\mathrm{d}s\right] \\
&- \mathbb{E}_{\theta_0}\left[ \int_{1}^t\gamma^2_s\Phi_{s,t}\partial_{\theta}\partial_{x}v^{i,N}(\theta_s^{i,N},\hat{x}_s^N)\nabla_{\theta} \hat{B}^{i,N}(\theta_s^{i,N},\hat{x}_s^N)\mathrm{d}s\right] \nonumber \\
&\leq K_{\theta_0}\left[\gamma_t + \int_1^t \left(\dot{\gamma}_s+\gamma_s^2\right)\Phi_{s,t}\mathrm{d}s\right]\leq K_{\theta_0}^{(3)}\gamma_t,\label{eq3168}
\end{align}
where in the penultimate inequality we have used the polynomial growth of $v^{i,N}(\theta,\hat{x}^N)$ and $\partial_{\theta}\partial_{x}v^{i,N}(\theta,\hat{x}^N)$, Condition \ref{assumption3}(ii) (which implies the polynomial growth of $\nabla_{\theta}\hat{B}^{i,N}(\theta,\hat{x}^N)$), Proposition \ref{prop_moment_bounds} (the moment bounds for $\hat{x}_t^N$), Lemma \ref{lemma_theta_moments} (the moment bounds for $\theta_s^{i,N}$), and in the final inequality we have used Condition \ref{assumption0} (the conditions on the learning rate). It remains only to bound $\smash{\Omega_{t,i,N}^{(4)}}$. For this term, using the same assumptions, we obtain
\begin{equation}
\mathbb{E}\left[\Omega_{t,i,N}^{(4)}\right] = \mathbb{E}_{\theta_0}\left[\int_1^t\gamma_s^2\Phi_{s,t} \big|\big|\nabla_{\theta} B(\theta_s,x_s,\mu_s)\big|\big|_{F}^2\mathrm{d}s\right]\leq K_{\theta_0}\int_1^t\gamma_s^2\Phi_{s,t}\mathrm{d}s\leq K_{\theta_0}^{(4)}\gamma_t. \label{eq3169}
\end{equation}
Combining inequalities \eqref{eq3155}, \eqref{eq4168}, \eqref{eq3168},  and \eqref{eq3169}, and setting $K_{\theta_0,1}^{\dagger} = \max\{K_{\theta_0}^{(1)},K^{(3)}\}$, $K_{\theta_0,2}^{\dagger} = K_{\theta_0}^{(4)}$, $K_{\theta_0,3}^{\dagger} = K_{\theta_0}^{(2,1)}$, and $K_{\theta_0,4}^{\dagger} = K_{\theta_0}^{(2,2)}$, we thus have that
\begin{align}
\mathbb{E}_{\theta_0}\left[||\theta_t^{i,N}-\theta_{0}||^2\right]
&\leq (K_{\theta_0,1}^{\dagger}+K_{\theta_0,2}^{\dagger})\gamma_t + {K_{\theta_0,3}^{\dagger}}\alpha(N) + \frac{K_{\theta_0,4}^{\dagger}}{N^{\frac{1}{2}}}, \label{eq_final}
\end{align}
which completes the proof of the first statement of the theorem. Let us now turn our attention to the second statement. The proof of this bound goes through almost verbatim. Let us briefly highlight the main points of difference. To begin, we now have the following decomposition of the parameter update equation
\begin{align}
\mathrm{d}\theta_t^{N} 
&= \gamma_t\nabla_{\theta}\tilde{\underline{\mathcal{L}}}(\theta_t^{N})\mathrm{d}t +\gamma_t \frac{1}{N}\sum_{i=1}^N\big(\nabla_{\theta}\tilde{\mathcal{L}}^{i,N}(\theta_t^{N}) - \nabla_{\theta}\tilde{\underline{\mathcal{L}}}(\theta_t^{N})\big)\mathrm{d}t \\
&+ \gamma_t\frac{1}{N}\sum_{i=1}^N\big(\nabla_{\theta}L(\theta_t^{N},x_t^{i,N},\mu_t^N)-\nabla_{\theta}\tilde{\mathcal{L}}^{i,N}(\theta_t^{N})\big)\mathrm{d}t + \gamma_t\frac{1}{N}\sum_{i=1}^N\nabla_{\theta}B(\theta_t^{i,N},x_t^{i,N},\mu_t^N)\mathrm{d}w_t^{i} .\nonumber
\end{align}
Using a first order Taylor expansion around $\theta_0$, defining $Z_t^N = \theta_t^N - \theta_0$, applying It\^o's formula to the function $||Z_t^N||^2$, and using the strong concavity of $\tilde{\mathcal{L}}(\theta)$, as in \eqref{eq_4_107} - \eqref{eq_4_111}, we obtain
\begin{align}
\mathrm{d}||Z_t^{N}||^2 + 2\eta\gamma_t||Z_t^{N}||^2\mathrm{d}t  &\leq \gamma_t \frac{1}{N}\sum_{i=1}^N \big\langle Z_t^{N}, \nabla_{\theta}\tilde{\mathcal{L}}^{N}(\theta_t^{N})-\nabla_{\theta}\tilde{\underline{\mathcal{L}}}(\theta_t^{N})\big\rangle\mathrm{d}t \\
&+\gamma_t\frac{1}{N}\sum_{i=1}^N\big\langle Z_t^{N}, \nabla_{\theta}L(\theta_t^{N},x_t^{i,N},\mu_t^N)-\nabla_{\theta}\tilde{\mathcal{L}}^{i,N}(\theta_t^{N})\big\rangle\mathrm{d}t \nonumber \\
&+ \gamma_t\frac{1}{N}\sum_{i=1}^N\big\langle Z_t^{N}, \nabla_{\theta}B(\theta_t^N,x_t^{i,N},\mu_t^N)\mathrm{d}w_t^{i}\big\rangle \\
&+ \gamma_t^2 \frac{1}{N^2}\sum_{i=1}^N\big|\big|\nabla_{\theta} B(\theta_t^N,x_t^{i,N},\mu_t^N)\big|\big|_{F}^2\mathrm{d}t \nonumber
\end{align}
Continuing to follow our previous arguments, we finally arrive at 
\begin{align}
\mathbb{E}_{\theta_0}\left[||Z_t^{N}||^2\right] 
&\leq\frac{1}{N}\sum_{i=1}^N\left[\mathbb{E}_{\theta_0}\left[\tilde{\Omega}_{t,i,N}^{(1)}\right] + \mathbb{E}_{\theta_0}\left[\tilde{\Omega}_{t,i,N}^{(2)}\right] + \mathbb{E}_{\theta_0}\left[\tilde{\Omega}_{t,i,N}^{(3)}\right]\right] + \frac{1}{N^2}\sum_{i=1}^N\mathbb{E}_{\theta_0}\left[\tilde{\Omega}_{t,i,N}^{(4)}\right]
\end{align}
where, modulo changes in the arguments from $\theta_{t}^{i,N}$ to $\theta_{t}^N$, $\smash{\tilde{\Omega}_{t,i,N}^{(1)},\dots,\tilde{\Omega}_{t,i,N}^{(4)}}$ are identical to $\smash{{\Omega}_{t,i,N}^{(1)},\dots,{\Omega}_{t,i,N}^{(4)}}$ as defined in \eqref{eq_defs} - \eqref{eq_defs_2}. We thus have, using \eqref{eq3155}, \eqref{eq4168}, \eqref{eq3168}, and \eqref{eq3169}, that 
\begin{align}
\mathbb{E}_{\theta_0}\left[||Z_t^{N}||^2\right] &\leq  \frac{1}{N}\sum_{i=1}^N \left(K_{\theta_0,1}^{\dagger}\gamma_t + {K_{\theta_0,3}^{\dagger}}\alpha(N) + \frac{K_{\theta_0,4}^{\dagger}}{N^{\frac{1}{2}}}\right) + \frac{1}{N^2}\sum_{i=1}^N K_{\theta_0,2}^{\dagger}\gamma_t \\
&=(K_{\theta_0,1}^{\dagger}+\frac{K_{\theta_0,2}^{\dagger}}{N})\gamma_t + {K_{\theta_0,3}^{\dagger}}\alpha(N) + \frac{K_{\theta_0,4}^{\dagger}}{N^{\frac{1}{2}}}. \label{eq_final2}
\end{align}
It remains to outline how to adapt this proof to obtain the results in Theorems \ref{theorem1_2}, \ref{theorem1_2_star}, and \ref{theorem2_2}. Let us start with Theorem \ref{theorem2_2}. In this case, we replace $\smash{\tilde{\underline{\mathcal{L}}}(\cdot)}$ by $\smash{\tilde{{\mathcal{L}}}^{i,N}(\cdot)}$ in \eqref{eq4110}, use a first order Taylor expansion for $\smash{\tilde{{\mathcal{L}}}^{i,N}(\cdot)}$ instead of $\smash{\tilde{\underline{\mathcal{L}}}(\cdot)}$ in \eqref{eq_4_107}, and the strong concavity of $\tilde{\mathcal{L}}^{i,N}(\cdot)$ (Condition \ref{assumption4''}) rather than $\smash{\tilde{\underline{\mathcal{L}}}(\cdot)}$ (Condition \ref{assumption4}) in \eqref{eq_4_111}. Following the subsequent arguments, we see that the second term $\smash{\Omega_{t,i,N}^{(2)}}$ in \eqref{eq_defs_2} vanishes, with all other terms unchanged. The result is that, in \eqref{eq_final} and \eqref{eq_final2}, only the first two terms appear. This completes the proof of Theorem \ref{theorem2_2}. 

We now turn our attention to Theorem \ref{theorem1_2_star}. In this case, we replace $x_t^{i,N}$ by $x_t^{i}$, $\mu_t^{N}$ by $\mu_t^{[N]}$, $\smash{\theta_t^{i,N}}$ by $\smash{\theta_t^{[i,N]}}$, $\smash{\theta_t^{N}}$ by $\smash{\theta_t^{[N]}}$, $\smash{\mathcal{L}_t^{i,N}(\cdot)}$ by $\smash{\mathcal{L}_t^{[i,N]}(\cdot)}$, $\smash{\mathcal{L}_t^{N}(\cdot)}$ by $\smash{\mathcal{L}_t^{[N]}(\cdot)}$, $\smash{\tilde{\mathcal{L}}^{i,N}(\cdot)}$ by $\smash{\tilde{\mathcal{L}}^{[i,N]}(\cdot)}$, and $\smash{\tilde{\mathcal{L}}^{N}(\cdot)}$ by $\smash{\tilde{\mathcal{L}}^{[N]}(\cdot)}$. The arguments in our proof now go through essentially unchanged. The only material difference appears in \eqref{l_bound_a}, where the third term vanishes. The result is that only the second term in \eqref{l_bound} and hence \eqref{eq4168} vanishes. Thus, in particular, the final term in \eqref{eq_final} and \eqref{eq_final2} no longer appears. This completes the proof of Theorem \ref{theorem1_2_star}. Finally, Theorem \ref{theorem1_2} can be deduced from Theorem \ref{theorem1_2_star} in the same way that Theorem \ref{theorem2_2} was obtained from Theorem \ref{theorem2_2_star} (see above).
\end{proof}

\section{Numerical Examples}
\label{sec:numerics}
To illustrate the results of Section \ref{sec:results}, we now provide two illustrative examples of parameter estimation in McKean-Vlasov SDEs, and the associated systems of interacting particles. In particular, we consider a one-dimensional linear mean-field model with two unknown parameters, and a stochastic opinion dynamics model with a single unknown parameter. In both cases, we simulate sample paths and implement the recursive MLE using a standard Euler-Maruyama scheme with $\Delta t = 0.1$. In addition, we will focus on the case in which we observe trajectories $(x_t^{i,N})_{t\geq 0}$ from the IPS (Case II) rather than independent trajectories $(x_t^{i})_{t\geq 0}$ from the McKean-Vlasov SDE (Case I). 

\subsection{Linear Mean Field Dynamics} \label{sec:numerics1}
We first consider a one-dimensional linear mean field model, parametrised by $\theta=(\theta_1,\theta_2)^T\in\mathbb{R}^2$, given by
\begin{align}
\mathrm{d}x_t &= -\left[\theta_1 x_t + \theta_2 \int_{\mathbb{R}} (x_t - y)\mu_t(\mathrm{d}y)\right]\mathrm{d}t + \sigma \mathrm{d}w_t, \label{linear_model_1} \\
\mu_t &= \mathcal{L}(x_t). \label{linear_model_2}
\end{align}
where $\sigma>0$ and $w=(w_t)_{t\geq 0}$ is a standard Brownian motion. We will assume that $x_0\in\mathbb{R}$. This is clearly of the form of the McKean-Vlasov SDE \eqref{MVSDE} - \eqref{MVSDE2} with $b(\theta,x) = - \theta_1 x$ and $\phi(\theta,x,y) = -\theta_2(x-y)$. The corresponding system of interacting particles is given by
\begin{equation}
\mathrm{d}x_t^{i,N} = -\left[\theta_1 x_t^{i,N} + \theta_2 \frac{1}{N}\sum_{j=1}^N (x_t^{i,N} - x_t^{j,N})\right]\mathrm{d}t + \sigma \mathrm{d}w_t^{i,N}~,~~~i=1,\dots,N.
\end{equation}
In this model, the parameter $\theta_1$ controls the strength of attraction of the non-linear process (or, in the IPS, of each individual particle) towards zero, while the strength of the parameter $\theta_2$ controls the strength of the attraction of the non-linear process (of each individual particle) towards its mean (the empirical mean). 

We will generally consider the case in which both parameters are unknown, and to be estimated from data. In this case, one can show that all of the assumptions required for our offline results (Theorems \ref{offline_theorem1} - \ref{offline_theorem2}) are satisfied, but that one of the assumptions required for some of our online results (Theorems \ref{theorem1_1_star} - 
\ref{theorem1_2_star} and Theorems \ref{theorem2_1_star} - \ref{theorem2_2_star}) are not satisfied. We discuss this further in Section \ref{subsec:linear_online}.

\subsubsection{Offline Parameter Estimation}
We begin by illustrating the performance of the offline MLE. In this case, since this model is linear in both of the parameters, it is possible to obtain the MLE in closed form as (see also \cite{Kasonga1990})
\begin{align}
\hat{\theta}_{1,t}^N = \frac{A_t^N - B_t^N}{C_t^N - D_t^N}~~~,~~~\hat{\theta}_{2,t}^N &= \frac{D_t^N A_t^N - C_t^N B_t^N}{(C_t^N)^2-C_t^ND_t^N}
\end{align}
where we have defined, writing $\smash{\bar{x}_s^N = \tfrac{1}{N}\sum_{j=1}^N x_s^{j,N}}$, 
\begin{alignat}{3}
A_t^N&=\int_0^t \sum_{i=1}^N (x_s^{i,N} - \bar{x}_s^N)\mathrm{d}x_s^{i,N}~&&,~~~&&B_t^N= \int_0^t \sum_{i=1}^N x_s^{i,N}\mathrm{d}x_s^{i,N} \\
C_t^N&= \int_0^t \sum_{i=1}^N(x_s^{i,N}-\bar{x}_s^N)^2\mathrm{d}s~&&,~~~ &&D_t^N=\int_0^t \sum_{i=1}^N (x_s^{i,N})^2\mathrm{d}s.
\end{alignat}
For our first simulation, we assume that the true parameter is given by $\theta^{*} = (1,0.5)^T$, and that the diffusion coefficient is equal to the identity, $\sigma = 1$. The performance of the MLE is visualised in Figure \ref{fig1}, in which we plot the mean squared error (MSE) of the offline parameter estimate for $t\in[0,30]$, and $N\in\{2,5,10,25,50,100\}$, averaged over 500 random trials. As expected, the parameter estimates converge to the true parameter values as $N$ increases with $t$ fixed (see Theorem \ref{offline_theorem1}), and also as $t$ increases with $N$ fixed (see, e.g., \cite{Borkar1982,Levanony1994}). 

\begin{figure}[!h]
\centering
\subfloat[$\hat{\theta}_{1,t}^N$. ]{\label{fig_1a}\includegraphics[width=.4\linewidth]{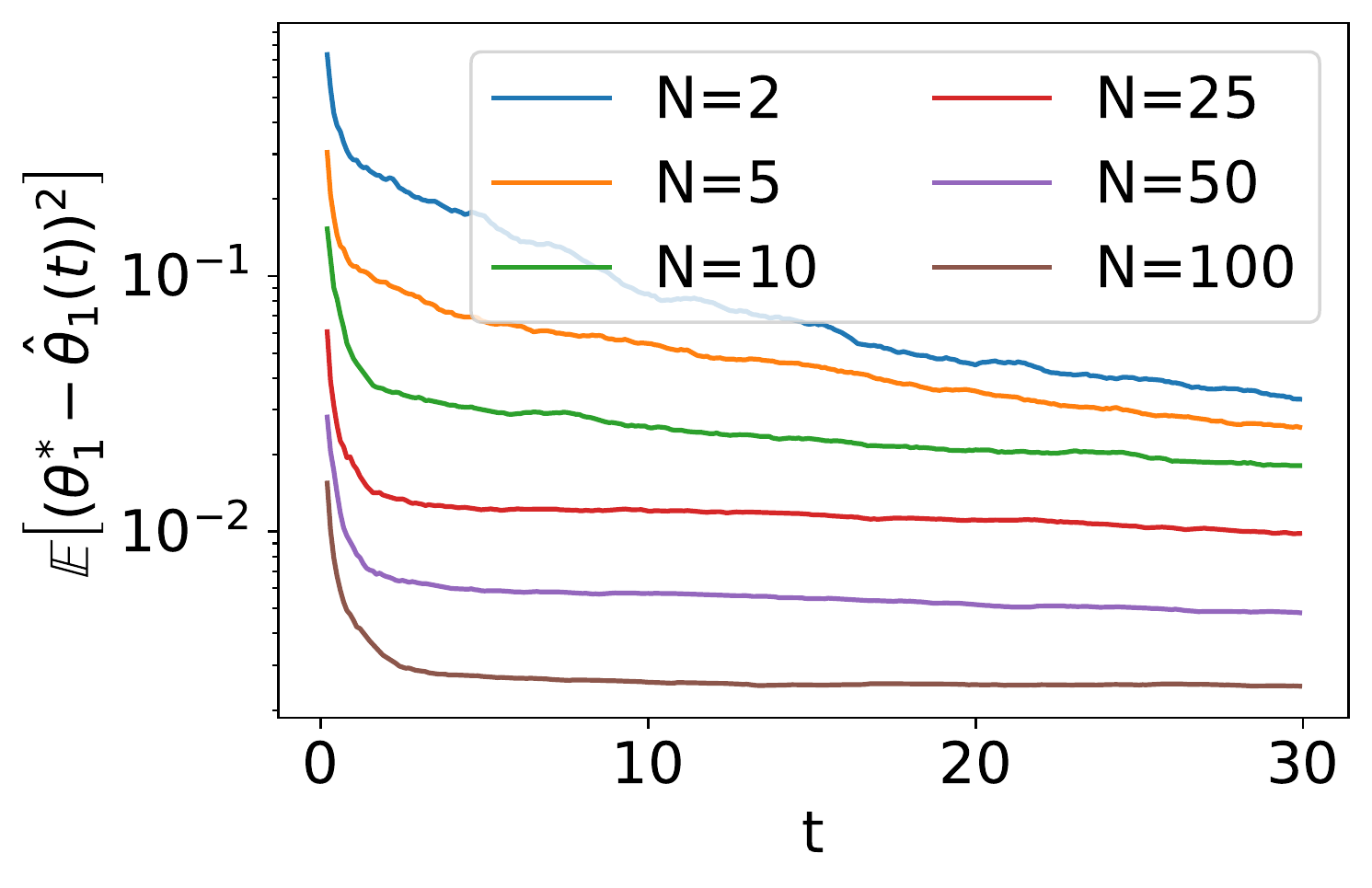}} 
\subfloat[$\hat{\theta}_{2,t}^N$. ]{{\label{fig_1b}\includegraphics[width=.4\linewidth]{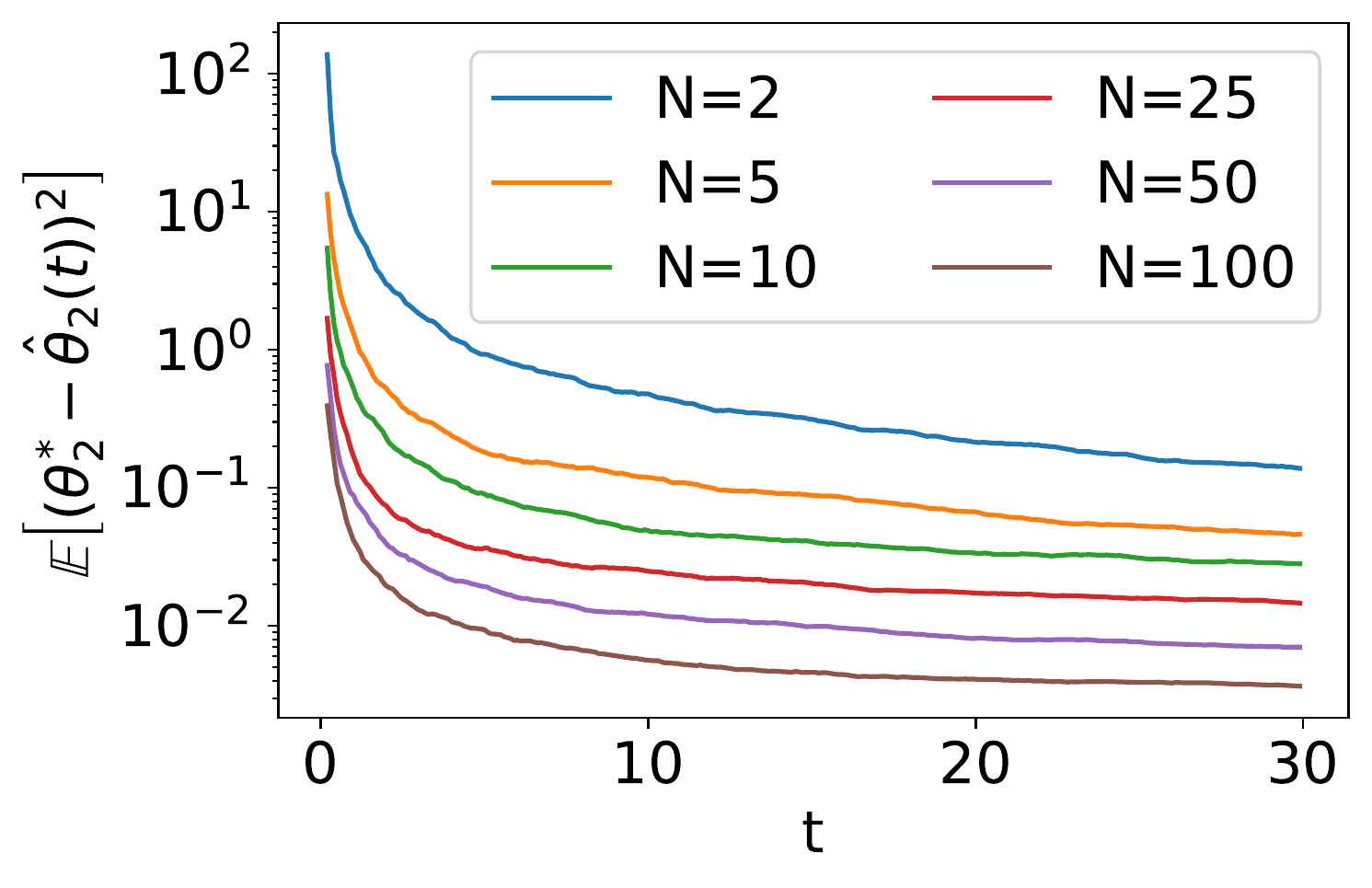}}} 
\vspace*{-2mm}
\caption{$\mathbb{L}^2$ error of the offline MLE for $t\in[0,30]$ and $N=\{2,5,10,25,50,100\}$. The $\mathbb{L}^2$ error is plotted on a log-scale.}
\label{fig1}
\end{figure}

We investigate the convergence rate of the offline MLE further in Figure \ref{fig2}, in which we plot the mean absolute error (MAE) of the offline parameter estimate for $N\in\{20,21,\dots,400\}$ with $t=5$, 
and also for $t\in[50,2000]$ with $N=2$, 
 averaged over 500 random trials. Our results suggest that the offline MLE for this model has an $\mathbb{L}^{1}$ convergence rate of order $O((Nt)^{-\frac{1}{2}})$. This is rather unsurprising: such a rate was recently established by Chen \cite{Chen2020} for a linear mean field model (of arbitrary dimension) in the absence of the global confinement term.

\begin{figure}[!h]
\centering
\subfloat[$\hat{\theta}_{1,t=5}^N$. ]{\label{fig_2a}\includegraphics[width=.38\linewidth]{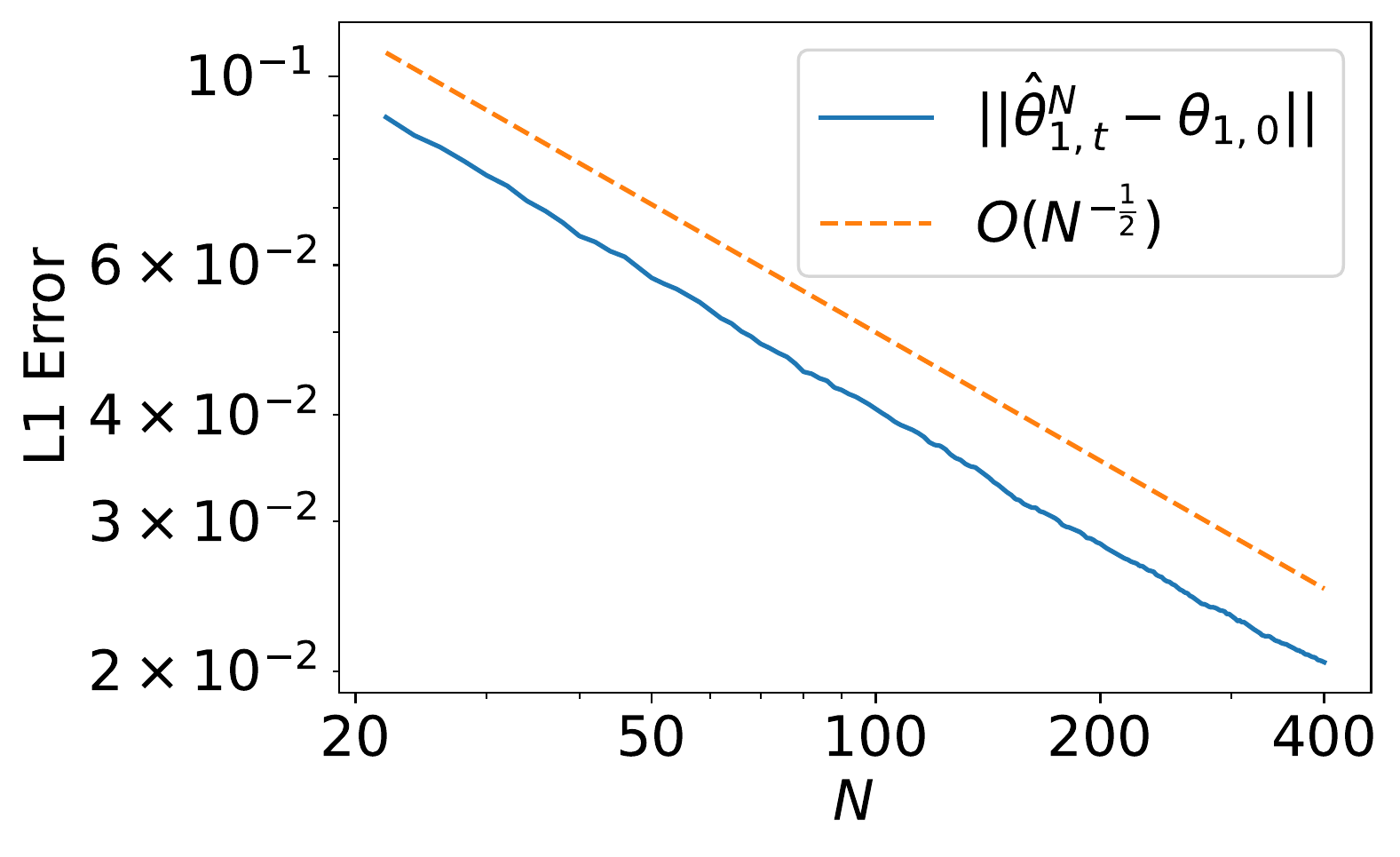}} \hspace{5mm} 
\subfloat[$\hat{\theta}_{2,t=5}^N$. ]{{\label{fig_2b}\includegraphics[width=.38\linewidth]{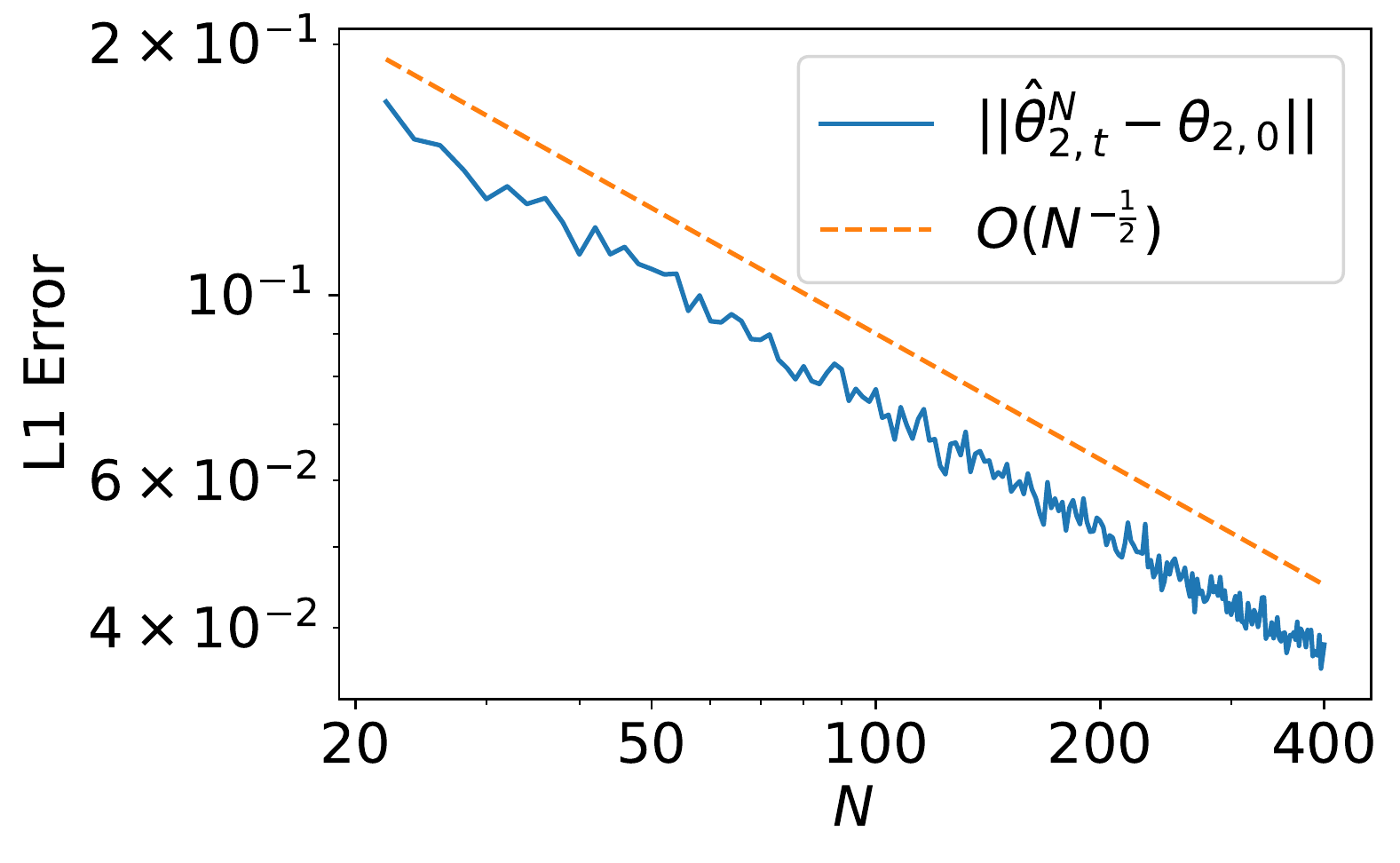}}} \\
~\subfloat[$\hat{\theta}_{1,t}^{N=2}$. ]{\label{fig_2c}\includegraphics[width=.38\linewidth]{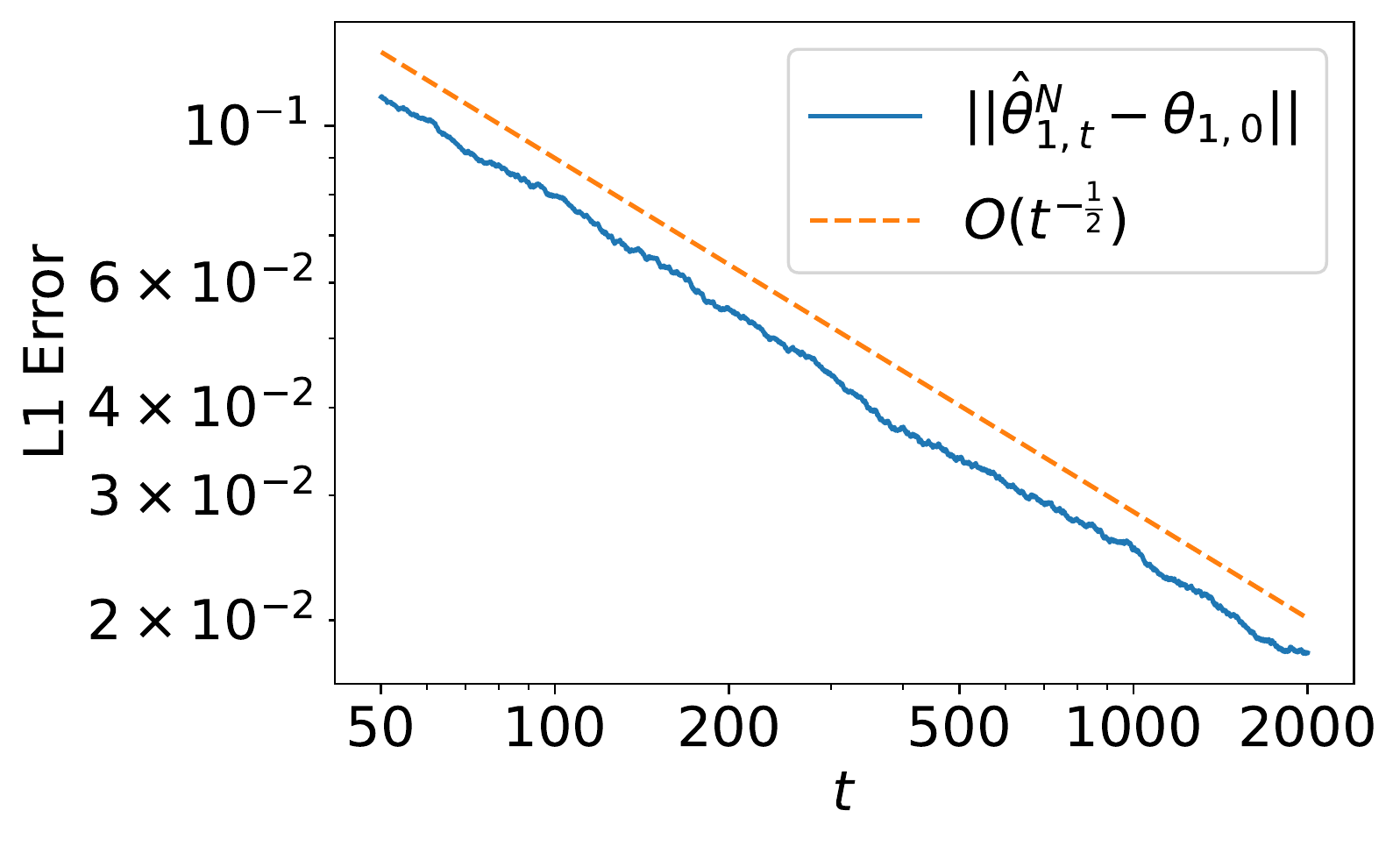}} \hspace{5mm}
\subfloat[$\hat{\theta}_{2,t}^{N=2}$. ]{{\label{fig_2d}\includegraphics[width=.38\linewidth]{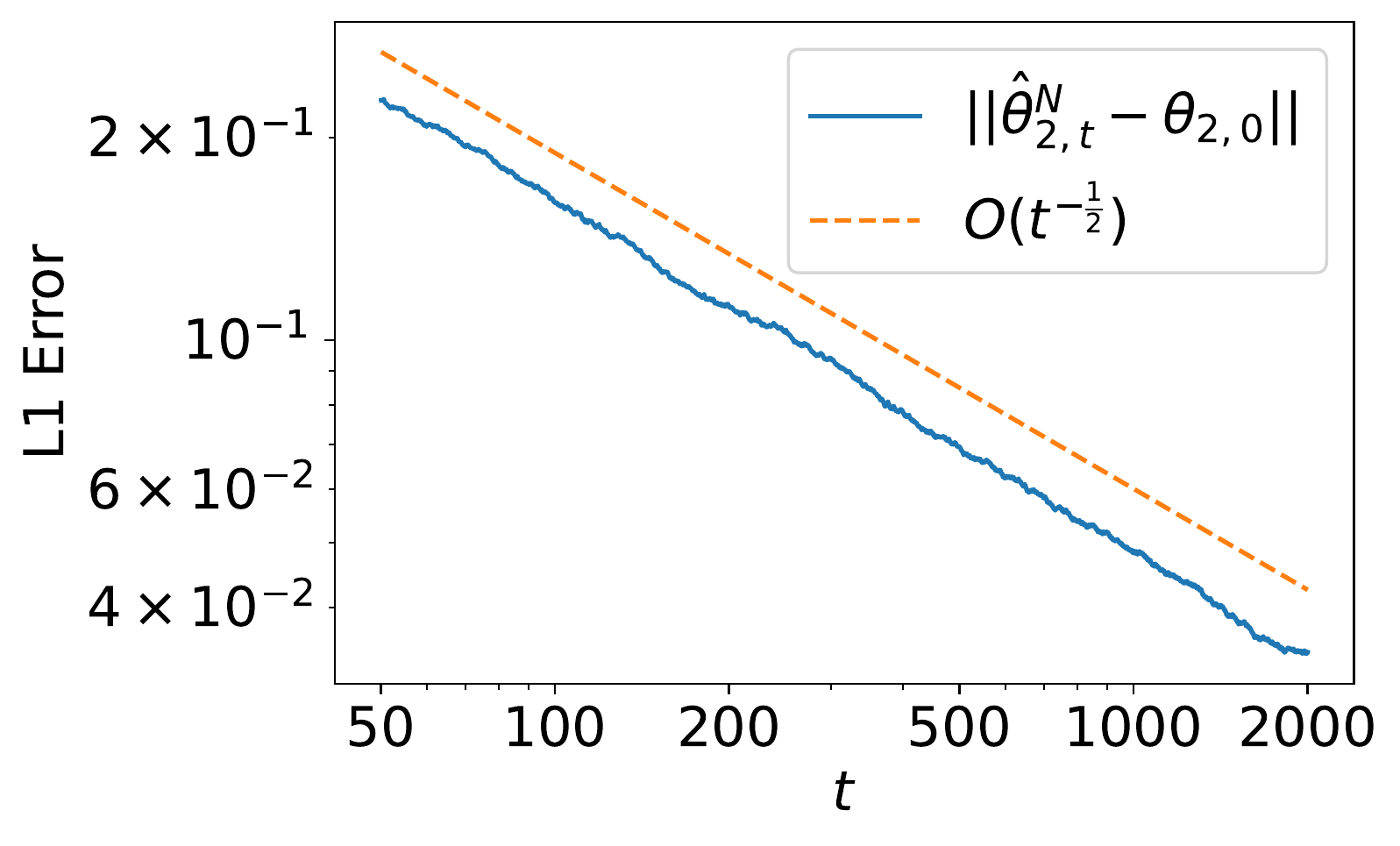}}} 
\caption{Log-log plot of the $\mathbb{L}^1$ error of the offline MLE for $t=5$ and $N\in\{20,\dots,400\}$ (top panel), and for $t\in[50,2000]$ and $N=2$ (bottom panel).}
\label{fig2}
\end{figure}

To conclude this section, we provide numerical confirmation of the asymptotic normality of the MLE (Theorem \ref{offline_theorem2}). For the linear mean field model of interest, it is in fact possible to obtain the asymptotic information matrix in closed form (see also \cite{Kasonga1990}). In particular, it is given by
\begin{equation}
I_t(\theta) = \left(\begin{array}{cc} D_t(\theta) & C_t(\theta) \\ C_t(\theta) & C_t(\theta) \end{array}\right),
\end{equation}
where, with $\gamma(\theta) = -2(\theta_1 + \theta_2)$, 
\begin{align}
C_t(\theta) 
&= \frac{1}{\gamma^2(\theta)} (e^{\gamma(\theta) t} - 1) - \frac{t}{\gamma(\theta)} + \frac{\sigma_0^2}{\gamma}(e^{\gamma(\theta) t} - 1), \\
D_t(\theta) 
&=  \frac{1}{\gamma^2(\theta)} (e^{\gamma(\theta) t} - 1) - \frac{t}{\gamma(\theta)} + \frac{\sigma_0^2}{\gamma(\theta)}(e^{\gamma(\theta) t} - 1) -\frac{\mu_0^2}{2\theta_1} (e^{-2\theta_1 t}-1).
\end{align}
As such, in Figure \ref{fig3}, we are able to provide a direct comparison of the asymptotic normal distribution of the MLE, and the approximate normal distribution obtained using a finite number of particles. 
\begin{figure}[!h]
\centering
\subfloat[Asymptotic \& approximate marginals.]{\label{fig_3a}\includegraphics[width=.38\linewidth]{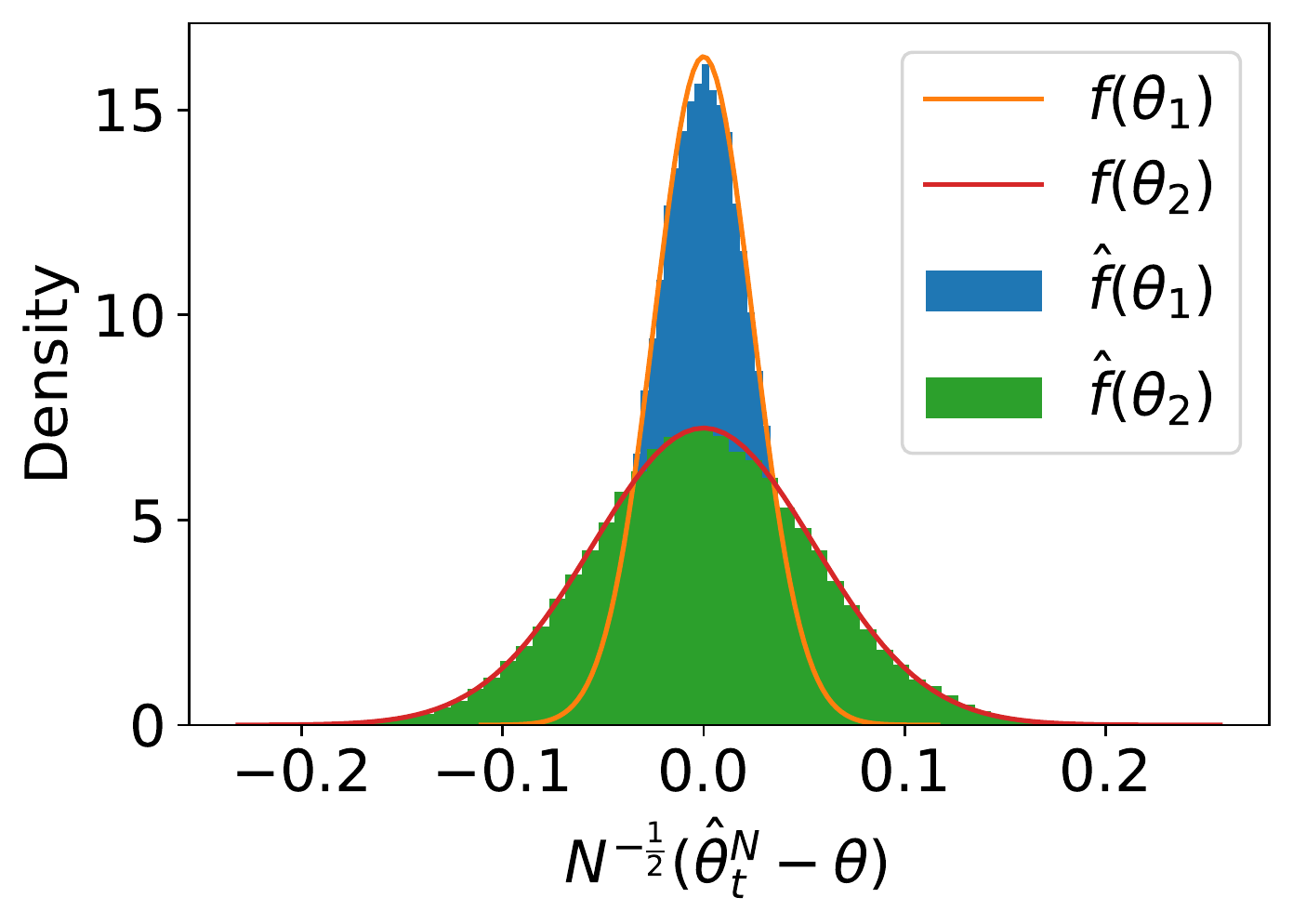}} \\
\subfloat[Approximate bivariate density.]{{\label{fig_3b}\includegraphics[width=.36\linewidth]{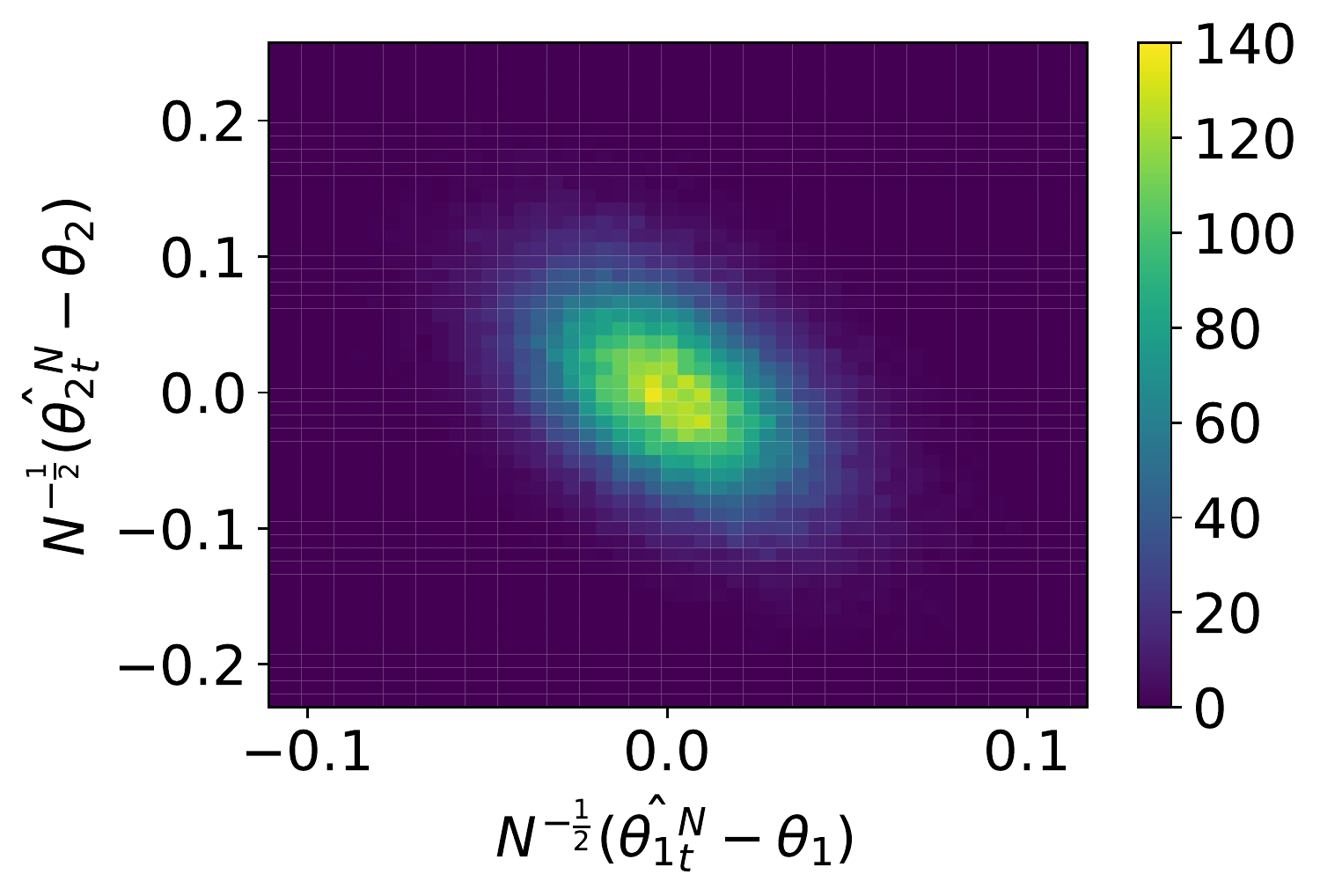}}} 
\subfloat[Asymptotic bivariate density.]{{\label{fig_3c}\includegraphics[width=.36\linewidth]{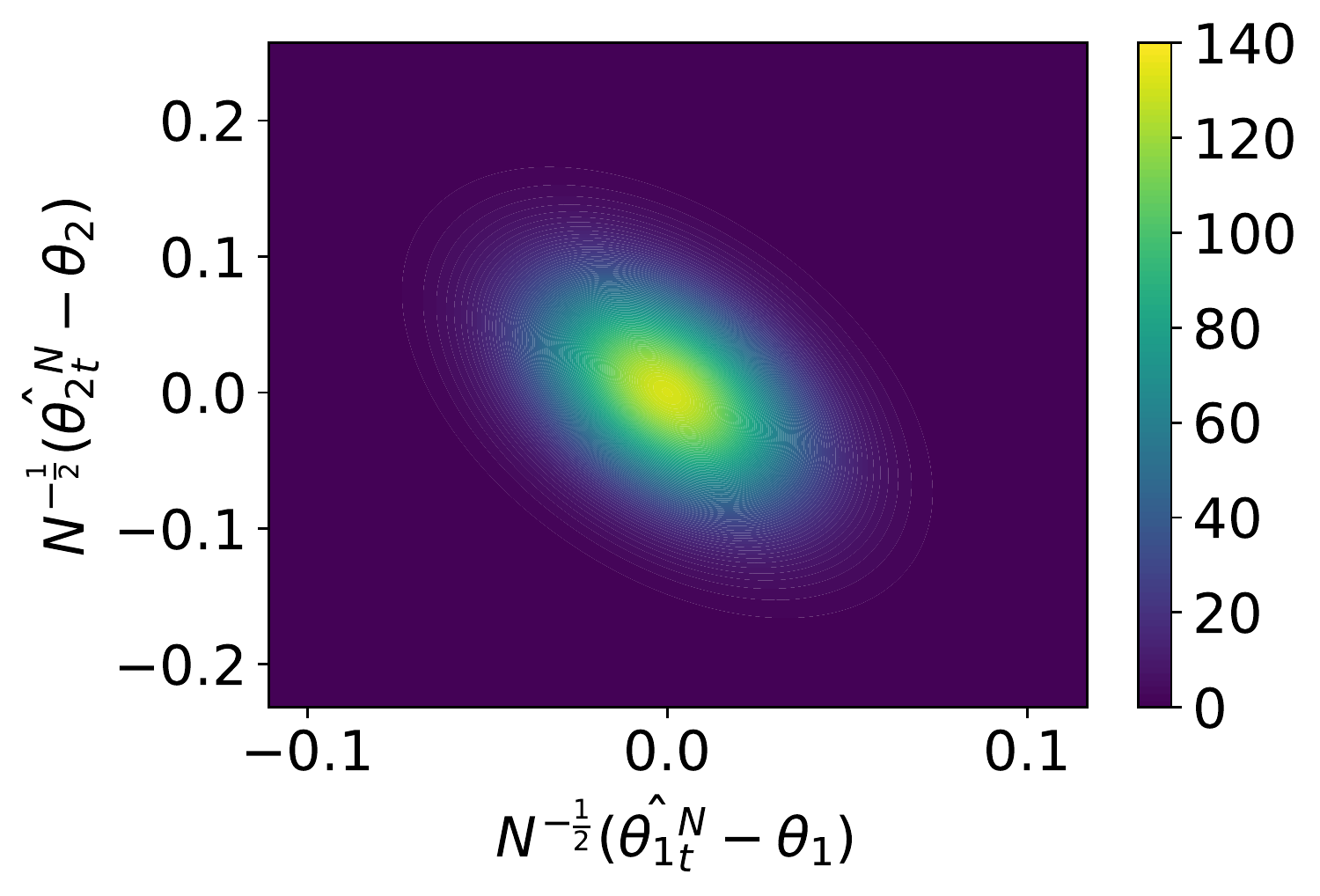}}} 
\caption{A comparison between the asymptotic normal distribution and the approximate normal distribution of the MLE for $N=500$ particles. The histograms were obtained using $10^5$ independent runs.}
\label{fig3}
\end{figure}

\subsubsection{Online Parameter Estimation}
\label{subsec:linear_online}
We now turn our attention to online parameter estimation. In particular, we will utilise the update equations in \eqref{theta_approx0}, which for this model are given by 
\begin{align}
\mathrm{d}\theta^N_{1,t} &= \frac{\gamma_{1,t}}{N \sigma^2}\sum_{i=1}^N\left[-x_t^{i,N} \mathrm{d}x_t^{i,N} - x_t^{i,N}(\theta^N_{1,t} x_t^{i,N}+ \theta^N_{2,t}(x_t^{i,N} - \bar{x}_t^N))\mathrm{d}t\right], \\
\mathrm{d}\theta^N_{2,t} &= \frac{\gamma_{2,t}}{N \sigma^2}\sum_{i=1}^N\left[-(x_t^{i,N}-\bar{x}_t^N) \mathrm{d}x_t^{i,N} - (x_t^{i,N}-\bar{x}_t^N)(\theta^N_{1,t} x_t^{i,N}+ \theta^N_{2,t}(x_t^{i,N} - \bar{x}_t^N))\mathrm{d}t\right].
\end{align}

We will initially assume that one of the parameters is fixed (and equal to the true value), while the other parameter is to be estimated. The true parameters are given by $\theta_1^{*} = 0.5$ and $\theta_2^{*} = 0.1$. Meanwhile, the initial parameter estimates are randomly generated according to $\theta^{0}_1,\theta^{0}_2\sim \mathcal{U}([2,5])$. Finally, the learning rates are given by $\gamma_{i,t} =\min\{\gamma_{i}^{0},\gamma_{i}^{0}t^{-\alpha}\}$, $i=1,2$, where $\gamma_{1}^{0} = 0.05$, $\gamma_{2}^{0}=0.30$, and $\alpha=0.51$. The performance of the stochastic gradient descent algorithm is visualised in Figures \ref{fig4}, in which we plot the MSE of the online parameter estimates for $t\in[0,1000]$ and $N=\{2,5,10,25,50,100\}$. The results are computed over 500 independent random trials. Interestingly, increasing the number of particles can result in a relatively significant reduction in the MSE of the interaction parameter $\theta_2$, but has little consequence for the error of the confinement parameter $\theta_1$.

\begin{figure}[!h]
\centering
\subfloat[${\theta}_{1,t}^N$. ]{\label{fig_4a}\includegraphics[width=.34\linewidth]{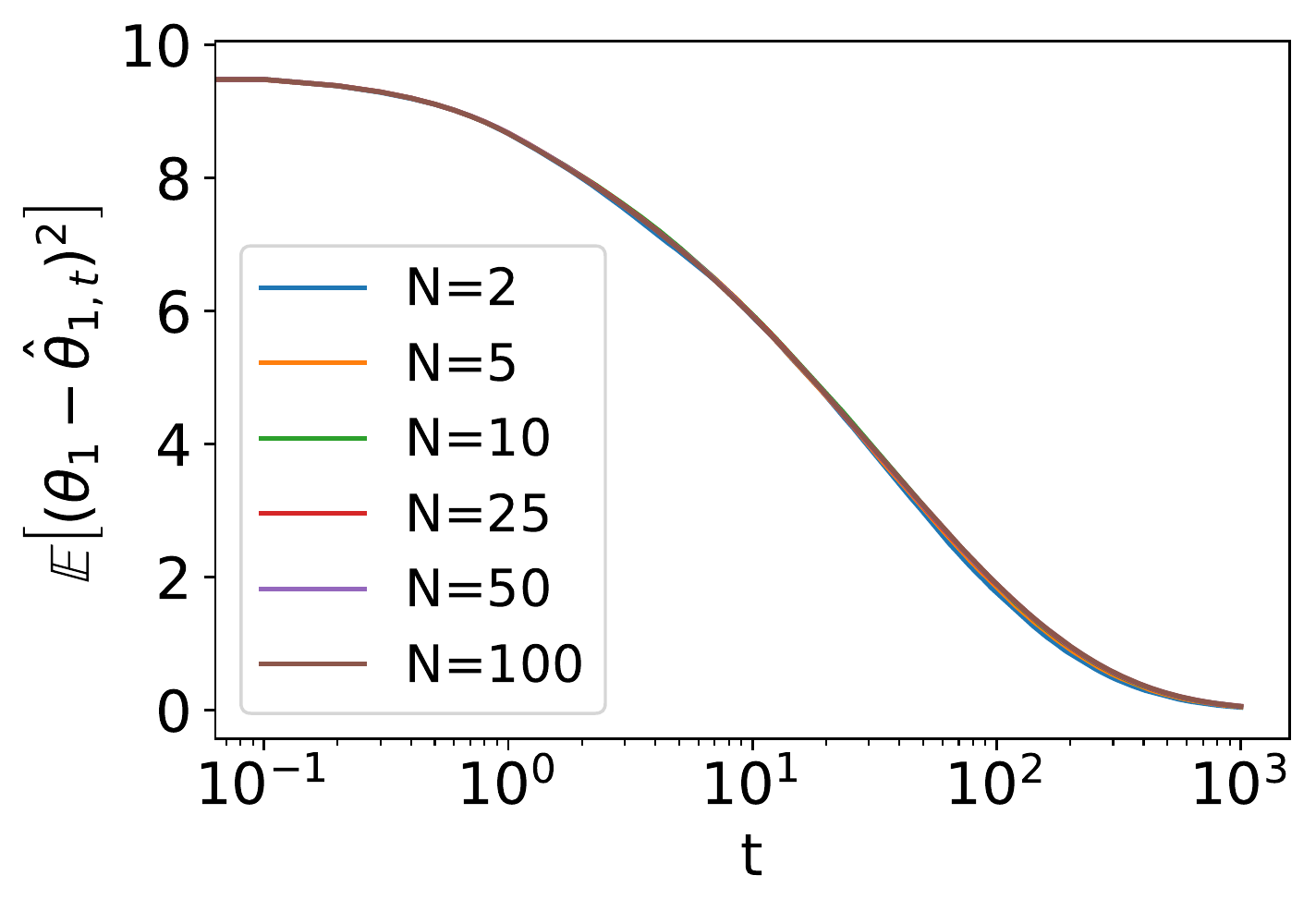}} \hspace{5mm}
\subfloat[${\theta}_{2,t}^N$. ]{{\label{fig_4b}\includegraphics[width=.34\linewidth]{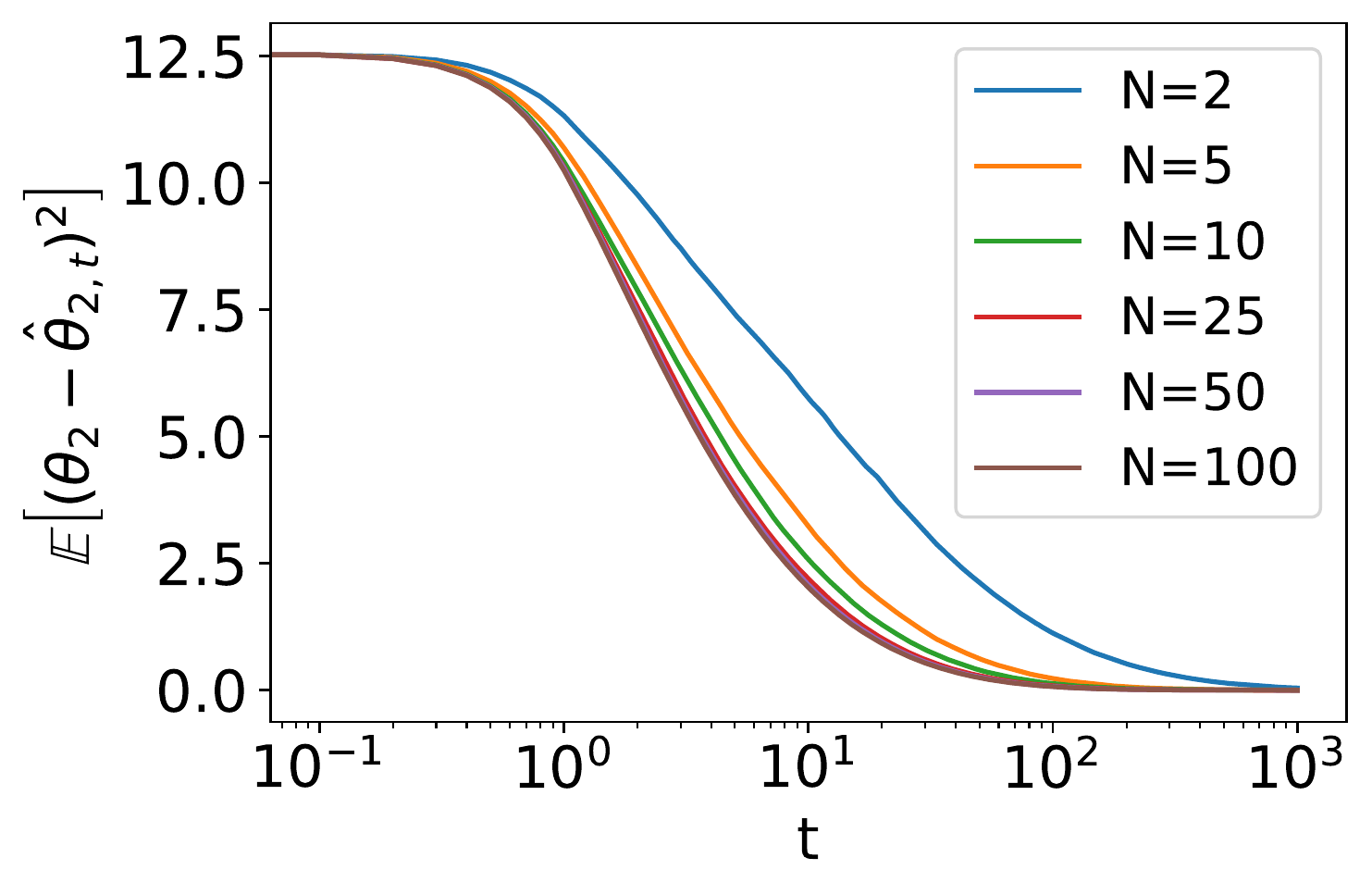}}} 
\vspace*{-2mm}
\caption{$\mathbb{L}^2$ error of the online parameter estimates for $t\in[0,1000]$ and $N=\{2,5,10,25,50,100\}$.} 
\label{fig4}
\end{figure}

Figures \ref{fig6} and \ref{fig7} provide a numerical illustration of why the finite-time performance of the online estimator improves with the number of particles, and why this improvement is more pronounced for the interaction parameter $\theta_2$. As the number of particles increases, we observe that the time weighted average of the 
log-likelihood of the IPS $\mathcal{L}_t^N(\theta)$ (the noisy objective function) much more closely resembles the asymptotic log-likelihood of the IPS $\tilde{\mathcal{L}}^N(\theta)$ (the true objective function), even for small time values. This means, in particular, that fluctuations terms of the form 
\begin{equation}
\int_0^t \gamma_s (\nabla_{\theta} \tilde{\mathcal{L}}^N(\theta^N_s)-\frac{1}{N}\sum_{i=1}^N\nabla_{\theta}L(\theta^N_s,x^N_s,\mu^N_s))\mathrm{d}s, 
\end{equation}
converge more rapidly to zero (as a function of time), for larger values of $N$. This disparity in the convergence rate of the log-likelihood (as a function of the time), for different values of $N$, appears to be much more significant for the interaction parameter $\theta_2$ (Figure \ref{fig7}) than it is for the confinement parameter $\theta_1$ (Figure \ref{fig6}). Consequently, the online parameter estimate $\smash{\theta_{2,t}^N}$ converges more rapidly as $N$ increases, while there is little difference in the convergence rate of $\smash{\theta_{1,t}^N}$. 

\begin{figure}[!h]
\centering
\subfloat[$T=1.0$.]{\label{fig_6a}\includegraphics[width=.23\linewidth]{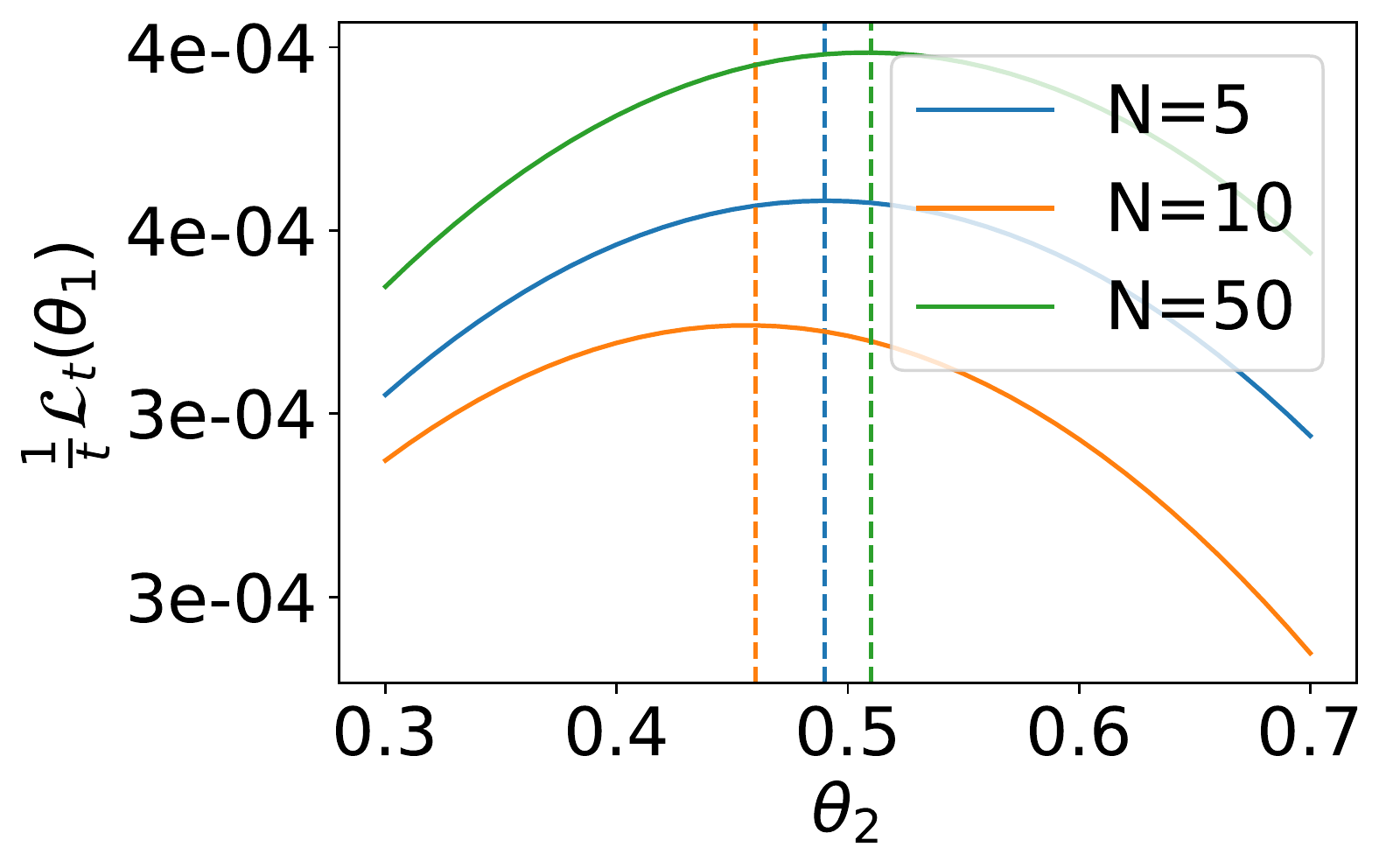}}
\subfloat[$T=2.5$.]{{\label{fig_6b}\includegraphics[width=.23\linewidth]{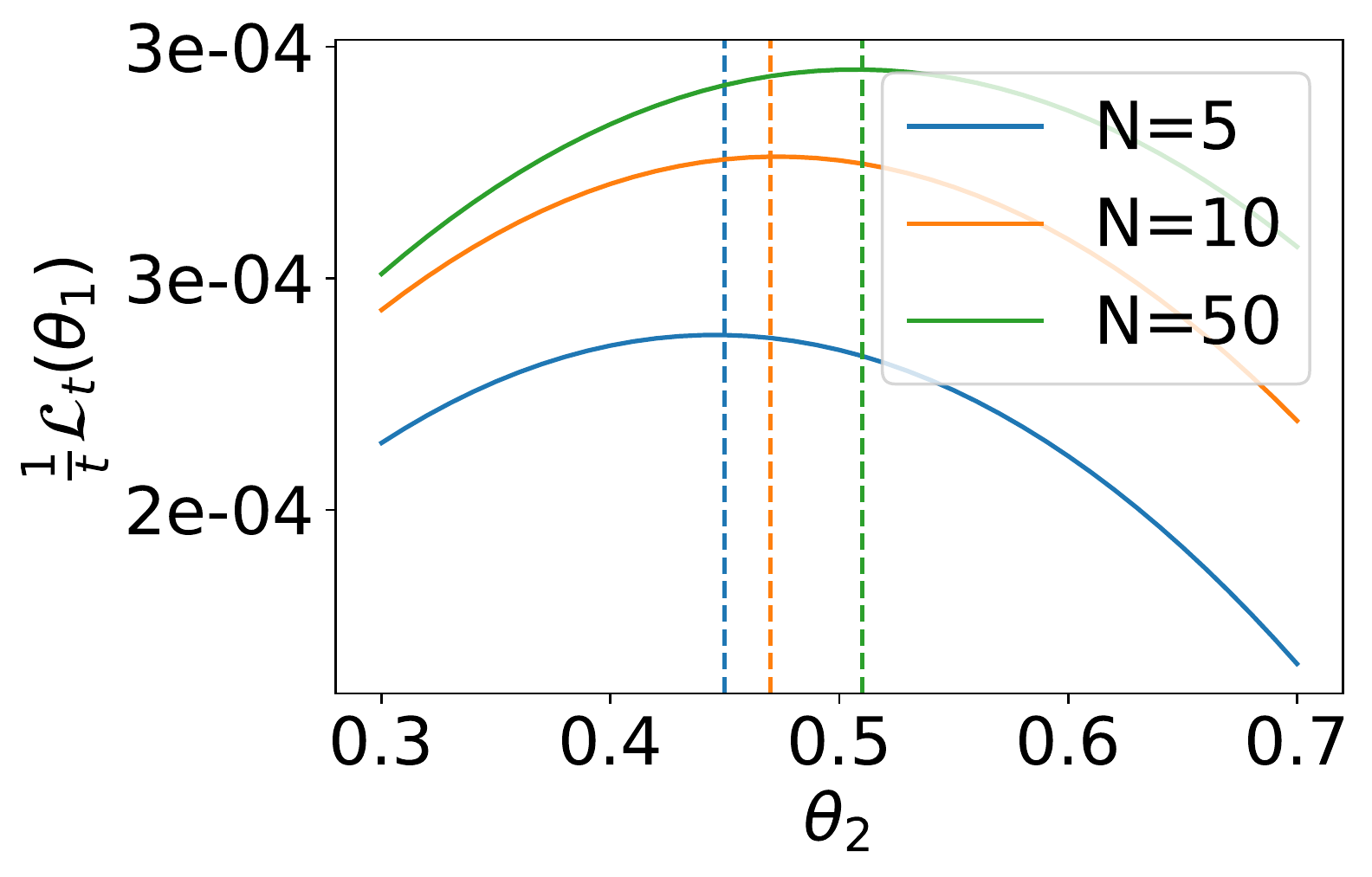}}} 
\subfloat[$T=5.0$.]{{\label{fig_6c}\includegraphics[width=.23\linewidth]{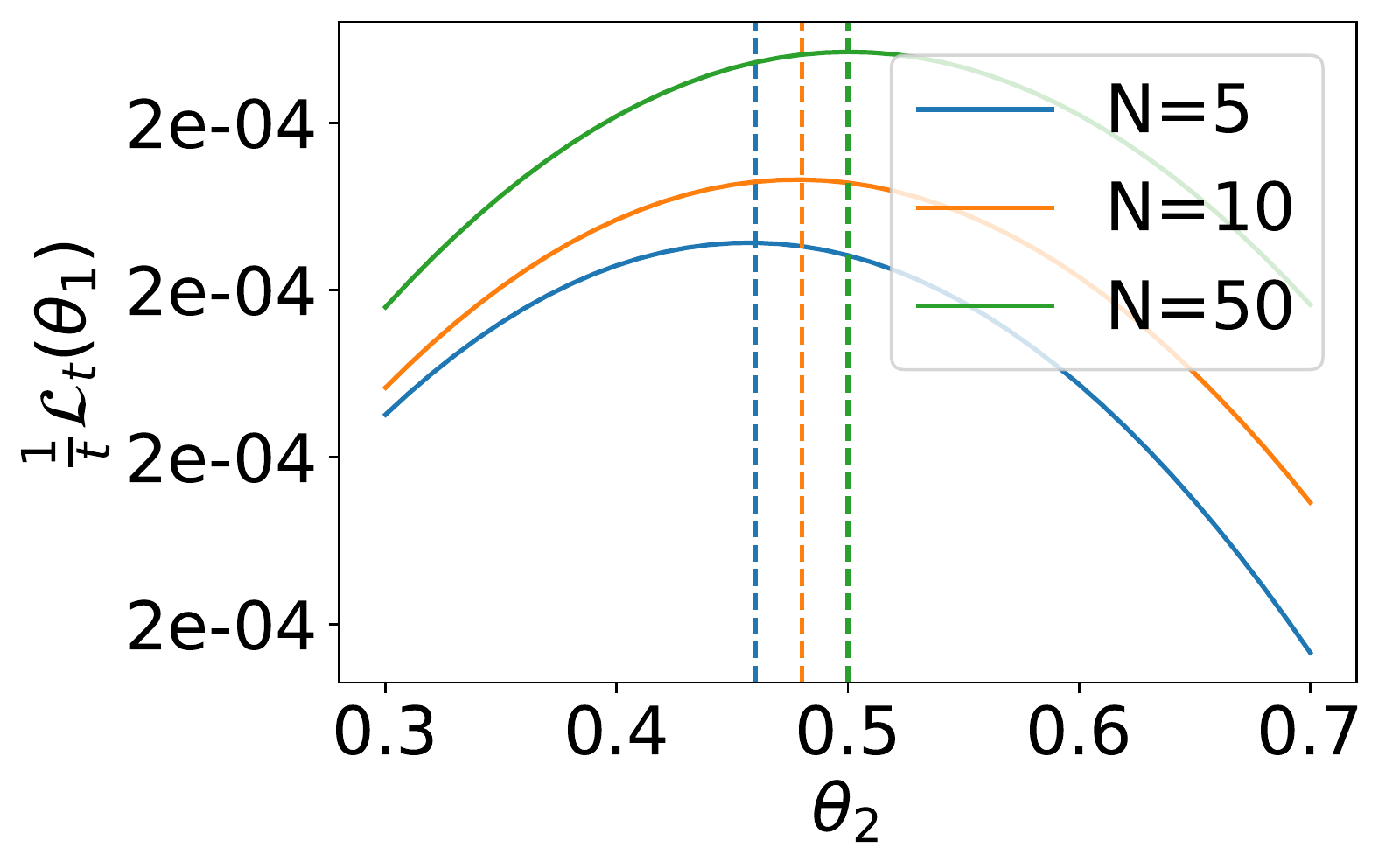}}} 
\subfloat[$T=7.5$. ]{{\label{fig_6d}\includegraphics[width=.23\linewidth]{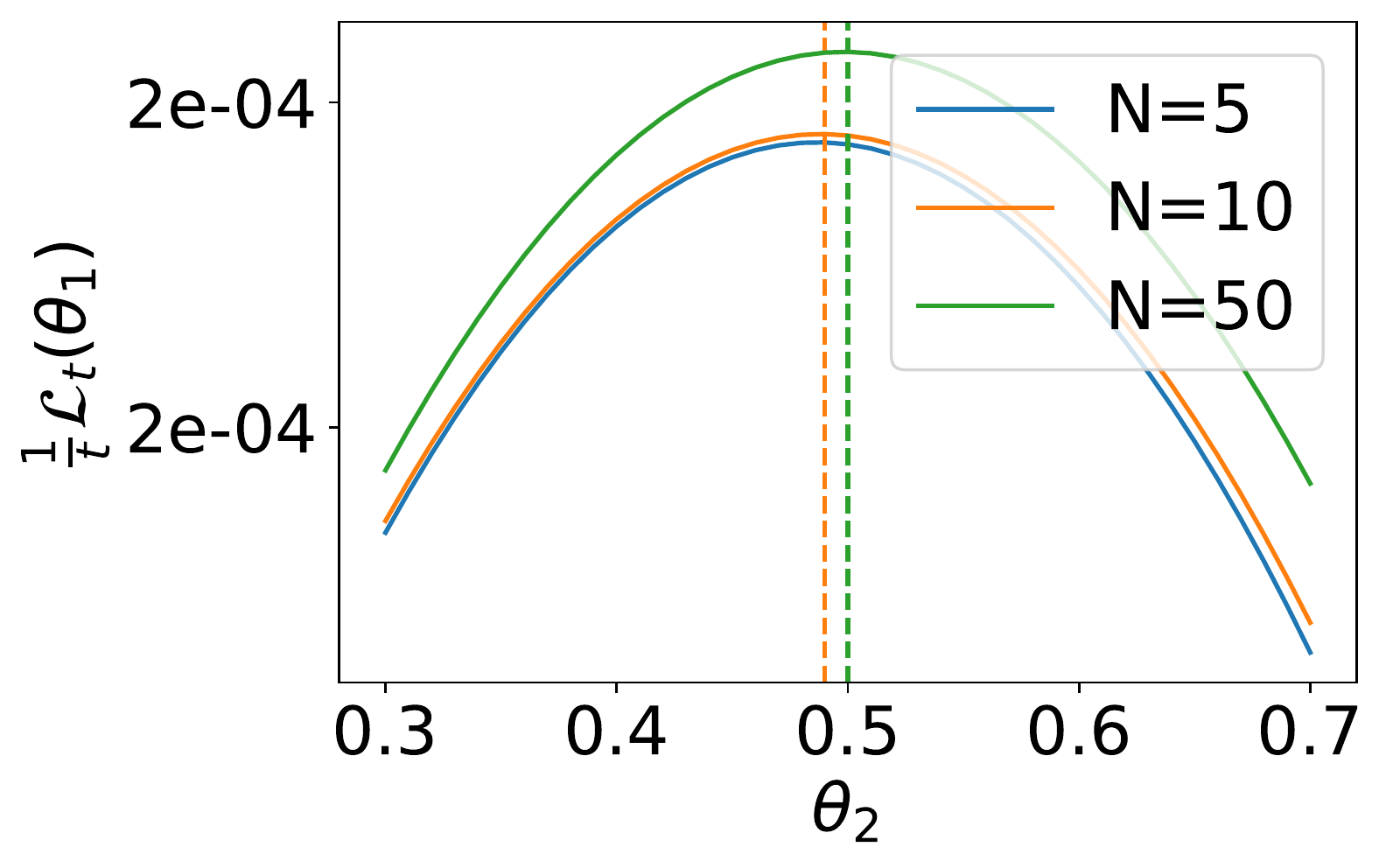}}} 
\vspace*{-1mm}
\caption{Plots of the average log-likelihood, $\frac{1}{T}\mathcal{L}_T^N(\theta_1)$, for $T =\{1,2.5,5.7.5\}$ and $N=\{5,10,50\}$.}
\label{fig6}
\vspace*{-1mm}
\end{figure}

\begin{figure}[!h]
\centering
\subfloat[$T=1.0$.]{\label{fig_7a}\includegraphics[width=.23\linewidth]{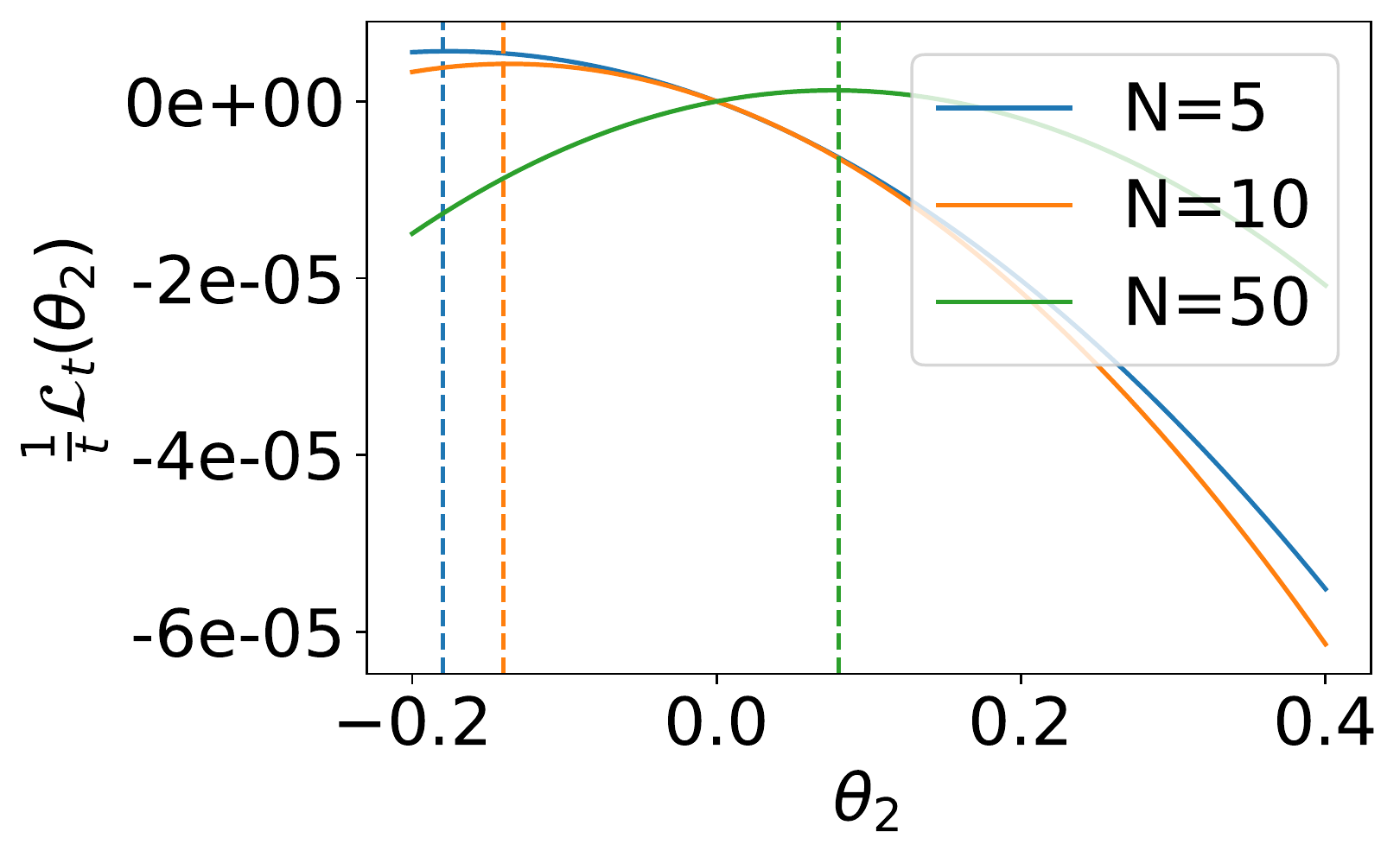}}
\subfloat[$T=2.5$.]{{\label{fig_7b}\includegraphics[width=.23\linewidth]{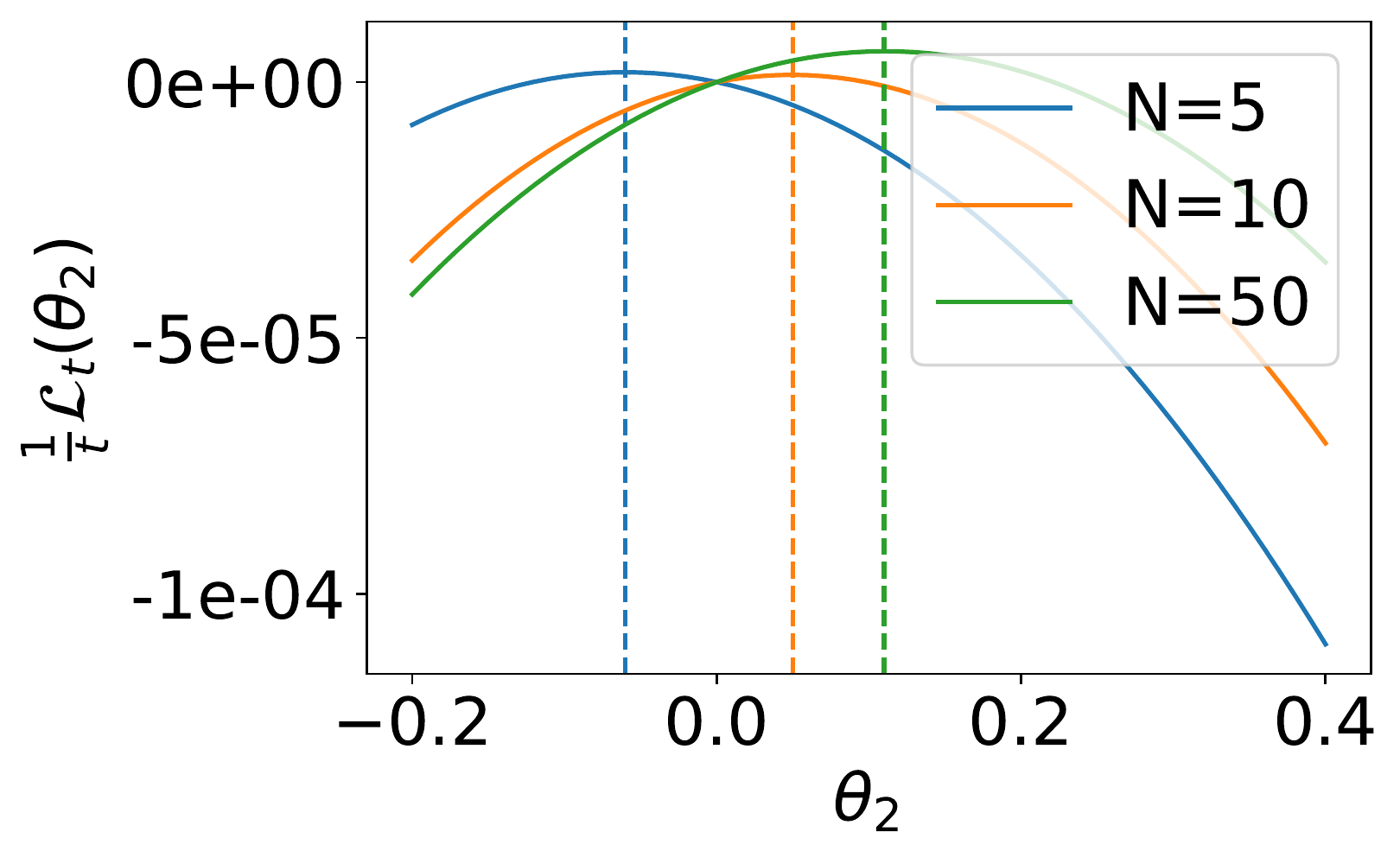}}} 
\subfloat[$T=5.0$.]{{\label{fig_7c}\includegraphics[width=.23\linewidth]{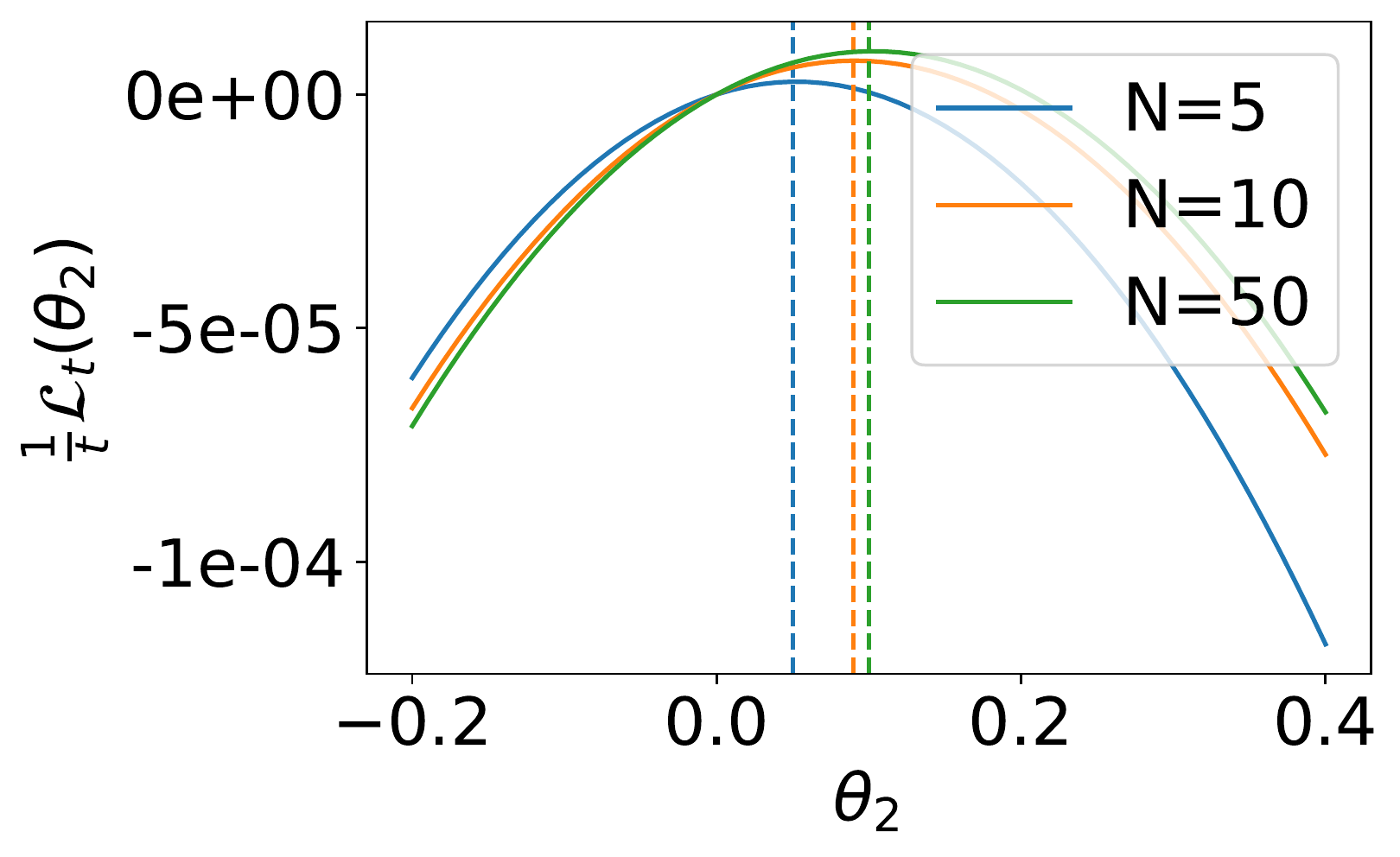}}} 
\subfloat[$T=7.5$.]{{\label{fig_7d}\includegraphics[width=.23\linewidth]{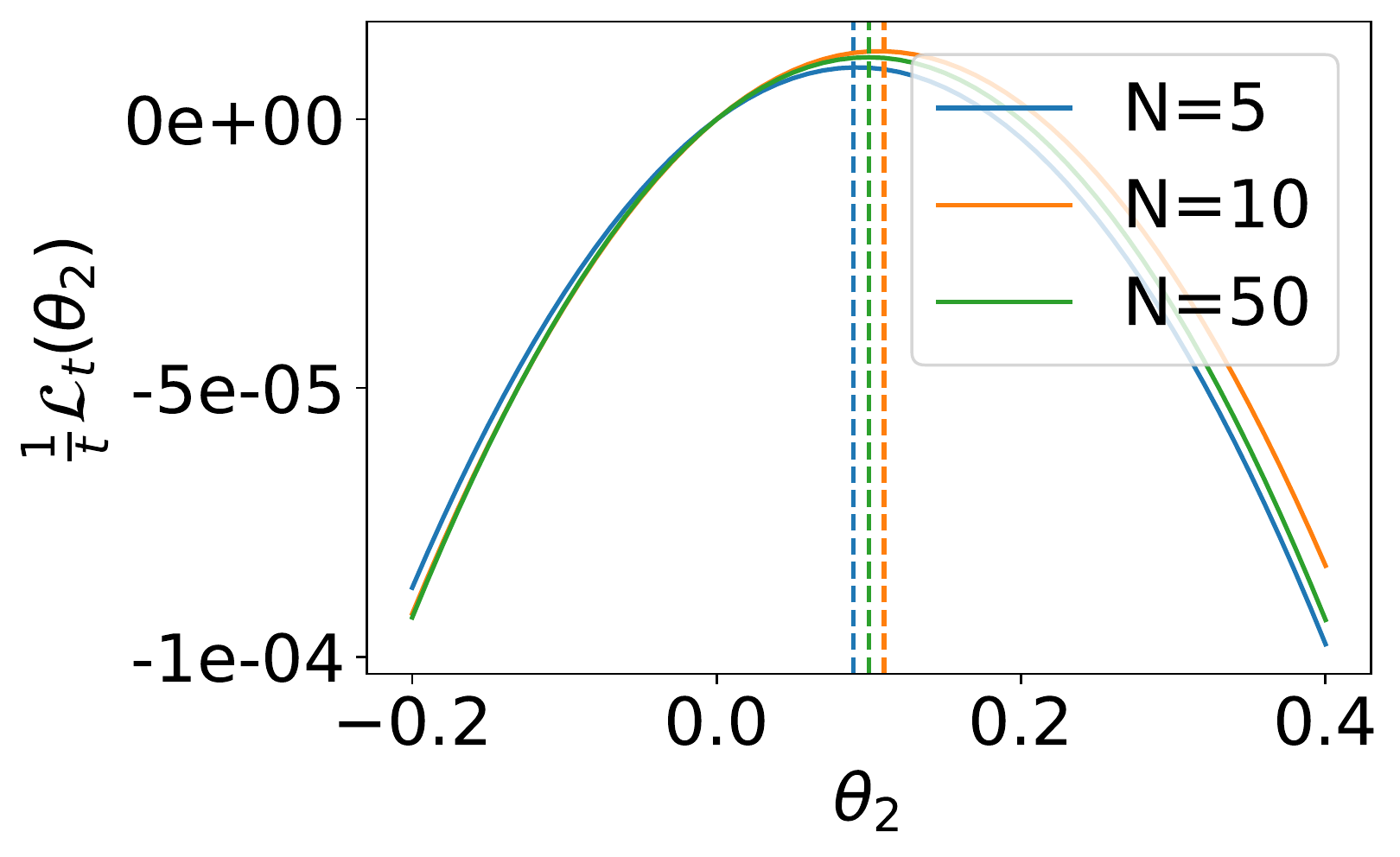}}} 
\vspace*{-1mm}
\caption{Plots of the average log-likelihood, $\frac{1}{T}\mathcal{L}_T^N(\theta_2)$, for $T =\{1,2.5,5.7.5\}$ and $N=\{5,10,50\}$.}
\label{fig7}
\vspace*{-1mm}
\end{figure}

We should remark that, in the linear mean field model, the asymptotic log-likelihood of the IPS is strongly concave for all values of $N$, with unique global maximum at the true parameter value. This is visualised in Figures \ref{fig_6d} and \ref{fig_7d}, in which we have plotted approximations of profile asymptotic log-likelihood of the IPS for several values of $N$. We are thus in the regime of Theorems \ref{theorem2_1} - \ref{theorem2_2}, meaning $\theta_t^N$ will eventually converge to the true parameter as $t\rightarrow\infty$, regardless of the value of the number of particles.

Let us now turn our attention to the case in which both parameters are unknown, and to be estimated from the data. 
In this case, one can show that the linear mean-field model satisfies Conditions \ref{assumption_init}, \ref{assumption1} - \ref{assumption2}, \ref{assumption3}, \ref{assumption_bound} - \ref{assumption_bound2}, \ref{assumption4''} (but not Condition \ref{assumption4}), and \ref{assumption0} - \ref{assumption0_1} (see Appendix \ref{app:verification}). Thus, the conditions of Theorems \ref{theorem2_1} - \ref{theorem2_2} and Theorem \ref{theorem2_1_star} are satisfied, but those of Theorem \ref{theorem2_2_star} are not.  In practice, this means that the asymptotic log-likelihood of the IPS admits a unique maximiser $\theta_0$, and thus the online parameter estimate is guaranteed to converge to $\theta_0$ as $t\rightarrow\infty$, for finite $N\in\mathbb{N}$. On the other hand, the asymptotic log-likelihood of the McKean-Vlasov SDE may not admit a unique maximiser, and thus the online parameter estimate may not convergence to $\theta_0$ as $t\rightarrow\infty$ and $N\rightarrow\infty$.

For the sake of comparison, we will once more assume that that the true parameter is given by $\theta^{*} = (\theta_1^{*},\theta_2^{*})=(0.5,0.1)$. The initial parameter estimates are now generated according to $\theta^{0}_1\sim \mathcal{U}([-1,2])$ and $\theta^0_2\sim\mathcal{U}([-2,2])$. Finally, we use constant learning rates, with $\gamma_{1,t} = 0.1$ and $\gamma_{2,t}=0.2$. The performance of the stochastic gradient descent algorithm is illustrated in Figure \ref{fig9}, in which we plot the MSE of the online parameter estimates for both of the unknown parameters, averaged over 500 random trials. 

\begin{figure}[!h]
\centering
\subfloat[${\theta}_{1,t}^N$. ]{\label{fig_9a}\includegraphics[width=.37\linewidth]{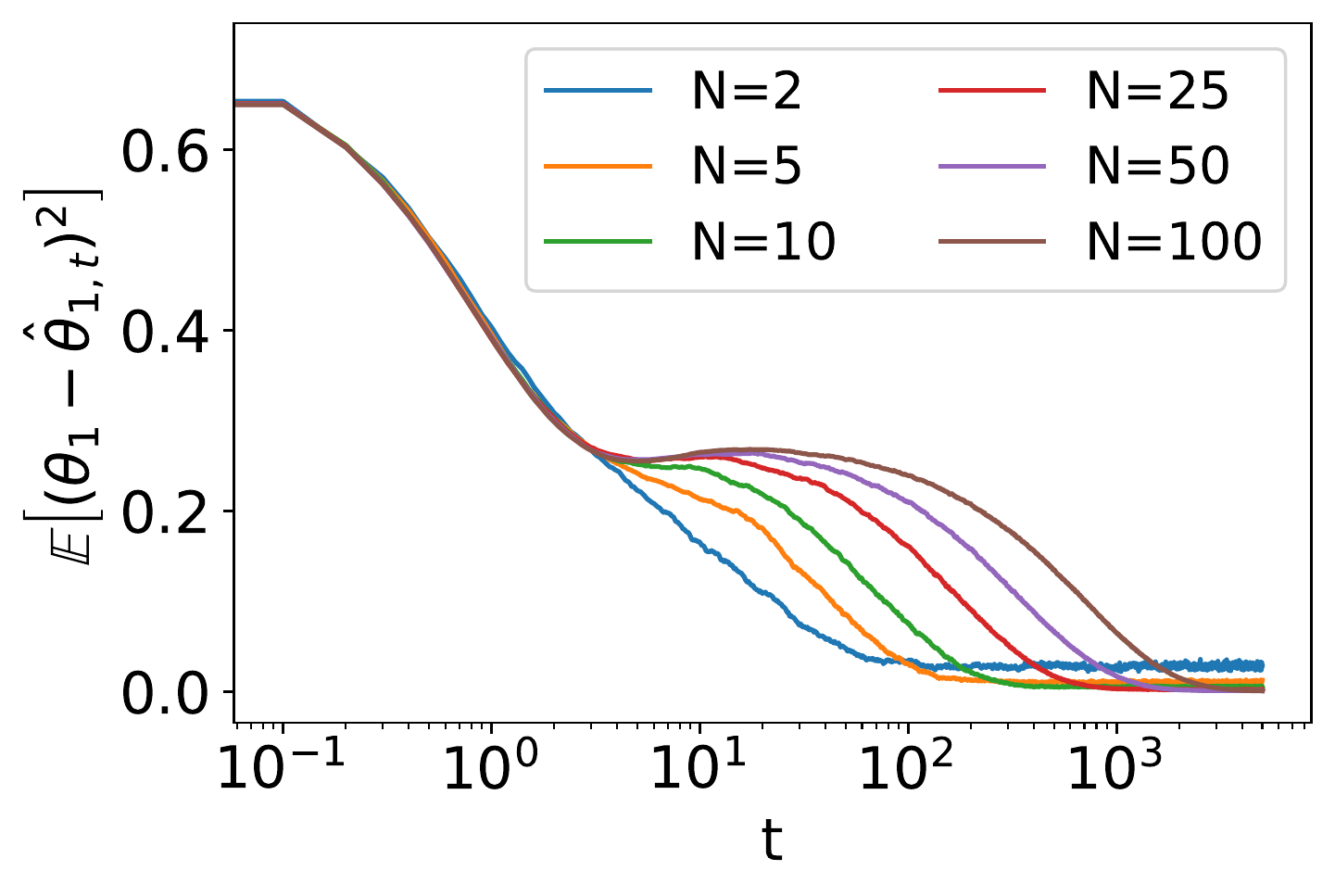}} \hspace{5mm}
\subfloat[${\theta}_{2,t}^N$. ]{{\label{fig_9b}\includegraphics[width=.37\linewidth]{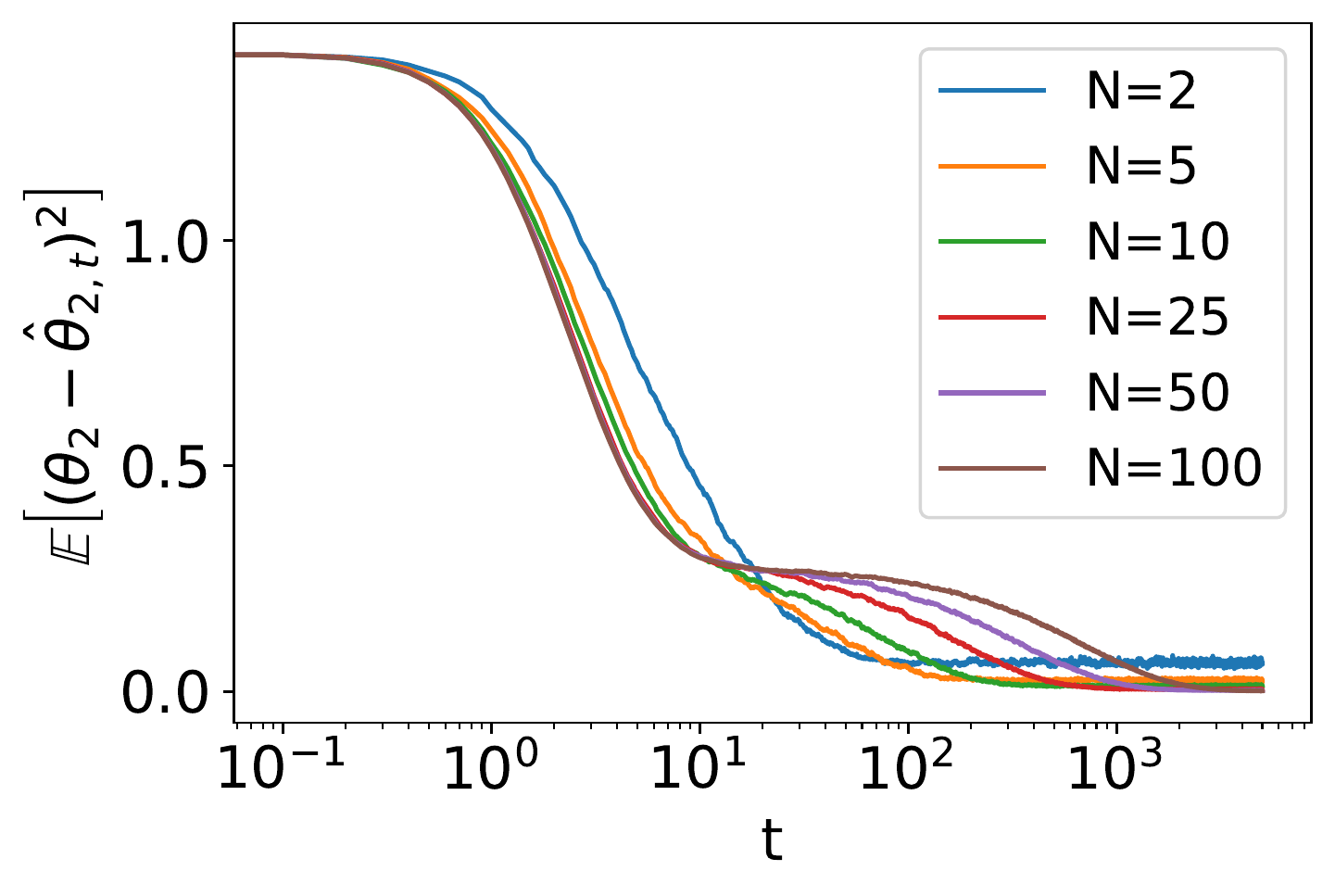}}} 
\caption{$\mathbb{L}^2$ error of the online MLEs for $T\in[0,5000]$ and $N=\{2,5,10,25,50,100\}$.} 
\label{fig9}
\end{figure}

In this case, the evolution of the MSE indicates three distinct learning phases. In the initial phase, the performance of the online estimator improves as a function of the number of particles, with this improvement being more noticeable for the interaction parameter $\theta_{2}$. In the middle phase, the online estimator performs significantly better for smaller values of $N$. In the final learning phase, the MSE of the online parameter estimate decreases as a function of the number of particles. 

These observations are readily explained by considering the asymptotic log-likelihood of the IPS for different numbers of particles (see Figure \ref{fig10}), and our theoretical results.  Regarding the initial learning phase, we note that, far from the global maximum at $\theta_0 = (0.5,0.1)$, the asymptotic log-likelihood decreases more steeply as $N$ increases. Thus, in this phase, the online parameter estimate moves more quickly towards the global maximum for larger $N$. On the other hand, close to the global maximum at $\theta_0$, the asymptotic log-likelihood exhibits an increasingly large plateau as $N$ increases; that is, the maximum is increasingly `flat'. In fact, in the mean field limit, the asymptotic log-likelihood does not even admit a unique maximum at $\theta_0$, but instead is maximised by any $\theta=(\theta_1,\theta_2)^T$ on the line $\theta_1 + \theta_2 = \theta_{0,1}+\theta_{0,2}$ (see Appendix \ref{app:verification}).  This region of the optimisation landscape is largely responsible for the middle learning phase, which explains the slower convergence of the estimator for larger numbers of particles. In the final learning phase, the MSE is governed by the results in Theorem \ref{theorem2_2}, which show that the asymptotic $\mathbb{L}_2$ error of the online parameter estimate implemented here decreases with the number of particles.
\begin{figure}[htbp]
\centering
\subfloat[$N=2$.]{\label{fig_10a}\includegraphics[width=.25\linewidth]{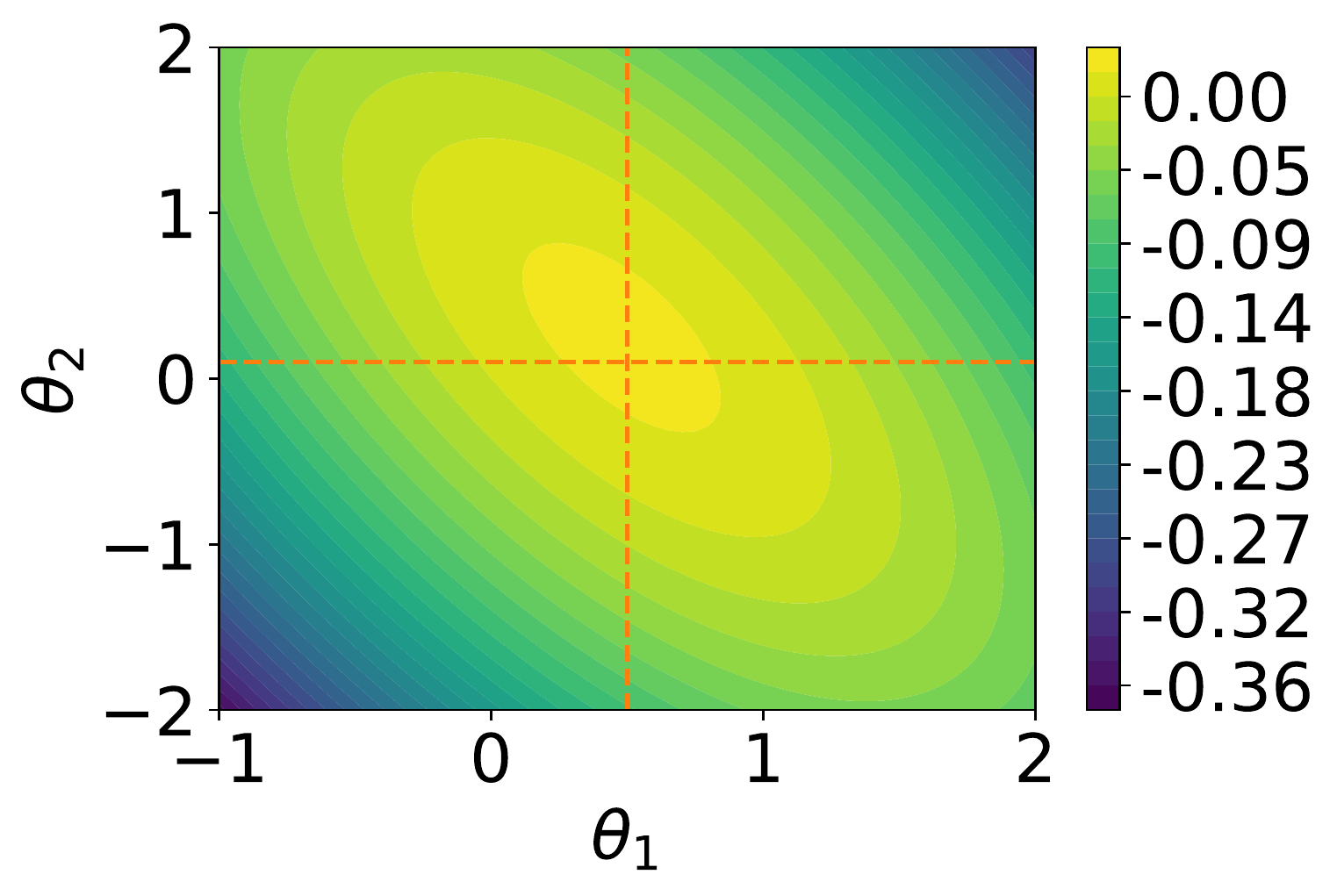}}
\subfloat[$N=5$.]{{\label{fig_10b}\includegraphics[width=.25\linewidth]{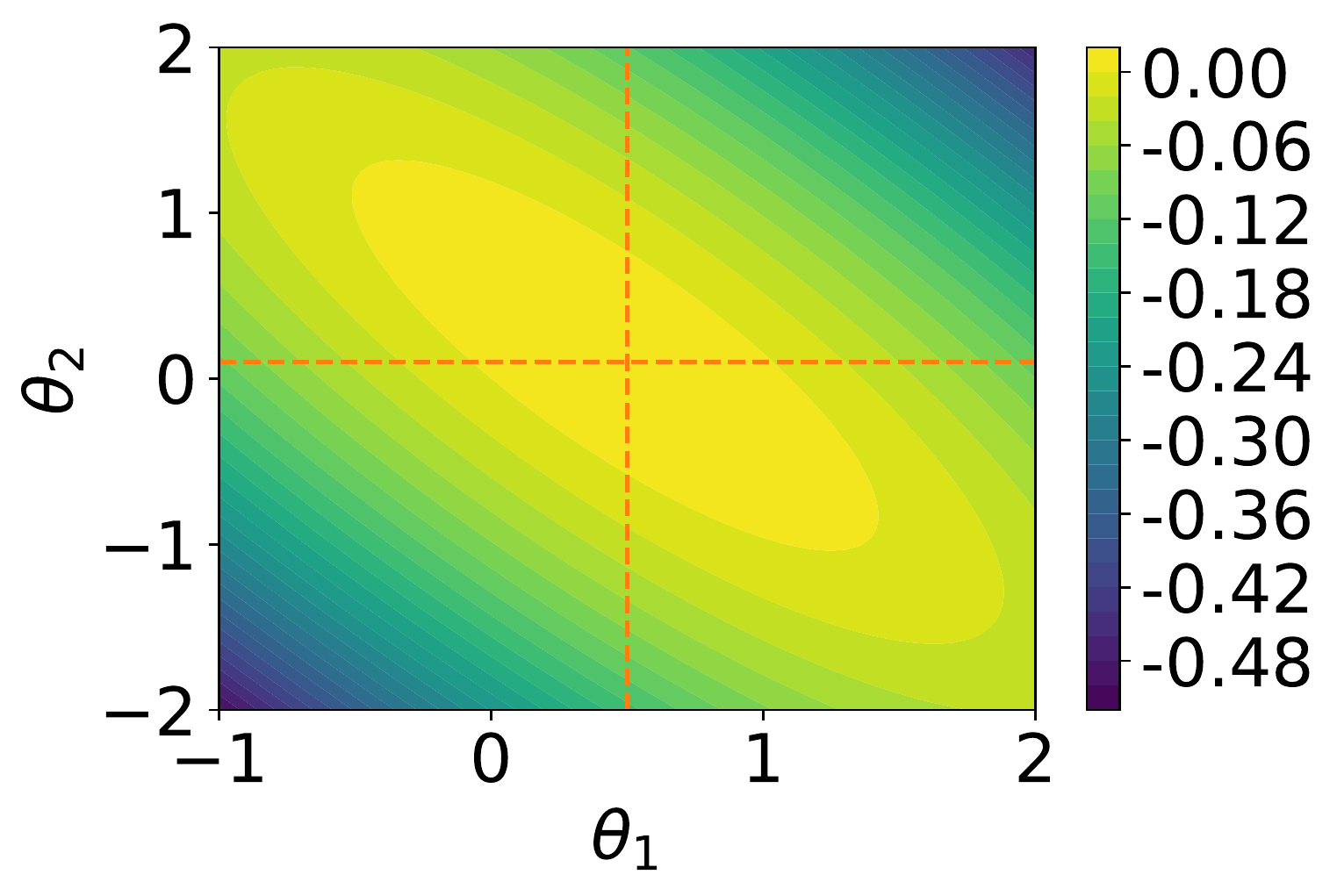}}} 
\subfloat[$N=10$.]{{\label{fig_10c}\includegraphics[width=.25\linewidth]{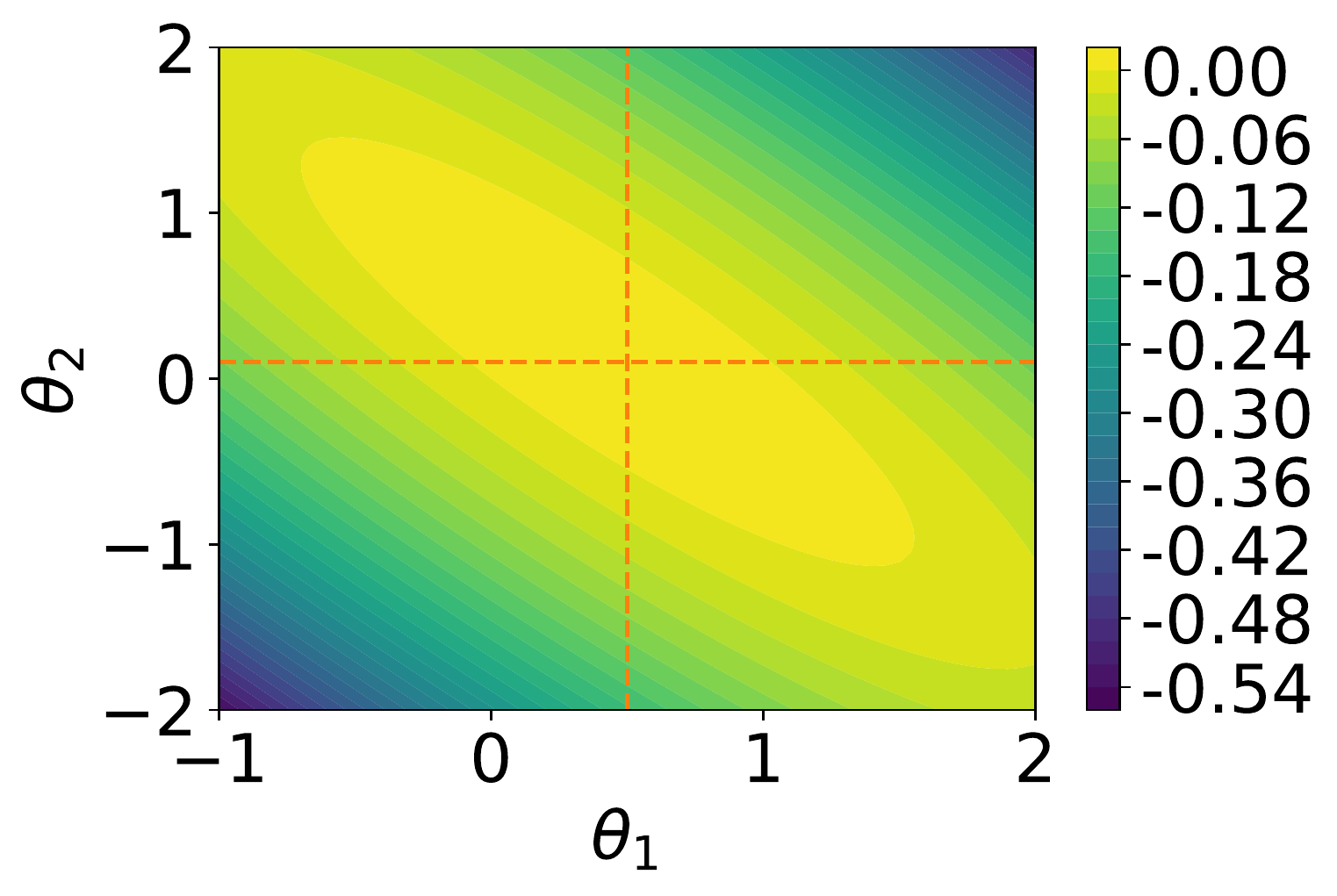}}}
\subfloat[$N=100$.]{{\label{fig_10d}\includegraphics[width=.25\linewidth]{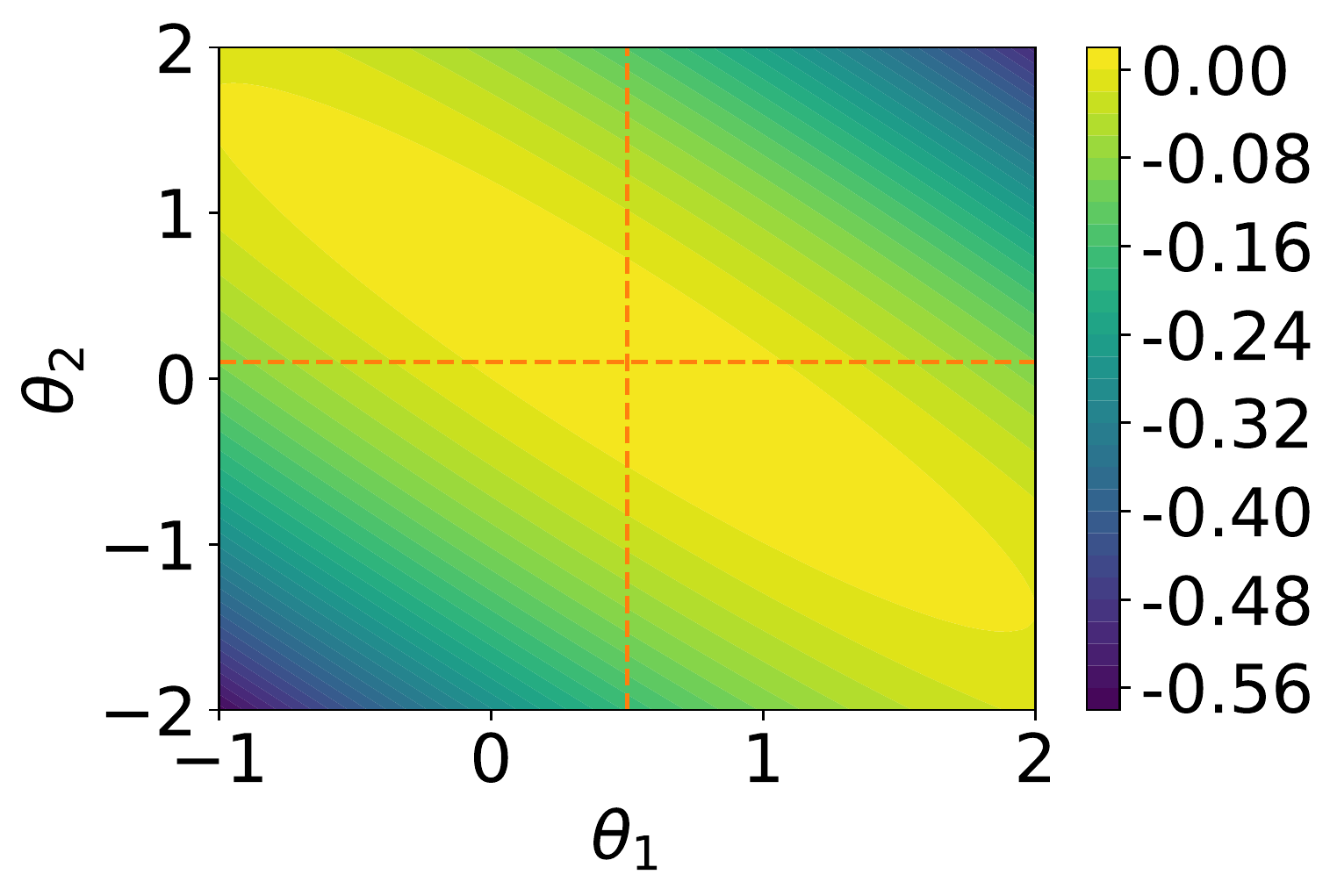}}}
\caption{Contour plots of the asymptotic log-likelihood $\tilde{\mathcal{L}}^{N}(\theta)$ for $N=\{2,5,10,100\}$.}
\label{fig10}
\end{figure}

\subsection{Stochastic Opinion Dynamics}

We now consider a one-dimensional stochastic opinion dynamics model, parametrised by $\theta = (\theta_1,\theta_2)^T\in\mathbb{R}^2$, of the form
\begin{equation}
\mathrm{d}x_t = -\left[\int_{\mathbb{R}}\varphi_{\theta}(||x_t - y||)(x_t-y)\mu_t(\mathrm{d}y)\right]\mathrm{d}t + \sigma\mathrm{d}w_t,
\end{equation}
where $\sigma>0$, $w=(w_t)_{t\geq 0}$ is a standard Brownian motion, $x_0\in\mathbb{R}$, and the {interaction kernel} $\varphi_{\theta}:\mathbb{R}_{+}\rightarrow\mathbb{R}_{+}$ is given by the scaled indicator function ${\varphi}_{\theta}(r) = \theta_1 \mathds{1}_{r\in[0,\theta_2]}$. In practice, we require a differentiable approximation to this function, and will thus use
\begin{equation}
\varphi_{\theta}(r) = \left\{ \begin{array}{ccc} \theta_1\exp\left[-\dfrac{0.01}{1-(r-\theta_2)^2}\right] & , & r>0 \\[2mm] 0 &, &  r\leq 0. \end{array}\right. \label{interac}
\end{equation}
This model is perhaps more frequently specified in terms of the corresponding system of interacting particles, which is given by
\begin{equation}
\mathrm{d}x_t^{i,N} = -\frac{1}{N}\sum_{j=1}^N \varphi_{\theta}(||x_t^{i,N} - x_t^{j,N}||)(x_t^{i,N}-x_t^{j,N})\mathrm{d}t + \sigma\mathrm{d}w_t.
\end{equation}
In this model, we can interpret $\theta_1$ as a scale parameter, which controls the strength of the attraction between particles, and $\theta_2$ as a range parameter, which determines the distance within which particles must be of one another in order to interact. 

Models of this form arise in various applications, from biology to the social sciences, in which $\varphi_{\theta}$ determines how the dynamics of one particle (e.g., the opinions of one person) may influence the dynamics of other particles (e.g., the opinions of other people). For a more detailed account of such models, we refer to \cite{Brugna2015,Chazelle2017,Garnier2017,Lu2021,Motsch2014} and references therein.  For deterministic models of this type, it is well known that, asymptotically, the particles merge into clusters, the number of which depends both on the interaction kernel (i.e., the range and strength of the interaction between particles) and the initialisation. In the stochastic setting, the random noise prohibits the formation of exact clusters; instead, the particles merge into metastable `soft clusters' (see also \cite{Lu2021}). This is shown in Figure \ref{fig11} and Figure \ref{fig12}.

\begin{figure}[htbp]
\centering
\subfloat[$\theta_2=0.0$.]{\label{fig_11a}\includegraphics[width=.25\linewidth]{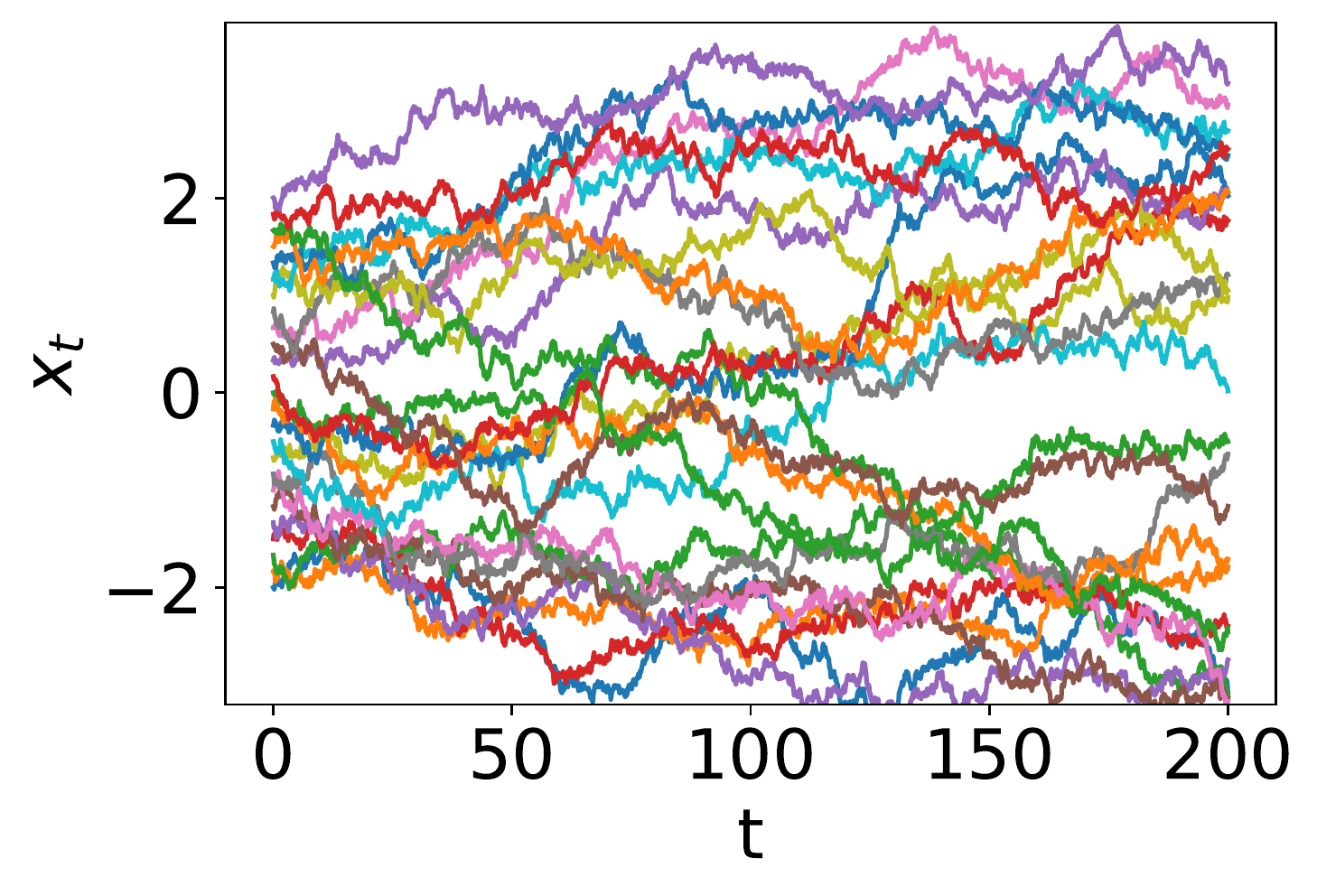}}
\subfloat[$\theta_2=0.3$.]{{\label{fig_11b}\includegraphics[width=.25\linewidth]{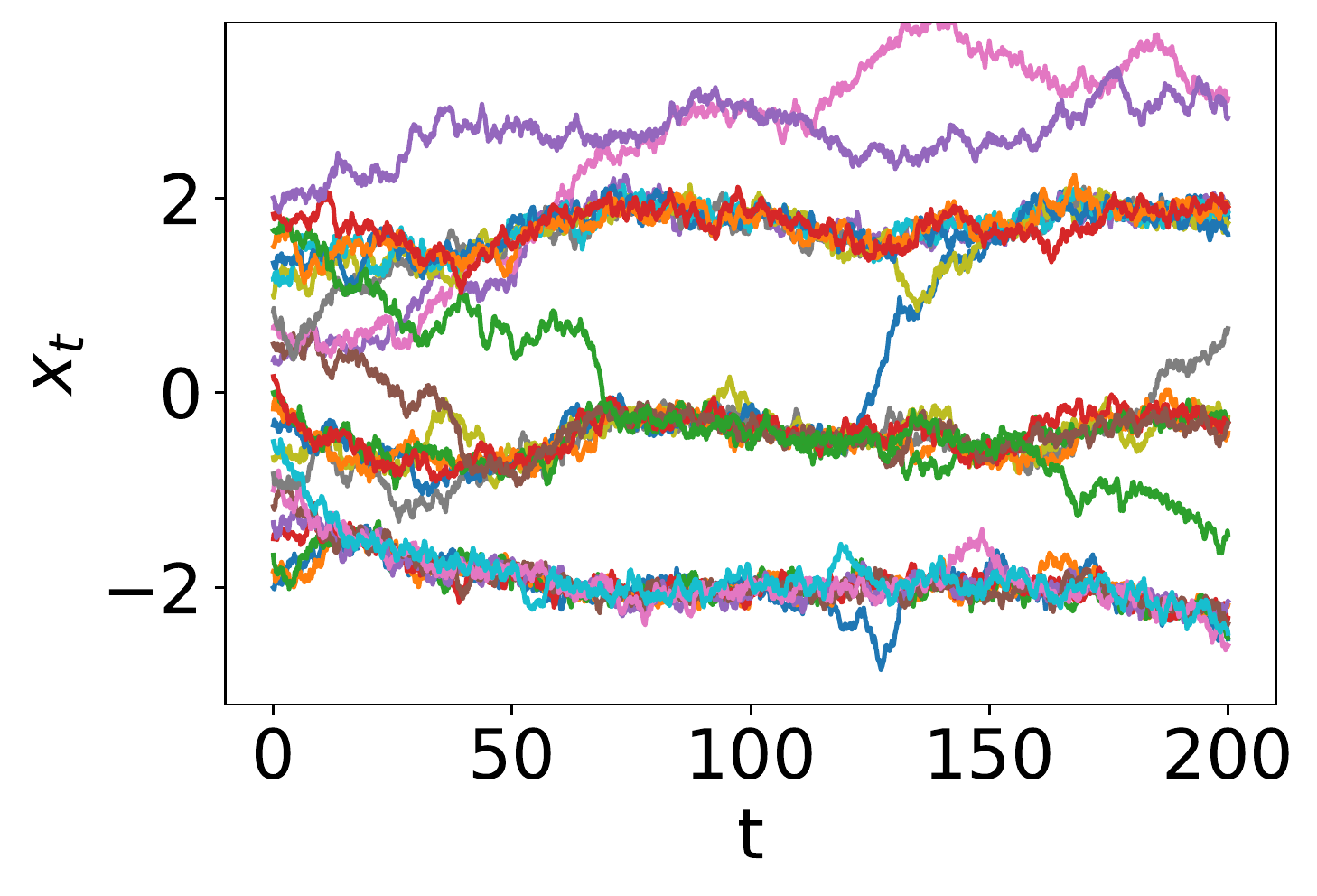}}} 
\subfloat[$\theta_2=0.5$.]{\label{fig_11c}\includegraphics[width=.25\linewidth]{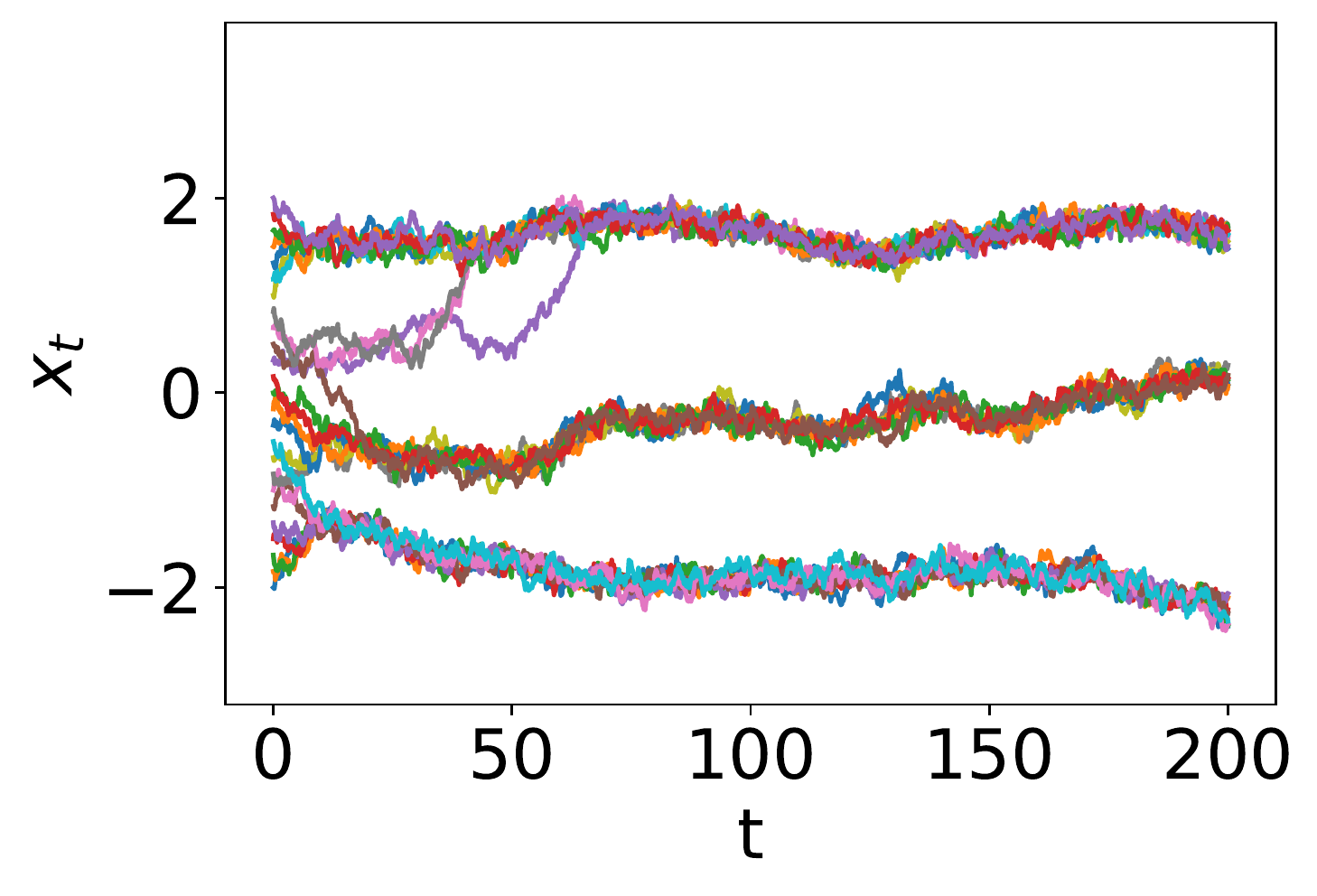}}
\subfloat[$\theta_2=1.0$.]{{\label{fig_11d}\includegraphics[width=.25\linewidth]{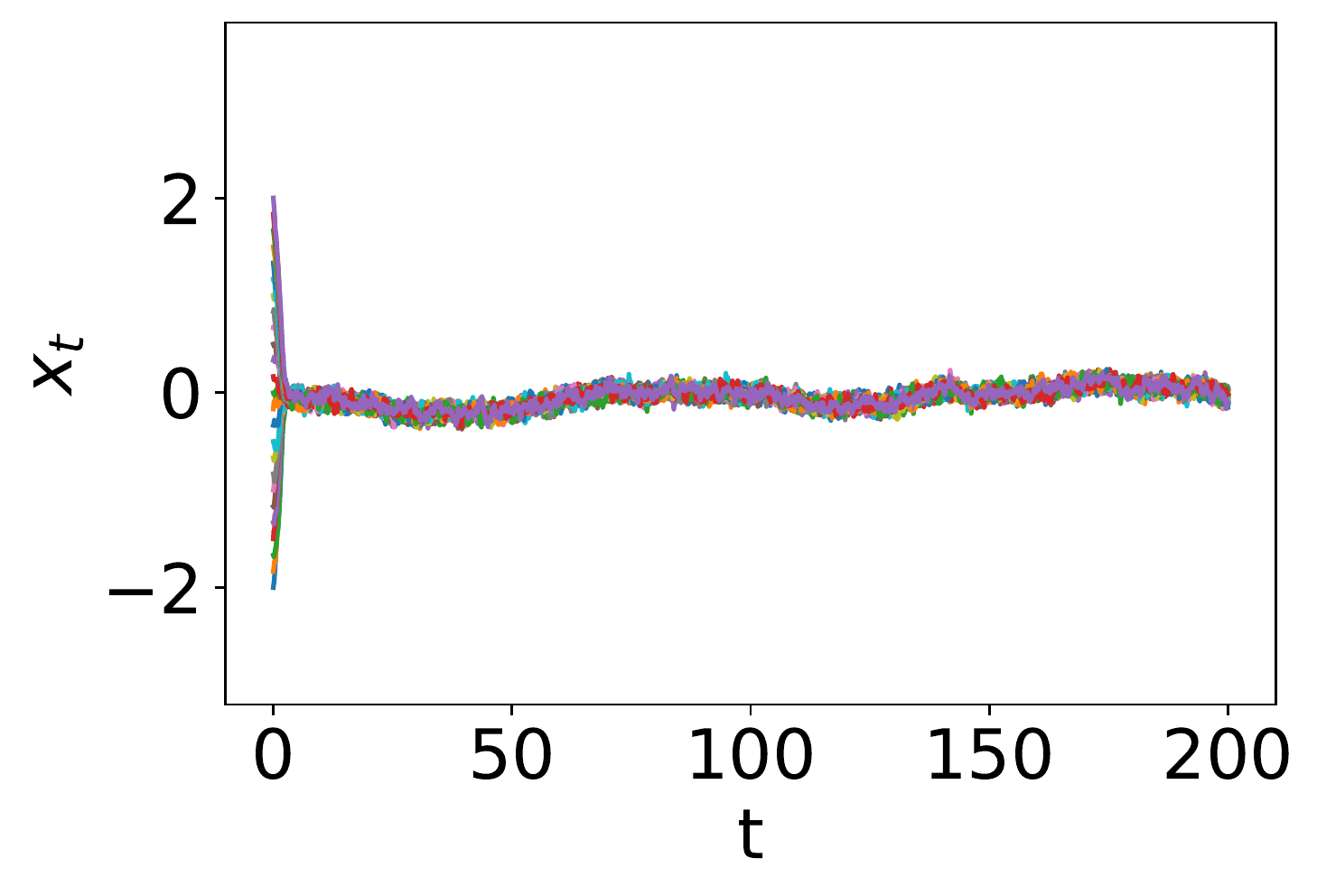}}}
\caption{Sample trajectories of the system of interacting particles for $\theta_2=\{0.0,0.3,0.5,1.0\}$. }
\label{fig11}
\end{figure}

We provide illustrative results for the case in which the scale parameter $\theta_1$ is fixed, and the range parameter $\theta_2$ is to be estimated. We assume that $\theta_{0}=(2,0.5)$. This corresponds to an interaction kernel with with compact support on $[0,0.5]$. The initial parameter estimates are generated uniformly at random on $[1.5,2.5]$. Finally, we use constant learning rates with $\gamma_{2,t} = 0.002$. The performance of the recursive MLE is illustrated in Figure \ref{fig13}, in which we plot the sequence of online parameter estimates for $\theta_2$, for several values of $N$, and for 50 different random initialisations. 

\begin{figure}[htbp]
\centering
\subfloat[$N=10$.]{\label{fig_12a}\includegraphics[width=.33\linewidth]{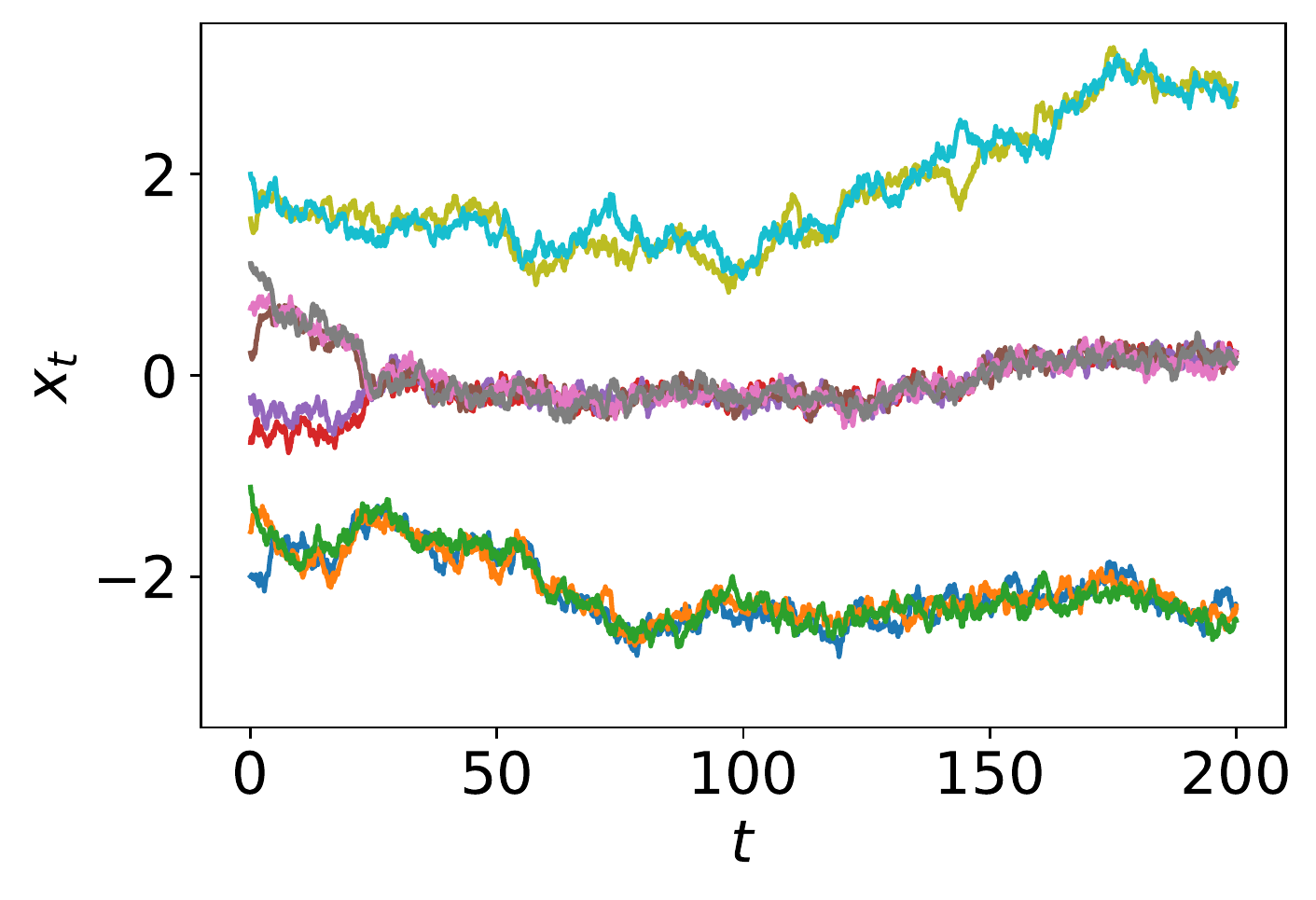}}
\subfloat[$N=20$.]{{\label{fig_12b}\includegraphics[width=.33\linewidth]{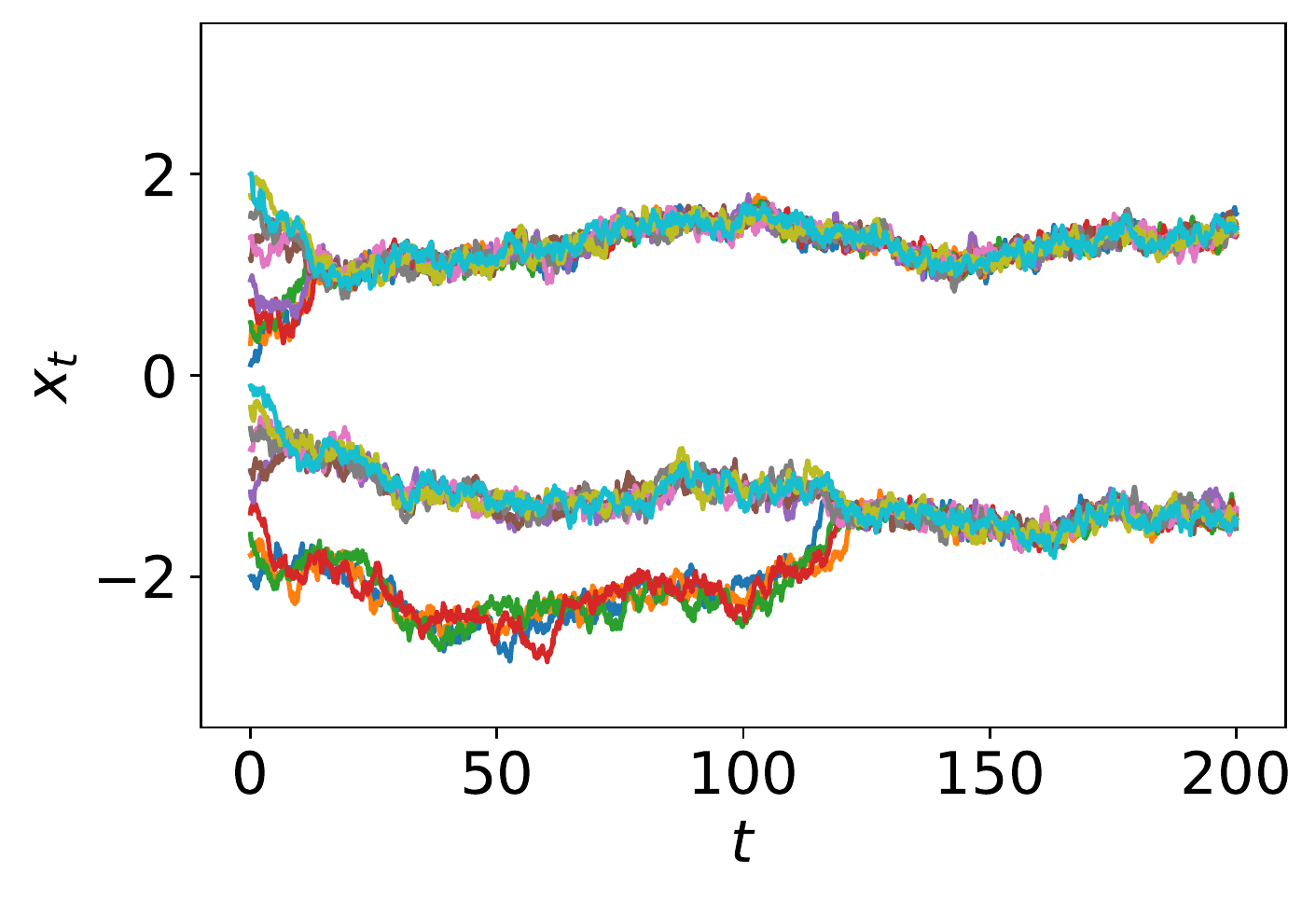}}} 
\subfloat[$N=50$.]{{\label{fig_12c}\includegraphics[width=.33\linewidth]{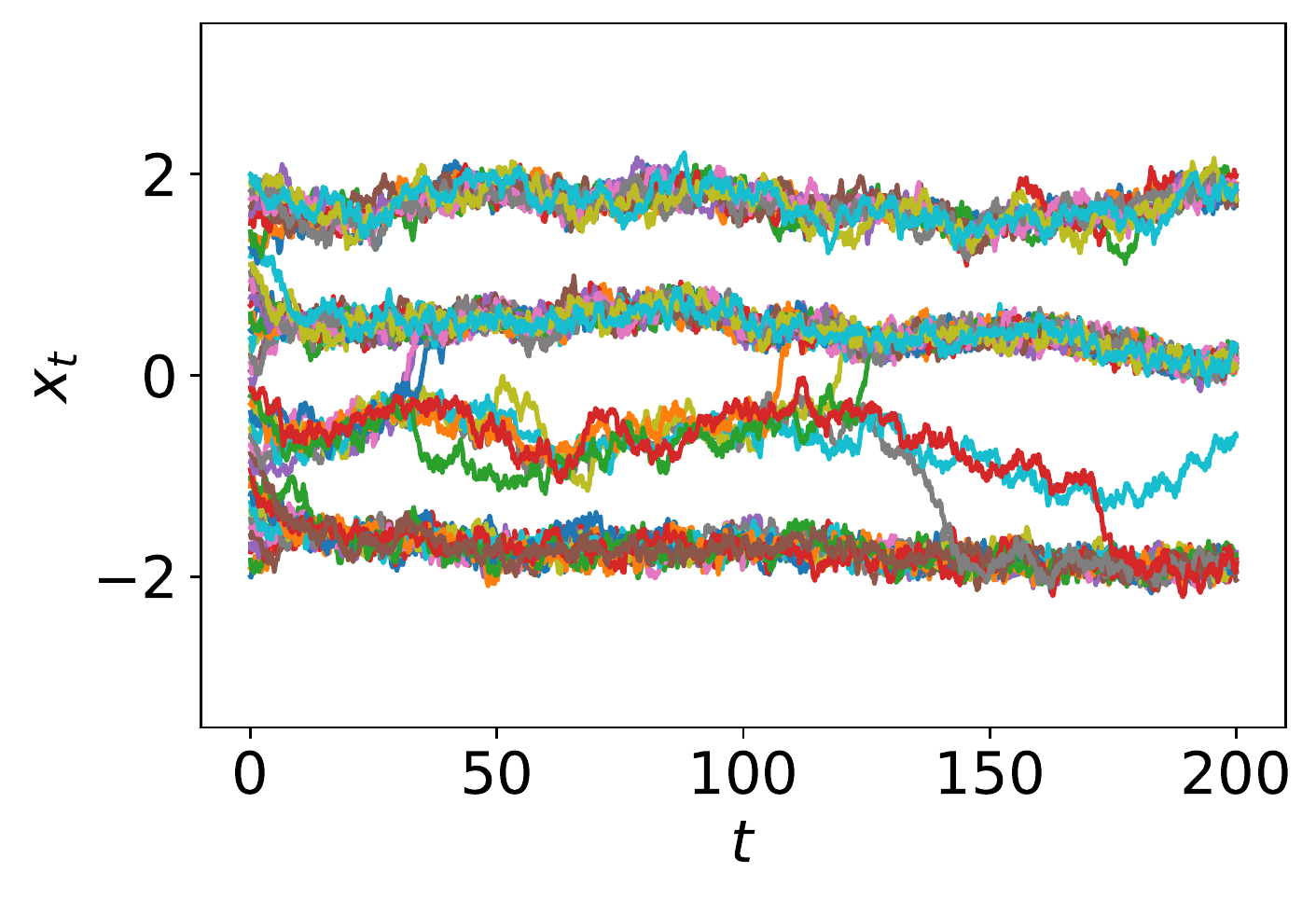}}}
\caption{Sample trajectories of the system of interacting particles for $N=\{10,20,50\}$. }
\label{fig12}
\end{figure}

Encouragingly, the vast majority of the online parameter estimates converge to within a small neighbourhood of the true value of the parameter, suggesting that it is indeed possible to estimate the range of the interaction kernel in an online fashion. As before, the performance of the online estimator improves as the number of particles is increased. We should remark that, in this case, the performance is highly dependent on the initial conditions of the particles. This should not come as a surprise; indeed, if the distance between particles is greater than the support of the interaction kernel, then the interaction kernel (and its gradient) are identically zero, and thus so too are all of the terms in the parameter update equation. Thus, the value of the parameter estimate will remain unchanged. We see this phenomenon in Figure \ref{fig13}, particularly when there are fewer particles.  

\begin{figure}[htbp]
\centering
\subfloat[$N=10$.]{\label{fig_13a}\includegraphics[width=.33\linewidth]{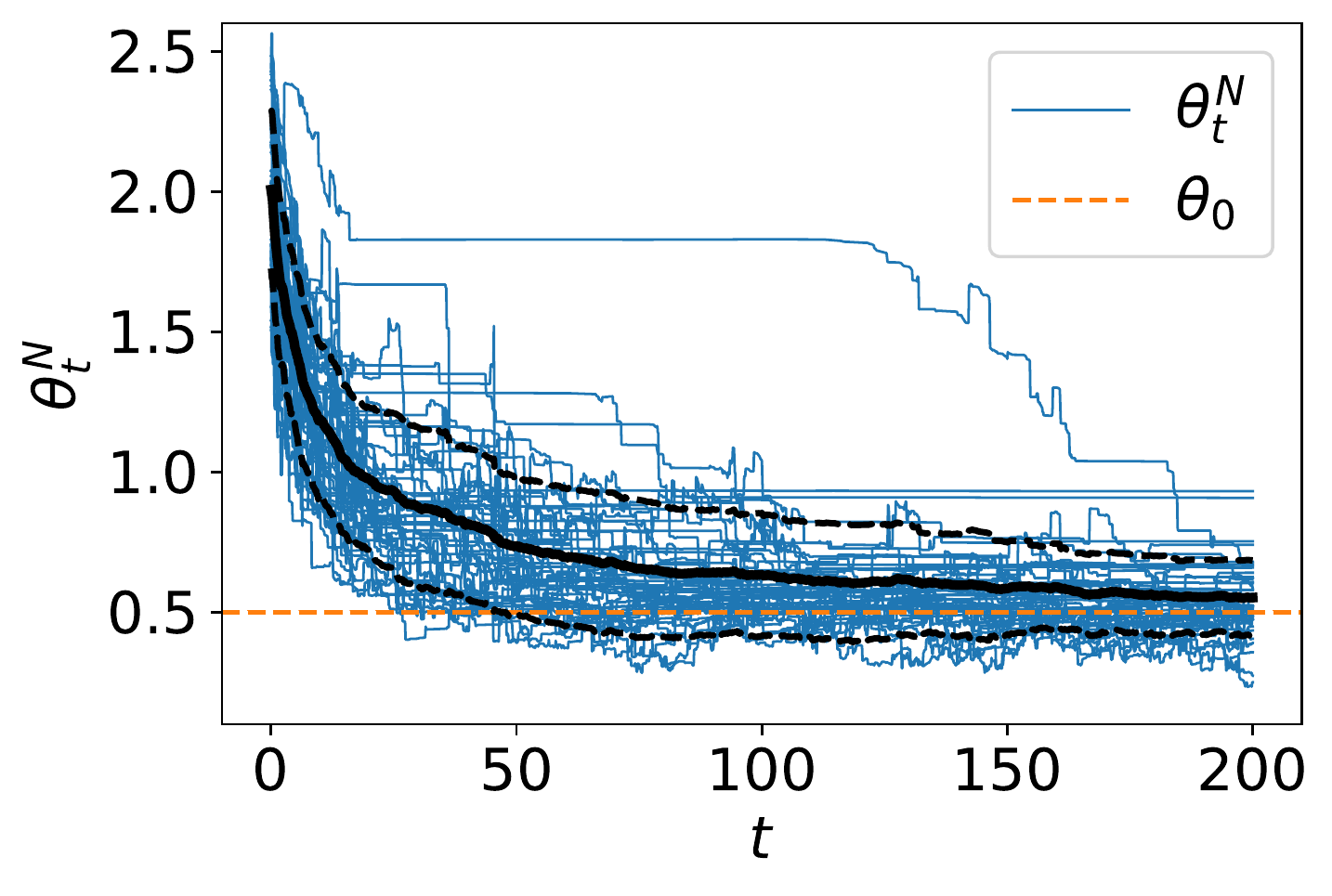}}
\subfloat[$N=20$.]{{\label{fig_13b}\includegraphics[width=.33\linewidth]{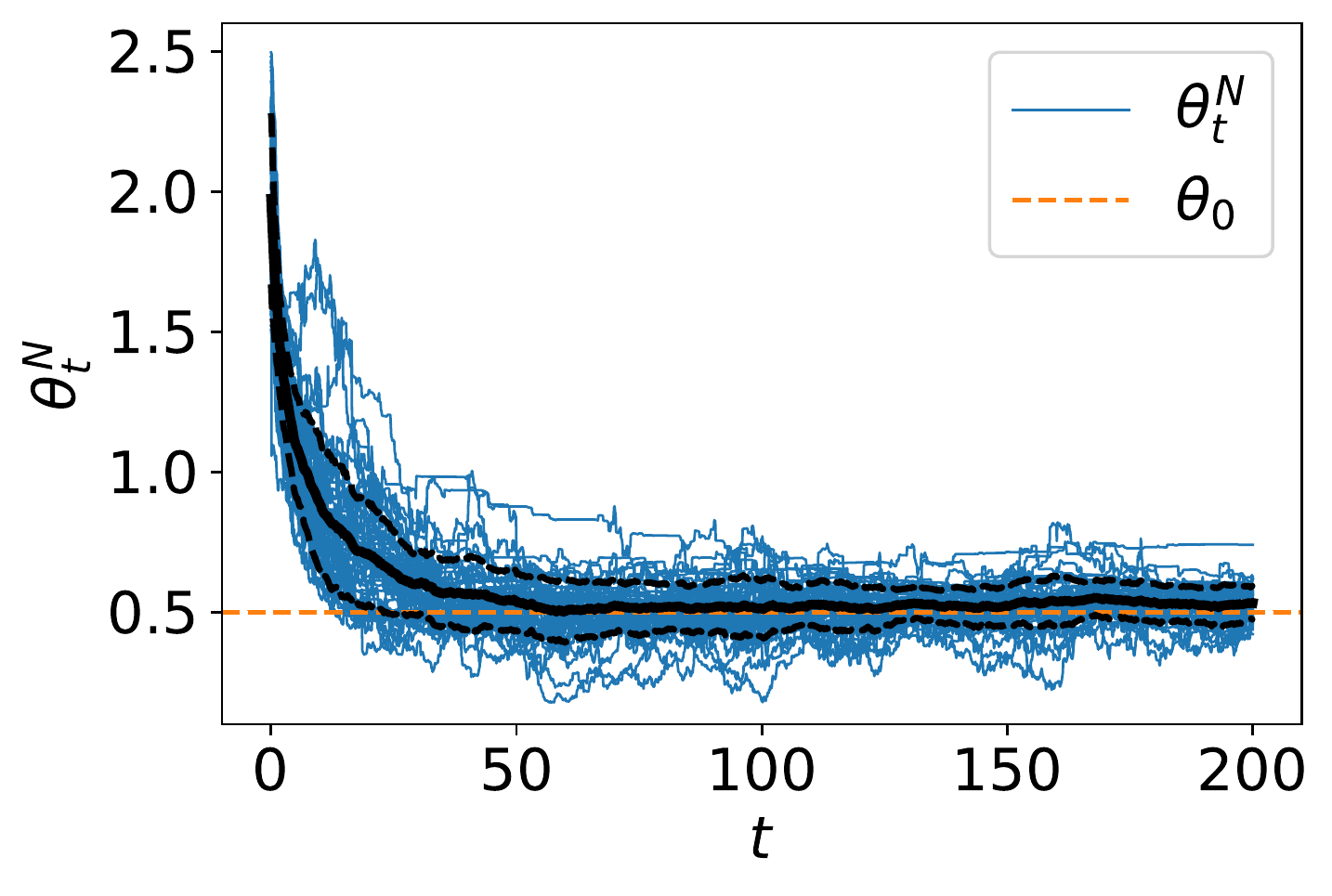}}} 
\subfloat[$N=50$.]{{\label{fig_13c}\includegraphics[width=.33\linewidth]{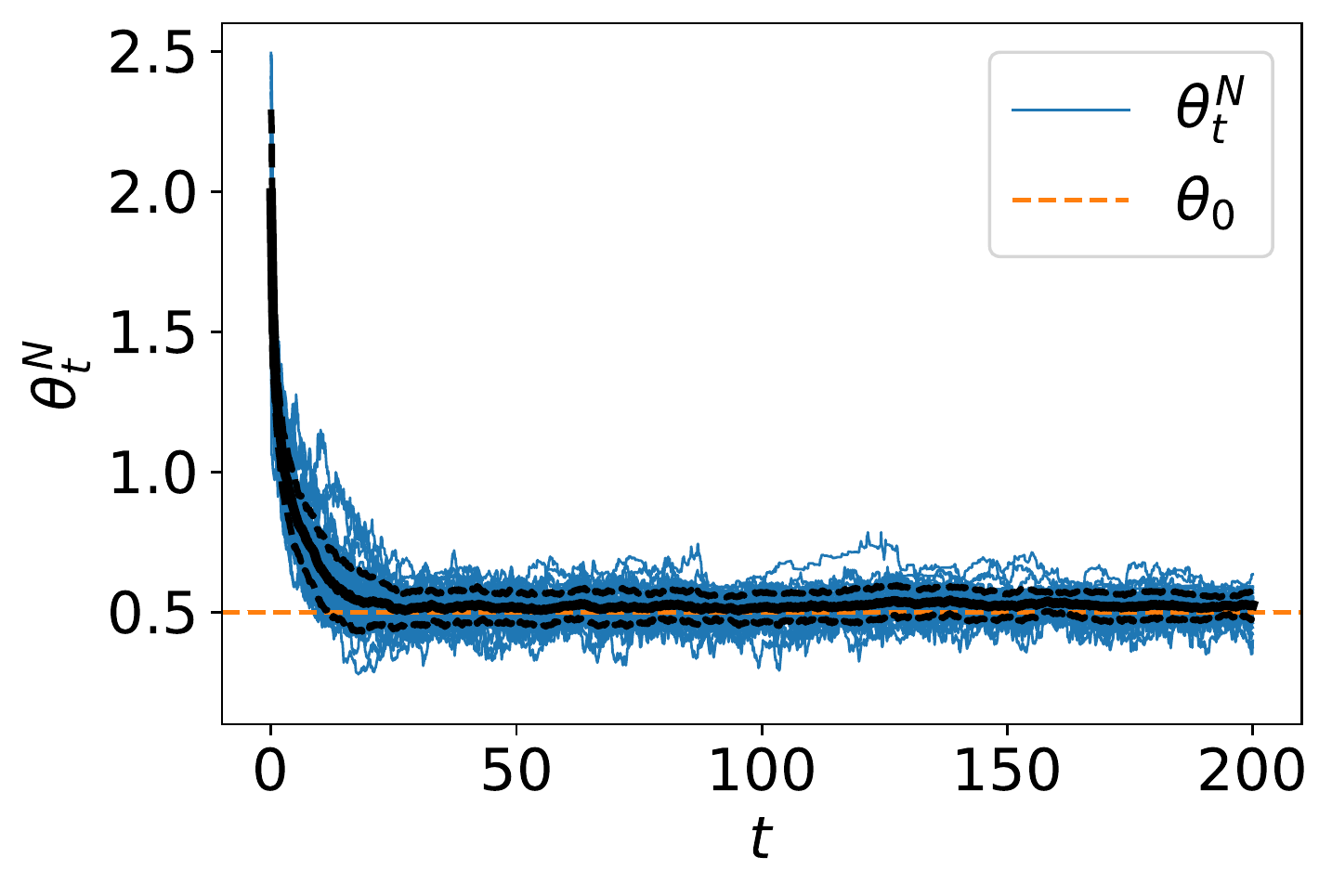}}}
\caption{Sequence on online parameter estimates (blue) for the range parameter $\theta_2$, for 50 different random initialisation $\theta_2^{0}\sim \mathcal{U}([1.5,2.5])$, and $N=\{10,20,50\}$. We also plot the true parameter value (orange), the mean online parameter estimate plus/minus one standard deviation (black: solid, dashed).}
\label{fig13}
\end{figure}

\section{Conclusions}
\label{sec:conclusions}
In this paper, we have considered the problem of parameter estimation for a stochastic McKean-Vlasov equation and the associated system of weakly interacting particles. We established consistency and asymptotic normality of the offline MLE for the IPS as the number of particles $N\rightarrow\infty$, extending classical results in \cite{Kasonga1990}. We also proposed an online estimator for the parameters of the stochastic McKean-Vlasov equation, and studied its asymptotic properties. 

Regarding directions for future research, in the offline case, it may be of interest to establish a non-asymptotic $\mathbb{L}^p$ convergence rate for the MLE in both the mean-field and long time regimes, extending the recent results in \cite{Chen2020} to a more general class of IPSs. In the online case, a natural extension of our results is to obtain a central limit theorem for the recursive estimator, extending the results in \cite{Sirignano2020a} to non-linear McKean-Vlasov diffusions. Alternatively, one could aim to extend our results to the case in which the diffusion coefficient is unknown, and must be estimated online (see \cite{Sirignano2017a} for online estimation of the diffusion coefficient in the linear case, and \cite{Huang2019} for offline estimation of the diffusion coefficient in IPSs). This is a particularly interesting problem given that, for a broad class of McKean-Vlasov SDEs, the uniqueness (or non-uniqueness) of the invariant measure(s) is known to depend on the magnitude of the noise coefficient (e.g., \cite{Herrmann2010,Herrmann2012,Tugaut2013}).


\appendix{

\section{Existing Results on the McKean-Vlasov SDE} \label{appendix_existing_results}
In this appendix, we present some existing results on the McKean-Vlasov SDE \eqref{MVSDE}, and the associated system of interacting particles \eqref{IPS}. For completeness, we also include the proofs of most of these results. These results are well known, and have appeared in a number of existing papers (e.g., \cite{Bolley2013,Cattiaux2008,Malrieu2001,Malrieu2003,Veretennikov2006}), albeit under conditions which are different to our own. Under our particular set of assumptions, these results can also be found in \cite{Bustos2008}.

\begin{proposition}[Moment Bounds] \label{prop_moment_bounds}
Assume that Conditions \ref{assumption_init} and \ref{assumption1}(i) - \ref{assumption2}(i) hold. Then, for all $N\in\mathbb{N}$, $i=1,\dots,N$, and for all $k\geq 0$, there exists $C_{k,\theta_0}>0$ such that 
\begin{align}
&\sup_{t\geq 0}\mathbb{E}_{\theta_0} ||x_t||^{k}\leq C_{k,\theta_0}\left(\int_{\mathbb{R}^d}||x||^{k}\mu_0(\mathrm{d}x)+1\right) \\
&\sup_{t\geq 0}\mathbb{E}_{\theta_0} ||x_t^{i,N}||^{k}\leq C_{k,\theta_0}\left(\int_{\mathbb{R}^d}||x||^{k}\mu_0(\mathrm{d}x)+1\right) \label{eq:bounds}
\end{align}
\end{proposition}
\begin{proof}
We follow the approach in {\cite[Lemma 2.3.1]{Bustos2008}} (see also \cite[Corollary 2.3]{Cattiaux2008}). We will establish the first bound in \eqref{eq:bounds}, with the second bound obtained similarly. We begin by applying It\^o's Lemma to $||\cdot||^{2k}$, and taking expectations, which yields
\begin{align}
\mathbb{E}_{\theta_0} \left[ ||x_t||^{2k} \right] = \mathbb{E}_{\theta_0}\left[||x_0||^{2k}\right] &+ 
\mathbb{E}_{\theta_0}\left[ \int_0^t 2k ||x_s||^{2k-2} \langle x_s, B(\theta_0,x_s,\mu_s)\rangle \mathrm{d}s\right] \\
&+[kd+2k(k-1)]\mathbb{E}_{\theta_0}\left[\int_0^t ||x_s||^{2k-2} \mathrm{d}s\right]  \label{exp}
\end{align}
Due to Conditions \ref{assumption1}(i) - \ref{assumption2}(i), there exists $C_{\theta_0}>0$ such that
$\langle x_s, B(\theta_0,x_s,\mu_s) \rangle \leq -(\alpha_{\theta_0} - L_{\theta_0,2}) ||x_s||^2 + C_{\theta_0}||x_s|| + L_{\theta_0,2}||x_s|| \mathbb{E}_{\theta_0}\left[ ||x_s||\right]$. 
It follows, defining $C_{k} = \frac{d}{2} + k-1 >0$, that
\begin{align}
\mathbb{E}_{\theta_0} \left[ ||x_t||^{2k} \right] &\leq \mathbb{E}_{\theta_0}\left[||x_0||^{2k}\right] -  
2k \mathbb{E}_{\theta_0}\left[ \int_0^t (\alpha_{\theta_0} - L_{\theta_0,2}) ||x_s||^{2k} \mathrm{d}s\right] \\
&+ 2k \mathbb{E}_{\theta_0}\left[ \int_0^t C_{\theta_0} ||x_s||^{2k-1} \mathrm{d}s\right] + 2k \mathbb{E}_{\theta_0}\left[ \int_0^t L_2||x_s||^{2k-1} \mathbb{E}_{\theta_0}\left[ ||x_s||\right]  \mathrm{d}s\right] \\
&+2k \mathbb{E}_{\theta_0}\left[\int_0^t C_{k} ||x_s||^{2k-2} \mathrm{d}s\right]  \nonumber
\end{align}
By Young's inequality, we then have 
\begin{align}
\mathbb{E}_{\theta_0} \left[ ||x_t||^{2k} \right] &\leq \mathbb{E}_{\theta_0}\left[||x_0||^{2k}\right] -  
2k \mathbb{E}_{\theta_0}\left[ \int_0^t (\alpha_{\theta_0} - L_{\theta_0,2}) ||x_s||^{2k} \mathrm{d}s\right] \\
&+ 2k \mathbb{E}_{\theta_0}\left[ \int_0^t C_{\theta_0} ||x_s||^{2k-1} \mathrm{d}s\right] + 2k \mathbb{E}_{\theta_0}\left[ \int_0^t L_2 \frac{2k-1}{2k} ||x_s||^{2k}  \mathrm{d}s\right] \\
&+ 2k \mathbb{E}_{\theta_0}\left[ \int_0^t L_2 \frac{1}{2k} ||x_s||^{2k}  \mathrm{d}s\right]+ 2k \mathbb{E}_{\theta_0}\left[\int_0^t C_{k} ||x_s||^{2k-2} \mathrm{d}s\right] 
\end{align}
It follows, writing $\lambda_{\theta_0} = \alpha_{\theta_0} - L_{\theta_0,2}$, that 
\begin{align}
\mathbb{E}_{\theta_0} \left[ ||x_t||^{2k} \right] &\leq \mathbb{E}_{\theta_0} \left[||x_0||^{2k}\right] +  
2k \left[ \int_0^t  \mathbb{E}_{\theta_0}\left[-\lambda_{\theta_0} ||x_s||^{2k}+C_{\theta_0} ||x_s||^{2k-1} +C_{k}  ||x_s||^{2k-2} \right] \mathrm{d}s \right] 
\end{align}
We next use the fact that, for any $a>0$, there exists $b$ such that $C_{\theta_0}||x||^{2k-1} + C_{k} ||x||^{2k-2} \leq a||x||^{2k} + b$. This implies, in particular, that there exists $C_{k,\theta_0,1}>0$, $C_{k,\theta_0,2}$ such that 
\begin{align}
\mathbb{E}_{\theta_0} \left[ ||x_t||^{2k} \right] &\leq \mathbb{E}_{\theta_0} \left[||x_0||^{2k}\right] +  
2k \left[ \int_0^t  -C_{k,\theta_0,1} \mathbb{E}_{\theta_0}\left[ ||x_s||^{2k} \right] + C_{k,\theta_0,2}  \mathrm{d}s \right].
\end{align}
 The result now follows via an extension of Gr\"onwall's inequality (see, e.g., \cite[Proposition 2.1]{Cattiaux2008}).
\end{proof}

\begin{proposition}[Asymptotic Moment Bounds] \label{lemma_invariant_moment_bounds}
Assume that Conditions \ref{assumption_init} and \ref{assumption1} - \ref{assumption2} hold. Then, for all $N\in\mathbb{N}$, for all $i=1,\dots,N$, and for all $k\in\mathbb{N}$, there exists a positive constant $K_{k,\theta_0}>0$ such that 
\begin{align}
\int_{\mathbb{R}^d} ||x||^k \mu_{\infty}(\mathrm{d}x)\leq K_{k,\theta_0}~~~,~~~\int_{(\mathbb{R}^d)^N} ||{x}^{i,N}||^k \hat{\mu}^{N}_{\infty}(\mathrm{d}\hat{x}^N)\leq K_{k,\theta_0}.
\end{align}
\end{proposition}

\begin{proof}
By Proposition \ref{prop_invariant_measure} (see below), the McKean-Vlasov SDE  \eqref{MVSDE} admits a unique equilibrium measure $\mu_{\infty}$ which is independent of $\mu_0$. By the ergodic theorem (e.g., \cite[Chapter X]{Revuz1999}), we thus have, for all $k\in\mathbb{N}$,
\begin{equation}
\lim_{t\rightarrow\infty}\frac{1}{t} \int_0^{t} ||x_s||^k\mathrm{d}s = \int_{\mathbb{R}^d}||x||^k\mu_{\infty}(\mathrm{d}x)~,~~~\mathrm{a.s.} \label{eq_b49}
\end{equation}
Using Jensen's inequality and the moment bounds in Proposition \ref{prop_moment_bounds}, we obtain uniform integrability of the family $\{\frac{1}{t}\int_0^t ||x_s||^k\mathrm{d}s\}_{t>0}$. It follows, taking expectations of \eqref{eq_b49}, using uniform integrability in order to interchange the limit and the expectation,  and once more making use of Proposition \ref{prop_moment_bounds}, that
\begin{align}
\int_{\mathbb{R}^d} ||x||^k \mu_{\infty}(\mathrm{d}x) 
&= \lim_{t\rightarrow\infty}\left[\frac{1}{t} \int_0^{t} \mathbb{E}_{\theta_0}\left[||x_s||^k\right]\mathrm{d}s \right] \\
&\leq \sup_{\mu_0\in\mathcal{P}_k(\mathbb{R}^d)} C_{k,\theta_0}\left(\int_{\mathbb{R}^d}||x||^{k}\mu_0(\mathrm{d}x)+1\right) \leq K_{k,\theta_0},
\end{align}
where in the final inequality we have used Condition \ref{assumption_init}. The proof of the bound for the IPS is identical, noting that all of the relevant results (Propositions \ref{prop_moment_bounds} and \ref{prop_invariant_measure}) apply to both the McKean-Vlasov SDE and the IPS.  
\end{proof}

\begin{proposition}[Unique Invariant Measure]
\label{prop_invariant_measure}
Assume that Conditions \ref{assumption_init} and \ref{assumption1} - \ref{assumption2} hold. Then, for $\theta=\theta_0$, the McKean-Vlasov SDE \eqref{MVSDE} and IPS \eqref{IPS} admit unique invariant measures $\smash{\mu_{\infty}}$ and $\smash{\hat{\mu}_{\infty}^{N}}$ which are independent of the initial conditions $\mu_0$ and $\smash{(\hat{\mu}_0)^{\otimes N}}$. Moreover, writing $\lambda_{\theta_0} = \alpha_{\theta_0} - 2L_{\theta_0,2}$, and $\smash{\hat{\mu}_t^{(k),N}}$ for the law of a subset of $1\leq k\leq N$ interacting particles, the following contraction rates hold
\begin{equation}
\mathbb{W}_2(\mu_t,\mu_{\infty})\leq e^{-\lambda_{\theta_0} t} \mathbb{W}_2(\mu_0,\mu_{\infty})~~~,~~~\mathbb{W}_2(\hat{\mu}_t^{(k),N},\hat{\mu}_{\infty}^{(k),N})\leq e^{-\lambda_{\theta_0} t}\mathbb{W}_2((\mu_0)^{\otimes k},\hat{\mu}_{\infty}^{(k),N}). \label{eq:ergodic_bounds}
\end{equation}
\end{proposition}

\begin{proof}
We follow the approach in {\cite[Theorem 2.3.3]{Bustos2008}} (see also \cite[Theorem 3.1]{Wang2018}). We prove the first inequality in \eqref{eq:ergodic_bounds}, with the second derived in a similar fashion. Let $({x}_t^{x})_{t\geq 0}$ and $({x}_t^{y})_{t\geq 0}$ be solutions of \eqref{MVSDE}, starting from ${x}$ and ${y}$, respectively. We will write $\mu_{t}$ for the law of $x_t^{x}$, and ${\nu}_t$ for the law of $x_t^{y}$. By It\^o's formula, we have
\begin{equation}
\frac{\mathrm{d}}{\mathrm{d}t} \mathbb{E}_{\theta_0}\left[ ||x_t^{x} - x_t^{y} ||^2\right] = 2 \mathbb{E}_{\theta_0}\left[ \left\langle x_t^{x} - x_t^{y}, B(\theta_0,x_t^{x},\mu_t) - B(\theta_0,x_t^{y},\nu_t) \right\rangle \right].
\end{equation}
Using Conditions \ref{assumption1} and \ref{assumption2}, it follows that 
\begin{equation}
\frac{\mathrm{d}}{\mathrm{d}t} \mathbb{E}_{\theta_0}\left[ ||x_t^{x} - x_t^{y} ||^2\right] \leq -2(\alpha_{\theta_0} - 2L_{\theta_0,2}) \mathbb{E}_{\theta_0}\left[ ||x_t^{x} - x_t^{y}||^2\right],
\end{equation}
and thus, writing $\lambda_{\theta_0} = \alpha_{\theta_0} - 2L_{\theta_0,2}$, that $\mathbb{E}_{\theta_0}[ ||x_t^{x} - x_t^{y} ||^2] \leq e^{-2\lambda_{\theta_0} t}||{x}-{y}||^2$. Let $\pi_0$ be an arbitrary coupling of $\mu_0$ and $\nu_0$. By Condition \ref{assumption_init}, we have that $\mu_0,\nu_0\in\mathcal{P}_2(\mathbb{R}^d)$. We thus have from the previous inequality that
\begin{equation}
\int_{\mathbb{R}^d} \mathbb{E}_{\theta_0}\left[||x_t^{x} - x_t^{y}||^2 \right] \pi_0(\mathrm{d}x,\mathrm{d}y) \leq e^{-2\lambda_{\theta_0} t} \int_{\mathbb{R}^d} ||x-y||^2\pi_0(\mathrm{d}x,\mathrm{d}y).
\end{equation}
It follows, taking the infimum over all coupling measures $\pi_0$, and taking square roots, that 
\begin{equation}
\mathbb{W}_2({\mu}_t,{\nu}_t) \leq e^{-\lambda_{\theta_0} t}\mathbb{W}_2({\mu}_0,{\nu}_0). \label{eq:contraction}
\end{equation}
We now establish the existence of an equilibrium measure for the McKean-Vlasov SDE. Let $(x_t)_{t\geq 0}$ denote a solution of \eqref{MVSDE} with initial law $\mu_0$, and $(y_t)_{t\geq 0}$ the solution of \eqref{MVSDE} with initial law $\mathcal{L}(x_s)$, for some $s>0$. Using \eqref{eq:contraction}, homogeneity, and Proposition \ref{prop_moment_bounds}, there exists $K>0$ such that
\begin{equation}
\mathbb{W}_2(\mathcal{L}({x}_t),\mathcal{L}({x}_{t+s})) = \mathbb{W}_2(\mathcal{L}({x}_t),\mathcal{L}({y}_t)) \leq e^{-\lambda_{\theta_0} t} \mathbb{W}_2({\mu}_0,\mathcal{L}({x}_s)) \leq K_{\theta_0}e^{-\lambda_{\theta_0} t}.
\end{equation} 
Thus, $\mathcal{L}({x}_t)$ converges to a unique limit ${\mu}_{\infty}$, which is independent of the initial measure. We will now demonstrate that this equilibrium measure is also the unique invariant measure of the McKean-Vlasov SDE. Let $0<t_1<t_2$. Consider the functional $P_{t_1}:\mathcal{P}(\mathbb{R}^d)\rightarrow\mathcal{P}(\mathbb{R}^d)$, defined such that $P_{t_1}(\mu)$ is the law of the solution of the McKean-Vlasov SDE \eqref{MVSDE} at time $t_1$ when the initial law is $\mu$. Similarly, define $P_{t_2}:\mathcal{P}(\mathbb{R}^d)\rightarrow\mathcal{P}(\mathbb{R}^d)$.
Following \eqref{eq:contraction}, each of these functionals admits a unique fixed point. We will denote these fixed points by $\smash{\bar{\mu}_{i}\in\mathcal{P}_2(\mathbb{R}^d)}$. Note that, by homogeneity of \eqref{MVSDE}, that $P_{nt_i}(\bar{\mu}_{i}) = \bar{\mu}_{i}$ for all $n\in\mathbb{N}$. It follows, again using \eqref{eq:contraction}, that
\begin{align}
\mathbb{W}_2(\bar{\mu}_1,{\mu}_{\infty}) &= \mathbb{W}_2(P_{nt_1}(\bar{\mu}_{1},{\mu}_{\infty}) \leq K_{\theta_0,1}e^{-\lambda_{\theta_0} nt_1} \\
\mathbb{W}_2(\bar{\mu}_2,{\mu}_{\infty}) &= \mathbb{W}_2(P_{nt_2}(\bar{\mu}_{2},{\mu}_{\infty})\leq K_{\theta_0,2}e^{-\lambda_{\theta_0} nt_2} 
\end{align}
Taking the limit as $N\rightarrow\infty$, and using the existence of a unique limit, we have that $\bar{\mu}_1 = \bar{\mu}_2 = {\mu}_{\infty}$. Since $t_1$ was chosen arbitrarily, it follows that $\mu_{\infty}$ is the unique invariant measure in $[t_1,\infty)$. It remains to check that $\mu_{\infty}$ that this is also true in $(0,t_1)$. This is indeed the case. Let $(x_t)_{t\geq 0}$ be a solution of \eqref{MVSDE} starting from $\bar{\mu}_1$, and write $\mathcal{L}(x_t)$ for the law of this solution. By construction, we have that $\mathcal{L}(x_{t_1}) = \bar{\mu}_1 = \bar{\mu}_2 = \mu_{\infty}$. Let $(y_t)_{t\geq 0}$ be a solution of \eqref{MVSDE} starting from $\mu_{t_1} = \mu_{\infty}$, with law $\mathcal{L}(y_t) = \mathcal{L}(x_{t+t_1})$.  Since $\mu_{\infty}$ is the unique invariant measure in $[t_1,\infty)$, and $\mathcal{L}(x_{t_1}) = \mu_{\infty}$, we must have $\mathcal{L}(x_{t+t_1}) = \mu_{\infty}$ for $t\in[0,\infty)$. But this implies that $\mathcal{L}(y_t) = \mathcal{L}(x_{t+t_1}) = \mu_{\infty}$ for all $t\in(0,\infty)$. This concludes the proof. 

The bound in \eqref{eq:ergodic_bounds} now follows straightforwardly by setting ${\nu}_0 = {\mu}_{\infty}$ in \eqref{eq:contraction}, and using the fact that $\mu_{\infty}$ is an invariant measure. 
\end{proof}

\begin{proposition}[Propagation of Chaos] \label{prop_chaos}
Let $x^{i}= (x_t^{i})_{t\geq 0}$ be $N$ independent copies of the solutions of \eqref{MVSDE}
driven by independent Brownian motions $w^{i} = (w_t^{i})_{t\geq 0}$. Assume that Conditions \ref{assumption_init} and \ref{assumption1} - \ref{assumption2} hold. Then there exist $0<C_{\theta_0}<\infty$, independent of time, such that  
\begin{equation}
\sup_{t\geq 0} \mathbb{E}_{\theta_0}\left[||x_t^{i,N}-x_t^{i}||^2\right]\leq \frac{C_{\theta_0}}{N}.
\end{equation}
\end{proposition}
\begin{proof}
We follow the approach in \cite[Lemma 2.4.1]{Bustos2008} (see also \cite[Theorem 3.3]{Malrieu2001}, \cite[Theorem 3.1]{Cattiaux2008}). Let $({x}_t)_{t\geq 0}$ be a solution of \eqref{MVSDE}, independent of $(x_t^{i})_{t\geq 0}$ for all $i=1,\dots,N$. Using It\^o's formula, 
and Condition \ref{assumption1}, we have
\begin{align}
\mathbb{E}_{\theta_0}\left[||x_t^{i,N}-x_t^{i}||^2\right] &\leq  -2\alpha_{\theta_0} \int_{0}^t \mathbb{E}_{\theta_0}\left[||x_s^{i}-x_s^{i,N}||^2\right] \mathrm{d}s \\
&+ 2\int_0^t \mathbb{E}_{\theta_0} \bigg[ ||x_s^{i,N} - x_s^{i}||\hspace{.5mm} \underbrace{|| \frac{1}{N} \sum_{j=1}^n \phi(\theta_0,x_s^{i,N},x_s^{j,N}) - \int_{\mathbb{R}^d} \phi(\theta_0,x_s^{i},y)\mu_s(\mathrm{d}y)||}_{\psi_s(N)} \bigg] \hspace{-5mm}
\end{align}
Meanwhile, using Condition \ref{assumption2}, and the triangle inequality, we then have 
\begin{equation}
\psi_s(N) \leq L_{\theta_0,2} \big[ ||x_s^{i,N} - x_{s}^{i} || + \frac{1}{N} \sum_{j=1}^N ||x_s^{j,N} - x_s^{j}||\big] + || \frac{1}{N} \sum_{j=1}^N \phi(\theta_0,x_s^{i},x_s^{j}) - \int_{\mathbb{R}^d} \phi(\theta_0,x_s^{i},y)\mu_s(\mathrm{d}y)||.
\end{equation}
Using this bound, the Hold\"er inequality, and that the particles are identically distributed, it follows that
\begin{align}
&\mathbb{E}_{\theta_0}\left[||x_t^{i,N}-x_t^{i}||^2\right] \leq  -2(\alpha_{\theta_0}-2L_{\theta_0,2}) \int_{0}^t \mathbb{E}_{\theta_0}\left[||x_s^{i}-x_s^{i,N}||^2\right] \mathrm{d}s \\
&\hspace{4mm}+ 2\int_0^t \mathbb{E}_{\theta_0} [ ||x_s^{i,N} - x_s^{i}||^2]^{\frac{1}{2}} \mathbb{E}_{\theta_0}[\big|\big| \frac{1}{N} \sum_{j=1}^N \phi(\theta_0,x_s^{i},x_s^{j}) - \int_{\mathbb{R}^d} \phi(\theta_0,x_s^{i},y)\mu_s(\mathrm{d}y)\big|\big|^2]^{\frac{1}{2}}\mathrm{d}s \label{eqbound}
\end{align}
Let $\xi_s^{j} = \phi(\theta_0,x_s^{i},x_s^{j}) - \int_{\mathbb{R}^d} \phi(\theta_0,x_s^{i},y)\mu_s(\mathrm{d}y)$. If $j\neq k$, one of them is not equal to $i$, and thus $\mathbb{E}_{\theta_0}\left[ \langle \xi_s^{j},\xi_s^{\theta,k}\rangle\right]=0$. 
It follows, using also Condition \ref{assumption2}, and the moment bounds in Proposition \ref{prop_moment_bounds}, that 
\begin{align}
&\mathbb{E}_{\theta_0}[\big|\big| \frac{1}{N} \sum_{j=1}^N \phi(\theta_0,x_s^{i},x_s^{j}) - \int_{\mathbb{R}^d} \phi(\theta_0,x_s^{i},y)\mu_s(\mathrm{d}y)\big|\big|^2] \\
&\leq \frac{1}{N^2} \sum_{j=1}^N \mathbb{E}_{\theta_0}\left[||\phi(\theta_0,x_s^{i},x_s^{j}) - \int_{\mathbb{R}^d} \phi(\theta_0,x_s^{i},y)\mu_s(\mathrm{d}y)||^2\right] \\
& \leq \frac{1}{N^2}\sum_{j=1}^N \mathbb{E}_{\theta_0}\left[\int_{\mathbb{R}^d} ||x_s^{j} - y||^2\mu_s(\mathrm{d}y) \right] \\
&\leq \frac{2}{N^2}\sum_{j=1}^N \mathbb{E}_{\theta_0}\left[||x_s^{j}||^2\right] + \mathbb{E}_{\theta_0}\left[ ||x_s||^2\right] \leq \frac{C_{\theta_0}}{N}.
\end{align}
We thus have, substituting this bound into our previous bound, and taking derivatives, that
\begin{equation}
\frac{\mathrm{d}}{\mathrm{d}t}\mathbb{E}_{\theta_0}\left[||x_t^{i,N}-x_t^{i}||^2\right] \leq  -2(\alpha_{\theta_0}-2L_{\theta_0,2})\mathbb{E}_{\theta_0}\left[||x_t^{i,N}-x_t^{i}||^2\right]  + \frac{C_{\theta_0}}{\sqrt{N}}\mathbb{E}_{\theta_0}\left[||x_t^{i,N}-x_t^{i}||^2\right]^{\frac{1}{2}}
\end{equation}
The result now follows from Gr\"onwall's inequality (see, e.g., \cite[Theorem 3.1]{Malrieu2001}). 
\end{proof}

\begin{proposition}[A Law of Large Numbers, {\cite[Theorem 1.2]{Coppini2020}}, \cite{Oelschlager1984}]
\label{prop_lln}
Assume that Conditions \ref{assumption1}(i) - \ref{assumption2}(i) hold. If $(\mu_0^N)_{N\in\mathbb{N}}$ converge weakly to $\mu_0$, then for all $g\in\mathcal{C}(\mathbb{R}^d)$ and for all $t\geq0$, as $N\rightarrow\infty$,
\begin{equation}
\lim_{N\rightarrow\infty}\left[\frac{1}{N}\sum_{i=1}^N g(x_t^{i,N})\right] \stackrel{\mathbb{P}}{=}\int_{\mathbb{R}^d}g(x)\mu_t(\mathrm{d}x).
\end{equation}
\end{proposition} 
\begin{proof}
See  {\cite[Lemma 9]{Oelschlager1984}}.
\end{proof}

\section{Additional Lemmas for Theorem \ref{offline_theorem1} and Theorem \ref{offline_theorem2}} 
\label{appendixB}
\begin{lemma} \label{theorem1_lemma1}
For all $T\geq 0$, for all $\theta\in \Theta\subseteq \mathbb{R}^p$, 
\begin{equation}
\lim_{N\rightarrow\infty}\sup_{0\leq t\leq T}\frac{1}{N}\sum_{i=1}^N \int_0^t \langle G(\theta,x_s^{i,N},\mu_s^N),\mathrm{d}w_s^{i}\rangle=0
\end{equation}
\end{lemma}
\begin{proof}
For ease of notation,  let us define 
\begin{align}
M_t^{N}(\theta) :=\frac{1}{N}\sum_{i=1}^N \int_0^t \langle G(\theta,x_s^{i,N},\mu_s^N),\mathrm{d}w_s^{i}\rangle.
\end{align}
Now, for all $N\in\mathbb{N}$, and for all $\theta\in\mathbb{R}^p$, $(M_t^{N}(\theta))_{t\geq 0}$ is a zero mean continuous square integrable martingale, with quadratic variation 
\begin{equation}
\big[ M^{N}(\theta)\big]_t  = \frac{1}{N^2} \sum_{i=1}^N \int_0^t ||G(\theta,x_s^{i,N},\mu_s^N)||^2\mathrm{d}s. \label{eqC3}
\end{equation}
It follow, using the elementary fact that $\sup_{x} \left[f(x) - g(x)\right] \geq \sup_{x} f(x) - \sup_{x}g(x)$, and the martingale inequality \cite[page 25]{McKean1969}, that
\begin{align}
\mathbb{P}\bigg(\sup_{0\leq t \leq T}M_t^N(\theta)-\sup_{0\leq t\leq T}\frac{\alpha}{2}\big[ M^N(\theta)\big]_{t}>\beta\bigg)&\leq \mathbb{P}\bigg(\sup_{0\leq t \leq T}\left\{M_t^N(\theta)-\frac{\alpha}{2}\big[ M^N(\theta)\big]_{t} \right\}>\beta\bigg) \hspace{-3mm} \\
&<e^{-\alpha\beta}. 
\end{align}
Thus, substituting \eqref{eqC3} and using symmetry, we have that
\begin{equation}
\mathbb{P}\left(\sup_{0\leq t\leq T} \big| M_t^{N}(\theta) \big|>\beta + \frac{\alpha}{2N^2} \sum_{i=1}^N \int_0^T ||G(\theta,x_s^{i,N},\mu_s^N)||^2\mathrm{d}s \right)<2e^{-\alpha\beta}.
\end{equation}
Let $\alpha = N^{a}$, $\beta = N^{-b}$, for some $0<a<b<1$. Then
\begin{equation}
\mathbb{P}\left(\sup_{0\leq t\leq T} |M_t^{N}(\theta)|>\frac{1}{N^b} + \frac{1}{2N^{1-a}} \frac{1}{N}\sum_{i=1}^N \int_0^T ||G(\theta,x_s^{i,N},\mu_s^N)||^2\mathrm{d}s \right)<2e^{-N^{a-b}}.
\end{equation}
By a repeated application of Proposition \ref{prop_lln} (the McKean-Vlasov Law of Large Numbers), we have that, as $N\rightarrow\infty$,
\begin{equation}
\frac{1}{N}\sum_{i=1}^N \int_0^T ||G(\theta,x_s^{i,N},\mu_s^N)||^2\mathrm{d}s\stackrel{\mathbb{P}}{\longrightarrow} \int_0^T \left[\int_{\mathbb{R}^d}  ||G(\theta,x,\mu_s)||^2 \mu_s(\mathrm{d}x)\right]\mathrm{d}s.
\end{equation}
By definition, Condition \ref{assumption3}(ii), and Proposition \ref{prop_moment_bounds} the limiting function on the RHS is finite and non-random. Moreover, we have that $\sum_{N=1}^{\infty} e^{-N^{a-b}}<\infty$. The Borel-Cantelli Lemma thus implies
\begin{equation}
\lim_{N\rightarrow\infty} \sup_{0\leq t\leq T}M_t^{N}(\theta) = 0.
\end{equation}
\end{proof}

\section{Additional Lemmas for Theorems \ref{theorem1_1}, \ref{theorem1_1_star}, \ref{theorem2_1}, and \ref{theorem2_1_star}} \label{sec:theorem1_lemmas}

\begin{lemma} \label{lemma1}
Assume that Conditions \ref{assumption_init} and \ref{assumption1} - \ref{assumption2} hold. 
Then, for all $N\in\mathbb{N}$, $i=1,\dots,N$, for all $t\geq 0$, and for all $k\geq 1$, there exists $K_{k,\theta_0}>0$ such that 
\begin{align}
\mathbb{E}_{\theta_0}\left[\sup_{0\leq s\leq t}||x_s||^k\right]  \leq K_{k,\theta_0} t^{\frac{1}{2}}~~~\text{and}~~~ 
\mathbb{E}_{\theta_0}\left[\sup_{0\leq s\leq t}||x_s^{i,N}||^k\right] \leq K_{k,\theta_0} t^{\frac{1}{2}}.
\end{align}
\end{lemma}

\begin{proof}
We will prove the first claim (the proof of the second being essentially identical). By It\^o's Lemma, we have 
\begin{align}
||x_t||^{2k} = ||x_0||^{2k} &+ \int_0^t 2k ||x_s||^{2k-2} \langle x_s, B(\theta_0,x_s,\mu_s)\rangle \mathrm{d}s \\
&+\int_0^t k||x_s||^{2k-2} \mathrm{Tr}[I_{d} + (k-2)[x_s^{i}x_s^{j}]_{i,j=1}^d||x_s||^{-2}] \mathrm{d}s \nonumber \\
&+ \int_0^t 2k ||x_s||^{2k-2} \langle \mathrm{d}x_s, \mathrm{d}w_s\rangle \nonumber
\end{align}
It follows, taking the supremum and taking expectations, that
\begin{align}
\mathbb{E}_{\theta_0}\left[\sup_{0\leq s\leq t}||x_t||^{2k}\right] &\leq \mathbb{E}_{\theta_0}\left[||x_0||^{2k}\right] + \underbrace{2k \int_0^t \mathbb{E}_{\theta_0}\left[\left| ||x_s||^{2k-2} \langle x_s, B(\theta_0,x_s,\mu_s)\rangle \right|\right]\mathrm{d}s}_{\Pi^{1}_t} \label{eqa4} \\
&+\underbrace{[kd + 2k(k-1)] \int_0^t \mathbb{E}_{\theta_0}\left[ ||x_s||^{2k-2}\right]\mathrm{d}s}_{\Pi^{2}_t} \\
&+ \underbrace{2k \mathbb{E}_{\theta_0}\left[\sup_{0\leq s \leq t} \int_0^t  ||x_s||^{2k-2} \langle x_s, \mathrm{d}w_s\rangle\right]}_{\Pi^{3}_t} \nonumber
\end{align}
We begin by bounding the first term. Due to Conditions \ref{assumption1} - \ref{assumption2}, there exist $C_{\theta_0,0}, C_{\theta_0,1}>0$ such that 
\begin{align}
\langle x_s, B(\theta,x_s,\mu_s)\rangle
&\leq -(\alpha_{\theta_0} - L_{\theta_0,2}) ||x_s||^2 + C_{\theta_0,0}||x_s|| + L_{\theta_0,2}||x_s||\mathbb{E}_{\theta_0}\left[||x_s||\right] \\
&\leq C_{\theta_0,1} \left[ ||x_s||^2 + ||x_s|| + ||x_s||\mathbb{E}_{\theta_0}\left[||x_s||\right]\right]  \nonumber
\end{align}
It follows straightforwardly, using the moment bounds in Proposition \ref{prop_moment_bounds}, that 
\begin{align}
\Pi_t^{2}&\leq 2k  \int_0^t C_{\theta_0,1} \left[ \mathbb{E}_{\theta_0}\left[||x_s||^{2k}\right] + \mathbb{E}_{\theta_0}\left[||x_s||^{2k-1}\right] + \mathbb{E}_{\theta_0}\left[||x_s||^{2k-2}\right] \mathbb{E}_{\theta_0}\left[||x_s||\right] \right]\mathrm{d}s\leq C_{k,\theta_0,1}t \label{eqa8}
\end{align}
Similarly, for the second term in \eqref{eqa4},  
\begin{align}
\Pi_t^{2}&\leq [kd+2k(k-1)]\int_0^t \mathbb{E}_{\theta_0}\left[||x_s||^{2k-2}\right]\mathrm{d}s\leq C_{k,\theta_0,2}t \label{eqa9}
\end{align}
Finally, for the final term in \eqref{eqa4}, we have
\begin{align}
\Pi_t^{3}
&\leq 2k \mathbb{E}_{\theta_0}\left[\int_0^t ||x_s||^{4k-4}||x_s||^2\mathrm{d}s\right]^{\frac{1}{2}} \leq 2k \left[\int_0^t \mathbb{E}_{\theta_0}[||x_s||^{4k-2}]\mathrm{d}s\right]^{\frac{1}{2}}\leq C_{k,\theta_0,3} t^{\frac{1}{2}}. \label{eqa20}
\end{align}
where we have used the Burkholder-Davis-Gundy inequality, and once more Proposition \ref{prop_moment_bounds}. Combining equations \eqref{eqa4}, \eqref{eqa8}, \eqref{eqa9}, and \eqref{eqa20}, and using the H\"older inequality, the conclusion follows. 
\end{proof}

\begin{lemma} \label{lemmab}
Assume that Conditions \ref{assumption_init}, \ref{assumption1} - \ref{assumption2}, \ref{assumption3} hold. Then, for $k=0,1,2$, there exist $K_{\theta_0}>0$ such that, for all $\theta\in\mathbb{R}^p$, $||\nabla_{\theta}^k\underline{\tilde{\mathcal{L}}}(\theta)||\leq K_{\theta_0}$ and $||\nabla_{\theta}^k\tilde{\mathcal{L}}^{i,N}(\theta)||\leq K_{\theta_0}$.
\end{lemma}
\begin{proof}
Using the definition of $\nabla_{\theta}^k\tilde{\underline{\mathcal{L}}}(\theta)$ (established in Lemma \ref{lemma0}), the polynomial growth of $\nabla_{\theta}^{k}L(\theta,x,\mu,\mu)$ (from Condition \ref{assumption3}), and the finite moments of the invariant measure of the McKean-Vlasov SDE (Proposition \ref{lemma_invariant_moment_bounds}), we have that
\begin{align}
||\nabla_{\theta}^k\tilde{\underline{\mathcal{L}}}(\theta)||
&\leq\int_{\mathbb{R}^d} ||\nabla_{\theta}^k L(\theta,x,\mu_{\infty},\mu_{\infty})||\mu_{\infty}(\mathrm{d}x) \\
&\leq K \int_{\mathbb{R}^d} \big[1+ ||x||^q + [\int_{\mathbb{R}^d} ||y||^q\mu_{\infty}(\mathrm{d}y)\big]\mu_{\infty}(\mathrm{d}y)
\leq K_{\theta_0}. \label{b_139} 
\end{align}
The bound for $\nabla_{\theta}^{k}\tilde{\mathcal{L}}^{i,N}(\theta)$ follows identically, this time using the defintion of $\nabla_{\theta}^{k}\tilde{\mathcal{L}}^{i,N}(\theta)$ (Lemma \ref{lemma0_1}), and the finite moments of the invariant measure of the IPS (Proposition \ref{lemma_invariant_moment_bounds}).
\end{proof}

\subsection{Additional Lemmas for Lemma \ref{lemmaC}}
\begin{lemma} \label{lemma_a3}
Assume that Conditions \ref{assumption_init}  and \ref{assumption1} - \ref{assumption2} hold. 
For all Lipschitz functions $\varphi$, there exists $K_{\theta_0}>0$ such that, for all $t\geq 0$, for all $N\in\mathbb{N}$, 
\begin{equation}
\mathbb{E}_{\theta_0}\left[\bigg|\bigg|\int_{\mathbb{R}^d}\varphi(y)\mu_t(\mathrm{d}y) - \frac{1}{N}\sum_{i=1}^N \varphi(x_t^{i,N})\bigg|\bigg|^2\right]\leq \frac{K_{\theta_0}}{N}
\end{equation}
\end{lemma}

\begin{proof}
Let $x_t^{i}$, $i=1,\dots,N$ denote independent solutions of the McKean-Vlasov SDE \eqref{MVSDE}. 
We then have, using the elementary inequality $||a+b||^2\leq 2||a||^2+2||b||^2$, that 
\begin{align}
\mathbb{E}_{\theta_0}\bigg[\big|\big|\int_{\mathbb{R}^d}\varphi(y)\mu_t(\mathrm{d}y) - \frac{1}{N}\sum_{i=1}^N \varphi(x_t^{i,N})\big|\big|^2\bigg] \leq 
&2\mathbb{E}_{\theta_0}\bigg[\big|\big|\int_{\mathbb{R}^d}\varphi(y)\mu_t(\mathrm{d}y) - \frac{1}{N}\sum_{i=1}^N \varphi(x_t^{i})\big|\big|^2\bigg] \\
&+2\mathbb{E}_{\theta_0}\bigg[\big|\big| \frac{1}{N}\sum_{i=1}^N \left(\varphi(x_t^{i}) -  \varphi(x_t^{i,N})\right)\big|\big|^2\bigg] \nonumber 
\end{align}
For the first term, using independence of $(x_t^{i})^{i=1,\dots,N}_{t\geq 0}$, the Lipschitz property of $\varphi$, and Proposition \ref{prop_moment_bounds} (bounded moments of the McKean-Vlasov SDE), we have that
\begin{align}
\mathbb{E}_{\theta_0}\bigg[\big|\big|\int_{\mathbb{R}^d}\varphi(y)\mu_t(\mathrm{d}y) - \frac{1}{N}\sum_{i=1}^N \varphi(x_t^{i})\big|\big|^2\big] 
&\leq \frac{1}{N}\mathbb{E}_{\theta_0}\left[\big|\big|\varphi(x_t^{1})-\mathbb{E}[\varphi(x_t^{1})]\big|\big|^2\right] \\
&\leq\frac{1}{N}\mathbb{E}_{\theta_0}\left[\big|\big|\varphi(x_t^{1})-\varphi(\mathbb{E}[x_t^{1}])\big|\big|^2\right] 
\leq \frac{K_{\theta_0}}{N}.
\end{align}
Meanwhile, for the second term, using the Cauchy-Schwarz inequality, the Lipschitz property of $\varphi$, and Proposition \ref{prop_chaos} (uniform-in-time propagation of chaos), we obtain
\begin{align}
\mathbb{E}_{\theta_0}\big[\big|\big| \frac{1}{N}\sum_{i=1}^N \big(\varphi(x_t^{i}) -  \varphi(x_t^{i,N})\big)\big|\big|^2\big] 
&\leq \frac{K_{\theta_0}}{N}\sum_{i=1}^N\mathbb{E}_{\theta_0}\big[\big|\big|x_t^{i}-x_t^{i,N}\big|\big|^2\big] \leq 
 \frac{K_{\theta_0}}{N}.
\end{align}
The result follows immediately.
\end{proof}

\begin{lemma} \label{lemma_a3'}
Assume that Conditions \ref{assumption_init} and \ref{assumption1} - \ref{assumption2} hold. Suppose also that $\mu_{0}\in\mathcal{P}_2(\mathbb{R}^d)$. Let $x_{t}^{i}$ denote a solution of the McKean-Vlasov SDE, driven by $w^{i}=(w_t^{i})_{t\geq 0}$. Then, for all Lipschitz functions $\varphi$, there exists $K_{\theta_0}>0$ such that, for all $t\geq 0$, for all $N\in\mathbb{N}$, 
\begin{equation}
\mathbb{E}\left[\bigg|\bigg|\int_{\mathbb{R}^d}\varphi(x_t^{i},y)\mu_t(\mathrm{d}y) - \frac{1}{N}\sum_{i=1}^N \varphi(x_t^{i,N},x_t^{j,N})\bigg|\bigg|^2\right]\leq \frac{K_{\theta_0}}{N}
\end{equation}
\end{lemma}

\begin{proof}
The is an immediate corollary of Lemma \ref{lemma_a3}. Using the H\"older inequality, and that $\varphi$ is Lipschitz, we have
\begin{align}
\big|\big|\int_{\mathbb{R}^d}\varphi(x_t^{i},y)\mu_t(\mathrm{d}y) - \frac{1}{N}\sum_{j=1}^N \varphi(x_t^{i,N},x_t^{j,N})\big|\big|^2 
&\leq 2\big|\big|\int_{\mathbb{R}^d}\varphi(x_t^{i},y)\mu_t(\mathrm{d}y) - \frac{1}{N}\sum_{j=1}^N \varphi(x_t^{i},x_t^{j,N})\big|\big|^2 \\
&~~+ \frac{2K_{\theta_0}}{N}\sum_{j=1}^N \big|\big|x_t^{i} - x_t^{i,N}\big|\big|^2
\end{align}
It follows immediately, as required, that
\begin{align}
&\mathbb{E}\left[\bigg|\bigg|\int_{\mathbb{R}^d}\varphi(x_t^{i},y)\mu_t(\mathrm{d}y) - \frac{1}{N}\sum_{i=1}^N \varphi(x_t^{i,N},x_t^{j,N})\bigg|\bigg|^2\right] \\
&\leq 2\underbrace{\mathbb{E}\left[\big|\big|\int_{\mathbb{R}^d}\varphi(x_t^{i},y)\mu_t(\mathrm{d}y) - \frac{1}{N}\sum_{j=1}^N \varphi(x_t^{i},x_t^{j,N})\big|\big|^2\right]}_{\leq \frac{K_{\theta_0}}{N} \text{ by Lemma \ref{lemma_a3}}}+ \frac{2K_{\theta_0}}{N}\underbrace{\mathbb{E}\left[\sum_{j=1}^N \big|\big|x_t^{i} - x_t^{i,N}\big|\big|^2\right]}_{\leq 
\frac{K_{\theta_0}}{N} \text{ by Proposition \ref{prop_chaos}}}
\leq \frac{K_{\theta_0}}{N}.
\end{align}
where, as elsewhere, we have allowed the value of the constant $K_{\theta_0}$ to increase from line to line.
\end{proof}

\subsection{Additional Lemmas for Lemma \ref{lemmaD}} 
\subsubsection{Main Lemmas}
The lemmas in this section are modified versions of Lemmas 3.1 - 3.5 in \cite{Sirignano2017a}. We provide the proofs of these results in full, appropriately adapted to the current setting.

\begin{lemma} \label{sub_lemma_1}
Assume that Conditions \ref{assumption_init}, \ref{assumption1} - \ref{assumption2}, \ref{assumption3}, and \ref{assumption0} hold. Define, with $\hat{x}^N = (x^{1,N},\dots,x^{N,N})$, 
\begin{equation}
\Gamma_{k,\eta} = \int_{\tau_k}^{\sigma_{k,\eta}} \gamma_s\left(\nabla_{\theta}\hat{L}^{i,N}(\theta_s^{i,N},\hat{x}_s^{N}) - \nabla_{\theta}\tilde{\mathcal{L}}^{i,N}(\theta_s^{i,N})\right)\mathrm{d}s.
\end{equation}
Then, almost surely, $||\Gamma_{k,\eta}||\rightarrow 0$ as $k\rightarrow\infty$.
\end{lemma}

\begin{proof}
Let $\hat{x}^N = (x^{1,N},\dots,x^{N,N})\in(\mathbb{R}^d)^N$. Consider $S^{i,N}(\theta,\hat{x}^N)
=\nabla_{\theta}\hat{L}^{i,N}(\theta,\hat{x}^N) - \nabla_{\theta}\tilde{\mathcal{L}}^{i,N}(\theta).$ We begin by noting that this function is `centred' with respect to the invariant measure $\hat{\mu}_{\infty}(\cdot)$,
using the definition of $\nabla_{\theta}\tilde{\mathcal{L}}^{i,N}(\cdot)$ from Lemma \ref{lemma0_1}. In addition, $S^{i,N}(\theta,\hat{x})\in\mathcal{C}^{2,\alpha}(\mathbb{R}^p,(\mathbb{R}^d)^N)$ and, due to Condition \ref{assumption3}, there exist positive constants $q,K>0$ such that, for $j=0,1,2$, $\smash{| \partial_{\theta}^{j} S^{i,N}(\theta,\hat{x}^N)|\leq K(1+||x_i||^q + \frac{1}{N}\sum_{j=1}^N ||x_j||^q)}$. Thus, the function $S^{i,N}:\mathbb{R}^p\times (\mathbb{R}^d)^N\rightarrow\mathbb{R}^{p}$ satisfies the conditions of Lemma \ref{lemma_poisson_3}. It follows that, 
for all $i=1,\dots,N$, the Poisson equation
\begin{equation}
\mathcal{A}_{\hat{x}} v^{i,N}(\theta,\hat{x}^N) = S^{i,N}(\theta,\hat{x}^N)~~~,~~~\int_{(\mathbb{R}^d)^N} v^{i,N}(\theta,\hat{x}^N)\hat{\mu}^N_{\infty}(\mathrm{d}\hat{x}^N)=0
\end{equation}
has a unique twice differentiable solution which satisfies $\sum_{j=0}^2 |\frac{\partial^j v^{i,N}}{\partial \theta^{i}}(\theta,\hat{x}^N)| + |\frac{\partial^2 v^{i,N}}{\partial \theta\partial x}(\theta,\hat{x}^N)|\leq K(1+||x^{i,N}||^q + \frac{1}{N}\sum_{j=1}^N ||x^{j,N}||^q)$.
Let $u^{i,N}(t,\theta,\hat{x}^N) = \gamma_tv^{i,N}(\theta,\hat{x}^N)$. Applying Ito's formula to each component of this vector-valued function, we obtain, for $l=1,\dots,p$, 
\begin{align}
u^{i,N}_l(t_2,\theta^{i,N}_{t_2},\hat{x}^N_{t_2}) - u^{i,N}_l(t_1,\theta^{i,N}_{t_1},\hat{x}^N_{t_1}) &= \int_{t_1}^{t_2}\partial_s u^{i,N}_l(s,\theta^{i,N}_{s},\hat{x}_s^N)\mathrm{d}s  \\
&+ \int_{t_1}^{t_2}\mathcal{A}_{\hat{x}}u^{i,N}_l(s,\theta^{i,N}_{s},\hat{x}_s^N)\mathrm{d}s + \int_{t_1}^{t_2}\mathcal{A}_{\theta}u^{i,N}_l(s,\theta^{i,N}_{s},\hat{x}_s^N)\mathrm{d}s \\
&+ \int_{t_1}^{t_2}\gamma_s \mathrm{Tr}\bigg[\nabla_{\theta}\hat{B}^{i,N}(\theta_s^{i,N},\hat{x}_s^N)\partial_{\theta}\partial_{\hat{x}}u^{i,N}_l(s,\theta^{i,N}_{s},\hat{x}_s^N)\bigg]\mathrm{d}s\\
&+\int_{t_1}^{t_2}\partial_{\hat{x}}u^{i,N}_l(s,\theta^{i,N}_{s},\hat{x}_s^N)\cdot \mathrm{d}\hat{w}_s^N \\
&+\int_{t_1}^{t_2}\gamma_s\partial_{\theta}u^{i,N}_l(s,\theta^{i,N}_{s},\hat{x}_s^N)\cdot \nabla_{\theta} \hat{B}^{i,N}(\theta_s^{i,N},\hat{x}_s^N)\mathrm{d}w_s^{i}
\end{align}
where $\mathcal{A}_{\hat{x}}$ and $\mathcal{A}_{\theta}$ are the infinitesimal generators of $\hat{x}^N$ and $\theta$, and we recall from \eqref{big_SDE} that $\hat{w}_t^N = (w_t^{1},\dots,w_t^N)^T$. Rearranging this identity, and also recalling that $v^{i,N}(\theta,\hat{x}^N)$ is the solution of the Poisson equation, we obtain
\begin{align}
\Gamma_{k,\eta} &= \int_{\tau_k}^{\sigma_{k,\eta}}\gamma_s\mathcal{A}_{\hat{x}}v^{i,N}(\theta^{i,N}_{s},\hat{x}_s^N)\mathrm{d}s \\
&= \gamma_{\sigma_{k,\eta}}v^{i,N}(\theta^{i,N}_{\sigma_{k,\eta}},\hat{x}^N_{\sigma_{k,\eta}}) - \gamma_{\tau_k}v^{i,N}(\theta^{i,N}_{\tau_k},\hat{x}^N_{\tau_k})  - \int_{\tau_k}^{\sigma_{k,\eta}}\dot{\gamma}_sv^{i,N}(\theta^{i,N}_{s},\hat{x}_s^N)\mathrm{d}s \nonumber \\
& - \int_{\tau_k}^{\sigma_{k,\eta}}\gamma_s\mathcal{A}_{\theta}v^{i,N}(\theta^{i,N}_{s},\hat{x}_s^N)\mathrm{d}s - \int_{\tau_k}^{\sigma_{k,\eta}}\gamma_s^2 \mathrm{Tr}\big[\nabla_{\theta} \hat{B}^{i,N}(\theta_s^{i,N},\hat{s}_x^{N})\partial_{\theta}\partial_{\hat{x}}v^{i,N}(\theta^{i,N}_{s},\hat{x}_s^N)\big]\mathrm{d}s \nonumber \\
&-\int_{\tau_k}^{\sigma_{k,\eta}}\gamma_s\partial_{\hat{x}}v^{i,N}(\theta^{i,N}_{s},\hat{x}_s^N)\cdot \mathrm{d}\hat{w}_s^N-\int_{\tau_k}^{\sigma_{k,\eta}}\gamma_s^2\partial_{\theta}v^{i,N}(\theta^{i,N}_{s},\hat{x}_s^N)\cdot \nabla_{\theta} \hat{B}^{i,N}(\theta_{s}^{i,N},\hat{x}_s^N)\mathrm{d}w^{i}_s \nonumber
\end{align} 
We now prove the convergence of each term on the right hand side of this equation. First define $J_t^{(1)} = \gamma_t ||v^{i,N}(\theta_t^{i,N},\hat{x}_t^N)||$. 
We have, using the polynomial growth of $v^{i,N}(\theta,\hat{x}^N)$, and Proposition \ref{prop_moment_bounds} (the bounded moments of the IPS), that
\begin{align}
\mathbb{E}_{\theta_0}[|J_t^{(1)}|^2]
\leq K\gamma_t^2\bigg(1+\mathbb{E}_{\theta_0}[||{x}^{i,N}_t||^q]+\frac{1}{N}\sum_{j=1}^N \mathbb{E}_{\theta_0}[||x^{j,N}_t||^q]\bigg)\leq K_{\theta_0}\gamma_t^2.
\end{align}
Applying the Borel-Cantelli argument as in \cite[Appendix B]{Sirignano2020a}, it follows that $J_t^{(1)}$ converges to zero with probability one. We next consider the term
\begin{align}
J_{0,t}^{(2)} &= \int_{0}^t\partial_s\dot{\gamma}_sv^{i,N}(\theta^{i,N}_s,\hat{x}^N_s)\mathrm{d}s+\int_{0}^t\gamma_s\mathcal{A}_{\theta}v^{i,N}(\theta^{i,N}_{s},\hat{x}_s^N)\mathrm{d}s \\
&+\int_0^{t}\gamma_s^2 \mathrm{Tr}\big[\nabla_{\theta} \hat{B}^{i,N}(\theta_{s}^{i,N},\hat{x}_s^N)\partial_{\theta}\partial_{x}v^{i,N}(\theta^{i,N}_{s},\hat{x}_s^N)\big]\mathrm{d}s \nonumber
\end{align}
This term obeys the bound
\begin{align}
\sup_{t>0}\mathbb{E}_{\theta_0}|J_{0,t}^{(2)}|&\leq K\int_0^{\infty} (|\dot{\gamma}_s|+\gamma_s^2)(1+\mathbb{E}_{\theta_0}[||{x}^{i,N}_s||^q]+\frac{1}{N}\sum_{j=1}^N \mathbb{E}_{\theta_0}[||x^{j,N}_s||^q])\mathrm{d}s\\
&\leq K_{\theta_0}\int_0^{\infty}(|\dot{\gamma}_s|+\gamma_s^2)\mathrm{d}s<\infty.
\end{align} 
Here, the first inequality follows from the growth properties of the $v^{i,N}(\theta,\hat{x}^N)$, 
 the second inequality from Proposition \ref{prop_moment_bounds} (the bounded moments of the IPS), and the final inequality from Condition \ref{assumption0} (the properties of the learning rate). It follows that there exists a finite random variable $\smash{J_{0,\infty}^{(2)}}$ such that, with probability one, 
\begin{equation}
J_{0,t}^{(2)}\rightarrow J_{0,\infty}^{(2)}~,~~~\text{as $t\rightarrow\infty$.} \label{eq4_116}
\end{equation}
The last term to consider is the stochastic integral
\begin{equation}
J_{0,t}^{(3)} = \int_{0}^{t}\gamma_s\partial_{\hat{x}}v^{i,N}(\theta^{i,N}_{s},\hat{x}_s^N)\cdot \mathrm{d}\hat{w}_s^N+\int_{\tau_k}^{\sigma_{k,\eta}}\gamma_s^2\partial_{\theta}v^{i,N}(\theta^{i,N}_{s},\hat{x}_s^N)\cdot \nabla_{\theta} \hat{B}^{i,N}(\theta_{s}^{i,N},\hat{x}_s^N)\mathrm{d}w^{i}_s
\end{equation}
In this case, using the BDG inequality, and the same bounds as above, we have
\begin{align}
\mathbb{E}_{\theta_0}\left[|J_{0,t}^{(3)}|^2\right]
&\leq K\int_0^{\infty} (\gamma_s^2+\gamma_s^4) \bigg[1+\mathbb{E}_{\theta_0}[||{x}^{i,N}_s||^q]+\frac{1}{N}\sum_{j=1}^N \mathbb{E}_{\theta_0}[||x^{j,N}_s||^q]\bigg]\mathrm{d}s \\
&\leq K_{\theta_0}\int_0^{\infty}\gamma_s^2\mathrm{d}s<\infty.
\end{align}
Thus, by Doob's martingale convergence theorem, there exists a square integrable random variable $\smash{J_{0,\infty}^{(3)}}$ such that, both almost surely and in $\mathbb{L}^2$,  
\begin{equation}
J_{0,t}^{(3)}\rightarrow J_{0,\infty}^{(3)}~,~~~\text{as $t\rightarrow\infty$.} \label{eq4_119}
\end{equation}
It remains only to observe, combining \eqref{eq4_116} and \eqref{eq4_119}, we have 
\begin{equation}
||\Gamma_{k,\eta}||\leq J_{\sigma_{k,\eta}}^{(1)}+ J_{\tau_k}^{(1)}+J^{(2)}_{\tau_k,\sigma_{k,\eta}}+J^{(3)}_{\tau_k,\sigma_{k,\eta}}\stackrel{k\rightarrow\infty}{\rightarrow}0.
\end{equation}
~
\end{proof}

\begin{lemma} \label{sub_lemma_2}
Assume that Conditions \ref{assumption1} - \ref{assumption2}, \ref{assumption3}, \ref{assumption_init} and \ref{assumption0} hold. Let $\rho>0$ be such that, for a given $\kappa>0$, it is true that $3\rho+\frac{\rho}{4\kappa} = \frac{1}{2L}$, where $L$ denotes the Lipschitz constant of $\nabla_{\theta}\tilde{\mathcal{L}}^{i,N}(\theta)$. For $k$ large enough, and for $\eta>0$ small enough (potentially random, and depending on $k$), one has 
\begin{equation}
\int_{\tau_k}^{\sigma_{k,\eta}}\gamma_s\mathrm{d}s>\rho~~~\text{and, a.s., }~~~\frac{\rho}{2}\leq \int_{\tau_k}^{\sigma_k}\gamma_s\mathrm{d}s\leq \rho.
\end{equation}
\end{lemma}

\begin{proof}
We proceed by contradiction. Let us assume that $\smash{\int_{\tau_k}^{\sigma_{k,\eta}}\gamma_s\mathrm{d}s\leq \rho}$. Choose arbitrary $\varepsilon>0$ such that $\varepsilon\leq \frac{\rho}{8}$. We begin with the observation that, via the It\^o isometry, we have that
\begin{align}
\sup_{t\geq 0}\mathbb{E}_{\theta_0}\big|\big|\int_0^t \gamma_s \frac{\kappa}{||\nabla\tilde{\mathcal{L}}^{i,N}(\theta^{i,N}_{\tau_k})||}\nabla_{\theta}\hat{B}^{i,N}(\theta_s^{i,N},\hat{x}_s^{N})\mathrm{d}w_s^{i}\big|\big|^2
&\leq \int_0^t K\gamma_s^2\left(1+\mathbb{E}_{\theta_0}\left[||\hat{x}_s^N||^{q}\right]\right)\mathrm{d}s<\infty
\end{align}
where, we have used the polynomial growth of $\nabla_{\theta}\hat{B}^{i,N}(\theta,\hat{x})$, 
Proposition \ref{prop_moment_bounds} (the bounded moments of the IPS), and Condition \ref{assumption0} (the properties of the learning rate). 
Thus, by the Doob's martingale convergence theorem, there exists a finite random variable $M$ such that, both almost surely and in $\mathbb{L}^2$, 
\begin{equation}
\int_0^t \gamma_s \frac{\kappa}{||\nabla\tilde{\mathcal{L}}^{i,N}(\theta^{i,N}_{\tau_k})||}\nabla_{\theta}\hat{B}^{i,N}(\theta_s^{i,N},\hat{x}_s^{N})\mathrm{d}w_s^{i}\rightarrow M
\end{equation}
It follows that, for the chosen $\varepsilon>0$, there exists $k$ such that 
\begin{equation}
\int_{\tau_k}^{\sigma_{k,\eta}}\gamma_s \frac{\kappa}{||\nabla\tilde{\mathcal{L}}^{i,N}(\theta^{i,N}_{\tau_k})||}\nabla_{\theta}\hat{B}^{i,N}(\theta_s^{i,N},\hat{x}_s^{N})\mathrm{d}w_s^{i}<\varepsilon. \label{eq391}
\end{equation}
Let us now also assume that, for the given $k$, $\eta$ is small enough such that for all $s\in[\tau_k,\sigma_{k,\eta}]$, we have $||\nabla_{\theta}\tilde{\mathcal{L}}^{i,N}(\theta_s^{i,N})||\leq 3||\nabla_{\theta}\tilde{\mathcal{L}}^{i,N}(\theta_{\tau_k}^{i,N})||$. We can then compute
\begin{align}
~ \label{d99} \\[-2mm] 
||\theta^{i,N}_{\sigma_{k,\eta}}-\theta^{i,N}_{\tau_k}||&= \big|\big|\int_{\tau_k}^{\sigma_{k,\eta}}\gamma_s\nabla_{\theta}\hat{L}^{i,N}(\theta_s^{i,N},\hat{x}_s^N)\mathrm{d}s+\int_{\tau_k}^{\sigma_{k,\eta}}\gamma_s\langle \nabla_{\theta} \hat{B}^{i,N}(\theta_s^{i,N},\hat{x}_s^N),\mathrm{d}w_s^{i}\rangle\big|\big|  \\[3mm]
&\leq 3||\nabla_{\theta}\tilde{\mathcal{L}}^{i,N}(\theta_{\tau_k}^{i,N})||\int_{\tau_k}^{\sigma_{k,\eta}}\gamma_s\mathrm{d}s+ 
\big|\big|\int_{\tau_k}^{\sigma_{k,\eta}}\gamma_s[\nabla_{\theta}\hat{L}^{i,N}(\theta_s^{i,N},\hat{x}_s^N)-\nabla_{\theta}\tilde{\mathcal{L}}^{i,N}(\theta_s^{i,N})]\mathrm{d}s\big|\big| \hspace{-30mm} \\[-1mm]
&+\frac{||\nabla\tilde{\mathcal{L}}^{i,N}(\theta^{i,N}_{\tau_k})||}{\kappa}
\big|\big|\int_{\tau_k}^{\sigma_{k,\eta}}\gamma_s\frac{\kappa}{||\nabla\tilde{\mathcal{L}}^{i,N}(\theta^{i,N}_{\tau_k})||}\langle \nabla_{\theta} B(\theta_s^{i,N},\hat{x}_s^N),\mathrm{d}w_s^{i}\rangle\big|\big| \hspace{-30mm} \\[2.5mm]
&\leq 3||\nabla_{\theta}\tilde{\mathcal{L}}^{i,N}(\theta_{\tau_k}^{i,N})||\rho + \varepsilon + \frac{||\nabla\tilde{\mathcal{L}}^{i,N}(\theta^{i,N}_{\tau_k})||}{\kappa}\varepsilon 
\leq ||\nabla_{\theta}\tilde{\mathcal{L}}^{i,N}(\theta_{\tau_k}^{i,N})||\left[{3\rho}+\frac{\rho}{4\kappa}\right] \label{d105}
\end{align}
where in the penultimate line we have used Lemma \ref{sub_lemma_1} and \eqref{eq391}, and in the final line we have used the fact that our choice of $\varepsilon$ satisfies $\varepsilon\leq \frac{\rho}{8}$. We thus obtain
\begin{equation}
||\theta^{i,N}_{\sigma_{k,\eta}}-\theta^{i,N}_{\tau_k}||\leq ||\nabla_{\theta}\tilde{\mathcal{L}}^{i,N}(\theta_{\tau_k}^{i,N})||\left[{3\rho}+\frac{\rho}{4\kappa}\right]\leq ||\nabla_{\theta}\tilde{\mathcal{L}}^{i,N}(\theta_{\tau_k}^{i,N})||\frac{1}{2L}.\label{d106}
\end{equation}
Thus, using also the definition of the Lipschitz constant $L$, we obtain
\begin{equation}
||\nabla_{\theta}\tilde{\mathcal{L}}^{i,N}(\theta_{\sigma_{k,\eta}}^{i,N})-\nabla_{\theta}\tilde{\mathcal{L}}^{i,N}(\theta_{\tau_k}^{i,N})||\leq L||\theta^{i,N}_{\sigma_{k,\eta}}-\theta^{i,N}_{\tau_k}||\leq \frac{1}{2}||\nabla_{\theta}\tilde{\mathcal{L}}^{i,N}(\theta_{\tau_k}^{i,N})||
\end{equation}
which then yields
\begin{equation}
\frac{1}{2}||\nabla_{\theta}\tilde{\mathcal{L}}^{i,N}(\theta_{\tau_k}^{i,N})||\leq ||\nabla_{\theta}\tilde{\mathcal{L}}^{i,N}(\theta_{\sigma_{k,\eta}}^{i,N})||\leq 2||\nabla_{\theta}\tilde{\mathcal{L}}^{i,N}(\theta_{\tau_k}^{i,N})||. \label{d108}
\end{equation}
But this implies that $\sigma_{k,\eta}\in[\tau_k,\sigma_k]$, which is a contradiction. Thus we do indeed have $\int_{\tau_k}^{\sigma_{k,\eta}}\gamma_s\mathrm{d}s>\rho$. We now turn our attention to the second part of the Lemma. By definition, we have that $\smash{\int_{\tau_k}^{\sigma_k}\gamma_s\mathrm{d}s\leq \rho}$. Thus, it remains only to show that $\smash{\frac{\rho}{2}\leq \int_{\tau_k}^{\sigma_k}\gamma_s\mathrm{d}s}$. From the first part of the Lemma, we have that $
\smash{\int_{\tau_k}^{\sigma_{k,\eta}}\gamma_s\mathrm{d}s>\rho}$. Moreover, for $k$ sufficiently large and $\eta$ sufficiently small, we must have $\smash{\int_{\sigma_k}^{\sigma_{k,\eta}}\gamma_s\mathrm{d}s\leq \frac{\rho}{2}}$. We thus obtain
\begin{equation}
\int_{\tau_k}^{\sigma_k}\gamma_s\mathrm{d}s\geq \rho - \int_{\sigma_k}^{\sigma_{k,\eta}}\gamma_s\mathrm{d}s\geq \rho - \frac{\rho}{2} = \frac{\rho}{2}.
\end{equation}
\end{proof}

\begin{lemma} \label{sub_lemma_3}
Assume that Conditions \ref{assumption1} - \ref{assumption2}, \ref{assumption3}, \ref{assumption_init} and \ref{assumption0} hold. Suppose that there are an infinite number of intervals $[\tau_k,\sigma_k)$. Then there exists a fixed constant $\beta=\beta(\kappa)>0$ such that, for $k$ large enough, almost surely,
\begin{equation}
\tilde{\mathcal{L}}^{i,N}(\theta^{i,N}_{\sigma_k}) - \tilde{\mathcal{L}}^{i,N}(\theta_{\tau_k}^{i,N})\geq \beta.
\end{equation}
\end{lemma}

\begin{proof}
By It\^o's formula, we have that
\begin{align}
\tilde{\mathcal{L}}^{i,N}(\theta^{i,N}_{\sigma_k}) - \tilde{\mathcal{L}}^{i,N}(\theta_{\tau_k}^{i,N}) \label{eq_d117} &= \underbrace{\int_{\tau_k}^{\sigma_k}\gamma_s ||\nabla_{\theta}\tilde{\mathcal{L}}^{i,N}(\theta_s^{i,N})||^2\mathrm{d}s }_{A_{1,k}^{i,N}} \\
&+ \underbrace{\int_{\tau_k}^{\sigma_k}\gamma_s\langle \nabla_{\theta}\tilde{\mathcal{L}}^{i,N}(\theta_s^{i,N}),\nabla_{\theta} \hat{L}^{i,N}(\theta_s^{i,N},\hat{x}_s^N) - \nabla_{\theta}\tilde{\mathcal{L}}^{i,N}(\theta_s^{i,N})\rangle \mathrm{d}s}_{A_{2,k}^{i,N}} \nonumber\\[-2mm]
&+ \underbrace{\int_{\tau_k}^{\sigma_k}\gamma_s\langle \nabla_{\theta}\tilde{\mathcal{L}}^{i,N}(\theta_s^{i,N}), \nabla_{\theta} \hat{B}^{i,N}(\theta_s^{i,N},\hat{x}_s^N)\mathrm{d}w_s^{i}\rangle}_{A_{3,k}^{i,N}} \\
&+ \underbrace{\int_{\tau_k}^{\sigma_k}\frac{1}{2}\gamma_s^2\mathrm{Tr}\left[\nabla_{\theta} \hat{B}^{i,N}(\theta_s^{i,N},\hat{x}_s^N)\nabla_{\theta} \hat{B}^{i,N}(\theta_s^{i,N},\hat{x}_s^N)^T\nabla^2_{\theta}\tilde{\mathcal{L}}^{i,N}(\theta_s^{i,N})\mathrm{d}s\right]}_{A_{4,k}^{i,N}} \hspace{-10mm}\nonumber
\end{align}
We will deal with each of these terms individually. First consider $A_{1,k}^{i,N}$. For this term, we have that
\begin{align}
A_{1,k}^{i,N} 
&= \int_{\tau_k}^{\sigma_k}\gamma_s||\nabla_{\theta}\tilde{\mathcal{L}}^{i,N}(\theta_s^{i,N})||^2\mathrm{d}s \geq \frac{||\nabla_{\theta}\tilde{\mathcal{L}}^{i,N}(\theta_{\tau_k}^{i,N})||^2}{4} \int_{\tau_k}^{\sigma_k}\gamma_s\mathrm{d}s \geq \frac{||\nabla_{\theta}\tilde{\mathcal{L}}^{i,N}(\theta_{\tau_k}^{i,N})||^2}{8} \rho   
\end{align}
where, in the first inequality, we have used the definition of the $\{\tau_k\}_{k\geq 0}$, namely, that $||\nabla_{\theta}\tilde{\mathcal{L}}^{i,N}(\theta_s^{i,N})||\geq \frac{1}{2}||\nabla_{\theta}\tilde{\mathcal{L}}^{i,N}(\theta_{\tau_k}^{i,N})||$ for all $s\in[\tau_k,\sigma_k]$, and in the second inequality we have used Lemma \ref{sub_lemma_2}. 

We now turn our attention to $\smash{A_{2,k}^{i,N}}$. We will handle this term using a very similar to approach to that used in the proof of Lemma \ref{sub_lemma_1}. Let us consider the function $T^{i,N}(\theta,\hat{x}^N) = \langle \nabla_{\theta}\tilde{\mathcal{L}}^{i,N}(\theta),\nabla_{\theta}\hat{L}^{i,N}(\theta,\hat{x}^N) - \nabla_{\theta}\tilde{\mathcal{L}}^{i,N}(\theta)\rangle$.
It is clear that $\smash{T^{i,N}(\theta,\hat{x}^N)\in\mathcal{C}^{2,\alpha}(\mathbb{R}^p,\mathbb{R}^d)}$, and that $||\partial_{\theta}^j T^{i,N}(\theta,\hat{x}^N)|\leq K(1+||x^{i,N}||^q + \frac{1}{N}\sum_{j=1}^N ||x_j||^q)$, for $j=0,1,2$. Moreover, it is straightforward to show that this function satisfies $\int_{(\mathbb{R}^d)^N}T^{i,N}(\theta,\hat{x}^N)\hat{\mu}_{\infty}(\mathrm{d}\hat{x}^N)=0$. Thus, Lemma \ref{lemma_poisson_3}, the Poisson equation
\begin{equation}
\mathcal{A}_{\hat{x}} v^{i,N}(\theta,\hat{x}^N) = T^{i,N}(\theta,\hat{x}^N)~~~,~~~\int_{(\mathbb{R}^d)^N} v^{i,N}(\theta,\hat{x}^N)\mu_{\infty}(\mathrm{d}\hat{x}^N)=0
\end{equation}
has a unique twice differentiable solution which satisfies $\sum_{j=0}^2 |\frac{\partial^j v^{i,N}}{\partial \theta^{i}}(\theta,\hat{x}^N)| + |\frac{\partial^2 v^{i,N}}{\partial \theta\partial x}(\theta,\hat{x}^N)|\leq K[1+||x^{i,N}||^q + \frac{1}{N}\sum_{j=1}^N ||x^{j,N}||^q]$.
Using the same steps as in the proof of Lemma \ref{sub_lemma_1}, we can then prove that, a.s., 
\begin{equation}
\left|\left|\int_{\tau_k}^{\sigma_k}\gamma_s\langle \nabla_{\theta}\tilde{\mathcal{L}}^{i,N}(\theta_s^{i,N}),\nabla_{\theta} \hat{L}^{i,N}(\theta_s^{i,N},\hat{x}_s^N) - \nabla_{\theta}\tilde{\mathcal{L}}^{i,N}(\theta_s^{i,N})\rangle \mathrm{d}s\right|\right|\stackrel{k\rightarrow\infty}{\rightarrow} 0.
\end{equation}
We next consider $\smash{A_{3,k}^{i,N}}$. 
Using It\^o's isometry, Lemma \ref{lemmab} (the bound on the asymptotic log-likelihood of the IPS), the polynomial growth of the function $\smash{\nabla_{\theta} \hat{B}^{i,N}(\theta,\hat{x})}$, 
Proposition \ref{prop_moment_bounds} (the moment bounds for solutions of the IPS) and Condition \ref{assumption0} (the square summability of the learning rate), we have that
\begin{align}
&\sup_{t\geq 0}\mathbb{E}_{\theta_0}\left[\left|\int_{0}^{t}\gamma_s\langle \nabla_{\theta}\tilde{\mathcal{L}}^{i,N}(\theta_s^{i,N}), \nabla_{\theta} \hat{B}^{i,N}(\theta_s^{i,N},\hat{x}_s^N)\mathrm{d}w_s^{i}\rangle\right|^2\right] \\
&\leq K_{\theta_0}\mathbb{E}_{\theta_0} \int_0^{\infty}\gamma_s^2||\nabla_{\theta} \hat{B}^{i,N}(\theta_s^{i,N},\hat{x}_s^N)||^2\mathrm{d}s \\
&\leq K_{\theta_0}\int_0^{\infty} \gamma_s^2 (1+\mathbb{E}_{\theta_0}\left[||{x}_s^{i,N}||^q\right] + \frac{1}{N}\sum_{j=1}^N \mathbb{E}_{\theta_0}\left[||{x}_s^{j,N}||^q\right]\mathrm{d}s<\infty.
\end{align} 
Thus, by Doob's martingale convergence theorem, there exists a finite random variable $A_{3,\infty}^{i,N}$ such that, both almost surely and in $\mathbb{L}^2$, 
\begin{equation}
\int_{0}^{t}\gamma_s\langle \nabla_{\theta}\tilde{\mathcal{L}}^{i,N}(\theta_s^{i,N}), \nabla_{\theta} \hat{B}^{i,N}(\theta_s^{i,N},\hat{x}_s^N)\mathrm{d}w_s^{i}\rangle\rightarrow A_{3,\infty}^{i,N}.
\end{equation}
as $t\rightarrow\infty$. It follows that $A_{3,k}^{i,N}\rightarrow 0$ a.s. as $k\rightarrow\infty$. Finally, we turn our attention to $A_{4,k}^{i,N}$. For this term, we observe that
\begin{align}
&\sup_{t\geq 0}\mathbb{E}_{\theta_0}\left|\left|\int_{0}^{t}\frac{1}{2}\gamma_s^2\mathrm{Tr}\left[\nabla_{\theta} \hat{B}^{i,N}(\theta_s^{i,N},\hat{x}_s^N)\nabla_{\theta} \hat{B}^{i,N}(\theta_s^{i,N},\hat{x}_s^N)^T\nabla^2_{\theta}\tilde{\mathcal{L}}^{i,N}(\theta_s^{i,N})\right]\mathrm{d}s\right|\right| \\
&\leq K_{\theta_0}\int_0^{\infty} \gamma_s^2(1+\mathbb{E}_{\theta_0}\left[||{x}_s^{i,N}||^q\right] + \frac{1}{N}\sum_{j=1}^N \mathbb{E}_{\theta_0}\left[||\hat{x}_s^{j,N}||^q\right])\mathrm{d}s<\infty, 
\end{align}
where, as above, we have used Lemma \ref{lemmab}, the polynomial growth of $\nabla_{\theta}\hat{B}^{i,N}(\theta,\hat{x})$, Proposition \ref{prop_moment_bounds}, and Condition \ref{assumption0}. It follows that the random variable
\begin{equation}
\int_{0}^{\infty}\frac{1}{2}\gamma_s^2\mathrm{Tr}\left[\nabla_{\theta} \hat{B}^{i,N}(\theta_s^{i,N},\hat{x}_s^N)\nabla_{\theta} \hat{B}^{i,N}(\theta_s^{i,N},\hat{x}_s^N)^T\nabla^2_{\theta}\tilde{\mathcal{L}}^{i,N}(\theta_s^{i,N})\right]\mathrm{d}s
\end{equation}
is finite a.s., which in turn implies that there exists a finite random variable $A_{4,\infty}^{i,N}$ such that
\begin{equation}
\int_{0}^{t}\frac{1}{2}\gamma_s^2\mathrm{Tr}\left[\nabla_{\theta} \hat{B}^{i,N}(\theta_s^{i,N},\hat{x}_s^N)\nabla_{\theta} \hat{B}^{i,N}(\theta_s^{i,N},\hat{x}_s^N)^T\nabla^2_{\theta}\tilde{\mathcal{L}}^{i,N}(\theta_s^{i,N})\right]\mathrm{d}s\rightarrow A_{4}^{\infty}
\end{equation}
almost surely. It follows, in particular, that $A_{4,k}^{i,N}\rightarrow 0$ a.s. as $k\rightarrow\infty$. Summarising, we thus have that, for all $\varepsilon>0$, there exists $k$ such that 
\begin{align}
\tilde{\mathcal{L}}^{i,N}(\theta^{i,N}_{\sigma_k}) - \tilde{\mathcal{L}}^{i,N}(\theta_{\tau_k}^{i,N}) 
&\geq A_{1,k}^{i,N} -||A_{2,k}^{i,N}|| - ||A_{3,k}^{i,N}|| - ||A_{4,k}^{i,N}|| 
=\frac{||\nabla_{\theta}\tilde{\mathcal{L}}^{i,N}(\theta_{\tau_k}^{i,N})||^2}{8} \rho - 3 \varepsilon
\end{align}
The claim follows by setting $\varepsilon = \frac{\rho(\kappa)\kappa^2}{32}$ and $\beta=\frac{\rho(\kappa)\kappa^2}{32}$. 
\end{proof}

\begin{lemma}{} \label{sub_lemma_4}
Assume that Conditions \ref{assumption_init}, \ref{assumption1} - \ref{assumption2}, \ref{assumption3}, and \ref{assumption0} hold. Suppose that there are an infinite number of intervals $[\tau_k,\sigma_k)$. Then there exists a fixed constant $0<\beta_1<\beta$ such that, for $k$ large enough, 
\begin{equation}
\tilde{\mathcal{L}}^{i,N}(\theta^{i,N}_{\tau_k}) - \tilde{\mathcal{L}}^{i,N}(\theta_{\sigma_{k-1}}^{i,N})\geq -\beta_1.
\end{equation}
\end{lemma}

\begin{proof}
Using It\^o's formula, we have that
\begin{align}
\tilde{\mathcal{L}}^{i,N}(\theta^{i,N}_{\tau_k}) - \tilde{\mathcal{L}}^{i,N}(\theta_{\sigma_{k-1}}^{i,N}&) \geq \underbrace{\int_{\sigma_{k-1}}^{\tau_k}\gamma_s\langle \nabla_{\theta}\tilde{\mathcal{L}}^{i,N}(\theta_s^{i,N}),\nabla_{\theta} \hat{L}^{i,N}(\theta_s^{i,N},\hat{x}_s^N) - \nabla_{\theta}\tilde{\mathcal{L}}^{i,N}(\theta_s^{i,N})\rangle \mathrm{d}s}_{B_{1,k}^{i,N}} \\[-3mm]
&+ \underbrace{\int_{\sigma_{k-1}}^{\tau_k}\gamma_s\langle \nabla_{\theta}\tilde{\mathcal{L}}^{i,N}(\theta_s^{i,N}), \nabla_{\theta} \hat{B}^{i,N}(\theta_s^{i,N},\hat{x}_s^N)\mathrm{d}w_s^{i}\rangle}_{B_{2,k}^{i,N}} \nonumber\\[-3mm]
&+ \underbrace{\int_{\sigma_{k-1}}^{\tau_k}\frac{1}{2}\gamma_s^2\mathrm{Tr}\left[\nabla_{\theta} \hat{B}^{i,N}(\theta_s^{i,N},\hat{x}_s^N)\nabla_{\theta} \hat{B}^{i,N}(\theta_s^{i,N},\hat{x}_s^N)^T\nabla^2_{\theta}\tilde{\mathcal{L}}^{i,N}(\theta_s^{i,N})\mathrm{d}s\right]}_{B_{3,k}^{i,N}}. \nonumber
\end{align}
Arguing as in the proof of Lemma \ref{sub_lemma_3}, the magnitude of each of the terms converges to zero a.s. as $k\rightarrow\infty$. This is sufficient for the conclusion.
\end{proof}

\subsubsection{Technical Lemmas: On A Related Poisson Equation}

\begin{lemma} \label{lemma_poisson_0}
Assume that Conditions \ref{assumption_init} and \ref{assumption1} - \ref{assumption2} hold. Suppose that, for all $\theta\in\mathbb{R}^p$, $f(\theta,\cdot):\mathbb{R}^d\rightarrow\mathbb{R}$ is locally Lipschitz, and satisfies a polynomial growth condition, viz
\begin{equation}
||f(\theta,{x}) - f(\theta,y)|| \leq K ||x-y|| \left[1+||x||^{q}+||y||^{q}\right].
\end{equation}
Then, for all $\theta\in\mathbb{R}^p$, $x,y\in\mathbb{R}^d$, $t\geq 0$, there exist positive constants $q,K_{\theta_0}>0$ such that 
\begin{align}
\bigg|\mathbb{E}_{\theta_0,x}\left[f(\theta,x_t)\right] - \int_{\mathbb{R}^d} f(\theta,z)\mu_{\infty}(\mathrm{d}z) \bigg| &\leq K_{\theta_0}[1+||x||^{q}]e^{-\lambda t}.\\
\bigg|\mathbb{E}_{\theta_0,x}\left[f(\theta,x_t)\right] - \mathbb{E}_{\theta_0,y}\left[f(\theta,x_t)\right]\bigg| &\leq K_{\theta_0}[1+||x||^{q}+||y||^{q}]e^{-\lambda t}.
\end{align}
Alternatively, suppose that, for all $\theta\in\mathbb{R}^p$, $f(\theta,\cdot):(\mathbb{R}^d)^N\rightarrow \mathbb{R}$ is locally Lipschitz and satisfies a polynomial growth condition in the sense that
\begin{align}
&\bigg|f(\theta,\hat{x}^N) - f(\theta,\hat{y}^N)\bigg|\leq K\bigg[1+||x^{i,N}||^{q} + ||y^{i,N}||^{q} + \frac{1}{N}\sum_{j=1}^N ||x^{j,N}||^{q} +  \frac{1}{N}\sum_{j=1}^N ||y^{j,N}||^{q}\bigg] \\[-.5mm]
&\hspace{41mm}\cdot\bigg[||x^{i,N}-y^{i,N}|| + \bigg(\frac{1}{N}\sum_{j=1}^N||x^{j,N} - y^{j,N}||^2\bigg)^{\frac{1}{2}}\bigg] \nonumber
\end{align}
where $\hat{x}^N = (x^{1,N},\dots,x^{N,N}) \in(\mathbb{R}^d)^N$. Then, for all $i=1,\dots,N$, and for all $\theta\in\mathbb{R}^p$, there exist positive constants $q,K_{\theta_0}>0$ such that
\begin{align}
&\bigg|\mathbb{E}_{\theta_0,\hat{x}^N}[f(\theta,\hat{x}^{N}_t)] - \int_{(\mathbb{R}^d)^N} f(\theta,\hat{z}^N)\hat{\mu}^{N}_{\infty}(\mathrm{d}\hat{z}^N)\bigg|\leq K_{\theta_0}\bigg[1+||x^{i,N}||^{q}+\frac{1}{N}\sum_{j=1}^N ||x^{j,N}||^{q}\bigg]e^{-\lambda t}  \\
&\bigg|\mathbb{E}_{\theta_0,\hat{x}^N}[f(\theta,\hat{x}^{N}_t)] - \mathbb{E}_{\theta_0,\hat{y}^N}[f(\theta,\hat{x}^{N}_t)] \bigg| \leq K_{\theta_0}\bigg[1+||x^{i,N}||^{q}+||y^{i,N}||^{q}+\frac{1}{N}\sum_{j=1}^N (||x^{j,N}||^{q}+||y^{j,N}||^{q})\bigg]e^{-\lambda t}  \hspace{-15mm}
\end{align}
for all $\hat{x}^N,\hat{y}^N\in(\mathbb{R}^d)^N$, and for all $t\geq 0$.
\end{lemma}
\begin{proof}
We will focus on the first statement of the first part of the lemma. Let $\mu,\nu\in\mathcal{P}(\mathbb{R}^d)$, and $\pi\in\Pi(\mu,\nu)$. Then, using the H\"older inequality and the local Lipschitz assumption, it follows that
\begin{align}
&\int_{\mathbb{R}^d\times\mathbb{R}^d} \left|f(\theta,y)-f(\theta,z)\right|\pi(\mathrm{d}y,\mathrm{d}z) \label{d20} \\
&\leq K\bigg[\int_{\mathbb{R}^d\times\mathbb{R}^d} ||y-z||^2\pi(\mathrm{d}y,\mathrm{d}z)\bigg]^{\frac{1}{2}} \bigg[1+\bigg[\int_{\mathbb{R}^d\times\mathbb{R}^d}||y||^{2q}\pi(\mathrm{d}y,\mathrm{d}z)\bigg]^{\frac{1}{2}}+\bigg[\int_{\mathbb{R}^d\times\mathbb{R}^d}||y||^{2q}\pi(\mathrm{d}y,\mathrm{d}z)\bigg]^{\frac{1}{2}}\bigg]  \hspace{-10mm}  \\
&= K\bigg[\int_{\mathbb{R}^d\times\mathbb{R}^d} ||y-z||^2\pi(\mathrm{d}y,\mathrm{d}z)\bigg]^{\frac{1}{2}} \bigg[1+\bigg[\int_{\mathbb{R}^d}||y||^{2q}\mu(\mathrm{d}y)\bigg]^{\frac{1}{2}}+\bigg[\int_{\mathbb{R}^d}||y||^{2q}\mu(\mathrm{d}z)\bigg]^{\frac{1}{2}}\bigg]  \hspace{-10mm} \label{d22} 
\end{align}
Let $(x_t)_{t\geq 0}$ be a solution of the McKean-Vlasov SDE \eqref{MVSDE} starting from $x\in\mathbb{R}^d$. Let $\mu_t^{x}$ denote the law of $x_t$, and let $\mu_{\infty}$ denote the invariant measure. Moreover, let $\pi_t^{x,\infty}$ denote an arbitrary coupling of $\mu_t^{x}$ and $\mu_{\infty}$. It then follows straightforwardly from the previous inequality that
\begin{align}
&\bigg|\mathbb{E}_{\theta_0,x}\left[f(\theta,x_t)\right] - \int_{\mathbb{R}^d}f(\theta,z)\mu_{\infty}(\mathrm{d}z)\bigg|  = \bigg| \int_{\mathbb{R}^d} f(\theta,y)\mu_t^{x}(\mathrm{d}y) - \int_{\mathbb{R}^d} f(\theta,z)\mu_{\infty}(\mathrm{d}z) \bigg|  \\
&\leq \int_{\mathbb{R}^d\times\mathbb{R}^d} |f(\theta,y)-f(\theta,z)|\pi_t^{x,\infty}(\mathrm{d}y,\mathrm{d}z) \hspace{-2mm} \label{eq_b_54}  \\
&\leq K\bigg[\int_{\mathbb{R}^d\times\mathbb{R}^d}||y-z||^2\pi_t^{x,\infty}(\mathrm{d}y,\mathrm{d}z)\bigg]^{\frac{1}{2}}\cdot \bigg[1+\bigg[\int_{\mathbb{R}^d}||y||^{2q}\mu_t^{x}(\mathrm{d}y)\bigg]^{\frac{1}{2}} + \bigg[\int_{\mathbb{R}^d}||z||^{2q}\mu_{\infty}(\mathrm{d}z)\bigg]^{\frac{1}{2}}\bigg] \hspace{-14mm} \nonumber
\end{align}
Finally, using the fact that the chosen coupling was arbitrary, Proposition \ref{prop_moment_bounds} (the moment bounds for the McKean-Vlasov SDE), Proposition \ref{lemma_invariant_moment_bounds} (the bounded moments of the invariant measure of the McKean-Vlasov SDE), and Proposition \ref{prop_invariant_measure} (exponential contractivity of the McKean-Vlasov SDE), the previous inequality implies
\begin{align}
&\bigg|\mathbb{E}_{\theta_0,x}\left[f(\theta,x_t)\right] - \int_{\mathbb{R}^d}f(\theta,z)\mu_{\infty}(\mathrm{d}z)\bigg| \\
&\leq K\mathbb{W}_2(\mu_{t}^{x},\mu_{\infty})\bigg[1+\bigg[\int_{\mathbb{R}^d}||y||^{2q}\mu_t^{x}(\mathrm{d}y) \bigg]^{\frac{1}{2}}+\bigg[\int_{\mathbb{R}^d}||z||^{2q}\mu_{\infty}(\mathrm{d}z)\bigg]^{\frac{1}{2}}\bigg] \label{eq26}  \hspace{-5mm} \\[1mm]
&\leq K\mathbb{W}_2(\mu_0^{x},\mu^{\theta}_{\infty}) e^{-\lambda t} \bigg[1+ \bigg[C_{2q,\theta_0}(||x||^{2q} + 1)\bigg]^{\frac{1}{2}} + K_{2q,\theta_0}^{\frac{1}{2}} \bigg]  \\
&\leq K_{\theta_0} \left[ \int_{\mathbb{R}^d}||x-y||^2 \mu_{\infty}(\mathrm{d}y)\right]^{\frac{1}{2}}\left[ 1 + ||x||^{q}\right] e^{-\lambda t} \\
&\leq K_{\theta_0} \left[ ||x||^2 + \int_{\mathbb{R}^d} ||y||^2 \mu_{\infty}(\mathrm{d}y) \right]^{\frac{1}{2}} \big[1+||x||^{q}\big]e^{-\lambda t} \label{eqD27} \leq K_{\theta_0} \left[ 1 + ||x||^{q}\right] 
\end{align}
where, as elsewhere, we have allowed the value of the constants $q$ and $K_{\theta_0}$ to increase from line to line, and we have explicit the dependence of $K_{\theta_0}$ on the true parameter.
This completes the proof of the first statement of the first part of the lemma. The proof of the second statement is essentially identical, this time considering an arbitrary coupling of ${\mu}_t^{x}$ and ${\mu}_t^y$, and making use of the bound $\mathbb{W}_2({\mu}_t^{x},{\mu}_{t}^y)\leq e^{-\lambda t}\mathbb{W}_2({\mu}_0^x,{\mu}_0^{y})$. Finally, the proof of the second part of the Lemma follows closely the previous proof, now using the statements in Proposition \ref{lemma_invariant_moment_bounds}, Proposition \ref{prop_moment_bounds}, and Proposition \ref{prop_invariant_measure} that are relevant to the IPS.
\end{proof}

\begin{lemma} \label{lemma_poisson_2}
Assume that Conditions \ref{assumption_init} and \ref{assumption1} - \ref{assumption2} hold. Suppose that, for all $\theta\in\mathbb{R}^p$, $f(\theta,\cdot):(\mathbb{R}^d)^N\rightarrow \mathbb{R}$ satisfies a polynomial growth condition of the form
\begin{equation}
||f(\theta,\hat{x}^N)||\leq K\bigg(1+||x^{i,N}||^q+\frac{1}{N}\sum_{j=1}^N ||x^{j,N}||^q\bigg)
\end{equation}
Moreover, suppose that $f(\theta,\cdot)$ is centred, in the sense that $\int_{(\mathbb{R}^d)^N} f(\theta,\hat{x}^N) \mathrm{d}\hat{\mu}_{\infty}^{N}(\mathrm{d}\hat{x}^N)=0$.
Then, for all $N\in\mathbb{N}$, the function 
\begin{equation}
F(\theta,\hat{x}^N)= \int_0^{\infty} \mathbb{E}_{\hat{x}^N,\theta_0}\left[f(\theta,\hat{x}^N_t)\right] \mathrm{d}t \label{poisson_solution}
\end{equation}
is a well defined, continuous function of Sobolev class $\cap_{p\geq 1} W_{p,\mathrm{loc}}^2$, which satisfies the Poisson equation
\begin{equation}
\mathcal{A}_{\hat{x}^N,\theta^{*}} F(\theta,\hat{x}^N) = -f(\theta,\hat{x}^N). \label{poisson_equation_1}
\end{equation}
Moreover, $F$ is centred, in the sense that $\int_{(\mathbb{R}^d)^N} F(\theta,\hat{x}^N)\hat{\mu}_{\infty}^N(\mathrm{d}\hat{x}^N) = 0$, 
and there exist constants $q,K>0$ such that
\begin{align}
|F(\theta,\hat{x}^N)| &\leq K_{\theta_0}\bigg[1+||x^{i,N}||^q + \frac{1}{N}\sum_{j=1}^N ||x^{j,N}||^q\bigg] \label{bounds} \\
|\nabla_{\hat{x}^N} F(\theta,\hat{x})^N| &\leq K_{\theta_0} \bigg[1+||x^{i,N}||^q + \frac{1}{N}\sum_{j=1}^N ||x^{j,N}||^q\bigg] \label{bounds_2}
\end{align} 
\end{lemma}

 \emph{Remark}.  
This is essentially a statement of \cite[Theorem 1]{Pardoux2001}, adapted appropriately to the current statement. In our case, however, since we are interested in the solution of the Poisson equation associated with the generator of the IPS $\hat{x}^N = (x^{1,N},\dots,x^{N,N})\in(\mathbb{R}^d)^N$ for any $N\in\mathbb{N}$, a little care is needed in places to ensure that arguments in the proof of \cite[Theorem 1]{Pardoux2001}, in particular those used to establish that the solution is well defined, and that it satisfies the bounds in \eqref{bounds} - \eqref{bounds_2}, are independent of $N$. Indeed, we are interested in the solution of this Poisson equation for arbitrarily large $N$, since we will later take the limit as $N\rightarrow\infty$. As an example, if we were to use \cite[Theorem 1]{Pardoux2001} directly, we would only have, in place of \eqref{bounds}, the bound $|F(\theta,\hat{x})|\leq K(1+||\hat{x}^N||^q)$, 
which, due to the $||\hat{x}^N||^q$ term, is unbounded in the limit as $N\rightarrow\infty$.

\begin{proof}
We begin by showing that the function $F(\theta,\hat{x}^N)$ is well defined, and that it satisfies \eqref{bounds}. Let $\hat{x}_t^N$ denote a solution of the IPS starting from $\hat{x}^N\in(\mathbb{R}^d)^N$. Let $\hat{\mu}_t^N$ denote the law of $\hat{x}_t^N$. Using the bounds in Lemma \ref{lemma_poisson_0}, and that $f$ is centred, we have
\begin{align}
\bigg| \mathbb{E}_{\theta_0,\hat{x}^N}\big[f(\theta,\hat{x}_t^N)\big] \bigg|&= \bigg| \mathbb{E}_{\theta_0,\hat{x}^N}\big[f(\theta,\hat{x}_t^N)\big] - \int_{(\mathbb{R}^d)^N} f(\theta,\hat{z}^N)\hat{\mu}_{\infty}^N(\mathrm{d}\hat{z}^N) \bigg|  \\
&\leq K_{\theta_0} \bigg[1+||x^{i,N}||^q + \frac{1}{N}\sum_{j=1}^N ||x^{j,N}||^q\bigg]e^{-\lambda t} \label{b_84}
\end{align}
We remark that, crucially, the constants $q,K,\lambda>0$ are independent of $N$. Thus, for all $N\in\mathbb{N}$, the function $F$, as defined in \eqref{poisson_solution}, is absolutely integrable, and thus well defined. Moreover, via the triangle inequality, we immediately obtain the bound in \eqref{bounds}.

The remaining statements in Lemma \ref{lemma_poisson_2} now follow directly from \cite[Theorem 1]{Pardoux2001}. 
In particular, the arguments in the proof of \cite[Theorem 1(b), 1(c), 1(d), 1(f)]{Pardoux2001} show that \eqref{poisson_solution} defines a continuous, centred solution, unique in the class of solutions belonging to $\cap_{p\geq 1}W_{p,\mathrm{loc}}^2$, of the Poisson equation \eqref{poisson_equation_1}. 

Finally, we can obtain the bound in \eqref{bounds_2} using the argument in the proof of \cite[Theorem 1(e)]{Pardoux2001}, replacing the intermediate bound on $||F(\theta,\cdot,\cdot)||$ by \eqref{bounds}, and the intermediate bound on $||f(\theta,\cdot,\cdot)||$ by our condition on the polynomial growth of $f(\cdot)$.\footnote{In the original notation, these are the bounds on $||u||$ and $||Lu||$, respectively. See \cite[pg. 1070]{Pardoux2001}}. This completes the proof.
\end{proof}

\begin{lemma} \label{lemma_poisson_3}
Assume that Conditions \ref{assumption_init} and \ref{assumption1} - \ref{assumption2} hold. Suppose that the function $f(\theta,\hat{x}^N)\in\mathcal{C}^{\alpha,2}(\mathbb{R}^p,(\mathbb{R}^d)^N)$, for some $\alpha>0$, is centred in the same sense as Lemma \ref{lemma_poisson_2}, and satisfies
\begin{equation}
|f(\theta,\hat{x}^N)|+|\partial_{\theta} f(\theta,\hat{x}^N)| + |\partial_{\theta}^2 f(\theta,\hat{x}^N)|\leq K\bigg[1+||x^{i,N}||^q+\frac{1}{N}\sum_{j=1}^N ||x^{j,N}||^q\bigg]
\end{equation}
where $\hat{x} = (x^{1,N},\dots,x^{N,N})$. Then the solution \eqref{poisson_solution} of the Poisson equation \eqref{poisson_equation_1} satisfies $F(\cdot,\hat{x}^N)\in\mathcal{C}^2$ for all $\hat{x}^N\in(\mathbb{R}^d)^N$. Moreover, there exist $q',K'>0$ such that 
\begin{equation}
\sum_{k=0}^2\bigg|\frac{\partial^{k}F}{\partial \theta^{k}}\bigg| + \bigg|\frac{\partial^2 F}{\partial \hat{x}\partial\theta}\bigg|\leq K_{\theta_0}\bigg[1+||x^{i,N}||^q+\frac{1}{N}\sum_{j=1}^N ||x^{j,N}||^q\bigg]
\end{equation}
\end{lemma}

\begin{proof}
The first statement of the Theorem follows directly from \cite[Theorem 3]{Pardoux2003}. Now, observe that, since $\partial_{\theta}^k f$, $k=0,1,2$, satisfies a polynomial growth condition in the required sense, $\partial_{\theta}^kf^{i,N}$ can be shown to satisfy bounds of the form given in Lemma \ref{lemma_poisson_0}. It follows, arguing as in 
\eqref{b_84}, that 
\begin{align}
&\bigg| \mathbb{E}_{\theta_0,\hat{x}^N} \bigg[\frac{\partial^kf}{\partial{\theta}^k}(\theta,\hat{x}_t^N)\bigg]\bigg|\leq K_{\theta_0}\bigg[ 1+||x^{i,N}||^{q}+\frac{1}{N}\sum_{j=1}^N ||x^{j,N}||^q\bigg]  e^{-\lambda t}
\end{align}
We thus have that, allowing the value of the constant $K$ to change from line to line, that
\begin{align}
\bigg|\frac{\partial^{k}F}{\partial \theta^{k}}(\theta,\hat{x}^N)\bigg| &\leq \int_0^{\infty} \bigg| \mathbb{E}_{\theta_0,\hat{x}^N} \bigg[\frac{\partial^kf}{\partial{\theta}^k}(\theta,\hat{x}^N)\bigg]\bigg|\mathrm{d}t \\
&\leq K_{\theta_0} \int_0^{\infty} \bigg[ 1+||x^{i,N}||^{q}+\frac{1}{N}\sum_{j=1}^N ||x^{j,N}||^q\bigg]  e^{-\lambda t}\mathrm{d}t \label{b_92} \\
&\leq K_{\theta_0} \bigg[ 1+||x^{i,N}||^{q}+\frac{1}{N}\sum_{j=1}^N ||x^{j,N}||^{q}\bigg].
\end{align}
Finally, the bound on the mixed derivative follows from \eqref{bounds_2} in Lemma \ref{lemma_poisson_2}.
\end{proof}

\subsection{Additional Lemmas for Theorem \ref{theorem2_2_star}}
\begin{lemma}  \label{lemma_theta_moments}
Assume that Conditions \ref{assumption_init}, \ref{assumption1} - \ref{assumption2}, \ref{assumption3}, \ref{assumption_bound} - \ref{assumption_bound2}, and \ref{assumption0} hold. Then, for all $q\geq 1$, for all $i=1,\dots,N$, for all $N\in\mathbb{N}$, there exists $K$ such that 
\begin{align}
\sup_{t>0}\mathbb{E}_{\theta_0}\left[||\theta_t||^{q}\right]\leq K ~~~\text{and}~~~\sup_{t>0}\mathbb{E}_{\theta_0}\left[||\theta_t^{i,N}||^{q}\right]\leq K.
\end{align}
\end{lemma}

\begin{proof}
This Lemma follows straightforwardly as an extension of \cite[Lemma A.1]{Sirignano2020a}, making use of the appropriate bounds in Conditions \ref{assumption_bound} - \ref{assumption_bound2}.
\end{proof}

\section{Verification of Conditions for the Linear Mean Field Model} \label{app:verification}
In this Appendix, we verify that the conditions of Theorems \ref{offline_theorem1} - \ref{offline_theorem2} (offline parameter estimation), as well as Theorem \ref{theorem2_1}, Theorem \ref{theorem2_1_star}, and Theorem \ref{theorem2_2} (online parameter estimation) hold for the linear one-dimensional mean field model studied in Section \ref{sec:numerics1}. We also demonstrate that one of the conditions of Theorem \ref{theorem2_2_star} (Assumption \ref{assumption4}) is not satisfied. 

\subsection{Main Conditions}

\begin{manualassumption}{A.1}
\emph{
We assume that $x_0\in\mathbb{R}$ and thus this condition is trivially satisfied. We note that this condition would also be satisfied if $x_0\sim \mathcal{N}(\mu,\sigma^2)$ for some $\mu,\sigma\in\mathbb{R}$. 
}
\end{manualassumption}

\begin{manualassumption}{B.1}
\emph{
For this model, we have $b(\theta,\cdot):\mathbb{R}\rightarrow\mathbb{R}$, with $b(\theta,x) = -\theta_1 x$. This function is Lipschitz continuous with constant $\theta_1$, and satisfies $\langle x-x', b(\theta,x) - b(\theta,x')\rangle = \langle x-x',-\theta_1(x-x')\rangle = - \theta_1||x-x'||^2$. This verifies Condition \ref{assumption1}, provided $\theta_1>0$. 
}
\end{manualassumption}

\begin{manualassumption}{B.2}
\emph{
For this model, we have $\phi(\theta,\cdot,\cdot):\mathbb{R}\times\mathbb{R}\rightarrow\mathbb{R}$, with $\phi(\theta,x,y) = -\theta_2(x-y)$. This function is twice differentiable with respect to both of its arguments, and is globally Lipschitz with constant $|\theta_2|$. This verifies Condition \ref{assumption2}, provided $|\theta_2|\leq \frac{1}{2}\theta_1$.
}
\end{manualassumption}

\begin{manualassumption}{C.1}
\emph{
The functions $b:\mathbb{R}\times\mathbb{R}\rightarrow\mathbb{R}$ and $\phi:\mathbb{R}\times\mathbb{R}\times\mathbb{R}$ are infinitely differentiable with respect to all of their arguments. Moreover, we have that $||\nabla_{\theta}^{i}b(\theta,x)||=||\nabla_{\theta}^{i}\phi(\theta,x,y)||=0$ for $i=1,2,3$. Finally, $||b(\theta,x)- b(\theta',x)||\leq ||\theta-\theta'|| ||x||$, $||\phi(\theta,x,y) - \phi(\theta',x,y)||\leq ||\theta-\theta'|| (||x||+||y||)$.This verifies Condition \ref{assumption3}.
}
\end{manualassumption}

\subsection{Offline Parameter Estimation}
\begin{manualassumption}{D.1}
\emph{
For this model, we have $B(\theta,x,\mu_s) 
 =- \theta_1 x - \theta_2(x-\mathbb{E}_{\theta_0}[x_s])$ and thus $G(\theta,x,\mu_s,\mu_s) 
= -(\theta_1-\theta_{1,0})x - (\theta_2-\theta_{2,0})(x-\mathbb{E}_{\theta_0}[x_s])$, where $\theta_0 = (\theta_{1,0},\theta_{2,0})\in\mathbb{R}^2$ denotes the true value of the parameter. We can then compute
\begin{align}
~ \label{l1}  \\[-3mm]
L(\theta,x,\mu_s) 
&= -\frac{1}{2}\left[(\theta_1-\theta_{1,0})x + (\theta_2-\theta_{2,0})(x-\mathbb{E}_{\theta_0}[x_s])\right]^2  \\
&= -\frac{1}{2}\left[(\theta_1-\theta_{1,0})^2x^2 + 2(\theta_1-\theta_{1,0})(\theta_2-\theta_{2,0})x(x-\mathbb{E}_{\theta_0}[x_s])+(\theta_2-\theta_{2,0})^2(x-\mathbb{E}_{\theta_0}[x_s])^2\right] \hspace{-10mm}
\end{align}
and thus
\begin{align}
m_t(\theta) &= \int_0^t \int_{\mathbb{R}^d} L(\theta,x,\mu_s)\mu_s(\mathrm{d}x) \mathrm{d}s \\
&= - \frac{1}{2}\int_0^t (\theta_1-\theta_{1,0})^2\mathbb{E}_{\theta_0}\left[x_s^2\right] + 2(\theta_1-\theta_{1,0})(\theta_2-\theta_{2,0})\mathrm{Var}_{\theta_0}(x_s) + (\theta_{2}-\theta_{2,0})^2\mathrm{Var}_{\theta_0}(x_s)\mathrm{d}s \label{mt_theta} \hspace{-5mm} \\ 
&= - \frac{1}{2}\int_0^t \left[(\theta_1-\theta_{1,0}) + (\theta_2 - \theta_{2,0})\right]^2\mathrm{Var}_{\theta_0}(x_s) + (\theta_1-\theta_{1,0})^2\mathbb{E}_{\theta_0}\left[x_s\right]^2\mathrm{d}s. \label{mt_theta_2}
\end{align}
Let $\mathbb{E}_{\theta_0}\left[x_0\right] = \mu_0$ and $\mathrm{Var}_{\theta_0}(x_0) = \sigma_0^2>0$. It is then relatively straightforward to compute (e.g., \cite{Kasonga1990}), defining $\gamma(\theta_0) = - 2(\theta_{1,0}+\theta_{2,0})$, 
\begin{align}
\mathbb{E}_{\theta_0}\left[x_s\right]^2 &=  \mu_0^2 e^{-2\theta_{1,0} s} \label{G6} \\
\mathrm{Var}_{\theta_0}(x_s)&= \sigma_0^2e^{\gamma(\theta_0) s} + \frac{e^{\gamma(\theta) s} - 1}{\gamma(\theta_0)} \label{G7}  
\end{align}
We thus have $\mathbb{E}_{\theta_0}\left[x_s\right]^2>0$, provided $\mu_0> 0$, and $\mathrm{Var}_{\theta_0}(x_s)>0$. It follows that $m_t(\theta)\leq 0$, with equality if and only if $\theta_1 = \theta_{1,0}$ and $\theta_{2} = \theta_{2,0}$.\footnote{We remark that, if $\mu_0 = 0$, then $\mathbb{E}_{\theta_0}\left[x_s\right]^2=0$ for all $s\geq 0$. Thus, while we certainly still have $m_t(\theta)\leq 0$, we now have equality whenever $(\theta_1 - \theta_{1,0}) + (\theta_2-\theta_{2,0}) = 0$. That is, whenever $\theta_{1}+\theta_{2} = \theta_{1,0} +\theta_{2,0}$. Thus, in this case, $\theta_1$ and $\theta_2$ are no longer jointly identifiable.} That is, equivalently, $\inf_{||\theta-\theta_0||>\delta}m_t(\theta)<0$ a.s. $\forall \delta>0$. This verifies Condition \ref{offline_assumption_2_1}.
}
\end{manualassumption}

\begin{manualassumption}{D.2}
\emph{
For this model, as noted above, we have $B(\theta,x,\mu_s) = -\theta_1 x - \theta_2(x-\mathbb{E}_{\theta_0}[x_s])$, and thus $\nabla_{\theta} B(\theta,x,\mu_s) = [-x, -(x-\mathbb{E}_{\theta_0}[x_s])]$. It follows that
\begin{align}
I_t(\theta_0) &=  \int_0^t \int_{\mathbb{R}^d}\nabla_{\theta} B(\theta_0,x,\mu_s) \otimes \nabla_{\theta} B(\theta_0,x,\mu_s) \mu_s(\mathrm{d}x) \mathrm{d}s \\
&= \int_0^t \begin{pmatrix} 
\mathbb{E}_{\theta_0}[x_s^2] & \mathrm{Var}_{\theta_0}(x_s) \\
 \mathrm{Var}_{\theta_0}(x_s) & \mathrm{Var}_{\theta_0}(x_s) 
\end{pmatrix} 
\mathrm{d}s= \begin{pmatrix} 
D_t(\theta_0) & -C_t(\theta_0)\\
-C_t(\theta_0) & C_t(\theta_0)
\end{pmatrix} 
\end{align}
where, using \eqref{G6} - \eqref{G7}, and integrating, we can obtain $C_t(\theta_0)$ and $D_t(\theta_0)$ explicitly as  
\begin{align}
C_t(\theta_0)
&= \frac{1}{\gamma^2(\theta_0)} (e^{\gamma(\theta_0) t} - 1) - \frac{t}{\gamma(\theta_0)} + \frac{\sigma_0^2}{\gamma(\theta_0)}(e^{\gamma(\theta_0) t} - 1), \\
D_t(\theta_0) 
&=  \frac{1}{\gamma^2(\theta_0)} (e^{\gamma(\theta_0) t} - 1) - \frac{t}{\gamma(\theta_0)} + \frac{\sigma_0^2}{\gamma(\theta_0)}(e^{\gamma(\theta_0) t} - 1) -\frac{\mu_0^2}{2\theta_{1,0}} (e^{-2\theta_{1,0} t}-1),
\end{align}
It remains to show that this matrix is positive-definite, and that for all $\lambda = (\lambda_1,\lambda_2)\in\mathbb{R}^2$, $\lambda^T I_t(\theta_0)\lambda$ is increasing as a function of $t$. Observe that 
\begin{align}
\lambda^T I_t(\theta) \lambda &= \lambda_1^2 D_t(\theta_0) - 2\lambda_1\lambda_2C_t(\theta_0) + \lambda_2^2C_t(\theta_0) \\
&= \lambda_1^2(D_t(\theta_0)-C_t(\theta_0)) + (\lambda_1-\lambda_2)^2C_t(\theta_0)>0
\end{align}
where, to obtain the final inequality, we have use the fact that $C_t(\theta_0)= \int_{0}^t \mathrm{Var}_{\theta_0}(x_s)\mathrm{d}s>0$ and $D_t(\theta_0) - C_t(\theta_0) = \int_{0}^t [\mathbb{E}_{\theta_0}[x_s^2] - \mathrm{Var}_{\theta_0}(x_s)]\mathrm{d}s = \int_0^t \mathbb{E}_{\theta_0}\left[x_s\right]^2\mathrm{d}s>0$ for all $s\geq 0$. Thus, $I_t(\theta)$ is positive definite. Finally, it is straightforward to see that $\lambda^TI_t(\theta_0)\lambda $ is increasing as a function of $t$, and that $I_0(\theta_0)=0$. This verifies Condition \ref{offline_assumption_2_2}.
}
\end{manualassumption}

\subsection{Online Parameter Estimation}
\begin{manualassumption}{E.1}
\emph{
For this model, we have that $L(\theta,x,\mu,\mu) = -\frac{1}{2}[(\theta_1-\theta_{1,0}) x+ (\theta_2-\theta_{2,0})(x-\mathbb{E}_{\mu}[x])]^2$. 
For simplicity, let us focus on the `pure interaction' case, in which $\theta_{1} = \theta_{1,0} = 0$. In this case, we have $L(\theta,x,\mu,\mu) = -\frac{1}{2}(\theta_2-\theta_{2,0})^2(x-\mathbb{E}_{\mu}[x])^2$, and thus $\nabla_{\theta}L(\theta,x,\mu,\mu) = -(\theta_2-\theta_{2,0})(x-\mathbb{E}_{\mu}[x])^2$. 
It is then straightforward to compute
\begin{equation}
\langle\nabla_{\theta}L(\theta,x,\mu,\mu),\theta\rangle = -\theta_2(\theta_2-\theta_{2,0})(x-\mathbb{E}_{\mu}[x])^2 = -\big(1-\frac{\theta_{2,0}}{\theta_{2}}\big)(x-\mathbb{E}_{\mu}[x])^2\theta_{2}^2 .
\end{equation}
It follows that, for all $||\theta_2||\geq ||\theta_{2,0}||$, we have
$\langle\nabla_{\theta}L(\theta,x,\mu,\mu),\theta\rangle\leq - 2(x-\mathbb{E}_{\mu}[x])^2||\theta_2||^2 = -\kappa(x,\mu)||\theta_2||^2$.
}
\end{manualassumption}

\begin{manualassumption}{E.2}
\emph{
For this model, we recall that $\nabla_{\theta} B(\theta,x,\mu) = [-x,-(x-\mathbb{E}_{\mu}[x])]$. We thus have
\begin{equation}
 \tau(\theta,x,\mu) = \big\langle \nabla_{\theta} B(\theta,x,\mu)\nabla_{\theta}B^T(\theta,x,\mu)\frac{\theta}{||\theta||},\frac{\theta}{||\theta||}\big\rangle^{\frac{1}{2}} = \left[x^2 + (x-\mathbb{E}_{\mu}(x))^2\right]^{\frac{1}{2}}.
\end{equation}
Thus implies, in particular, that $|\tau(\theta,x,\mu) - \tau(\theta',x,\mu)| = 0$, which verifies Condition \ref{assumption_bound2}.
}
\end{manualassumption}

\begin{manualassumption}{F.1}
\emph{The unique invariant measure $\hat{\mu}^N_{\infty}(\mathrm{d}\hat{x}^N)$ of the interacting particle system associated with the linear mean field model is multivariate normal, with mean 0 and covariance $\Sigma^N_{\infty}$ given by the solution of the Lyapunov equation (e.g., \cite[Chapter 3]{Pavliotis2014}) $A \Sigma^N_{\infty} + \Sigma^N_{\infty} A^T = I$, 
where $A =  (\theta_{1,0} + \theta_{1,0}) 1_{N}  - \frac{\theta_{2,0}}{N} J_N$, with $I_N$ the $\mathbb{R}^{N\times N}$ identity matrix, and $J_N$ the $\mathbb{R}^{N\times N}$ matrix of ones.
We also have that
\begin{equation}
\hat{L}^{i,N}(\theta,\hat{x}^N) = -\frac{1}{2} \bigg[ (\theta_{1}-\theta_{1,0})x^{i,N} + (\theta_2-\theta_{2,0}) (x^{i,N} - \bar{x}_N)\bigg]^2
\end{equation}
where $\bar{x}_N = \frac{1}{N}\sum_{j=1}^N x^{j,N}$. We can thus obtain the asymptotic log-likelihood as
\begin{align}
\tilde{\mathcal{L}}^{i,N}(\theta) =& -\frac{1}{2}(\theta_1-\theta_{1,0})^2 \underbrace{\int_{\mathbb{R}^N}  (x^{i,N})^2\hat{\mu}^N_{\infty}(\mathrm{d}\hat{x}^N)}_{C_1^{i,N}}  \\
&- (\theta_1 - \theta_{1,0}) (\theta_2 - \theta_{2,0}) \underbrace{\int_{\mathbb{R}^N} x^{i,N} \left(x^{i,N} - \bar{x}_N\right)\hat{\mu}^N_{\infty}(\mathrm{d}\hat{x}^N)}_{C_2^{i,N}} \\[-1mm]
& - \frac{1}{2} (\theta_2 - \theta_{2,0})^2 \underbrace{\int_{\mathbb{R}^N} \left(x^{i,N} - \bar{x}_N\right)^2\hat{\mu}^N_{\infty}(\mathrm{d}\hat{x}^N)}_{C_3^{i,N}}
\end{align}
The Hessian can now be computed as
\begin{equation}
\nabla^2_{\theta}\tilde{\mathcal{L}}^{i,N}(\theta) = \begin{pmatrix} -C_{1}^{i,N} & -C_2^{i,N} \\ -C_2^{i,N} & -C_3^{i,N}
 \end{pmatrix}
 \end{equation}
 In order to establish strong concavity, we are required to show that there exists $\eta>0$ such that, for all $\theta\in\mathbb{R}^2$, $\nabla_{\theta}^2 \tilde{\mathcal{L}}^{i,N}(\theta)\preceq -\eta I$. That is, the matrix $\nabla_{\theta}^2 \tilde{\mathcal{L}}^{i,N}(\theta)+ \eta I$ is negative semi-definite. We will show, equivalently, that $\mathrm{trace}(\nabla_{\theta}^2 \tilde{\mathcal{L}}^{i,N}(\theta)+ \eta I)<0$ and $\mathrm{det}(\nabla_{\theta}^2 \tilde{\mathcal{L}}^{i,N}(\theta)+ \eta I)>0$. First note that 
 \begin{align}
 \mathrm{trace}(\nabla_{\theta}^2 \tilde{\mathcal{L}}^{i,N}(\theta)+ \eta I) &= -C_{1}^{i,N} - C_{3}^{i,N} + 2\eta \leq -C_1^{i,N} + 2\eta.
 \end{align}
Since $\Sigma^N_{\infty}$ is a covariance matrix, we must have $C_{1}^{i,N} = [\Sigma^N_{\infty}]_{ii}>0$. Thus, for any $\eta\in(0,\frac{1}{2}C_{1}^{i,N})$, we do indeed have $\mathrm{trace}(\nabla_{\theta}^2 \tilde{\mathcal{L}}^{i,N}(\theta)+ \eta I)<0$. We now turn our attention to 
\begin{align}
\mathrm{det}(\nabla_{\theta}^2 \tilde{\mathcal{L}}^{i,N}(\theta)+ \eta I) 
&= \eta^2 - \left(C_{1}^{i,N} + C_{2}^{i,N}\right) \eta + \left(C_{1}^{i,N}C_{3}^{i,N} - (C_{2}^{i,N})^2 \right)
\end{align}
The roots of this quadratic are given by
\begin{equation}
\eta_{\pm} = \frac{\left(C_{1}^{i,N} + C_{2}^{i,N}\right) \pm \sqrt{\left(C_{1}^{i,N} + C_{2}^{i,N}\right)^2 - 4\left(C_{1}^{i,N}C_{3}^{i,N} - (C_{2}^{i,N})^2 \right)}}{2}
\end{equation}
These roots are real, since $(C_{1}^{i,N} + C_{2}^{i,N})^2 - 4(C_{1}^{i,N}C_{3}^{i,N} - (C_{2}^{i,N})^2 ) =  (C_{1}^{i,N} - C_{2}^{i,N})^2 + 2(C_2^{i,N})^2>0$. It follows that $\mathrm{det}(\nabla_{\theta}^2 \tilde{\mathcal{L}}^{i,N}(\theta)+ \eta I)>0$ for $\eta\in(-\infty,\eta_{-})\cup(\eta_{+},\infty)$. We claim that $\eta_{-}$ is positive, and thus $\mathrm{det}(\nabla_{\theta}^2 \tilde{\mathcal{L}}^{i,N}(\theta)+ \eta I)>0$ for all $\eta\in(0,\eta_{-})$. To see this, note that 
\begin{align}
C_{1}^{i,N}C_{3}^{i,N} - (C_{2}^{i,N})^2 &= \int_{\mathbb{R}^N}  (x^{i,N})^2\hat{\mu}^N_{\infty}(\mathrm{d}\hat{x}^N)\int_{\mathbb{R}^N} \left(x^{i,N} - \bar{x}_N\right)^2\hat{\mu}^N_{\infty}(\mathrm{d}\hat{x}^N) \\
&- \left(\int_{\mathbb{R}^N} x^{i,N} \left(x^{i,N} - \bar{x}_N\right)\hat{\mu}^N_{\infty}(\mathrm{d}\hat{x}^N)\right)^2>0
\end{align}
by the Cauchy-Schwarz inequality. It follows that $(C_{1}^{i,N} + C_{2}^{i,N})^2 - 4(C_{1}^{i,N}C_{3}^{i,N} - (C_{2}^{i,N})^2 )<(C_{1}^{i,N} + C_{2}^{i,N})^2$, and thus $\eta_{-}$ is indeed positive. This completes the verification of Condition \ref{assumption4''}, since for any $0<\eta<\min(\frac{1}{2}C_{1}^{i,N},\eta_{-})$, we have both $\mathrm{trace}(\nabla_{\theta}^2 \tilde{\mathcal{L}}^{i,N}(\theta)+ \eta I)<0$ and $\mathrm{det}(\nabla_{\theta}^2 \tilde{\mathcal{L}}^{i,N}(\theta)+ \eta I)>0$.
}
\end{manualassumption}

\begin{manualassumption}{F.2}
\emph{
The unique invariant measure $\mu_{\infty}(\mathrm{d}x)$ of the linear mean-field model is normal, with mean $\mu_{\infty} = 0$ and variance $\sigma_{\infty}^2 = \frac{1}{2(\theta_{1,0} + \theta_{2,0})}$ (e.g., \cite{Kasonga1990}). 
It follows that $L(\theta,x,\mu_{\infty},\mu_{\infty}) = -\frac{1}{2}[(\theta_1-\theta_{1,0}) x+ (\theta_2-\theta_{2,0})x]^2$.
We can thus compute the asymptotic contrast function as
\begin{align}
\tilde{\underline{\mathcal{L}}}(\theta) 
&=-\frac{1}{2} [(\theta_1+\theta_2) - (\theta_{1,0}+\theta_{2,0})]^2\int_{\mathbb{R}} x^2 \mu_{\infty}(\mathrm{d}x) = - \frac{[(\theta_1+\theta_2)-(\theta_{1,0}+\theta_{2,0})]^2}{4(\theta_{1,0}+\theta_{2,0})}. 
\end{align}
This function is not strongly concave, and does not admit a unique maximiser. Indeed, it is maximised by any $\theta=(\theta_1,\theta_2)^T$ such that $\theta_1+\theta_2 = \theta_{1,0} + \theta_{2,0}$. This demonstrates that Condition \ref{assumption4} is \emph{not} satisfied by the linear mean field model.
}
\end{manualassumption}

\begin{manualassumption}{G.1 - G.2}
\emph{
Conditions \ref{assumption0} - \ref{assumption0_1} are satisfied by $\gamma_t = \min\{\gamma^0,\gamma^0t^{-\delta}\}$, $\gamma^0\in[0,\infty)$, $\delta\in(\frac{1}{2},1]$.
}
\end{manualassumption}



\bibliographystyle{siamplain}
\bibliography{references}

\end{document}